\documentclass{amsart}

\pdfoutput=1

\usepackage{amsmath, amsthm, amssymb, amsfonts}
\usepackage{enumerate}
\usepackage{latexsym}
\usepackage{mathrsfs}
\usepackage{amscd}
\usepackage{hyperref}
\usepackage{enumitem}
\usepackage{braket}
\usepackage{xy}
\usepackage{graphicx}
\usepackage{thmtools}
\usepackage{thm-restate}
\usepackage{hyperref}
\usepackage{cleveref}

\usepackage{mathtools}

\usepackage{tikz}
\usepackage{tikz-cd}
\usepackage{color}

\newtheorem{theorem}{Theorem}[section]
\newtheorem{definition}[theorem]{Definition}
\newtheorem{lemma}[theorem]{Lemma}
\newtheorem{proposition}[theorem]{Proposition}
\newtheorem{corollary}[theorem]{Corollary}
\newtheorem{remark}[theorem]{Remark}
\newtheorem{example}[theorem]{Example}
\newtheorem{assumption}[theorem]{Assumption}
\newtheorem{conjecture}[theorem]{Conjecture}

\numberwithin{equation}{section}
\numberwithin{figure}{section}

\title[Functors of wrapped Fukaya categories]{Functors of wrapped Fukaya categories from Lagrangian correspondences}

\author[Yuan Gao]{Yuan Gao\textsuperscript{1}}

\address{\textsuperscript{1}Department of Mathematics, Stony Brook University, Stony Brook NY, 11794, USA}

\email{ygao@math.stonybrook.edu}

\begin{document}

\begin{abstract}
	We study wrapped Floer theory on product Liouville manifolds and prove that the wrapped Fukaya categories defined with respect to two different kinds of natural Hamiltonians and almost complex structures are equivalent. The implication is we can do quilted version of wrapped Floer theory, based on which we then construct functors between wrapped Fukaya categories of Liouville manifolds from certain classes of Lagrangian correspondences, by enlarging the wrapped Fukaya categories appropriately, allowing exact cylindrical Lagrangian immersions. For applications, we present a general K\"{u}nneth formula, and also identify the Viterbo restriction functor with the functor associated to the completed graph of embedding of a Liouville sub-domain.
\end{abstract}

\maketitle

\section{Introduction}

	This paper pushes further the discussion in \cite{Gao1} studying Lagrangian correspondences between Liouville manifolds. The goal is to understand the natural functors between wrapped Fukaya categories arising from Lagrangian correspondences, with a motivation in investigating the functoriality properties of wrapped Fukaya categories, as well as the relation to the well-established functoriality properties of the devired categories of coherent sheaves via homological mirror symmetry. \par

\subsection{Floer theory on product manifolds}
	As Lagrangian correspondences are simply Lagrangian submanifolds of the product manifold, it is natural to study first of all the wrapped Fukaya category of a product Liouville manifold and relate it to the wrapped Fukaya categories of both factors, $\mathcal{W}(M)$ and $\mathcal{W}(N)$. To understand such a relation, it is natural to setup up wrapped Floer theory using the split Hamiltonian, which is the sum of admissible Hamiltonians on both factors, i.e. Hamiltonian of the form $\pi_{M}^{*}H_{M} + \pi_{N}^{*}H_{N}$. This together with the product almost complex structure defines a version of wrapped Fukaya category of $M \times N$, which we call the split model of wrapped Fukaya category and denote by $\mathcal{W}^{s}(M \times N)$. \par
	However, as already addressed in \cite{Gao1}, there is a technical issue: the split Hamiltonian is a priori not admissible for wrapped Floer theory, meaning that it might have defined a different category compared to the quadratic Hamiltonian on the product $M \times N$ with respect to a natural cylindrical end. This issue was resolved on the cohomology level there. In this paper, we carry out a chain-level discussion, confirming that the wrapped Fukaya category defined with respect to the split Hamiltonians is quasi-isomorphic to the one defined with respect to a quadratic Hamiltonian. Thus up to canonical quasi-equivalence, there is no ambiguity in mentioning the wrapped Fukaya category of the product manifold. This would simplify many functoriality arguments in wrapped Floer theory, and expectantly in the study of homological mirror symmetry. \par

\begin{theorem}
\label{two models of wrapped Fukaya categories of product manifolds are equivalent}
	The split model $\mathcal{W}^{s}(M \times N)$ is quasi-equivalent to the ordinary wrapped Fukaya category $\mathcal{W}(M \times N)$. \par
	To be more precise, let $\mathbb{L}$ be a countable collection of admissible Lagrangian submanifolds of $M \times N$ and $\mathcal{W}^{s}(\mathbb{L})$ (resp. $\mathcal{W}(\mathbb{L})$) be the full subcategory consisting of objects in $\mathbb{L}$. Then there is a natural quasi-equivalence
\begin{equation}
R: \mathcal{W}^{s}(\mathbb{L}) \to \mathcal{W}(\mathbb{L}).
\end{equation} 
\end{theorem}

	As this $A_{\infty}$-functor is constructed using an action-filtration argument, it will be called the action-restriction functor. \par

\subsection{Functors associated to Lagrangian correspondences}
	Next, we investigate a specific class of Lagrangian correspondences, which are either products or cylindrical with respect to a natural choice of cylindrical end of the product $M^{-} \times N$. Associated to every such Lagrangian correspondence $\mathcal{L} \subset M^{-} \times N$, we would like to construct an $A_{\infty}$-functor from the wrapped Fukaya category of $M$ to that of $N$. Technically, this is not always possible. To overcome this, we include wider class of objects in the wrapped Fukaya category, which are to be introduced in section \ref{A-infinity functors associated to Lagrangian correspondence}. Details are to be presented later, but let us first illustrate the main spirit below. \par
	Using a quilted version wrapped Floer cohomology, we first construct an $A_{\infty}$-functor 
\begin{equation*}
\mathcal{W}(M^{-} \times N) \to (\mathcal{W}(M), \mathcal{W}(N))^{bimod}
\end{equation*}
from the wrapped Fukaya category of the product to the $A_{\infty}$-category of $(\mathcal{W}(M), \mathcal{W}(N))$-bimodules. By purely algebraic considerations, this gives rise to an $A_{\infty}$-functor
\begin{equation*}
\mathcal{W}(M^{-} \times N) \to func(\mathcal{W}(M), \mathcal{W}(N)^{l-mod})
\end{equation*}
by applying the Yoneda embedding on the second factor $\mathcal{W}(N)$. Concretely, to each admissible Lagrangian correspondence $\mathcal{L} \subset M^{-} \times N$, we associate an $A_{\infty}$-bimodule $P_{\mathcal{L}}$ over $(\mathcal{W}(M), \mathcal{W}(N))$ whose value on a pair $(L, L')$ is the quilted wrapped Floer cochain complex $CW^{*}(L, \mathcal{L}, L')$. Regarding $L' \subset N$ as a testing object, with small amount of homological algebra argument we then immediately get an $A_{\infty}$-functor
\begin{equation}
\Phi_{\mathcal{L}}: \mathcal{W}(M) \to \mathcal{W}(N)^{l-mod}.
\end{equation} \par
	Then we study the geometric composition of Lagrangian correspondences, which is in general a Lagrangian immersion. In order to be able to include these immersed Lagrangian submanifolds as objects of the wrapped Fukaya category, in section \ref{the immersed wrapped Fukaya category} we define the immersed wrapped Fukaya category of $M$, denoted by $\mathcal{W}_{im}(M)$, whose objects are unobstructed proper exact cylindrical Lagrangian immersions of $M$ (with transverse self-intersections), together with bounding cochains for them. Floer theory for such Lagrangian immersions turns out to work over $\mathbb{Z}$, so that $\mathcal{W}_{im}(M)$ is an $A_{\infty}$-category over $\mathbb{Z}$ in the usual sense. We also extend this theory to cylindrical Lagrangian immersions with clean self-intersections in section \ref{section: immersed wrapped Floer theory in the case of clean intersections}. \par
	The ordinary wrapped Fukaya category $\mathcal{W}(M)$ can be embedded into $\mathcal{W}_{im}(M)$ as a full sub-category. Also, the $A_{\infty}$-functors $\Phi_{\mathcal{L}}$ and $\Phi$ can be extended to the category of modules over the immersed wrapped Fukaya category
\begin{equation}
\Phi_{\mathcal{L}}: \mathcal{W}(M) \to \mathcal{W}_{im}(N)^{l-mod},
\end{equation}
and
\begin{equation}
\Phi: \mathcal{W}(M^{-} \times N) \to func(\mathcal{W}(M), \mathcal{W}_{im}(N)^{l-mod}).
\end{equation} \par

\begin{theorem}\label{functor associated to a Lagrangian correspondence}
	Let $\mathcal{L} \subset M^{-} \times N$ be an admissible Lagrangian correspondence between Liouville manifolds $M$ and $N$, such that the projection $\mathcal{L} \to N$ is proper. Then under some further generic geometric conditions, namely Assumption \ref{assumption on the geometric composition}, we have:
\begin{enumerate}[label=(\roman*)]

\item For every object $L \in Ob \mathcal{W}(M)$, there is a curved $A_{\infty}$-algebra associated to the geometric composition $L \circ \mathcal{L}$, defined in terms of wrapped Floer theory for Lagrangian immersions.

\item The geometric composition $L \circ \mathcal{L}$ is always unobstructed, with a canonical choice of bounding cochain $b$ for it. Thus $(L \circ \mathcal{L}, b)$ becomes an object of $\mathcal{W}_{im}(N)$. This $b$ is unique such that the next condition is satisfied.

\item There is a natural $A_{\infty}$-functor
\begin{equation}
\Theta_{\mathcal{L}}: \mathcal{W}(M) \to \mathcal{W}_{im}(N),
\end{equation}
which represents $\Phi_{\mathcal{L}}$. On the level of objects, it sends any Lagrangian submanifold $L \in Ob \mathcal{W}(M)$ to the pair $(L \circ \mathcal{L}, b) \in Ob \mathcal{W}_{im}(N)$.

\end{enumerate}

\end{theorem}

	We expect that the above $A_{\infty}$-functor extends over $\mathcal{W}_{im}(M)$ in a canonical and unique way. However, as that is not one of the main subjects of this paper and involves quite a lot of technicality, we will not try to give a proof. \par
	In a more functorial form, the assignment of $A_{\infty}$-functors to Lagrangian correspondences is functorial in the wrapped Fukaya category of the product manifold $M^{-} \times N$: \par

\begin{theorem}
\label{functoriality for Lagrangian correspondences}
	Let $\mathcal{A}(M^{-} \times N)$ be the full $A_{\infty}$-subcategory whose objects are Lagrangian correspondences $\mathcal{L}$ from $M$ to $N$ such that the projection $\mathcal{L} \to N$ is proper, which further satisfies Assumption \ref{assumption on the geometric composition}. Then there is a canonical $A_{\infty}$-functor
\begin{equation} \label{A-infinity functor from product to functor category}
\Theta: \mathcal{A}(M^{-} \times N) \to func(\mathcal{W}(M), \mathcal{W}_{im}(N)),
\end{equation}
such that
\begin{enumerate}[label=(\roman*)]

\item $\Theta$ represents $\Phi$;

\item $\Theta(\mathcal{L}) = \Theta_{\mathcal{L}}$ for every $\mathcal{L} \in Ob \mathcal{A}(M^{-} \times N)$.

\end{enumerate}

\end{theorem}

\subsection{Some applications}

	As a particular application of the construction of functors, we present a well-known and expected (but not fully established) K\"{u}nneth formula for wrapped Fukaya categories, which relate $\mathcal{W}(M \times N)$ to the $A_{\infty}$-tensor product $\mathcal{W}(M) \otimes \mathcal{W}(N)$ in an appropriate sense.
We also show that under the condition that both $\mathcal{W}(M)$ and $\mathcal{W}(N)$ have finite collections of split-generators, there is a quasi-equivalence between $\mathcal{W}(M \times N)$ and $\mathcal{W}(M) \otimes \mathcal{W}(N)$. These results will be discussed in a more formal way in subsection \ref{section: Kunneth formula}. \par

	An especially important instance of Lagrangian correspondence comes from Liouville sub-domains $U_{0} \subset M_{0}$: the graph of the natural inclusion can be completed to a Lagrangian correspondence between $U$ and $M$. This is called the graph correspondence, denoted by $\Gamma$. As an admissible Lagrangian correspondence from $M$ to $U$, it satisfies the hypothesis of Theorem \ref{functor associated to a Lagrangian correspondence}.
Thus we obtain an $A_{\infty}$-functor 
\begin{equation}
\Theta_{\Gamma}: \mathcal{W}(M) \to \mathcal{W}_{im}(U).
\end{equation}
There is a full sub-category on which this functor takes a simpler form.
First, there is the full sub-category $\mathcal{B}(M)$ consisting of exact cylindrical Lagrangian submanifolds $L$ whose primitive $f$ is locally constant near both $\partial M$ and $\partial U$ (Assumption \ref{strong exactness condition}). Then, we may further consider a full sub-category $\mathcal{B}_{0}(M)$ of $\mathcal{B}(M)$ whose objects satisfy an additional geometric condition (Assumption \ref{invariance assumption}). Restricted to this sub-category, the functor $\Theta_{\Gamma}$ induces a functor
\begin{equation*}
\Theta_{\Gamma}: \mathcal{B}_{0}(M) \to \mathcal{W}(N),
\end{equation*}
whose image lies in the ordinary wrapped Fukaya category consisting of properly embedded exact cylindrical Lagrangian submanifolds, with zero bounding cochains. On the other hand, recall that the Viterbo restriction functor is defined on this sub-category $\mathcal{B}(M)$. More detailed definitions will be given in section \ref{section: sub-domains and the restriction functors}. We shall see that their linear terms agree, when the Viterbo restriction functor is further restricted to $\mathcal{B}_{0}(M)$. \par

\begin{theorem}\label{Viterbo functor as a correspondence functor}
	The $A_{\infty}$-functor $\Theta_{\Gamma}$ associated to the graph correspondence $\Gamma \subset M^{-} \times U$ of the Liouville sub-domain restricts to an $A_{\infty}$-functor on the full sub-category $\mathcal{B}_{0}(M)$
\begin{equation}
\Theta_{\Gamma}: \mathcal{B}_{0}(M) \to \mathcal{W}(U).
\end{equation}
The linear term $\Theta_{\Gamma}^{1}$ is chain homotopic to the linear term $r^{1}$ of the Viterbo restriction functor $r$. \par
\end{theorem}

	It is expected that $\Theta_{\Gamma}$ and $r$ are in fact homotopic as $A_{\infty}$-functors, when restricted to the full sub-category $\mathcal{B}_{0}(M)$. Our current method of proof needs improvement in order to prove this more general statement. Moreover, they should be indeed homotopic on the bigger sub-category $\mathcal{B}(M)$ in appropriate sense, if we can manage to show that the geometric composition with the graph correspondence is Floer-theoretically equivalent to the actual restriction of Lagrangian submanifolds. Such points are to be discussed in \cite{Gao2}. \par
	Going back to the definition of the Viterbo restriction functor, we recall that it is only defined for those Lagrangian submanifolds which satisfy Assumption \ref{strong exactness condition}. However, the restriction functor $\Theta_{\Gamma}$ is defined on the whole wrapped Fukaya category, though the image of a Lagrangian submanifold is not necessarily simply its restriction to the sub-domain. However, it is conjectured that the Viterbo restriction functor can be extended to the whole wrapped Fukaya category via some kind of deformation theory, as stated in Conjecture \ref{conjecture on extension of the Viterbo restriction functor}. It is not completely known but a very interesting question to ask in what cases the extension agrees with the functor $\Theta_{\Gamma}$. This is a question for future research. \par

	The structure of this paper is as follows. Sections 2 to 5 provide fundamental materials for the theory to work: section 2 reviews basic homological algebra regarding $A_{\infty}$-modules and presents a version of $A_{\infty}$-direct limit; section 3 recalls the definition of the wrapped Fukaya category and introduces a new class of objects necessary to establish representability of the functors; section 4 defines the immersed wrapped Fukaya category and section 5 generalizes the story to Lagrangian immersions with clean self-intersections. Section 6 studies Floer theory on product manifolds and proves the two models of wrapped Fukaya categories are equivalent. Section 7 formally sets up the framework of functoriality in Lagrangian correspondences and proves representability, with a discussion on the K\"{u}nneth formula at the end. Section 8 concerns sub-domains and how the Viterbo restriction functor is related to our framework. \par
	
\paragraph{\textit{Acknowledgment.}} This research is part of a project on studying functorial properties of wrapped Fukaya categories, which the author carries out in the graduate school at Stony Brook University. The author is grateful to his advisor Kenji Fukaya for his inspiration on this project, as well as for numerous helpful suggestions and ideas that author learned from him. The author would also like to express thanks to Mohammed Abouzaid for his insight that the theory should work over the integers, as well as to Mark McLean for his explanation and discussion on various constructions and structures related to wrapped Floer cohomology and the Viterbo functor.

\section{Homological algebra preliminaries}

\subsection{$A_{\infty}$-modules and bimodules}
	Most of the results about $A_{\infty}$-categories in this section are well-known, which are included for the purpose of fixing notations and conventions for $A_{\infty}$-categories.
	Let $k$ be a field. We shall work with $A_{\infty}$-categories that are linear over $k$. Let $Ch$ be the dg-category of chain complexes over $k$, regarded as an $A_{\infty}$-category whose higher order operations $m^{k}, k \ge 3$ are all zero. Here chain complexes are not necessarily bounded. \par
	By definition, a non-unital left $A_{\infty}$-module over $\mathcal{A}$ is a non-unital $A_{\infty}$-functor $\mathcal{A} \to Ch^{op}$, where $Ch^{op}$ is the opposite category of $Ch$. All left $A_{\infty}$-modules over $\mathcal{A}$ also form an $A_{\infty}$-category (in fact a dg-category), $\mathcal{A}^{nu-l-mod} = nu-func(\mathcal{A}, Ch^{op})$. \par
	There are also non-unital right $A_{\infty}$-modules over $\mathcal{A}$, which are functors from $\mathcal{A}^{op}$ to $Ch$. These also form a dg-category, $\mathcal{A}^{nu-r-mod} = func(\mathcal{A}^{op}, Ch)$. \par
	When $\mathcal{A}$ comes with strict units, homotopy units or cohomological units, there are also unital versions of left and right $A_{\infty}$-modules:
\begin{equation*}
\mathcal{A}^{l-mod} = func(\mathcal{A}, Ch^{op}),
\end{equation*}
\begin{equation*}
\mathcal{A}^{r-mod} = func(\mathcal{A}^{op}, Ch),
\end{equation*}
as unital $A_{\infty}$-functors. \par

\subsection{Representable modules}
	First, we recall the non-unital version of the Yoneda embedding. The left Yoneda functor is a non-unital $A_{\infty}$-functor
\begin{equation}
\mathfrak{y}_{l}: \mathcal{A} \to \mathcal{A}^{nu-mod}
\end{equation}
which sends an object $Y \in Ob\mathcal{A}$ to its left Yoneda module $\mathcal{Y}^{l} \in \mathcal{A}^{nu-mod}$, which is defined as follows. For objects $X \in Ob\mathcal{A}$,
\begin{equation}
\mathcal{Y}^{l}(X) = hom_{\mathcal{A}}(Y, X),
\end{equation}
and the module structure is given by the $A_{\infty}$-structure maps of $\mathcal{A}$:
\begin{equation}
n_{\mathcal{Y}^{l}}^{d}: \hom_{\mathcal{A}}(X_{d-2}, X_{d-1}) \otimes \cdots \otimes \hom_{\mathcal{A}}(X_{0}, X_{1}) \otimes \mathcal{Y}^{r}(X_{0}) \to \mathcal{Y}^{r}(X_{d-1}),
\end{equation}
\begin{equation}
n_{\mathcal{Y}^{l}}^{d}(a_{d-1}, \cdots, a_{1}, b) = m^{d+1}_{\mathcal{A}}(a_{d-1}, \cdots, a_{1}, b).
\end{equation}
On the level of morphisms, $\mathfrak{y}_{l}^{1}$ assigns to a morphism $c \in hom_{\mathcal{A}}(Y_{0}, Y_{1})$ a pre-module homomorphism:
\begin{equation*}
(\mathfrak{y}_{l}^{1}(c))^{d}: \hom_{\mathcal{A}}(X_{d-2}, X_{d-1}) \otimes \cdots \otimes \hom_{\mathcal{A}}(X_{0}, X_{1}) \otimes \mathcal{Y}_{1}^{l}(X_{0}) \to \mathcal{Y}_{0}^{l}(X_{d-1}),
\end{equation*}
\begin{equation}
(\mathfrak{y}_{l}^{1}(c))^{d}(a_{d-1}, \cdots, a_{1}, b) = m^{d+1}_{\mathcal{A}}(a_{d-1}, \cdots, a_{1}, b, c).
\end{equation} 
Higher order terms $\mathfrak{y}_{l}^{k}$ are defined in by analogous formulas:
\begin{equation*}
(\mathfrak{y}_{l}^{k}(c_{k}, \cdots, c_{1}))^{d}: \hom_{\mathcal{A}}(X_{d-2}, X_{d-1}) \otimes \cdots \otimes \hom_{\mathcal{A}}(X_{0}, X_{1}) \otimes \mathcal{Y}_{k}^{l}(X_{0}) \to \mathcal{Y}_{0}^{l}(X_{d-1}),
\end{equation*}
\begin{equation}
(\mathfrak{y}_{l}^{k}(c_{k}, \cdots, c_{1}))^{d}(a_{d-1}, \cdots, a_{1}, b) = m^{d+k}_{\mathcal{A}}(a_{d-1}, \cdots, a_{1}, b, c_{k}, \cdots, c_{1}).
\end{equation}
	In case $\mathcal{A}$ is cohomologically unital, the image of the Yoneda functor lies in the $A_{\infty}$-subcategory of $\mathcal{A}^{nu-l-mod}$ consisting of c-unital right $\mathcal{A}_{\infty}$-modules over $\mathcal{A}$. We denote this $A_{\infty}$-subcategory by $\mathcal{A}^{l-mod}$. \par
	An important notion is the representability of an $A_{\infty}$-functor, in the sense of \cite{Fukaya1}. We briefly recall it here. \par

\begin{definition}
	A left $A_{\infty}$-module over $\mathcal{A}$, namely an $A_{\infty}$-functor $\mathcal{M}: \mathcal{A} \to Ch^{op}$ is said to be representable, if it is homotopic (or equivalently quasi-isomorphic) to a left Yoneda module $\mathcal{Y}^{l} = \mathfrak{y}_{l}(Y)$ for some object $Y$ of $\mathcal{A}$, as $A_{\infty}$-functors $\mathcal{A} \to Ch^{op}$.
\end{definition}

	There is also a right Yoneda functor
\begin{equation}
\mathfrak{y}_{r}: \mathcal{A} \to \mathcal{A}^{nu-r-mod},
\end{equation}
whose c-unital version becomes
\begin{equation}
\mathfrak{y}_{r}: \mathcal{A} \to \mathcal{A}^{r-mod}.
\end{equation}
Similarly, a right $A_{\infty}$-module over $\mathcal{A}$ is said to be representable, if it is homotopic to a right Yoneda module $\mathcal{Y}^{r} = \mathfrak{y}_{r}(Y)$. \par
	We will see in the following subsection that the notion of representability does not have ambiguity, up to quasi-isomorphism. That is, the representative is unique up to quasi-isomorphism. \par

\subsection{Yoneda lemma}
	The key point related to Yoneda embedding we want to emphasize here is that whenever $\mathcal{A}$ is c-unital, the Yoneda embedding is cohomologically fully faithful, and therefore is an $A_{\infty}$-homotopy equivalence to its image. This was first proved in \cite{Fukaya1} under the assumption that $\mathcal{A}$ is strictly unital, with a new proof given in \cite{Seidel} in case $\mathcal{A}$ is only cohomologically unital. To see this, let $\mathcal{M}$ be any c-unital left $A_{\infty}$-module over $\mathcal{A}$ and consider the following cochain map
\begin{equation}
\lambda: \mathcal{M}(Y) \to hom_{\mathcal{A}^{mod}}(\mathcal{Y}^{l}, \mathcal{M})
\end{equation}
\begin{equation}
\lambda(c)^{d}(a_{d-1}, \cdots, a_{1}, b) = n_{\mathcal{M}}^{d+1}(a_{d-1}, \cdots, a_{1}, b, c).
\end{equation}
Note in fact this definition also makes sense for general non-unital $\mathcal{A}_{\infty}$-modules, but we will only emphasize its importance in the c-unital case. \par

\begin{lemma}
	Then the above cochain map $\lambda$ is a quasi-isomorphism, for any object $Y$ of $\mathcal{A}$.
\end{lemma}
\begin{proof}
	The mapping cone of the cochain map $\lambda$ is the following cochain complex:
\begin{equation} \label{mapping cone of Yoneda embedding}
(\mathcal{M}(Y) \oplus \hom_{\mathcal{A}^{l-mod}}(\mathcal{Y}^{l}, \mathcal{M})[-1], \begin{pmatrix}
n^{1}_{\mathcal{M}} & 0\\
\lambda & -m^{1}_{\mathcal{A}^{l-mod}}
\end{pmatrix}
)[1].
\end{equation}
Define a filtration on this cochain complex by first taking the subcomplex $\hom_{\mathcal{A}^{l-mod}}(\mathcal{Y}^{l}, \mathcal{M})[-1]$, then filtering it by its natural length filtration. Denote $A = H(\mathcal{A}), M = H(\mathcal{M})$. Associated to this filtration there is a spectral sequence whose $E_{1}$-page is
\begin{equation}
E_{1}^{rs} =
\begin{cases}
M^{s}(Y), \text{ if } r=0,\\
\begin{split}
\prod_{X_{0}, \cdots, X_{r-1}} & Hom^{s}_{R}(Hom_{A}(X_{r-2}, X_{r-1}) \otimes \cdots \otimes Hom_{A}(X_{0}, X_{1})\\
&\otimes Hom_{A}(Y, X_{0}), M(X_{r-1})), \text{ if } r>0.
\end{split}
\end{cases}
\end{equation}
The differential $d = d_{1}^{rs}: E_{1}^{rs} \to E_{1}^{r+1, s}$ is given by $d(c)(b) = (-1)^{|b|}(c)$ if $r=0$, and
\begin{equation}
\begin{split}
d(t)(a_{r}, \cdots, a_{1}, b) &= (-1)^{|b|+ *_{r}}t(a_{r}, \cdots, a_{1})b + (-1)^{|b|}t(a_{r}, a_{r-1}, \cdots, a_{1}b)\\
&+ \sum_{n} (-1)^{\Delta}t(a_{r}, \cdots, a_{n+2}a_{n+1}, \cdots, a_{1}, b),
\end{split}
\end{equation}
where $*_{r} = |a_{1}| + \cdots + |a_{r}| - r$, and $\Delta = |a_{n+2}| + \cdots + |a_{r}| + |b| + n + 1 - r$. In the above expression, $a_{i}a_{j}$ is the induced composition in the cohomology category $A$, which is associative, and $ab$ is the induced $A$-module structure on $M$ from the structure of left $A_{\infty}$-module on $\mathcal{M}$. This is the standard bar resolution of the cochain complex \eqref{mapping cone of Yoneda embedding}, which in the presence of cohomological unit of $\mathcal{A}$, admits a contracting homotopy
\begin{equation}
\kappa: E_{1}^{r+1, s} \to E_{1}^{rs}
\end{equation}
\begin{equation}
\kappa(t)(a_{r-1}, \cdots, a_{1}, b) = t(a_{r-1}, \cdots, a_{1}, b, e_{Y}),
\end{equation}
where $e_{Y}$ is a cochain in $\hom_{\mathcal{A}}(Y, Y)$ representing the identity morphism in $H(\hom_{\mathcal{A}}(Y, Y))$. This shows that the spectral sequence degenerates at the $E_{1}$-page, which imples that the cochain complex \eqref{mapping cone of Yoneda embedding} is acyclic, and therefore the cochain map $\lambda$ is a quasi-isomorphism.
\end{proof}

\begin{corollary}
	If $\mathcal{A}$ is c-unital, then the (c-unital) left Yoneda functor
\begin{equation*}
\mathfrak{y}_{l}: \mathcal{A} \to \mathcal{A}^{l-mod}
\end{equation*}
is cohomologically fully faithful.
\end{corollary}

	There is a parallel discussion for right Yoneda functors. That is, the right Yoneda functor
\begin{equation*}
\mathfrak{y}_{r}: \mathcal{A} \to \mathcal{A}^{r-mod}
\end{equation*}
is cohomologically fully faithful. \par

	For this reason, we also call the Yoneda functor the Yoneda embedding. And there is no ambiguity for the notion of representability. This allows us to make the following definition. \par

\begin{definition}
	The $A_{\infty}$-category $\mathcal{A}^{rep-l-mod}$ (resp. $\mathcal{A}^{rep-r-mod}$) of representable left (resp. right) $A_{\infty}$-modules over $\mathcal{A}$, is the full $A_{\infty}$-subcategory of $\mathcal{A}^{l-mod}$ whose objects are representable left (resp. right) $A_{\infty}$-modules. Equivalently, it is the image of the left (resp. right) Yoneda embedding $\mathfrak{y}_{l}$ (resp. $\mathfrak{y}_{r}$) inside $\mathcal{A}^{l-mod}$ (resp. $\mathcal{A}^{r-mod}$).
\end{definition}

\subsection{Bimodules and functors}\label{section: bimodules and functors}
	Let us relate the story of representable $A_{\infty}$-modules to that of $A_{\infty}$-functors. We shall be considering $A_{\infty}$-functors of the following kind
\begin{equation}
\mathcal{F}_{m}: \mathcal{A} \to \mathcal{B}^{l-mod}.
\end{equation}
Such an $A_{\infty}$-functor is called a module-valued functor, which says that it suffices to verify representability on objects. \par

\begin{definition}
	A module-valued functor
\begin{equation*}
\mathcal{F}_{m}: \mathcal{A} \to \mathcal{B}^{l-mod}
\end{equation*}
is said to be representable, if there exists an $A_{\infty}$-functor
\begin{equation}
\mathcal{F}: \mathcal{A} \to \mathcal{B},
\end{equation}
such that $\mathfrak{y}_{l} \circ \mathcal{F}$ is homotopic to $\mathcal{F}_{m}$ as $A_{\infty}$-functors.
\end{definition}

	We have the following concrete criterion for representability of a module-valued functor. \par 

\begin{lemma}\label{criterion for representability in terms of objects}
	A module-valued functor
\begin{equation*}
\mathcal{F}_{m}: \mathcal{A} \to \mathcal{B}^{l-mod}
\end{equation*}
is representable, if and only if $\mathcal{F}_{m}(X)$ is a representable left $\mathcal{B}$-module of any object $X \in Ob\mathcal{A}$.
\end{lemma}
\begin{proof}
	The "only if" part is obvious by definition. \par	
	Now consider the "if" part. Since for every $X \in Ob\mathcal{A}$, the $A_{\infty}$-module $\mathcal{F}_{m}(X)$ over $\mathcal{B}$ is representable, there exists an object $Y = Y(X)$ and an $A_{\infty}$-module homomorphism
\begin{equation}
T_{X}: \mathcal{F}_{m}(X) \to \mathcal{Y}^{l},
\end{equation}
which is a quasi-isomorphism of $A_{\infty}$-modules. We make a choice of a homotopy inverse $K_{X}$ of $T_{X}$ for each $X$. \par

	Let us recall what it means for $\mathcal{F}_{m}$ to be representable. There should exist an $A_{\infty}$-functor
\begin{equation*}
\mathcal{F}: \mathcal{A} \to \mathcal{B},
\end{equation*}
as well as an $A_{\infty}$-natural transformation of degree $1$
\begin{equation*}
T \in \hom_{func(\mathcal{A}, \mathcal{B}^{l-mod})}(\mathcal{F}_{m}, \mathfrak{y}_{l} \circ \mathcal{F}),
\end{equation*}
such that $T$ is a homotopy between $\mathcal{F}_{m}$ and $\mathfrak{y}_{l} \circ \mathcal{F}$. \par
	To define these, we need to pick a homotopy inverse of the Yoneda embedding
\begin{equation*}
\mathfrak{y}_{l}: \mathcal{B} \to \mathcal{B}^{l-mod},
\end{equation*}
when restricted to the image. That is, if we regard the Yoneda embedding as an $A_{\infty}$-functor
\begin{equation*}
\mathfrak{y}_{l}: \mathcal{B} \to \mathcal{B}^{rep-l-mod},
\end{equation*}
it is a quasi-isomorphism, so that we can choose a homotopy inverse,
\begin{equation}
\lambda_{\mathcal{B}}: \mathcal{B}^{rep-l-mod} \to \mathcal{B}.
\end{equation} \par
	We define $\mathcal{F}$ as follows. On objects, $\mathcal{F}(X) = Y = Y(X)$. For $X_{0}, X_{1} \in Ob\mathcal{A}$, we define
\begin{equation}
\mathcal{F}^{1}: \hom_{\mathcal{A}}(X_{0}, X_{1}) \to \hom_{\mathcal{B}}(Y_{0}, Y_{1})
\end{equation}
as follows. For each $a \in \hom_{\mathcal{A}}(X_{0}, X_{1})$, the morphism 
\begin{equation*}
\mathcal{F}_{m}^{1}(a) \in \hom_{\mathcal{B}^{l-mod}}(\mathcal{F}_{m}(X_{0}), \mathcal{F}_{m}(X_{1}))
\end{equation*}
is a $A_{\infty}$-pre-module homomorphism, which can be composed with $K_{X_{0}}$ and $T_{X_{1}}$ in the dg-category $\mathcal{B}^{l-mod}$ by the structure map $m^{2}_{\mathcal{B}^{l-mod}}$ to get an $A_{\infty}$-pre-module homomorphism
\begin{equation}
\mathcal{G}_{m}(a) = m^{2}_{\mathcal{B}^{l-mod}}(T_{X_{1}}, m^{2}_{\mathcal{B}^{l-mod}}(\mathcal{F}_{m}(a), K_{X_{0}})) \in \hom_{\mathcal{B}^{l-mod}}(\mathcal{Y}_{0}^{l}, \mathcal{Y}_{1}^{l}).
\end{equation}
Since $m^{2}_{\mathcal{B}^{l-mod}}$ is associative, there is no ambiguity of this composition, and therefore this is well-defined in a unique way, once we fix a choice of a homotopy inverse $K_{X}$ of $T_{X}$ for every $X \in Ob\mathcal{A}$, as this homotopy inverse is indepedent of $a$. Note that
\begin{equation*}
\mathcal{G}_{m}(a) \in \hom_{\mathcal{B}^{l-mod}}(\mathcal{Y}_{0}^{l}, \mathcal{Y}_{1}^{l})
\end{equation*}
is an $A_{\infty}$-pre-module homomorphism between left Yoneda modules, which lie in the sub-category $\mathcal{B}^{rep-l-mod}$. Thus we can apply $\lambda_{\mathcal{B}}$ to $\mathcal{G}_{m}(a)$ to obtain
\begin{equation}
\mathcal{F}^{1}(a) = \lambda_{\mathcal{B}}^{1}(\mathcal{G}_{m}(a)).
\end{equation}
This defines $\mathcal{F}$ on morphisms. \par
	For higher order terms $\mathcal{F}^{k}$, we again follow the same strategy. That is, we define
\begin{equation}
\mathcal{F}^{k}: \hom_{\mathcal{A}}(X_{k-1}, X_{k}) \otimes \cdots \otimes \hom_{\mathcal{A}}(X_{0}, X_{1}) \to \hom_{\mathcal{B}}(Y_{0}, Y_{d})
\end{equation}
to be the image of the composition of $\mathcal{F}_{m}^{k}$ with $K_{X_{0}}$ and $T_{X_{k}}$ under $\lambda_{\mathcal{B}}$. That is, for $a_{i} \in \hom_{\mathcal{A}}(X_{i-1}, X_{i})$, we define
\begin{equation}
\mathcal{F}^{k}(a_{k}, \cdots, a_{1}) = \lambda_{\mathcal{B}}^{1}(m^{2}_{\mathcal{B}^{l-mod}}(T_{X_{k}}, m^{2}_{\mathcal{B}^{l-mod}}(\mathcal{F}_{m}^{k}(a_{k}, \cdots, a_{1}), K_{X_{0}}))).
\end{equation} \par
	It is a straightforward computation to check that $\mathcal{F} = \{\mathcal{F}^{k}\}_{k=1}^{\infty}$ satisfies the $A_{\infty}$-functor equations. \par
	Then we need to define the homotopy between $\mathfrak{y}_{l} \circ \mathcal{F}$ and $\mathcal{F}_{m}$. But this is clear: we simply take

\end{proof}

	A natural source of module-valued functors is given by $A_{\infty}$-bimodules. Let $\mathcal{P}$ be a left-$\mathcal{B}$ right-$\mathcal{A}$ $A_{\infty}$-bimodule, or simply called an $A_{\infty}$-bimodule over $(\mathcal{A}, \mathcal{B})$. This notation is slightly misleading as we write $\mathcal{A}$ on the left and $\mathcal{B}$ on the right, but we shall keep this convention as it is fitted into Floer theory which will be discussed later on. From $\mathcal{P}$ we can define an $A_{\infty}$-functor
\begin{equation}
\mathcal{F}_{\mathcal{P}}: \mathcal{A} \to \mathcal{B}^{l-mod}
\end{equation}
as follows. For each object $X \in Ob\mathcal{A}$, we set
\begin{equation*}
\mathcal{F}_{\mathcal{P}}(X) = \mathcal{P}(X, \cdot).
\end{equation*}
That is, $\mathcal{F}_{\mathcal{P}}(X)$ is the left-$\mathcal{B}$ module that takes value $\mathcal{P}(X, Y)$ for each object $Y \in Ob\mathcal{B}$, and has module structure maps
\begin{equation}
n^{k}_{\mathcal{F}_{\mathcal{P}}(X)}: \hom_{\mathcal{B}}(Y_{k-1}, Y_{k}) \otimes \hom_{\mathcal{B}}(Y_{0}, Y_{1}) \otimes \mathcal{P}(X, Y_{0}) \to \mathcal{P}(X, Y_{k}),
\end{equation}
\begin{equation}
n^{k}_{\mathcal{F}_{\mathcal{P}}(X)}(b_{k}, \cdots, b_{1}, p) = n^{k, 0}_{\mathcal{P}}(b_{k}, \cdots, b_{1}, p),
\end{equation}
where $n^{k, l}_{\mathcal{P}}$ are the $A_{\infty}$-bimodule structure maps of $\mathcal{P}$. Next we define the action of $\mathcal{F}_{\mathcal{P}}$ on morphism spaces
\begin{equation}
\mathcal{F}_{\mathcal{P}}^{l}: \hom_{\mathcal{A}}(X_{l-1}, X_{l}) \otimes \cdots \otimes \hom_{\mathcal{A}}(X_{0}, X_{1}) \to \hom_{\mathcal{B}^{l-mod}}(\mathcal{F}_{\mathcal{P}}(X_{l}), \mathcal{F}_{\mathcal{P}}(X_{0}))
\end{equation}
by the formula
\begin{equation}
(\mathcal{F}_{\mathcal{P}}^{l}(a_{l}, \cdots, a_{1}))^{k}(b_{k}, \cdots, b_{1}, p) = n^{k, l}_{\mathcal{P}}(b_{k}, \cdots, b_{1}, p, a_{l}, \cdots, a_{1}).
\end{equation}
By the $A_{\infty}$-bimodule equations for $n^{k, l}_{\mathcal{P}}$, it is straightforward to verify: \par

\begin{lemma}
	The multilinear maps $\{\mathcal{F}_{\mathcal{P}}^{l}\}_{l=1}^{\infty}$ form an $A_{\infty}$-functor
\begin{equation*}
\mathcal{F}_{\mathcal{P}}: \mathcal{A} \to \mathcal{B}^{l-mod}.
\end{equation*}
\end{lemma}
	
	All $A_{\infty}$-bimodules over $(\mathcal{A}, \mathcal{B})$ form an $A_{\infty}$-category, denoted by $(\mathcal{A}, \mathcal{B})^{bimod}$. The above construction of a module-valued functor $\mathcal{F}_{\mathcal{P}}$ associated to an $A_{\infty}$-bimodule $\mathcal{P}$ can be made on the level of the whole category $(\mathcal{A}, \mathcal{B})^{bimod}$. \par

\begin{proposition}
	There is a canonical $A_{\infty}$-functor
\begin{equation}
\mathcal{F}: (\mathcal{A}, \mathcal{B})^{bimod} \to func(\mathcal{A}, \mathcal{B}^{l-mod}),
\end{equation}
such that $\mathcal{F}(\mathcal{P}) = \mathcal{F}_{\mathcal{P}}$ for every object $\mathcal{P} \in Ob(\mathcal{A}, \mathcal{B})^{bimod}$.
\end{proposition}

	In general, it is not possible to expect $\mathcal{F}(X)$ is representable for every $X \in Ob\mathcal{A}$. But we might say that it is representable for some of the objects. This leads to the following definition. \par

\begin{definition}\label{definition of representability over a subcategory}
	Let $\mathcal{A}_{0}$ be a full $A_{\infty}$-subcategory of $\mathcal{A}$. We say that 
\begin{equation*}
\mathcal{F}: (\mathcal{A}, \mathcal{B})^{bimod} \to func(\mathcal{A}, \mathcal{B}^{l-mod})
\end{equation*}
is representable on $\mathcal{A}_{0}$, if $\mathcal{F}(X)$ is a representable left $\mathcal{B}$-module for every $X \in Ob\mathcal{A}_{0}$.
\end{definition}

	By the criterion stated in Lemma \ref{criterion for representability in terms of objects}, this definition makes sense. \par

\subsection{Cyclic element and bounding cochain}\label{section: cyclic element and bounding cochain}
	In this subsection, we introduce the notion of a cyclic element, which is key to proving the existence and uniqueness of a bounding cochain under certain assumptions. One origianl formulation due to Fukaya (Proposition 3.5 of \cite{Fukaya2}), works with filtered $A_{\infty}$-algebras and filtered $A_{\infty}$-modules over the Novikov ring. In our case, we consider curved $A_{\infty}$-algebras and $A_{\infty}$-modules over $\mathbb{Z}$. There is in general no intrinsic way of making the method work for arbitrary curved $A_{\infty}$-algebras and $A_{\infty}$-modules, so additional structures are required, to be explained below. \par
	Let $(C, m^{k})$ be a curved $A_{\infty}$-algebra over $\mathbb{Z}$ ($m^{0} \neq 0$), where $C$ is a free $\mathbb{Z}$-module (of finite or infinite rank), and $(D, n^{k})$ a left $A_{\infty}$-module over $(C, m^{k})$, where $D$ is also a free $\mathbb{Z}$-module. We need an additional condition, similar to the gappedness condition for filtered $A_{\infty}$-algebras and filtered $A_{\infty}$-modules. Although we work with usual $A_{\infty}$-algebras and $A_{\infty}$-modules, we still want to have some inductive structures, similar to those for filtered $A_{\infty}$-algebras and filtered $A_{\infty}$-modules. For us, the notion we need is a filtration which satisfies certain analogous conditions. \par

\begin{definition}
\begin{enumerate}[label=(\roman*)]

\item A filtration $F_{C}$ on $(C, m^{k})$ is an $\mathbb{R}$-filtration such that
\begin{equation*}
F_{C}^{\lambda'}C \subset F_{C}^{\lambda}C, \text{ if } \lambda < \lambda',
\end{equation*}
and furthermore,
\begin{equation}
m^{k}(F_{C}^{\lambda_{k}}C \otimes \cdots F_{C}^{\lambda_{1}}C) \subset F_{C}^{\sum_{j = 1}^{k} \lambda_{j}}C.
\end{equation}

\item Given a filtration $F_{C}$ on $(C, m^{k})$, a compatible filtration $F_{D}$ on $(D, n^{k})$ is an $\mathbb{R}$-filtration such that
\begin{equation*}
F_{D}^{\lambda'}D \subset F_{D}^{\lambda}D, \text{ if } \lambda < \lambda',
\end{equation*}
and furthermore,
\begin{equation}
n^{k}(F_{C}^{\lambda_{k}}C \otimes F_{C}^{\lambda_{1}}C \otimes F_{D}^{\lambda'}) \subset F_{D}^{\lambda' + \sum_{j=1}^{k} \lambda_{j}}D.
\end{equation}

\end{enumerate}

\end{definition}
	
\begin{definition}
	A filtration $F_{C}$ is said to be discrete, if there is a discrete subset $\Lambda_{C}$ of $\mathbb{R}$ such that
\begin{equation}
F_{C}^{\lambda} C = F_{C}^{\lambda'} C, \text{ if } [\lambda, \lambda'] \cap \Lambda_{C} = \varnothing,
\end{equation}
for any $\lambda < \lambda'$. A similar definition applies to a compatible filtration $F_{D}$. \par
	A discrete compatible filtration $F_{D}$ is said to be strictly compatible, if $\Lambda_{C} = \Lambda_{D}$.
\end{definition}

	In order to deform a curved $A_{\infty}$-algebra to a non-curved $A_{\infty}$-algebra, we need the notion of a bounding cochain. \par

\begin{definition}
	A bounding cochain for the curved $A_{\infty}$-algebra $(C, m^{k})$ is an element $b$, such that the inhomogeneous Maurer-Cartan equation is satisfied
\begin{equation}\label{MC equation for curved A-infinity algebra}
\sum_{k=0}^{\infty} m^{k}(b, \cdots, b) = 0,
\end{equation}
where the sum stops at a finite stage. That is, there exists $K$ such that for all $k > 0$, $m^{k}(b, \cdots, b) = 0$ and
\begin{equation}\label{MC equation for a nilpotent element}
\sum_{k=0}^{K} m^{k}(b, \cdots, b) = 0.
\end{equation}
In other words, $b$ is assumed to be nilpotent.
\end{definition}
	
	There is a very important class of examples for which nilpotent elements naturally exist. For example, if $C$ as a $\mathbb{Z}$-module has finitely many generators which lie in $F_{C}^{0}$, then any such generator in the strictly positive part of the filtration is nilpotent, because the $A_{\infty}$-structure maps increase the filtration. For geometric applications, we shall see that wrapped Floer cochain spaces have this property. Thus it does not harm to give a name for such a filtration. \par

\begin{definition}
	We say that the filtration $F_{C}$ is bounded above, if $\Lambda_{C}$ is bounded above. Equivalently, there exists $\lambda_{+}$ such that
\begin{equation*}
F_{C}^{\lambda_{+}}C = 0. 
\end{equation*}
\end{definition}

	Suppose that we have chosen a discrete filtration $F_{C}$ on $(C, m^{k})$ as well as a discrete strictly compatible filtration $F_{D}$ on $(D, n^{k})$. For simplicity, we assume that $0 \in \Lambda_{D}$. Now we are going to introduce the key notion. \par

\begin{definition}
	We say that $u \in D$ is a cyclic element, if the following properties are satisfied:
\begin{enumerate}[label=(\roman*)]

\item the map $C \to D$ sending $x$ to $n^{1}(x; u)$ is a filtration-preserving isomorphism of $\mathbb{Z}$-modules, with its inverse also filtration-preserving;

\item $u \in F_{D}^{0}D$, but $n^{0}(u) \notin F_{D}^{0}D$.

\end{enumerate}

\end{definition}

	Note that $n^{0}(F_{D}^{\lambda}D) \subset F_{D}^{\lambda}D$ for every $\lambda$. Thus the second condition means when $n^{0}$ acts on $u$, it strictly increases the filtration. \par

	The following result is proved in \cite{Fukaya2} in the case of filtered $A_{\infty}$-algebras and filtered $A_{\infty}$-modules (Proposition 3.5 of \cite{Fukaya2}) over the Novikov ring. We have a similar result in our case, which yields bounding cochains of "finite type". \par

\begin{lemma} \label{cyclic element and bounding cochain}
	Let $(C, m^{k})$ be a curved $A_{\infty}$-algebra, and $(D, n^{k})$ a left $A_{\infty}$-module over $(C, m^{k})$. Suppose that the filtrations $F_{C}$ and $F_{D}$ are bounded above. Suppose $u \in D$ is a cyclic element such that $u \in F_{D}^{0}$. Then there exists a unique nilpotent bounding cochain $b$ of $(C, m^{k})$ such that
\begin{equation*}
b \in F_{C}^{\lambda_{1}} C
\end{equation*}
for some $\lambda_{1} > 0$, and
\begin{equation}\label{cyclic element is closed under the deformed differential}
d^{b}(u) = 0.
\end{equation}
Here $d^{b}: D \to D$ is defined by
\begin{equation*}
d^{b}(y) = \sum_{k=0}^{\infty} n^{k}(b, \cdots, b; y).
\end{equation*}
\end{lemma}
\begin{proof}
	Because the filtration $F_{D}$ is strictly compatible with $F_{C}$, it is possible to find free generators $y_{j}$ of $D$ and $x_{j}$ of $C$, such that $y_{j} \in F_{D}^{\lambda_{j}}D$ and $x_{j} \in F_{C}^{\lambda_{j}}C$, where $\lambda_{j} \in \Lambda_{D} = \Lambda_{C}$. \par
	We write $u = u_{0} y_{0} + \cdots + u_{l} y_{l}$ in a unique way, where $u_{j} \in \mathbb{Z}$ and $y_{j} \in D$ are free generators of the free $\mathbb{Z}$-module $D$, such that $y_{0} \in F_{D}^{0}D$, and $y_{j} \in F_{D}^{\lambda_{j}}D$ but $y_{j} \notin F_{D}^{\lambda}D$ if $\lambda > \lambda_{j}$, where $0 < \lambda_{1} < \cdots < \lambda_{l}$. Because of condition (ii) for a cyclic element, we have that $n^{0}(u_{0}) = 0$, so that $n^{0}(u) = \sum_{j = 1}^{l} u_{j} n^{0}(y_{j}) \in F_{D}^{\lambda_{1}}$. \par
	Also, $b$ has a unique expression $b = b_{1} x_{1} + \cdots b_{l} x_{l}$, where $b_{j} \in \mathbb{Z}$ and $x_{j} \in C$ are free generators of the free $\mathbb{Z}$-module $C$, such that $x_{j} \in F_{C}^{\lambda_{j}}C$ but $x_{j} \notin F_{C}^{\lambda}C$ if $\lambda > \lambda_{j}$, where $\lambda_{1} < \cdots < \lambda_{m}$. \par
	Let us try to solve the equation \eqref{cyclic element is closed under the deformed differential} for such an element $b$. This breaks down to a system of equations
\begin{equation}
\sum_{k} \sum_{i_{1}, \cdots, i_{k}} \sum_{j} b_{i_{1}} \cdots b_{i_{k}} u_{j} n^{k}(x_{i_{1}}, \cdots, x_{i_{k}}; y_{j}) = 0.
\end{equation}
Here the term $n^{k}(x_{i_{1}}, \cdots, x_{i_{k}}; y_{j}) \in F_{D}^{\lambda_{i_{1}} + \cdots + \lambda_{i_{k}} + \lambda_{j}}D$, but not in a $F_{D}^{\lambda'}D$ for any smaller $\lambda'$. Because $F_{D}$ is bounded above, there can be at most finitely many such terms which are non-zero. Thus this system can be separated to finitely many equations, according to the filtration. We order these equations in an increasing order in terms of filtration. Rewrite every equation as the form
\begin{equation}
b_{i} u_{0} n^{1}(x_{i}; y_{0}) + \sum b_{i_{1}} \cdots b_{i_{k}} u_{j} n^{k}(x_{i_{1}}, \cdots, x_{i_{k}}; y_{j}) = 0
\end{equation}
where the second sum is taken over all possible indices so that $\lambda_{i_{1}} + \cdots + \lambda_{i_{k}} + \lambda_{j} \le \lambda_{i}$. In particular, all $i_{j} < i$. When $i = 1$, there are no $b$'s in the second sum, and the only possibly nonzero terms are of the form $u_{j} n^{0}(u_{j})$ such that $\lambda_{j} \le \lambda_{i} + \lambda_{0}$. Because $n^{1}(\cdot; u)$ is a filtration-preserving isomorphism of $\mathbb{Z}$-modules, the coefficient of the first term $u_{0}$ is invertible: $u_{0} = \pm 1$. Thus we can solve for a unique $b_{1}$. Then we consider $i = 2$ and argue in the same way to solve for $b_{2}$. We repeat this process until we solve for all $b_{i}, i = 1, \cdots, l$. \par
	It remains to prove that this solution is also a solution to the inhomogeneous Maurer-Cartan equation \eqref{MC equation for curved A-infinity algebra}. Since $b \in F_{C}^{\lambda_{1}}C$ for $\lambda_{1} > 0$, and $F_{C}$ is bounded above, $b$ is automatically nilpotent, and the equation \eqref{MC equation for curved A-infinity algebra} can be reduced to \eqref{MC equation for a nilpotent element}. The strategy is to prove that
\begin{equation}\label{MC equation modulo lambda}
\sum_{k} m^{k}(b, \cdots, b) \in F_{C}^{\lambda'}C
\end{equation}
for every $\lambda' > 0$, which implies that it must vanish because $F_{C}$ is bounded above: for $\lambda'$ large, $F_{C}^{\lambda'}C = 0$.
From the $A_{\infty}$-equations for $n^{k}$ and $m^{k}$ we get the following equation
\begin{equation}
\sum n^{k_{1}}(b, \cdots, b, n^{k_{2}}(b, \cdots, b; u) + \sum n^{k_{1}}(b, \cdots, m^{k_{2}}(b, \cdots, b), \cdots, b); u) = 0.
\end{equation}
The first sum vanishes because $\sum n^{k}(b, \cdots, b; u) = 0$. Following the same kind of argument as before, we rewrite the above equation as the following:
\begin{equation}\label{MC equation written in terms of basis}
\sum b_{i_{1}} \cdots b_{i_{k}} u_{j} n^{k_{1}}(x_{i_{1}}, \cdots, x_{i_{s}}, m^{k_{2}}(x_{i_{s+1}}, \cdots, x_{i_{s+k_{2}}}), \cdots, x_{i_{k}}; y_{j}) = 0.
\end{equation}
We may further write this equation as a system of equations, ordered by the filtration.
Let us try to prove \eqref{MC equation modulo lambda} by induction on $\lambda$, namely we take a discrete sequence $\lambda'_{n} \to +\infty$ and prove \eqref{MC equation modulo lambda} for $\lambda'_{n}$. Recall that we have
\begin{equation}
b_{i} u_{0} n^{1}(x_{i}; y_{0}) + \sum b_{i_{1}} \cdots b_{i_{k}} u_{j} n^{k}(x_{i_{1}}, \cdots, x_{i_{k}}; y_{j}) = 0
\end{equation}
If we assume \eqref{MC equation modulo lambda} holds for any $\lambda'$ smaller than $\lambda_{i}$, it follows that the terms on the left hand side of \eqref{MC equation written in terms of basis} cancel with each other except that the following terms are left over:
\begin{equation}
b_{i} u_{0} n^{0}(m^{1}(x_{i}); y_{0}) + \sum b_{i_{1}} \cdots b_{i_{k}} u_{0} n^{k}(m^{k}(x_{i_{1}}, \cdots, x_{i_{k}}); y_{0})
\end{equation}
where the second sum is taken over all possible indices such that $\lambda_{i_{1}} + \cdots + \lambda_{i_{k}} \le \lambda_{i}$. As the map $n^{0}(\cdot; u)$ is a filtration-preserving isomorphism, and $y_{0}$ is the summand of $u$ in lowest filtration, this implies that
\begin{equation*}
b_{i} m^{1}(x_{i}) + \sum b_{i_{1}} \cdots b_{i_{k}} m^{k}(x_{i_{1}}, \cdots, x_{i_{k}}) = 0.
\end{equation*}
Since $m^{1}$ does not decrease the filtration, this implies that \eqref{MC equation modulo lambda} holds for $\lambda_{i}$. \par
\end{proof}

\subsection{Homotopy direct limit} \label{A-infinity homotopy direct limit}
	In this subsection, we discuss an $A_{\infty}$-analogue of homotopy direct limit, which is necessary for extending the action-restriction functor to the whole wrapped Fukaya categories, which is the main concern of section \ref{section: product manifolds}. The basic idea based on the existing $A_{\infty}$-functors \eqref{action-restriction functor for finitely many Lagrangians} is to take some kind of limit over all admissible Lagrangian submanifolds in $M \times N$. \par
	We temporarily depart from geometry and formalize the idea in a general situation of $A_{\infty}$-categories and functors, not limited to Fukaya-type categories. Let $\mathcal{A}, \mathcal{B}$ be $A_{\infty}$-categories consisting of countably many objects and let
\begin{equation*}
i_{d}: \mathcal{A}_{d} \to \mathcal{A},
\end{equation*}
\begin{equation*}
j_{d}: \mathcal{B}_{d} \to \mathcal{B}
\end{equation*} be full $A_{\infty}$-subcategories of $\mathcal{A}$ and $\mathcal{B}$ respectively, both of which consist of $d$ objects: $X_{1}, \cdots, X_{d}$ for $\mathcal{A}_{d}$ and $Y_{1}, \cdots, Y_{d}$ for $\mathcal{B}_{d}$. We shall make the following assumption:
\begin{assumption} \label{direct limit assumption on subcategories}
For $d < d'$ the subcategory $\mathcal{A}_{d}$ is also a full $A_{\infty}$-subcategory of $\mathcal{A}_{d'}$ via embedding
\begin{equation*}
i_{d, d'}: \mathcal{A}_{d} \to \mathcal{A}_{d'}
\end{equation*}
and that the sequence of $A_{\infty}$-subcategories
\begin{equation}
\mathcal{A}_{1} \subset \cdots \subset \mathcal{A}_{d} \subset \cdots
\end{equation}
form a directed system such that
\begin{equation}
\mathcal{A} = \lim\limits_{\substack{\longrightarrow\\ d \to \infty}} \mathcal{A}_{d}.
\end{equation}
The same assumption should hold for $\mathcal{B}$.
\end{assumption}

	Now suppose that there exists an $A_{\infty}$-functor
\begin{equation}
\mathcal{F}_{d}: \mathcal{A}_{d} \to \mathcal{B}_{d}
\end{equation}
sending $X_{i}$ to $Y_{i}$, and induces an isomorphism on cohomology categories. To be able to take the limit over $d$ to obtain an $A_{\infty}$-functor
\begin{equation} \label{homotopy direct limit functor}
\mathcal{F}: \mathcal{A} \to \mathcal{B},
\end{equation}
we shall also need certain compatibility condition on the sequence of $A_{\infty}$-functors $\mathcal{F}_{d}$:
\begin{assumption} \label{direct limit assumption on functors}
For $d < d'$ the $A_{\infty}$-functor $\mathcal{F}_{d'}$ maps the subcategory $\mathcal{A}_{d}$ of $\mathcal{A}_{d'}$ to the subcategory $j_{d, d'}(\mathcal{B}_{d})$ of $\mathcal{B}_{d'}$. Thus we may define
\begin{equation}
j^{-1}_{d, d'} \circ \mathcal{F}_{d'} \circ i_{d, d'}: \mathcal{A}_{d} \to \mathcal{B}_{d}.
\end{equation}
This should be homotopic to $\mathcal{F}_{d}$ as $A_{\infty}$-functors from $\mathcal{A}_{d}$ to $\mathcal{B}_{d}$.
\end{assumption}
	Under Assumptions \ref{direct limit assumption on subcategories} and \ref{direct limit assumption on functors}, it is almost straightforward to obtain the $A_{\infty}$-functor \eqref{homotopy direct limit functor} by taking the homotopy direct limit, in the sense we describe below. The following lemma is a straightforward consequence of these assumptions, which can be proved by using homological perturbation lemma. \par

\begin{lemma} \label{auto-equivalences correcting homotopy commutativity to strict commutativity}
	For each $d$ and $d'$, there exists an auto-equivalence $\mathcal{H}_{d, d'}$ of $\mathcal{B}_{d}$ such that the composition $j_{d, d'} \circ \mathcal{H}_{d, d'} \circ \mathcal{F}_{d}$ strictly coincides with $\mathcal{F}_{d'} \circ i_{d, d'}$. 
\end{lemma}

	Moreover, we make the following observation: \par
\begin{lemma}
	$j_{d+1, d+2} \circ \mathcal{H}_{d+1, d+2} \circ j_{d, d+1} \circ \mathcal{H}_{d, d+1}$ is homotopic to $j_{d, d+2}$.
\end{lemma}
\begin{proof}
	We claim that the auto-equivalence $\mathcal{H}_{d, d'}$ in Lemma \ref{auto-equivalences correcting homotopy commutativity to strict commutativity} is actually homotopic to the identity functor, and that thus the lemma follows. The claim follows from the fact that $j_{d, d'} \circ \mathcal{F}_{d}$ is homotopic to $\mathcal{F}_{d'} \circ i_{d, d'}$, and the fact that the auto-equivalence $\mathcal{H}_{d, d'}$ of $\mathcal{B}_{d}$ is introduced to make them strictly agree.
\end{proof}

	 We are thus led to the situation where we have homotopy directed systems of $A_{\infty}$-categories and diagrams of $A_{\infty}$-functors
\begin{equation}
\begin{CD}
\mathcal{A}_{1} @>i_{1, 2}>> & \mathcal{A}_{2} @>i_{2, 3}>> & \mathcal{A}_{3} @>i_{3, 4}>> & \cdots \\
@VV\mathcal{F}_{1}V & @VV\mathcal{F}_{2}V & @VV\mathcal{F}_{3}V\\
\mathcal{B}_{1} @>j_{1, 2} \circ \mathcal{H}_{1, 2}>> & \mathcal{B}_{2} @>j_{2, 3} \circ \mathcal{H}_{2, 3}>> & \mathcal{B}_{3} @>j_{3, 4} \circ \mathcal{H}_{3, 4}>> & \cdots
\end{CD}
\end{equation}
such that each square is strictly commutative. Then we can take the direct limits:
\begin{equation*}
\lim\limits_{\substack{\longrightarrow \\ d \to \infty}} (\mathcal{A}_{d}, i_{d, d+1}) = \mathcal{A},
\end{equation*}
\begin{equation*}
\lim\limits_{\substack{\longrightarrow \\ d \to \infty}} (\mathcal{B}_{d}, j_{d, d+1} \circ \mathcal{H}_{d, d+1}) = \tilde{\mathcal{B}},
\end{equation*}
where $\tilde{\mathcal{B}}$ is quasi-equivalent to $\mathcal{B}$, as well as
\begin{equation}
\lim\limits_{\substack{\longrightarrow \\ d \to \infty}} \mathcal{F}_{d} = \mathcal{F}: \mathcal{A} \to \tilde{\mathcal{B}}.
\end{equation}
This is the limit $A_{\infty}$-functor obtained from the sequence $\mathcal{F}_{d}$, with target different from but quasi-equivalent to $\mathcal{B}$. \par

\section{The wrapped Fukaya category: revisited} \label{wrapped Fukaya category}

\subsection{Overview}
	There are several ways of defining the wrapped Fukaya category of a Liouville manifold. The basic approaches grow out of the definition of wrapped Floer cohomology, which either uses a cofinal family of Hamiltonians linear at infinity \cite{Abouzaid-Seidel}, or a single Hamiltonian quadratic at infinity \cite{Abouzaid1}, or more generally a single Hamiltonian whose growth at infinity is faster than linear. A more functorial approach is via categorical colimits and localization, without having to specify a particular choice of Hamiltonian as introduced in \cite{Ganatra-Pardon-Shende}. Of course, all these approaches give quasi-equivalent $A_{\infty}$-categories. \par
	At first, we set up the wrapped Fukaya category using a single Hamiltonian quadratic at infinity. Although this is not the best definition as it relies on working with a specific choice of Hamiltonian, it is convenient for us to compare two versions of wrapped Fukaya categories of the product manifold, as we just have to work with some chosen Hamiltonians, so that the difficulty in dealing with complicated triple colimits of functors is bypassed. \par
	Also, this setup simplifies the construction of functors from Lagrangian correspondences, as they can be constructed in a single step, without having to be defined as a colimit of functors. However, regarding applications and calculation, it is better to give a definition in the framework of wrapped Fukaya category defined with respect to linear Hamiltonians. This definition, due to \cite{Abouzaid-Seidel} will be reviewed in subsection \ref{section: linear Hamiltonians}. \par

\subsection{Basic geometric setup}
	The geometric setup considered here is exactly the same as that in \cite{Gao1}. We briefly mention it here mainly for the purpose of fixing notations. Consider a Liouville manifold $M$ which is the completion of a Liouville domain $M_{0}$ with boundary $\partial M$, which has a collar neighborhood $\partial M \times (\epsilon, 1]$ so that the Liouville vector field is equal to $\frac{\partial}{\partial r}$ in that neighborhood. We assume that $M$ is symplectically Calabi-Yau, namely $2c_{1}(M) = 0 \in H^{2}(M; \mathbb{Z})$. \par
	The admissible Lagrangian submanifolds are either closed exact Lagrangian submanifolds in the interior $M_{0}$, or cylindrical Lagrangian submanifolds of the form $L = L_{0} \cup \partial L \times [1, +\infty)$ where $\partial L \subset \partial M$ is a Legendrian submanifold with respect to the contact structure induced from the Liouville one-form. To be more specific, for the latter kind of Lagrangian submanifold $L_{0}$ of $M_{0}$, there should be a function $f$ on it so that $df = \lambda |_{L_{0}}$, where $\lambda$ is the Liouville form. Moreover, we require that $f$ has an extension to a neighborhood of $L$ in $M$ such that it is locally constant near $\partial L \times (\epsilon, +\infty)$. In addition, we shall make the assumption that
\begin{equation}
2c_{1}(M, L) = 0 \in H^{2}(M, L; \mathbb{Z}),
\end{equation}
which ensures the existence of gradings and spin structures on $L$. We will fix a choice of grading and spin structure for every admissible Lagrangian submanifold. These conditions will ensure that the wrapped Fukaya category of $M$ is defined over $\mathbb{Z}$, and carries $\mathbb{Z}$-gradings. \par

\subsection{Floer data and consistency}
	The moduli space of surfaces controlling the algebraic operations and relations in the Fukaya categories are the moduli spaces of stable marked nodal disks, which was studied in \cite{FOOO1} and proved to be cellular isomorphic to the moduli spaces of stable metric ribbon trees introduced by Stasheff [Stasheff], which is known to be the operad controlling $A_{\infty}$-algebras. \par
	Let $\bar{\mathcal{M}}_{k+1}$ be the compactified moduli space of stable $(k+1)$-marked disks. It is proved in \cite{FOOO1} that $\bar{\mathcal{M}}_{k+1}$ is a compact smooth manifold with corners and a neighborhood of the stratum $\mathcal{M}_{T}$ with combinatorial type modeled on a stable ribbon tree $T$ with one root and $k$-leaves is covered by the image of the gluing map:
\begin{equation} \label{tree gluing}
(-1, 0]^{e(T)} \times \mathcal{M}_{T} \supset U_{T} \to \bar{\mathcal{M}}_{k+1},
\end{equation}
which is smooth and a diffeomorphism onto the image by shrinking $U_{T}$ if necessary. \par

	To define $A_{\infty}$-operations on Floer cochain spaces, we need to study moduli spaces of (perturbed) pseudoholomorphic maps from nodal disks to $M$ with boundary in Lagrangian submanifolds. For this purpose, we also need to include the case where the domain is unstable, or contains unstable components. The unstable curve involved here is the infinite strip $Z$, whose automorphism is the additive group $\mathbb{R}$. To write down inhomogeneous Cauchy-Riemann equations and achieve transversality of the moduli spaces of solutions, we need several auxilliary data. We briefly recall the notion here. \par

\begin{definition}
	Given a semistable $(k+1)$-marked nodal disk $S \in \bar{\mathcal{M}_{k}}$, a Floer datum $P_{S}$ for $S$ consists of
\begin{enumerate}[label=(\roman*)]

\item A collection of positive integers $w_{0}, \cdots, w_{k}$.

\item A time-shifting function $\rho_{S}: \partial S \to [1, +\infty)$, which takes the value $w_{j}$ over the $j$-th strip-like end $\epsilon_{j}$.

\item A basic one form $\alpha_{S}$, whose restriction to every smooth component of $S$ is closed, and whose pullback by $\epsilon_{j}$ agrees with $w_{j}dt$.

\item A Hamiltonian perturbation $H_{S}: S \to \mathcal{H}(M)$, whose pullback by $\epsilon_{j}$ agrees with $\frac{H \circ \phi^{w_{j}}}{w_{j}^{2}}$.

\item A domain-dependent perturbation of almost complex structures $J_{S}: S \to \mathcal{J}(M)$, whose pullback by $\epsilon_{j}$ agrees with $(\phi^{w_{j}})^{*}J_{t}$.

\end{enumerate}
such that over unstable components of $S$, i.e. strips, all the three data restricts to translation-invariant data.
\end{definition}

	In order to ensure that the various operations constructed from moduli spaces of marked inhomogeneous pseudoholomorphic disks satisfy the $A_{\infty}$-equations, we need to make sure that the Floer data chosen for the underlying semistable marked nodal disks are compatible with respect to gluing maps \eqref{tree gluing}. Therefore the following notion is useful: for $k \ge 3$, a universal and consistent choice of Floer data is a choice of Floer data for all $S \in \bar{\mathcal{M}}_{k+1}$ that varies smoothly with respect to $S$ in the compactified moduli space. The notion of universal and consistent choice of Floer data is extended also to the strip $Z$, as follows: when we glue in a strip at a boundary marked point of a stable marked nodal disk $S$, we require that the (translation-invariant) Floer datum chosen on $Z$ agree with that on the strip-like end for $S$ near that marked point. The space of choices of Floer data is convex, therefore by induction on the strata of the moduli space of stable marked nodal disks, we can construct universal and consistent choices of Floer data. More detailed explanation is given in \cite{Seidel}. \par

\subsection{Inhomogeneous pseudoholomorphic disks}
	To define the $A_{\infty}$-operations on the wrapped Fukaya category, we need to study the moduli spaces of inhomogeneous pseudoholomorphic disks with boundary mapped to several Lagrangian submanifolds. Make universal and consistent choices of Floer data $P$ for all semistable marked nodal disks. Denote by $S$ an element in the smooth part of the moduli space $\mathcal{M}_{k+1}$. That is, $S$ is a smooth disk with boundary marked points $(z_{0}, \cdots, z_{k})$ that are cyclically ordered on the boundary. Given admissible Lagrangian submanifolds $L_{0}, \cdots, L_{k}$, consider the following inhomogeneous Cauchy-Riemann equation, for both $S$ and $u$ as variables:
\begin{equation}
\begin{cases}
(du - \alpha_{S} \otimes X_{H_{S}})^{0, 1} = 0;\\
u(z) \in \phi^{\rho_{S}(z)}L_{j}, \text{ if $z$ lies in between $z_{j}$ and $z_{j+1}$};\\
\lim\limits_{s \to -\infty} u \circ \epsilon_{0}(s, \cdot) = \phi^{w_{0}}x_{0}(\cdot) \in \mathcal{X}(\phi^{w_{0}}L_{0}, \phi^{w_{0}}L_{k});\\
\lim\limits_{s \to +\infty} u \circ \epsilon_{j}(s, \cdot) = \phi^{w_{j}}x_{j}(\cdot) \in \mathcal{X}(\phi^{w_{j}}L_{j-1}, \phi^{w_{j}}L_{j}), & j = 1, \cdots, k.
\end{cases}
\end{equation} 
The solutions will sometimes also be called Floer's disks. \par
	Suppose for the moment $k \ge 2$. Let $\mathcal{M}_{k+1}(L_{0}, \cdots, L_{k}; x_{0}, \cdots, x_{k}; P)$ be the moduli space of solutions $(S, u)$ to the above equation with respect to the chosen Floer data $P$, and let $\bar{\mathcal{M}}_{k+1}(L_{0}, \cdots, L_{k}; x_{0}, \cdots, x_{k}; P)$ be its stable map compactification. It is proved in \cite{Abouzaid1} that for a generic choice of Floer data $P$, the zero-dimensional and one-dimensional components of $\bar{\mathcal{M}}_{k+1}(L_{0}, \cdots, L_{k}; x_{0}, \cdots, x_{k}; P)$ are compact smooth manifolds with corners of dimension 
\begin{equation*}	
2 - k + \deg(x_{0}) - \deg(x_{1}) - \cdots - \deg(x_{k}).
\end{equation*} \par
	In the unstable case $k=1$, there is no moduli of $S$, so we consider the set of solutions $u$ to the above equation. Since the Floer datum $P_{Z}$ on the strip is chosen to be translation-invariant, we can quotient the parametrized moduli space by this automorphism group. We denote the quotient moduli space by $\mathcal{M}_{2}(L_{0}, L_{1}; x_{0}, x_{1}; P_{Z})$ as well, and the corresponding Gromov bordification by $\bar{\mathcal{M}}_{2}(L_{0}, L_{1}; x_{0}, x_{1}; P_{Z})$. \par

\subsection{Identification of Floer cochain spaces with different weights}
	The "count" of rigid elements in the moduli spaces $\bar{\mathcal{M}}_{k+1}(L_{0}, \cdots, L_{k}; x_{0}, \cdots, x_{k}; P)$ defines operations of the following kind
\begin{equation}
\begin{split}
&CW^{*}(\phi^{w_{k}}L_{k-1}, \phi^{w_{k}}L_{k}; \frac{H \circ \phi^{w_{k}}}{w_{k}}) \otimes \cdots \otimes CW^{*}(\phi^{w_{1}}L_{0}, \phi^{w_{1}}L_{1}; \frac{H \circ \phi^{w_{1}}}{w_{1}})\\
\to &CW^{*}(\phi^{w_{0}}L_{0}, \phi^{w_{0}}L_{k}; \frac{H \circ \phi^{w_{0}}}{w_{0}}).
\end{split}
\end{equation}
	In order define an honest $A_{\infty}$-category with morphism spaces being $CW^{*}(L, L'; H)$ between the two objects $L$ and $L'$ so that the operations happen on these morphism spaces, we have to identify $CW^{*}(\phi^{w}L, \phi^{w}L'; \frac{H \circ \phi^{w}}{w})$ with $CW^{*}(L, L'; H)$, in a canonical way. Because the Hamiltonian $H$ is quadratic in the radial coordinate of the cylindrical end, such an identification is easily achieved by noting that $\frac{H \circ \phi^{w}}{w^{2}}$ behaves the same as $H$ in the cylindrical end where Reeb dynamics occur. Technically, the rescaled Hamiltonian differs from $H$ by a small amount that is supported in the compact part of $M$, and this can be taken care of by using a compactly supported deformation of Hamiltonian functions, which gives rise to continuation maps that form an $A_{\infty}$-quasi-isomorphism (of $A_{\infty}$-bimodule structures on the wrapped Floer cochain spaces). \par
	To summarize, these arguments in previous subsections together imply that the wrapped Fukaya category $\mathcal{W}(M)$ is well-defined, up to quasi-isomorphism. \par

\subsection{Winding Lagrangian submanifolds}
	Let us introduce a new class of Lagrangian submanifolds in the wrapped Fukaya category. These Lagrangian submanifolds come from geometric compositions of Lagrangian correspondences to be discussed in detail in section \ref{geometric composition admissible for wrapped Floer theory}. This class includes in particular $H$-perturbed cylindrical Lagrangian submanifolds in $M$, i.e. $\phi_{H}^{1}(L)$ for some cylindrical Lagrangian submanifold $L = L_{0} \cup \partial L \times [1, +\infty)$, where the Hamiltonian perturbation is the same as the one used to defined $\mathcal{W}(M)$ - these $H$-perturbed cylindrical Lagrangian submanifolds are geometric compositions with the diagonal. The picture of such a Lagrangian submanifold is one that winds around in the cylindrical end of $M$. The point is, we want to include both cylindrical Lagrangian submanifolds and these perturbed cylindrical Lagrangian submanifolds as object in our wrapped Fukaya category, with respect to the same Liouville structure, although one can easily show that the perturbed ones are cylindrical with respect to a different Liouville structure. \par
	Now let $\phi_{H}(L)$ be an $H$-perturbed Lagrangian submanifold. Geometrically, it looks like winding around a cylindrical Lagrangian submanifold in the cylindrical end in accelerating speeds with respect to the radial coordinate. The first task is to define its self wrapped Floer cohomology. On the level of the underlying cochain groups, this is fairly straightforward: the self wrapped Floer complex of $\phi_{H}(L)$ can be thought of as the Lagrangian intersection Floer cochain complex of $\phi_{H}^{2}(L)$ with $\phi_{H}(L)$. As the underlying cochain group, this is canonically isomorphic to $CF^{*}(\phi_{H}(L), L)$ by applying the exact symplectomorphism $\phi_{H}^{-1}$. However, the differential is a bit tricky to define. Recall that we have defined the differential on $CW^{*}(L)$ using inhomogeneous $J$-holomorphic strips with boundary on $L$ that converge to $H$-chords at infinity. If one somehow wants to "transport" the differential of $CW^{*}(L)$ to one of $CW^{*}(\phi_{H}(L))$, geometrically that will be to consider moduli space of inhomogeneous $(\phi_{H})^{*}J$-holomorphic strips, instead of $J$-holomorphic strips. On possible way to resolve this is to refer to the Lagrangian intersection setup, where one defines $CW^{*}(\phi_{H}(L))$ as the Lagrangian intersection Floer complex of $\phi_{H}^{2}(L)$ with $\phi_{H}(L)$. Thus the differential will be to use homogeneous $J$-holomorphic strips with boundary on $\phi_{H}^{2}(L)$ and $\phi_{H}(L)$ which converge to the corresponding intersection points. An appropriate version of maximum principle in this setting is necessary. To establish that, we compare $\phi_{H}^{2}(L)$ with $\phi_{2H}(L)$ and note they could only possibly differ by a compact-supported Hamiltonian isotopy. Since $2H$ is also admissible, the maximum principle applies to such $J$-holomorphic strips. \par
	Alternatively, we could directly try to prove the compactness results of the moduli spaces of inhomogeneous pseudoholomorphic disks with boundary on (appropriate rescalings of) $\phi_{H}(L)$ by appealing to the action-energy equality. The most important fact is the following equality computing the new primitive function for $\phi_{H}(L)$ in terms of that of for $L$. \par

\begin{lemma}\label{changing primitive under Hamiltonian isotopy}
Let $f$ be a primitive for $L$. Then the following function
\begin{equation}
f + \iota_{X}\lambda
\end{equation}
is a primitive for $\phi_{H}(L)$. Here $X$ is the Hamiltonian vector field of $H$.
\end{lemma}
\begin{proof}
	The proof is a straightforward calculation, based on the well-known fact that a Hamiltonian symplectomorphism is exact and adds to the primitive of the symplectic form the following:
\begin{equation}
d\int_{0}^{1} \iota_{X_{t}}\lambda dt.
\end{equation}
Now since our Hamiltonian is time-independent, this is simply equal to
\begin{equation}
d \iota_{X}\lambda.
\end{equation}
Thus a primitive for $\phi_{H}(L)$ can be taken to be
\begin{equation}
f + \iota_{X}\lambda
\end{equation}
\end{proof}

	The following estimate on the new primitive is the crucial step in proving compactness result for the relevant moduli spaces in wrapped Floer theory. \par

\begin{lemma}
$\iota_{X}\lambda$ equals $2r^{2}$ in the cylindrical end $\partial M \times [1, +\infty)$. In particular, it is constant on any level hypersurface $\partial M \times \{r\}$.
\end{lemma}
\begin{proof}
This is a straightforward elementary calculation.
\end{proof}

	Looking back at the action-energy equality (the action is computed with respect to the original primitive),
\begin{equation*}
\mathcal{A}_{H, \phi_{H}(L)}(x) = \int_{0}^{1} -x^{*}\lambda + H(x(t))dt + f_{\phi_{H}(L)}(x(1)) - f_{\phi_{H}(L)}(x(0)),
\end{equation*}
we find that for an $H$-chord from $\phi_{H}(L)$ to itself which is contained in a level hypersurface, the last two terms contribute zero because the extra term $\iota_{X}\lambda$ in the primitive is constant on this hypersurface. Thus, the same estimates apply as if we were in the case of a cylindrical Lagrangian submanifold, and consequently the action-energy equality implies compactness results for the moduli spaces of inhomogeneous pseudoholomorphic strips with boundary on $\phi_{H}(L)$. Therefore, the self wrapped Floer cohomology of $\phi_{H}(L)$ is well-defined. \par
	The second task is to define the wrapped Floer cohomology $HW^{*}(\phi_{H}(L), L')$, for any pair $(\phi_{H}(L), L')$ where $L'$ is either a closed exact Lagrangian submanifold in the interior $M_{0}$, or a cylindrical Lagrangian submanifold. Reasoning in the same way as above, we know that this can be defined as the Lagrangian intersection Floer cohomology of $\phi_{H}^{2}(L)$ with $L'$. On the other hand, by a rescaling argument introduced in the previous subsection, that group is isomorphic to $HW^{*}(L, L')$. There is also an alternative construction by direct analysis on the compactness of the moduli spaces of inhomogeneous pseudoholomorphic strips with boundary on the pair $(\phi_{H}(L), L')$ using estimates on the primitives, which can be done in a similar way to that for a single $H$-perturbed Lagrangian submanifold. \par
	The third task is to define the wrapped Floer cohomology $HW^{*}(L', \phi_{H}(L))$. The story is parallel to the previous case, but there is an asymmetry. This time, $HW^{*}(L', \phi_{H}(L))$ should be regarded as the Lagrangian intersection Floer cohomology of $\phi_{H}(L')$ with $\phi_{H}(L)$. In a generic situation where $L$ and $L'$ intersect transversely and the Legendrian boundaries of $L'$ and $L$ do not intersect, there is in fact a canonical chain-level isomorphism between $CF^{*}(\phi_{H}(L'), \phi_{H}(L))$ and $CF^{*}(L', L)$, where both are defined with respect to $J$. This is because there are only finitely many intersections which are contained in the compact domain $M_{0}$, so distinguishing $J$ and $\phi_{H}^{*}J$ is irrelevant. In the special case where $L = L'$, further discussion is needed. We may define a Morse-Bott Floer complex by introducing a chain model for $L$, say a Morse complex of a Morse function $f: L \to \mathbb{R}$ without any critical point in the cylindrical end $\partial L \times [1, +\infty)$. For example, we can take $f$ to be the restriction of $H$ to $L$, provided $H$ is generic. The differential counts $(J, H)$-pseudoholomorphic strips, which become gradient flow lines provided $H$ is $C^{2}$-small. Thus the cohomology $HW^{*}(L, \phi_{H}(L))$ is isomorphic to the ordinary cohomology of $L$. \par
	The next task is to introduce $A_{\infty}$-structures on the cochain groups underlying the above-mentioned wrapped Floer cohomology groups. There are several cases:
\begin{enumerate}[label=(\roman*)]

\item The self wrapped Floer cochain complex $CW^{*}(\phi_{H}(L))$ should be equipped with a structure of an $A_{\infty}$-algebra. Because of our previous definition of this cochain complex as the Lagrangian intersection Floer cochain complex of $\phi_{H}^{2}(L)$ with $L$, it is inappropriate to use inhomogeneous pseudoholomorphic disks with boundary on $L$.

\item Regarding the wrapped Floer cochain complex $CW^{*}(\phi_{H}(L), L')$ for a pair, it is equipped with an $A_{\infty}$-bimodule structure over $(CW^{*}(\phi_{H}(L)), CW^{*}(L'))$, defined by the moduli spaces of $(J, H)$-holomorphic strips with punctures on both boundary components that are mapped to $(\phi_{H}(L), L')$, together with $(J, H)$-holomorphic disks with punctures with boundary on $\phi_{H}(L)$ or $L'$, which are joint to the strips at their boundary punctures.

\item As for $CW^{*}(L', \phi_{H}(L))$, there is a small difference. To equip this with a natural left-$CW^{*}(L')$ and right-$CW^{*}(\phi_{H}(L))$ $A_{\infty}$-bimodule structure, we need a small modification of the setup of Floer complex for the pair $(L', \phi_{H}(L))$, by defining the differential using the original setup for wrapped Floer cohomology, instead of $J$-holomorphic strips with boundary on $(L', L)$. We consider $CW^{*}(L', \phi_{H}(L))$ as being generated by time-one $H$-chords from $L'$ to $\phi_{H}(L)$, which are in one-to-one correspondence with intersection points $L \cap L'$, hence finite. The differential counts rigid $(J, H)$-holomorphic strips with boundary on $(L', \phi_{H}(L))$, which is well-defined because there are only finitely many time-one $H$-chords. Now the moduli space of $(J, H)$-holomorphic strips with boundary on $(L', \phi_{H}(L))$ is compatible with the moduli spaces of $(J, H)$-holomorphic disks with boundary on $L'$ or $\phi_{H}(L)$, meaning that the relevant pseudoholomorphic curves can be glued together in a coherent way. Therefore we obtain the desired $A_{\infty}$-bimodule structure.

\end{enumerate}

\section{Wrapped Floer theory for Lagrangian immersions}\label{the immersed wrapped Fukaya category}

\subsection{Overview of immersed Lagrangian Floer theory}
	In this section, we extend wrapped Floer theory to certain classes of Lagrangian immersions. The main purpose of such an extension is to prove representability of functors associated to Lagrangian correspondences in general (to be discussed in section \ref{A-infinity functors associated to Lagrangian correspondence}), though in many concrete and interesting cases, it is sufficient to study embedded Lagrangian submanifolds. \par
	In order for the Lagrangian immersions in consideration to have well-behaved Floer theory, we must impose some conditions: they should satisfy a condition similar to being exact, be embedded in the cylindrical end of $M$, and possibly have transverse self-intersections in the interior part of $M$. Without loss of generality, we assume these self-intersections are at most double points. In general, there will be pseudoholomorphic disks bounded by the image of such a Lagrangian immersion, and these disks will interact with inhomogeneous pseudoholomorphic disks (solutions to Floer's equation). Therefore, we should pick a good model for the compactifications of the relevant moduli spaces of disks. Fortunately, this can be done fairly directly, as the Lagrangian immersions we are going to consider still satisfy an "exactness" condition, which will be introduced in the next subsection. \par
	Wrapped Floer theory should contain information about Reeb dynamics on the boundary contact manifold, in addition to the cohomological generators of the Lagrangian submanifolds. The construction of the $A_{\infty}$-structure maps involves both inhomogeneous pseudoholomorphic disks and homogeneous pseudoholomorphic disks. The entire picture would be an analogue of the setup of holomorphic curves in relative symplectic field theory (see \cite{BEHWZ}). Fortunately, there is a purely Floer-theoretic formulation, where we can construct moduli spaces of maps which satisfy certain Floer's equation, and the virtual techniques used in proving transversality does not go beyond the theory of Kuranishi structures, because the pseudoholomorphic curves that we are dealing with are all of genus zero with connected boundary, considered by Fukaya-Oh-Ohta-Ono \cite{FOOO1}, \cite{FOOO2}. \par
	Wrapped Floer theory assigns to such a Lagrangian immersion $\iota: L \to M$ a curved $A_{\infty}$-algebra $(CW^{*}(L, \iota; H), m^{k})$ over $\mathbb{Z}$ (compare to the case of general compact Lagrangian immersions studied by Akaho and Joyce \cite{Akaho-Joyce}). To define the immersed wrapped Fukaya category, we shall consider unobstructed Lagrangian immersions, i.e. those for which the curved $A_{\infty}$-algebra $CW^{*}(L, \iota; H)$ has a bounding cochain. \par

\subsection{Gradings and spin structures}
	For Floer theory to carry an absolute $\mathbb{Z}$-grading as well as to have coefficients in $\mathbb{Z}$, we need to introduce the notions of gradings and spin structures. \par

\begin{definition}\label{definition of graded Lagrangian immersion}
	Say that the Lagrangian immersion $\iota: L \to M$ is graded, if the square phase function $\alpha_{L}: L \to S^{1}$ has a lift $\tilde{\alpha}_{L}$ to $\mathbb{R}$. Here the square phase function is defined by sending $x \in L$ to $(d\iota)_{x}(T_{x}L)$, an element in the Lagrangian Grassmannian $\mathcal{LAG}(TM)$, then mapping that to $S^{1}$ by pairing any orthonormal basis for the Lagrangian plane $(d\iota)_{x}(T_{x}L)$ with the quadratic volume form, which is independent of the choice of an orthonormal basis. Such a lift is called a grading for this Lagrangian immersion $\iota: L \to M$. 
\end{definition}

	From now on we shall make the following assumption. \par

\begin{assumption}\label{graded and spin Lagrangian immersions}
	The manifold $L$ is spin with a chosen spin structure $v$. Also, the Lagrangian immersion is graded in the sense of Definition \ref{definition of graded Lagrangian immersion} with a chosen grading $\tilde{\alpha}_{L}$.
\end{assumption}
	
	A grading for the Lagrangian immersion $\iota: L \to M$ defines an absolute Maslov index for each generator $c$ (note that if $c$ is a critical point of $f$ together with a capping half-disk $w$, the disk Maslov index of $w$ agrees with the Morse index of $f$ at $p$), which endows with the wrapped Floer cochain space a $\mathbb{Z}$-grading. A spin structure $v$ determines orientations on the moduli spaces we are going to introduce in subsection \ref{section: moduli space of disks bounded by immersed Lagrangian submanifolds}. \par

\begin{remark}
	The condition that the immersion $\iota: L \to M$ be graded implies that the Maslov class of $\iota: L \to M$ is zero. However, it does not prohibit the existence of holomorphic disks with boundary on $\iota(L)$ of non-zero Maslov indices.
\end{remark}

\subsection{The wrapped Floer cochain space for a cylindrical Lagrangian immersion}\label{section: wrapped Floer cochain space for a cylindrical Lagrangian immersion}
	Let $\iota: L \to M$ be a Lagrangian immersion. To develop wrapped Floer theory for it, we shall make some geometric assumptions. \par

\begin{definition}
	A Lagrangian immersion $\iota: L \to M$ is said to have transverse self-intersections, if the self fiber product $L \times_{\iota} L$ is a smooth submanifold of $L \times L$, and has a decomposition
\begin{equation}
L \times_{\iota} L  = \Delta_{L} \coprod \coprod_{a} L_{a}
\end{equation}
where $\Delta_{L} \cong L$ is the diagonal, and the disjoint union is a union of isolated points. \par	
	Let $\iota: L \to M$ be a Lagrangian immersion with transverse self-intersections which are at most double points. It is said to be cylindrical, if over the cylindrical end $\partial M \times [1, +\infty)$, $\iota$ is an embedding, whose image is of the form $l \times [1, +\infty)$ for some Legendrian submanifold $l \subset \partial M$. More generally speaking, being cylindrical means the image is invariant under the Liouville flow outside of a compact set. \par
	The Lagrangian immersion $\iota: L \to M$ is said to be exact, if there exists a function $f: L \to \mathbb{R}$ such that $df = \iota^{*} \lambda_{M}$. \par
\end{definition}

	Let $\iota: L \to M$ be an exact cylindrical Lagrangian immersion with transverse self-intersections. From now on we assume that all self-intersections are at most double points, in which case the decomposition of the self fiber product takes the form
\begin{equation} \label{decomposition of the self fiber product of Lagrangian immersion with transverse self-intersections}
L \times_{\iota} L = \Delta_{L} \coprod \coprod_{x} \{(p_{-}, p_{+}): p_{-} \neq p_{+}, \iota(p_{-}) = \iota(p_{+}) = x\}.
\end{equation}
\par
	The more classical version of Floer theory gives rise to a deformation of the singular chain complex of the Lagrangian submanifold. In addition to that, Reeb chords on the boundary contact manifold and the Legendrian boundary of the Lagrangian submanifolds should also be included as generators in wrapped Floer theory. In the case of a Lagrangian immersion, the self-intersection points bring in extra generators. \par
	Given the above general idea, in order to choose a suitable chain model for the wrapped Floer cochain space for an exact cylindrical Lagrangian immersion, we consider the following geometric setup. We take a smooth Hamiltonian $H: M \to \mathbb{R}$ which is zero in the interior part $M_{0}$ of the Liouville manifold $M$. This in particular implies that all the Hamiltonian chords from $\iota(L)$ to itself which are contained in the interior part of $M$ are constant. Over the cylindrical end $\partial M \times [1, +\infty)$, more precisely on a smaller subset $\partial M \times [1 + \epsilon, +\infty)$ for some small $\epsilon > 0$, the Hamiltonian $H$ is quadratic, i.e. of the form $r^{2}$. \par
	Let $\mathcal{P}(M; \iota(L))$ be the space of paths in $M$ from $\iota(L)$ to itself. Fix a choice of a basepoint $x_{*}$ in every connected component of $\mathcal{P}(M; \iota(L))$. Let $x$ be a time-one $H$-chord from $\iota(L)$ to itself (either constant or non-constant), which lies in a connected component of $\mathcal{P}(M; \iota(L))$ where $x_{*}$ is located. A capping half-disk of $w$ with reference to $x_{*}$ is a map
\begin{equation}
w: [0, 1] \times [0, 1] \to M,
\end{equation}
such that $w(s, i) \in \iota(L)$ for $i = 0, 1$, and $w(0, t) = x_{*}(t), w(1, t) = x(t)$. Now if $x$ is a non-constant time-one $H$-chord contained in the cylindrical end $\partial M \times [1, +\infty)$, there is a unique homotopy class of capping half-disk for $x$, since $\iota$ is an embedding there and is exact in the usual sense. For such a Hamiltonian chord, it is not necessary to specify the homotopy class of capping half-disks, just as in the setup for an embedded exact Lagrangian submanifold. \par
	Because the Hamiltonian $H$ is constantly zero in the interior part, it is natural to introduce a Morse-Bott setup for wrapped Floer theory. There are many approaches, depending on the chain model for $M$. If we use the singular chain model for the complex which computes the cohomology of $M$, the transversality argument as in \cite{FOOO2} involves quite complicate process of choosing countably generated sub-complexes. To minimize the amount the work, we shall use the Morse complex for computing the cohomology of $M$, so that there are finitely many generators of the complex which computes the cohomology of $M$. Correspondingly, we define the wrapped Floer cochain space $CW^{*}(L, \iota; H)$, which is a graded $\mathbb{Z}$-module, as follows. \par

\begin{definition}\label{definition of wrapped Floer cochain space for a single Lagrangian immersion}
	Denote by $S(L, \iota)$ the set of pairs $(p_{-}, p_{+})$ so that $\iota(p_{-}) = \iota(p_{+})$ but $p_{-} \neq p_{+}$. Let $f$ be an auxiliary $C^{2}$-small Morse function on the fiber product $L \times_{\iota} L = \Delta_{L} \coprod S(L, \iota)$. \par
	We define the wrapped Floer cochain space $CW^{*}(L, \iota; H)$ to be the free $\mathbb{Z}$-module generated by the following two kinds of generators:
\begin{enumerate}[label=(\roman*)]

\item pairs $(p, w)$ where $p \in Crit(f)$ is a critical point of $f$, and $w$ is a $\Gamma$-equivalence class of capping half-disk for $p$, i.e. a map from a one-punctured disk to $M$ with boundary on $\iota(L)$ which converges to $p$ at the puncture (see \cite{FOOO1} Chapter 2 for the definition of $\Gamma$-equivalence);

\item non-constant time-one $H$-chords $x$ from $\iota(L)$ to itself contained in the cylindrical end $\partial M \times [1, +\infty)$. The set of these non-constant Hamiltonian chords is denoted by $\mathcal{X}_{+}(L, \iota; H)$.

\end{enumerate}

	For simplicity, we shall denote by $c$ any kind of generator.
\end{definition}

	The following lemma explains why this is a reasonable and meaningful definition. \par

\begin{lemma}\label{interior chords corresponds to Morse critical points on the fiber product}
	Let $K$ be a $C^{2}$-small generic perturbation of $H$ supported in the interior part $M_{0}$ of $M$, which is also assumed to be non-degenerate. Then all the $K$-chords starting from $\iota(L)$ to itself which are contained in the interior part of $M$ are constant and non-degenerate (hence isolated). Moreover, there is a natural one-to-one correspondence between these constant chords, and the critical points of the lift of $K$ to the fiber product $L \times_{\iota} L = \Delta_{L} \coprod S(L, \iota)$.
\end{lemma}
\begin{proof}
	Since $K$ is $C^{2}$-small in the interior part, all time-one $K$-chords from $\iota(L)$ to itself which are contained in that region are constant. Non-degeneracy follows from the assumption, which then implies that all such $K$-chords are isolated. \par
	Note that as in \eqref{decomposition of the self fiber product of Lagrangian immersion with transverse self-intersections}, the fiber product is the union of the diagonal $\Delta_{L} \cong L$ and the set of pairs $(p_{-}, p_{+})$, with $p_{-} \neq p_{+}$, each of which corresponds to a self-intersection $p$ of the immersion $\iota$. The lift of $K$ to this fiber product is simply the pullback of $K$ by $\iota$ on the main component $\Delta_{L}$, and constants on each isolated point $(p_{-}, p_{+})$ with value equal to $K(p)$ where $p = \iota(p_{-}) = \iota(p_{+}) \in M$. Since all the interior $K$-chords from $\iota(L)$ to itself are constant, every such constant chord $x$ has image being a point $p$ on $\iota(L)$. If $p$ is not a self-intersection point of $\iota$, then it satisfies the condition that $dK_{\iota(L)}(p)=0$, which implies that it is a critical point of the lift of $K$ to the fiber product $L \times_{\iota} L$.
\end{proof}

	On the other hand, non-constant time-one $H$-chords are located in the cylindrical end $\partial M \times [1, +\infty)$, and correspond to Reeb chords on the contact boundary $\partial M$, from the Legendrian submanifold $\iota(L) \cap \partial M$ to itself. Thus Definition \ref{definition of wrapped Floer cochain space for a single Lagrangian immersion} is a good definition. \par

\subsection{Pearly trees}
	The moduli spaces we use to set up wrapped Floer theory for a cylindrical Lagrangian immersion $\iota: L \to M$ combine two geometric configurations: both pseudoholomorphic disks and gradient flow trees. Moreover, there are two different kinds of disks - both inhomogeneous pseudoholomorphic disks and homogeneous pseudoholomorphic disks. To describe elements in the moduli spaces, we introduce the following objects as the underlying domains of "pseudoholomorphic maps" to $M$. \par

\begin{definition}
	A colored rooted tree with $k$-leaves ($k \ge 0$) consists of the following data:
\begin{enumerate}[label=(\roman*)]

\item a planar oriented metric ribbon tree $(T, V, E, r)$ with $k+1$ ends, where $V$ is the set of vertices, $E$ is the set of edges, and $r: E \to \mathbb{R}_{\ge 0} \cup \{\infty\}$ is a length function;

\item a decomposition of the set $E$ of edges into the set of exterior edges $E_{ext}$ and the set of interior edges $E_{int}$, such that $E_{ext}$ consists of $k+1$ semi-infinite edges corresponding to the $k+1$ ends: one is called the root $e_{0}$ and the other $k$ are called the leaves $e_{1}, \cdots, e_{k}$, while the interior edges are finite;

\item a coloring $c: V \to \{0, 1\}$ and a coloring $d: E_{ext} \to \{0, 1\}$, which satisfy the property that if $d(e) = 1$ for an enterior edge $e$, then its endpoint must have color $1$.

\end{enumerate}
\end{definition}

	Note that we should also allow one exceptional case: $T$ has no vertices and only one edge $e$ which is infinite in both directions, with color $d(e) = 0$. This does not quite fit into the definition of a colored rooted tree, but we shall still call it one. This infinite edge $e$ should be thought of as joining the root and one leaf together, so that it also comes with a preferred orientation. \par
	The orientation on $T$ induces an orientation on every interior edge $e \in E_{int}$, so the two endpoints of each edge can be naturally distinguished - one is called the source, denoted by $s(e)$, the other is called the target, denoted by $t(e)$. On the other hand, exterior edge are semi-infinite, and the root $e_{0}$ has only the target $t(e_{0})$ as its endpoint, while each $e_{i}$ has only the source $s(e_{i})$. \par

\begin{remark}
	Note that in our definition, we do not require the valency $val(v)$ of a vertex $v$ to be greater than or equal to $2$. In fact, the presence of vertices $v$ with $val(v) = 1$ will be important in the story.
\end{remark}

\begin{definition}
	A colored rooted tree $T$ is called admissible, if it is obtained from a colored rooted tree $T_{0}$ whose vertices all have color $1$ by attachment of colored rooted trees $T_{j}$ whose vertices all have color $0$. These $T_{j}$'s are attached to $T_{0}$ by edges $e_{j}$ (not the leaves) whose color are $0$. Moreover, $T_{0}$ and $T_{j}$'s are subtrees of $T$.
\end{definition}

	For our purpose of setting up wrapped Floer theory, we shall consider only admissible colored rooted trees, and simply call them colored rooted trees by abuse of name. \par
	Given a colored rooted tree with $k$-leaves as above, we can construct from it a topological space $S_{T}$ canonically in the following way. \par
	To every vertex $v$, we assign a punctured disk $S_{v} = D \setminus \{z_{v, 0}, \cdots, z_{v, val(v)-1}\}$, where each puncture $z_{v, j}$ corresponds to an edge adjacent to $v$. \par
	To every interior edge $e$, we assign a finite interval $I_{e}$ of length $r(e)$, joining the two disks (possibly with punctures) associated to $s(e), t(e)$ at the punctures on $S_{s(e)}$ and $S_{t(e)}$ which correspond to $e$. The length $r(e)$ is allowed to be zero, in which case $I_{e}$ topologically becomes a point, but we still think of $e$ as an edge combinatorially. \par
	To every exterior edge $e$ with $d(e) = 0$, we assign a semi-infinite interval $I_{e} = (-\infty, 0]$ if $e$ is the root $e_{0}$, or $I_{e} = [0, +\infty)$ if $e$ is any of the leaves $e_{i}, i = 1, \cdots, k$. The semi-infinite interval $I_{e}$ is attached at $\{0\}$ to the corresponding puncture on $S_{s(e)}$ or $S_{t(e)}$.  \par
	Finally, to every exterior edge $e$ with $d(e) = 1$, we assign a semi-infinite strip $Z_{e} = (-\infty, 0] \times [0, 1]$ if $e$ is the root $e_{0}$, or $Z_{e} = [0, +\infty) \times [0, 1]$ if $e$ is any of the leaves $e_{i}$. This semi-infinite strip should be identified with a strip-like end near the corresponding puncture on $S_{s(e)}$ or $S_{t(e)}$. \par
	The topological space $S_{T}$ is the union of all the above configurations, which are glued together according to the combinatorial data of the tree $T$. \par
	One special case is when $T$ has only one vertex $v$ one root $e_{0}$ and one leaf $e_{1}$ with colors $d(e_{0}) = d(e_{1}) = 1$. In this case, $S_{T}$ should be a disk with two boundary punctures, which is identified with $Z = \mathbb{R} \times [0, 1]$. \par
	In the exceptional case, i.e. when $T$ has no vertices and only one edge $e$ which is infinite in both directions, we assign $S_{T} = I_{e} = \mathbb{R}$. \par
	In order for such geometric objects $S_{T}$ to have a reasonable moduli problem, we should then equip $S_{T}$ with an additional structure - a complex structure $j_{v}$ on each disk component $D_{v}$. We briefly denote that by $j$. We call $(S_{T}, j)$ or simply $S_{T}$ a (rooted) pearly tree with $k$-leaves. However, these pearly trees do not have a moduli space because we do not impose stability conditions at the moment. \par
	Pearly trees will be the underlying domains of the pseudoholomorphic maps in wrapped Floer theory of cylindrical Lagrangian immersions. However, they are not enough, as positive-dimensional families of pseudoholomorphic maps can degenerate to broken pseudoholomorphic maps. To describe those, we introduce broken colored rooted trees as well as broken pearly trees. \par
	We need some terminology when talking about degeneration of colored rooted trees. When the length of an interior edge $e$ of a colored rooted tree $T$ tends to infinity, we obtain a pair of colored rooted trees $T_{0}, T_{1}$, so that $e$ breaks into a new leaf $e_{0, new}$ of $T_{0}$ and the root $e_{1, 0}$ of $T_{1}$. In such a picture, we say that the pair of root and leaf $e_{1, 0}, e_{0, new}$ is connected at infinity. \par
	Now we formalize the definition of a broken colored rooted tree. \par

\begin{definition}
	A broken colored rooted tree is a tuple $(T_{0}, \cdots, T_{m})$, where each $T_{i}$ is an admissible colored rooted tree, such that it satisfies the following conditions:
\begin{enumerate}[label=(\roman*)]

\item (rooting) The root $e_{0, 0}$ of $T_{0}$ is not connected at infinity to any leaf of any $T_{i}, i \neq 0$.

\item (ordering) For each $j \neq 0$, there is a unique $l(j)$ (which can be $0$) such that the root $e_{j, 0}$ of $T_{j}$ is connected at infinity to some (unique) leaf $e_{l(j), p(j)}$ of $T_{l(j)}$.

\item (compatible coloring) The root $e_{j, 0}$ of $T_{j}$ and the leaf $e_{l(j), p(j)}$ of $T_{l(j)}$ which are connected at infinity should have the same coloring, i.e. $d(e_{j, 0}) = d(e_{l(j), p(j)})$.

\end{enumerate}
\end{definition}

	Given a broken colored rooted tree $(T_{0}, \cdots, T_{m})$ as above, as well as $m$ positive real numbers $\rho_{1}, \cdots, \rho_{m}$, we may perform a gluing construction as follows. For each $T_{j}, j \neq 0$ and the corresponding $l(j)$, recall that we have identifications
\begin{equation}
\begin{split}
&e_{j, 0} \cong (-\infty, 0],\\
&e_{l(j), p(j)} \cong [0, +\infty),
\end{split}
\end{equation}
cut off $(-\infty, -\rho_{j}/2]$ from $e_{j, 0}$ and $[\rho_{j}/2, +\infty)$ from $e_{l(j), p(j)}$, and glue the remaining intervals at the endpoints $\{-\rho_{j}/2\} \sim \{\rho_{j}/2\}$. We may suitably reparametrize the interval so it has a nicer form, but that is not important; the only important information is that the resulting edge has length $\rho_{j}$. After doing this process for all $j = 1, \cdots, m$, we obtain a colored rooted tree
\begin{equation}
T = \sharp_{\rho_{1}, \cdots, \rho_{m}}(T_{0}, \cdots, T_{m}).
\end{equation}
If the resulting colored rooted tree $T$ is admissible, we call this an admissible gluing, and call $(T_{0}, \cdots, T_{m})$ an admissible broken colored rooted tree. From now on we shall only consider admissible broken colored rooted trees, and call them broken colored rooted trees for simplicity. Partial gluings are also allowed, which again give us broken colored rooted trees. It can be defined in a similar way, but the gluing process is only done for a sub-collection of edges connected at infinity. Let $J \subset \{1, \cdots, m\}$ index such a sub-collection, and we denote the result of partial gluing by
\begin{equation}
\sharp_{\rho_{j}: j \in J}(T_{0}, \cdots, T_{m}).
\end{equation}\par
	Similar to the case of a colored rooted tree, we can assign to a broken colored rooted tree a topological space as follows. \par

\begin{definition}
	A broken pearly tree $(S_{T_{0}}, \cdots, S_{T_{m}})$ associated to a broken colored rooted tree $(T_{0}, \cdots, T_{m})$ is simply the union of pearly trees $S_{T_{i}}$ associated to each component.
\end{definition}

	As the underlying broken colored rooted tree $(T_{0}, \cdots, T_{m})$ can be glued root-to-leaf in an admissible way, we can also glue the associated broken pearly tree to get a pearly tree. There are two cases. If $d(e_{j, 0}) = d(e_{l(j), p(j)}) = 0$, the $0$-th end $\epsilon_{j, 0}$ of $S_{T_{j}}$ is the negative half-ray $(-\infty, 0]$ and the $p(j)$-th end $\epsilon_{l(j), p(j)}$-th is the positive half-ray $[0, +\infty)$. In this case perform the gluing in the same way as we have done for the underlying trees. If $d(e_{j, 0}) = d(e_{l(j), p(j)}) = 1$, the $0$-th end $\epsilon_{j, 0}$ of $S_{T_{j}}$ is the negative infinite half-strip $(-\infty, 0] \times [0, 1]$ and the $p(j)$-th end $\epsilon_{l(j), p(j)}$-th is the positive infinite half-strip $[0, +\infty) \times [0, 1]$. We cut off $(-\infty, \rho_{j}/2] \times [0, 1]$ from $\epsilon_{j, 0}$ and $[\rho_{j}/2, +\infty) \times [0, 1]$ from $\epsilon_{l(j), p(j)}$, then glue the resulting finite strips along the boundary intervals $\{-\rho_{j}/2\} \times [0, 1] \sim \{\rho_{j}/2\} \times [0, 1]$. Doing this process for all $j$, we obtain a pearly tree denoted by
\begin{equation}
S_{T} = \sharp_{\rho_{1}, \cdots, \rho_{m}}(S_{T_{0}}, \cdots, S_{T_{m}}).
\end{equation}
Also, we can perform partial gluing in a similar way as we have done for broken colored rooted trees. \par

\subsection{Moduli spaces of stable pearly trees} \label{section: moduli space of disks bounded by immersed Lagrangian submanifolds}
	The moduli spaces involved in wrapped Floer theory for the Lagrangian immersion $\iota: L \to M$ are analogues and modifications of those used by \cite{Bourgeois-Oancea} to set up linearized contact homology in Hamiltonian formulation, without circle action and symmetry in our case. In addition, the inhomogeneous pseudoholomorphic disks in the interior part which have asymptotic limits being the constant chords are also in consideration. \par
	Two issues bring up complication in the construction of the moduli spaces. Unlike the case of an embedded exact Lagrangian submanifold, the image $\iota(L)$ in general bounds $J$-holomorphic disks, and limits of inhomogeneous pseudoholomorphic disks might bubble off homogeneous pseudoholomorphic disks with boundary on $\iota(L)$. These should be suitably packaged into the moduli spaces. On the other hand, the $A_{\infty}$-structure maps are typically defined by appropriate counts of inhomogeneous pseudoholomorphic disks; in particular, we expect the zeroth order map $m^{0}$ of the curved $A_{\infty}$-algebra structure to be defined by inhomogeneous pseudoholomorphic disks with one puncture. This causes some potential problems, in particular in the verification of $A_{\infty}$-relations, as the elements in some boundary strata do not satisfy Floer's equations because they arise from disk bubbling. \par
	In order to treat homogeneous and inhomogeneous pseudoholomorphic disks with one puncture in a uniform way, we have chosen our Hamiltonian $H$ to be constantly zero in the interior part of $M$, so that inhomogeneous pseudoholomorphic disks which are contained in the interior part automatically satisfy the homogeneous Cauchy-Riemann equation. Because of the chain model we pick for the wrapped Floer cochain space, we shall construct a version of Morse-Bott moduli spaces combining pseudoholomorphic disks and gradient flow trees. \par
	The issues can be resolved using the exactness condition for the Lagrangian immersion $\iota: L \to M$. The key properties about the behavior of pseudoholomorphic disks are given by the following lemmas. \par

\begin{lemma}\label{finiteness of pseudoholomorphic disks passing through a self-intersection point}
	For any compatible almost complex structure $J$ of contact type near the boundary $\partial M$, all $J$-holomorphic disks with boundary on $\iota(L)$ are contained in the interior part of $M$, and have to pass through a self-intersection point of $\iota: L \to M$. \par
	For any self-intersection point $(p_{-}, p_{+})$ ($p_{-} \neq p_{+}$) such that $\iota(p_{-}) = \iota(p_{+}) = p$, there are finitely many relative homology classes $\beta$ of $J$-holomorphic disks bounded by $\iota(L)$, with one boundary marked point mapped to the point $p$.
\end{lemma}
\begin{proof}
	The first statement follows from the assumption that $\iota$ is an embedding over the cylindrical end $\partial M \times [1, +\infty)$ of an exact Lagrangian submanifold of the form $l \times [1, +\infty)$, using a standard argument by the maximum principle. \par
	The second statement follows from the exactness condition. For any such a $J$-holomorphic disk $u$, the exactness condition $df = \iota^{*} \lambda_{M}$ implies that its energy is fixed:
\begin{equation*}
E(u) = f(p_{-}) - f(p_{+}),
\end{equation*}
by integration by parts. On the other hand, the energy is also equal to $\omega(\beta)$, which implies that there can only be finitely many such homology classes $\beta$. \par
\end{proof}

	The next three lemmas all follow from maximum principle. \par

\begin{lemma}
	Let $u: D \to M$ be a $J$-holomorphic disk with boundary on $\iota(L)$. Then $u(D)$ cannot be entirely contained in the cylindrical end $\partial M \times [1, +\infty)$.
\end{lemma}

\begin{lemma}
	Let $u: S \to M$ be a $J$-holomorphic curve with boundary on $\iota(L)$, for any smooth Riemann surface $S$ with boundary $\partial S \neq \varnothing$. If $u(\partial D)$ is contained in a compact subset of the interior part of $M$ away from $\partial M$, then the image of the entire disk $u(D)$ must be contained in the interior part of $M$ away from $\partial M$.
\end{lemma}

\begin{lemma}
	Let $u: S \to M$ be a smooth map from a one-punctured disk to $M$ with boundary on $\iota(L)$, which satisfies the Floer's equation
\begin{equation}
(du - \gamma \otimes X_{H})^{0, 1} = 0.
\end{equation}
If $u$ asymptotically converges to either a constant $H$-chord, or a self-intersection point of $\iota$ at the puncture, then $u(S)$ must be entirely contained in the interior part of $M$, where the Hamiltonian $H$ vanishes. In particular, $u$ has a natural smooth extension to the closed disk $D$, and satisfies the $J$-holomorphic curve equation
\begin{equation}
(du)^{0, 1} = 0
\end{equation}
in the interior of $D$.
\end{lemma}

	Now let us proceed to describe elements in the moduli spaces. To write down the inhomogeneous Cauchy-Riemann equations, we shall need certain geometric data, e.g. Morse functions, Hamiltonians, and almost complex structures. Fix the original Hamiltonian $H$, and let $J = J_{L, \iota}$ be a compatible almost complex structure of contact type. The various geometric data needed for writing down the relevant equations are packaged in the following way. \par

\begin{definition}
	A Floer datum on a pearly tree $S_{T}$ consists of the following data:
\begin{enumerate}[label=(\roman*)]

\item A time-shifting function $\rho_{S_{T}}: \partial S_{T} \to [1, +\infty)$, where $\partial S_{T}$ denotes the union of boundary components of the disk components of $S_{T}$, as well as the intervals. $\rho_{S_{T}}$ should be equal to a constant over the strip-like end near each puncture.

\item For each vertex $v$ with $c(v) = 0$, a constant family of almost complex structures $J_{S_{v}} = J$.

\item For each vertex $v$ with $c(v) = 1$ and $val(v) = 1$, a constant family of Hamiltonians $H_{S_{v}} = H$ and a constant family of almost complex structures $J_{S_{v}} = J$.

\item For each vertex $v$ with $c(v) = 1$ and $val(v) = 2$, a time-dependent family of almost complex structures $J_{v, t}$, rescaled by weight $w_{v, 0} = w_{v, 1}$. This is a time-dependent perturbation of $J$ in the class of almost complex structures of contact type. Moreover, we require that the choices $J_{v, t} = J_{t}$ be the same for all such vertices $v$, but possibly rescaled by differente weights according to the values of $\rho_{S_{T}}$.

\item For each vertex $v$ with $c(v) = 1$ and $val(v) \ge 3$, a domain-dependent family of Hamiltonians $H_{S_{v}}$ and a domain-dependent family of almost complex structures $J_{S_{v}}$, such that $H_{S_{v}}$ and $J_{S_{v}}$ asymptotically agrees with $H$ and respectively $J_{t}$ rescaled by weight $w_{v, j}$, over the strip-like end near each puncture $z_{v, j}$. Moreover, we require that the family of Hamiltonians $H_{S_{v}}$ be a compactly-supported domain-dependent perturbation of $H$, i.e. $H_{S_{v}} = H$ in a neighborhood of the boundary $\partial S_{v}$.

\item For each interior edge $e$, an $s$-dependent family of Morse functions $f_{e, s}$ on $I_{e}$.

\item For each exterior edge $e$ with $d(e) = 0$, a family of Morse functions $f_{e, s}$ on $L$ parametrized by $s \in I_{e}$, which agrees with $f$ for $|s| \gg 0$.

\item For each exterior edge $e$ with $d(e) = 1$, a family of time-dependent Hamiltonians $H_{e, s, t} = H_{e, t}$, and a family of time-dependent almost complex structures $J_{e, s, t} = J_{e, t}$ which agree with $H$ and respectively $J$ for $|s| \gg 0$. Moreover, as the semi-infinite strips are glued to the punctured disk $S_{s(e)}$ or $S_{t(e)}$, we require that these data extend smoothly over the glued domain.

\end{enumerate}

	In the exceptional case where $T$ does not have vertices and has a single infinite edge $e$ of color $d(e) = 0$, so that $I_{e} = \mathbb{R}$, we require $f_{e, s} = f$ for all $s$.

\end{definition}

	In order to define the $A_{\infty}$-structures, a necessary condition is to make sure that the Floer data chosen for various on various pearly trees satisfy certain consistency conditions. \par
	The simplest consistency condition to state is when two colored rooted trees $T_{1}, T_{2}$ are glued together root-to-leaf in a way that the resulting colored rooted tree $T = T_{1} \sharp_{0, i, \rho} T_{2}$ is still admissible. The consistency condition means that the Floer datum on $S_{T}$ is obtained from gluing the Floer data on $S_{T_{1}}$ and $S_{T_{2}}$. There are higher consistency conditions, which we refer the reader to \cite{Seidel} for a detailed description. \par


\begin{remark}
	Such consistency conditions make sense because the Floer datum is designed so that perturbations are not put on the punctured disks $S_{v}$ for vertices $v$ with $c(v) = 0$, that is, the family of almost complex structure $J_{S_{v}}$ on such a component is constant equal to $J$. \par
	Furthermore, a consistent choice of Floer data is not for the purpose of achieving transversality, even the choices are "generic" for all the components that allow non-constant families of perturbations of Hamiltonians and almost complex structures. Such a consistent choice is made mainly to ensure that the moduli spaces of pearly trees have good compactifications so that the boundary strata of the compactifications are products of moduli spaces of the same type, as we shall see below.
\end{remark}

	Having made a consistent choice of Floer data as above, we now describe elements in the moduli spaces. \par

\begin{definition}\label{stable pearly tree map}
	Let $I \subset \{0, \cdots, k\}$ be a subset. Let $\alpha: I \to S(L, \iota)$ be a map, labeling those marked points which are mapped to some self-intersection point of $\iota$. A stable pearly tree map is a triple $(S_{T}, u, l)$ satisfying the following conditions:

\begin{enumerate}[label = (\roman*)]

\item $S_{T}$ is a pearly tree modeled on a colored rooted tree $T$ with $k$-leaves.

\item $u: S_{T} \to M$ is a continuous map.

\item For a vertex $v$ with $c(v) = 0$, let $u_{v}$ be the restriction of $u$ to the punctured disk $S_{v}$ associated to $v$. Then $u_{v}$ satisfies the homogeneous Cauchy-Riemann equation
\begin{equation} \label{homogeneous Cauchy-Riemann equation on the disk components}
(du_{v})^{0, 1} = 0,
\end{equation}
with respect to the family $J_{S_{v}}$ of almost complex structures. In case $val(v) = 2$, $J_{S_{v}} = J_{t}$.

\item For a vertex $v$ with $c(v) = 1$, let $u_{v}$ be the restriction of $u$ to the punctured disk $S_{v}$ associated to $v$. Then $u_{v}$ satisfies the inhomogeneous Cauchy-Riemann equation
\begin{equation} \label{inhomogeneous Cauchy-Riemann equation on disk components}
(du_{v} - \gamma_{v} \otimes X_{H_{S_{v}}})^{0, 1} = 0.
\end{equation}

\item For an interior edge $e$, let $u_{e}$ be the restriction of $u$ to the interval $I_{e}$ associated to $e$. Then $u_{e}$ comes with a preferred lift $\tilde{u}_{e}: I_{e} \to L$ and satisfies the gradient flow equation:
\begin{equation} \label{gradient flow equation on interval components}
\frac{d\tilde{u}_{e}}{ds} + \nabla f_{e, s}(\tilde{u}_{e}) = 0.
\end{equation}
In the exceptional case where $S_{T} = \mathbb{R}$, this gradient flow equation takes the form
\begin{equation}
\frac{d\tilde{u}_{e}}{ds}+ \nabla f(\tilde{u}_{e}) = 0,
\end{equation}
by our choice of Floer datum.

\item $u(z) \in \phi_{M}^{\rho_{S_{T}}(z)} \iota(L)$, for $z \in \partial S_{T}$.

\item $l: \partial S_{T} \to L$ is a continuous map, which specifies the boundary lifting condition, i.e. $\iota \circ l = u|_{\partial S_{T}}$.

\item If $e$ is an exterior edge such that $I_{e}$ is either half-infinite or $\mathbb{R}$, then $\lim\limits_{s \to \pm \infty} \tilde{u}_{e}(s, \cdot) = p$ for some critical point $p$ of $f$.

\item If $e_{i}$ is the $i$-th exterior edge with color $d(e_{i}) = 1$, $\lim\limits_{s \to \pm \infty} u_{v(e_{i})} \circ \epsilon_{i}(s, \cdot) = \phi_{M}^{w_{i}}x_{i}(\cdot)$, where $v(e_{i})$ is the endpoint of $e$, and $\epsilon_{i}$ is the strip-like end associated to $e_{i}$. Here $x_{i}$ is some non-constant time-one $H$-chord from $\iota(L)$ to itself.

\item $(\lim\limits_{\theta \uparrow 0} l(e^{\sqrt{-1}\theta}\zeta_{i}), \lim\limits_{\theta \downarrow 0} l(e^{\sqrt{-1}\theta}\zeta_{i})) = \alpha(i) \in S(L, \iota)$ for $i \in I$. In addition, $\iota(\alpha(i)) = \tilde{u}_{e_{i}}(z_{v(e_{i}), j(v(e_{i})})$. Here $\tilde{u}_{e_{i}}(z_{v(e_{i}), j(v(e_{i}))})$ means the asymptotic value of $\tilde{u}_{e_{i}}$ at the end of the half-infinite ray associated to the $i$-th exterior edge $e_{i}$, where $z_{v(e_{i}), j(v(e_{i}))}$ is the puncture on the punctured disk $S_{v(e_{i})}$ that corresponds to the exterior edge $e_{i}$.

\item If $d(e) = 0$, then the preferred lift $\tilde{u}_{e}$ of $u_{e}$ should be compatible with the restriction of $l$ to $I_{e}$, such that the above condition hold. That is, if $\tilde{u_{e}}$ is a constant map to a component of $S(L, \iota)$, i.e. a discrete point, then the restriction of $l$ to $I_{e}$ should satisfy the same limiting condition in $(x)$.

\item The homology class of $u$ together with its asymptotic non-constant Hamiltonian chords oppositely oriented, is $\beta \in H_{2}(M, \iota(L); \mathbb{Z})$.

\end{enumerate}
\end{definition}

	The last condition in the definition of a stable pearly tree map needs some explanation. Since some of the asymptotic convergence conditions are non-constant Hamiltonian chords, which might be non-contractible relative $\iota(L)$, the map $u$ does not define a homology class in $H_{2}(M, \iota(L); \mathbb{Z})$. However, if we compactify $S_{T}$ at infinity by adding $\{\pm \infty\} \times [0, 1]$ to all exterior strip-like ends, and glue the oppositely oriented Hamiltonian chords to the map $u$, then the resulting map defines a homology class in $H_{2}(M, \iota(L); \mathbb{Z})$.
\par
	There is an obvious notion of two triples $(S_{T}, u, l)$ and $(S'_{T'}, u', l')$ being isomorphic. First of all, there should be an isomorphism of rooted colored trees $\bar{\phi}: T \to T'$, an isomorphism $\phi: S_{T} \to S'_{T'}$ compatible with $\bar{\phi}$, such that $u' \circ \phi = u$, and the pullback of the Floer datum chosen for $S'_{T'}$ by $\phi$ agrees with that for $S_{T}$. \par
	Denote by $c$ either a pair $(p, w)$, where $p$ is a critical point of $f$ and $w$ is a homotopy class of capping half-disk, or a pair $(x, w)$, where $x$ is a non-constant $H$-chord, and $w$ is a homotopy class of capping half-disk with reference to some chosen based point in the connected component of $x$ in the space $\mathcal{P}(M; \iota(L))$ of paths in $M$ from $\iota(L)$ to itself. Let $\mathcal{M}_{k+1}(\alpha, \beta; J, H; c_{0}, \cdots, c_{k})$ be the moduli space of the above triples $(S_{T}, u, l)$ which satisfy the convergence conditions to $C_{0}, \cdots, c_{k}$ at the root $e_{0}$ and leaves $e_{1}, \cdots, e_{k}$. Clearly, there is a decomposition according to the combinatorial type of $T$:
\begin{equation}
\mathcal{M}_{k+1}(\alpha, \beta; J, H; c_{0}, \cdots, c_{k}) = \coprod_{T} \mathcal{M}_{T}(\alpha, \beta; J, H; c_{0}, \cdots, c_{k}).
\end{equation} 
This description is only for the purpose of visualizing elements in the moduli space, but not for the purpose of proving transversality results by induction on combinatorial types. \par

\subsection{Compactification: stable broken pearly trees}
	The moduli space of stable pearly trees $\mathcal{M}_{k+1}(\alpha, \beta; J, H; c_{0}, \cdots, c_{k})$ is generally non-compact, from which one cannot expect to extract invariants. In order to define the desired curved $A_{\infty}$-algebra whose structure constants are given by appropriate counts of elements in the moduli space, we need to compactify it. \par

\begin{lemma}\label{compactness and Hausdorffness}
	The moduli space $\mathcal{M}_{k+1}(\alpha, \beta; J, H; c_{0}, \cdots, c_{k})$ has a natural compactification $\bar{\mathcal{M}}_{k+1}(\alpha, \beta; J, H; c_{0}, \cdots, c_{k})$, which has a natural topology being compact and Hausdorff.
\end{lemma}

\begin{lemma}\label{boundary strata in the moduli space}
	The compactification $\bar{\mathcal{M}}_{k+1}(\alpha, \beta; J, H; c_{0}, \cdots, c_{k})$ consists of products of moduli spaces $\bar{\mathcal{M}}_{k_{i}+1}(\alpha_{i}, \beta_{i}; J, H; c'_{i, 0}, \cdots, c'_{i, k_{i}})$ of the same type. More precisely, the codimension-$m$ strata 
\begin{equation*}
	S_{m}\bar{\mathcal{M}}_{k+1}(\alpha, \beta; J, H; c_{0}, \cdots, c_{k}),
\end{equation*}
i.e. those points of codimension exactly $m$ consist of the union of fiber products of the form 
\begin{equation}\label{boundary strata of moduli space of pearly trees in the form of fiber products}
\bar{\mathcal{M}}_{k_{0}+1}(\alpha_{0}, \beta_{0}; J, H; c_{0, 0}, \cdots, c_{0, k_{0}}) \times \cdots \times \bar{\mathcal{M}}_{k_{m}+1}(\alpha_{m}, \beta_{m}; J, H; c_{m, 0}, \cdots, c_{m, k_{m}}),
\end{equation}
where $c_{j} = c_{i_{j}, k_{j}}$ for some $i_{j}, k_{j}$, for every $j = 0, \cdots, k$. Here fiber products are taken over a discrete set of generators, so they are written as products.
In fact, by the inductive construction, for any element in the codimension-$m$ strata, the underlying domain is a broken pearly tree $(S_{T_{0}}, \cdots, S_{T_{m}})$ consisting of exactly $m+1$ pearly trees. Denote by $u_{j}$ the restriction of the stable map on $S_{T_{j}}$. For each $j \neq 0$, there is a unique $l(j)$ for which $u_{j}$ converges over the $0$-th end $\epsilon_{j, 0}$ (corresponding to the root) of $S_{T_{j}}$ to the same generator as $u_{l(j)}$ converges over the end $\epsilon_{l(j), p(j)}$ (corresponding to some leaf) of $S_{T_{l(j)}}$. That is to say, $c_{j, 0} = c_{l(j), p(j)}$ for some $p(j)$.
\end{lemma}

\begin{proof}[Sketch of proof of Lemma \ref{compactness and Hausdorffness} and Lemma \ref{boundary strata in the moduli space}]
	The definition of the compactification is given in the statement of Lemma \ref{boundary strata in the moduli space}, which is valid because of the Gromov compactness theorem. This inductive fiber product structure makes sense, because the various factors satisfy an induction condition with respect to $(E, k)$ in lexicographic order, where $E$ the energy of the stable maps from broken pearly trees, and $k$ is the number of leaves. The fact that $\bar{\mathcal{M}}_{k+1}(\alpha, \beta; J, H; c_{0}, \cdots, c_{k})$ is compact follows from the standard maximum principle and Gromov compactness theorem, as explained below. \par
	The proof of the rest is basically standard, so we shall simply explain the crucial reason. Fixing the homology class $\beta$ provides a priori energy bound for all elements in the moduli space $\bar{\mathcal{M}}_{k+1}(\alpha, \beta; J, H; c_{0}, \cdots, c_{k})$, so that there are only finitely many types of product moduli spaces that appear in the boundary strata of the compactfication. Thus by the maximum principle, there is a compact subset $C$ of $M$ depending on $k, \alpha, \beta, J, H$ and the asymptotic convergence conditions $c_{0}, \cdots, c_{k}$ but not on individual maps in the moduli space, so that any element in $\bar{\mathcal{M}}_{k+1}(\alpha, \beta; J, H; c_{0}, \cdots, c_{k})$ is represented by a stable map whose image is contained in $C$. We may take $C$ to be some sub-level set $\{r \le A\}$ by possibly enlarging the subset. Then Gromov compactness theorem applies. \par
	Hausdorff-ness follows from the stability condition. For a detailed proof of Hausdorff-ness, we refer the reader to \cite{FOOO2}. \par
\end{proof}

	Note that $\bar{\mathcal{M}}_{k+1}(\alpha, \beta; J, H; c_{0}, \cdots, c_{k})$ naturally compactifies the Gromov bordification of the moduli space of inhomogeneous pseudoholomorphic disks with boundary on $\iota(L)$, encoding additional data that specify the boundary lifts to the preimage of the immersion, as introduced in \cite{Akaho-Joyce}. Here by the Gromov bordification, we mean the moduli space of broken stable inhomogeneous pseudoholomorphic disks as those in \cite{Abouzaid1} (see also \cite{Seidel} in the unwrapped setting), which in our case is not compact because the limit of a sequence of a broken stable inhomogeneous pseudoholomorphic disk can bubble off homogeneous pseudoholomorphic disks with one marked points (such disks necessarily pass through some self-intersection point of $\iota: L \to M$. In words, the above compactification is obtained by adding all such disks. \par

\subsection{Kuranishi structures on the moduli spaces of stable pearly tree maps} \label{section: Kuranishi structure on moduli spaces of stable pearly tree maps}

	Appropriate perturbation framework is required to prove transversality results for the moduli spaces 
\begin{equation*}
\bar{\mathcal{M}}_{k+1}(\alpha, \beta; J, H; c_{0}, \cdots, c_{k}),
\end{equation*}
so that we can extract algebraic structures from these moduli spaces. Because of the presence of disk bubbles, in particular those disks with one marked point, unless we impose further geometric assumptions, it seems difficult to use traditional transversality methods, which use generic choices of Floer data which depend underlying Riemann surfaces in the underlying moduli spaces of domains. However, the domains of those pseudoholomorphic disks with one marked point are unstable, which do not form a good moduli space. The usual inductive argument for a universal and consistent choice of Floer data would fail, as gluing in a disk with one marked point will decrease the number of marked points.
To overcome such difficulty, our strategy is to appropriately adapt the theory of Kuranishi structures \cite{Fukaya-Ono} \cite{FOOO1}, \cite{FOOO2}, \cite{FOOO3}, \cite{FOOO4} developed by Fukaya-Ono, Fukaya-Oh-Ohta-Ono. In our paper, nothing new about the theory of Kuranishi structures is invented; we shall only use their results to construct virtual fundamental chains of the relevant moduli spaces, so the argument is sketchy leaving technical details these monumental works. \par

\begin{proposition}\label{Kuranishi structure on the moduli space of pearly trees}
	The moduli space $\bar{\mathcal{M}}_{k+1}(\alpha, \beta; J, H; c_{0}, \cdots, c_{k})$ has an oriented Kuranishi structure with boundary and corners. Moreover, the induced Kuranishi structure on the codimension-$m$ stratum $S_{m}\bar{\mathcal{M}}_{k+1}(\alpha, \beta; J, H; c_{0}, \cdots, c_{k})$ agrees with the fiber product Kuranishi structures on the fiber products \eqref{boundary strata of moduli space of pearly trees in the form of fiber products}.
\end{proposition}

	Sometimes we shall call a Kuranishi structure with boundary and corners briefly a Kuranishi structure for short. The proof of this proposition is based on minor modification of the momentous machinery developed in \cite{FOOO2}, as the basic idea is the same despite the fact that the form of the moduli space presented is slightly different. As this is not the main focus of the current paper, we shall only outline the main steps. We expect to provide a more detailed treatment in the work on analytic setup of the currently used "Morse-Bott" type setup of wrapped Floer theory. \par
	The construction of a Kuranishi structure on $\bar{\mathcal{M}}_{k+1}(\alpha, \beta; J, H; c_{0}, \cdots, c_{k})$
is of an inductive manner, for which we outline the main steps as follows. For more detailed argument, in particular various estimates to make the gluing construction work, see Chapter 7 of \cite{FOOO2}. \par
	Given the geometric data $\alpha, \beta, J, H, c_{0}, \cdots, c_{k}$, there is a uniform upper bound on the energy of all elements in the moduli space, say bounded by $E$. The stable maps in 
\begin{displaymath}
\bar{\mathcal{M}}_{k'+1}(\alpha', \beta'; J, H; c'_{0}, \cdots, c'_{k'})
\end{displaymath}
which appears as a factor of the fiber product \eqref{boundary strata of moduli space of pearly trees in the form of fiber products} either have less energy, of fewer leaves. This fact allows one to build the Kuranishi structure by induction on lexicographic order of energy and number of leaves. \par
	To build Kuranishi charts, we first look at a subset 
\begin{equation}
	\bar{\mathcal{M}}_{k+1}^{reg}(\alpha, \beta; J, H; c_{0}, \cdots, c_{k}) \subset \bar{\mathcal{M}}_{k+1}(\alpha, \beta; J, H; c_{0}, \cdots, c_{k})
\end{equation}
consisting of elements whose underlying domain is a single pearly tree (these are called smooth points in the moduli space). \par

\begin{lemma}
	There exists a Kuranishi chart $(U(S_{T}, u, l), E_{(S_{T}, u, l)}, s_{(S_{T}, u, l)}, \Gamma_{(S_{T}, u, l)})$ at each point $[S_{T}, u, l] \in \bar{\mathcal{M}}_{k+1}^{reg}(\alpha, \beta; J, H; c_{0}, \cdots, c_{k})$, where the isotropy group $\Gamma_{(S_{T}, u, l)}$ is trivial.
\end{lemma}
\begin{proof}
	For each point $[S_{T}, u, l] \in \bar{\mathcal{M}}_{k+1}^{reg}(\alpha, \beta; J, H; c_{0}, \cdots, c_{k})$, the choice of Floer datum for $S_{T}$ defines a non-degenerate Fredholm problem, which is a combination (direct sum) of inhomogeneous Cauchy-Riemann operators and the linearized operators of the gradient flow equations. Let $\mathcal{B}$ be the Banach manifold of consisting maps $(S_{T}, u, l)$ (without satisfying the equations \eqref{inhomogeneous Cauchy-Riemann equation on disk components} and \eqref{gradient flow equation on interval components}) of suitable weighted Sobolev class $W^{1, p}$, and $\mathcal{E} \to \mathcal{B}$ be the Banach bundle whose fiber over $(S_{T}, u, l)$ is the space of one-forms on $S_{T}$ with values in $u^{*}TM$, which are of class $W^{0, p}_{loc}$. The equations \eqref{inhomogeneous Cauchy-Riemann equation on disk components} and \eqref{gradient flow equation on interval components} combine to define a section of this bundle, briefly denoted by $\bar{\partial}$. The linearization at $(S_{T}, u, l)$ is the previously mentioned Fredholm operator
\begin{equation}
D_{(S_{T}, u, l)}\bar{\partial}: T_{(S_{T}, u, l)}\mathcal{B} \to \mathcal{E}_{(S_{T}, u, l)}.
\end{equation}
Choose a finite-dimensional subspace $E_{(S_{T}, u, l)}$ of $\mathcal{E}_{(S_{T}, u, l)}$ containing the cokernel, so that
\begin{equation}
D_{(S_{T}, u, l)}\bar{\partial}: T_{(S_{T}, u, l)}\mathcal{B} \to \mathcal{E}_{(S_{T}, u, l)}/E_{(S_{T}, u, l)}
\end{equation}
is surjective. A careful choice of $E_{(S_{T}, u, l)}$, based on the unique continuation theorem of solutions to elliptic equations, can be made that every vector in $E_{(S_{T}, u, l)}$ is smooth and has compact support in the underlying domain $S_{T}$, meaning that the support of every vector in $E_{(S_{T}, u, l)}$ is contained in the union of the interiors of punctured disk components and the interiors of compact interval components. In fact, we may take $E_{(S_{T}, u, l)}$ so that every vector vanishes on the interval components, by choosing the families of Morse functions generic enough. However, we will not use this fact in our construction. \par
	Define $U((S_{T}, u, l))$ to be the space of nearby elements $(S'_{T'}, u', l') \in \mathcal{B}$ such that
\begin{equation}
\bar{\partial}((S'_{T'}, u', l')) \in E_{(S_{T}, u, l)},
\end{equation}
where $E_{(S_{T}, u, l)}$ is identified with a subspace of $\mathcal{E}_{(S'_{T'}, u', l')}$ by parallel transport along minimal geodesics. This will be a Kuranishi neighborhood at $[S_{T}, u, l]$, over which the obstruction space is $E_{(S_{T}, u, l)}$, so that the Kuranishi map is defined by $\bar{\partial}$. Regarding the isotropy group, we may take $\Gamma_{(S_{T}, u, l)} = \{1\}$ to be the trivial group. This choice is fine because of the following lemma. \par
\end{proof}

\begin{lemma}
	The automorphism group of a semistable pearly tree is torsion-free.
\end{lemma}

	Next, we shall extend the construction to other strata in the compactification 
\begin{displaymath}
\bar{\mathcal{M}}_{k+1}(\alpha, \beta; J, H; c_{0}, \cdots, c_{k}).
\end{displaymath}
We need to construct a Kuranishi chart at each point in 
\begin{displaymath}
\bar{\mathcal{M}}_{k+1}(\alpha, \beta; J, H; c_{0}, \cdots, c_{k}) \setminus \bar{\mathcal{M}}^{reg}_{k+1}(\alpha, \beta; J, H; c_{0}, \cdots, c_{k}).
\end{displaymath}\par

	This is standard application of gluing theorems, as explained in \cite{FOOO2}, which we briefly recall here. For simplicity, we consider the case of a pair of pearly tree maps $[(S_{T_{0}}, u_{0}, l_{0}), (S_{T_{1}}, u_{1}, l_{1})]$, which is an element in the codimension-one strata, such that the root $e_{1, 0}$ of $S_{T_{1}}$ can be glued to some leaf $e_{0, j}$ of $S_{T_{0}}$ in an admissible way, and the asymptotic convergence conditions of $u_{0}$ at $e_{0, j}$ and of $u_{1}$ at $e_{1, 0}$ agree. \par

\begin{lemma}
	At every such point $[(S_{T_{0}}, u_{0}, l_{0}), (S_{T_{1}}, u_{1}, l_{1})]$ in the moduli space $\bar{\mathcal{M}}_{k+1}(\alpha, \beta; J, H; c_{0}, \cdots, c_{k})$, there exists a Kuranishi chart 
\begin{equation*}
\begin{split}
&(U((S_{T_{0}}, u_{0}, l_{0}), (S_{T_{1}}, u_{1}, l_{1})), E_{[(S_{T_{0}}, u_{0}, l_{0}), (S_{T_{1}}, u_{1}, l_{1})]},\\
&s_{[(S_{T_{0}}, u_{0}, l_{0}), (S_{T_{1}}, u_{1}, l_{1})]}, \Gamma_{[(S_{T_{0}}, u_{0}, l_{0}), (S_{T_{1}}, u_{1}, l_{1})]}),
\end{split}
\end{equation*}
such that $\Gamma_{[(S_{T_{0}}, u_{0}, l_{0}), (S_{T_{1}}, u_{1}, l_{1})]}$ is trivial.
\end{lemma}
\begin{proof}
	 Consider the two Banach bundles $\mathcal{E}_{i} \to \mathcal{B}_{i}$, together with choices of obstruction spaces $E_{(S_{T_{i}}, u_{i}, l_{j})}$ making the linearized section surjective onto the quotient space $(\mathcal{E}_{i})_{(S_{T_{i}}, u_{i}, l_{j})}/(E_{i})_{(S_{T_{i}}, u_{i}, l_{j})}$ at the point $(S_{T_{i}}, u_{i}, l_{i})$. \par
	Consider the direct sum linearized operator
\begin{equation}
D_{(S_{T_{0}}, u_{0}, l_{0})}\bar{\partial}_{0} \oplus D_{(S_{T_{1}}, u_{1}, l_{1})}\bar{\partial}_{1}: T_{(S_{T_{0}}, u_{0}, l_{0})}\mathcal{B}_{0} \oplus T_{(S_{T_{1}}, u_{1}, l_{1})}\mathcal{B}_{1} \to (\mathcal{E}_{0})_{(S_{T_{0}}, u_{0}, l_{0})} \oplus (\mathcal{E}_{1})_{(S_{T_{1}}, u_{1}, l_{1})}.
\end{equation}
After projecting to the quotient spaces, we obtain a surjective Fredholm operator
\begin{equation}
\begin{split}
&D_{(S_{T_{0}}, u_{0}, l_{0}), (S_{T_{1}}, u_{1}, l_{1})}\bar{\partial}:
T_{(S_{T_{0}}, u_{0}, l_{0})}\mathcal{B}_{0} \oplus T_{(S_{T_{1}}, u_{1}, l_{1})}\mathcal{B}_{1}\\
&\to (\mathcal{E}_{0})_{(S_{T_{0}}, u_{0}, l_{0})}/(E_{0})_{(S_{T_{0}}, u_{0}, l_{0})} \oplus (\mathcal{E}_{1})_{(S_{T_{1}}, u_{1}, l_{1})}/(E_{1})_{(S_{T_{1}}, u_{1}, l_{1})}.
\end{split}
\end{equation}\par
	To make sure that the Kuranishi chart at the point $[(S_{T_{0}}, u_{0}, l_{0}), (S_{T_{1}}, u_{1}, l_{1})]$ is compatible with the Kuranishi charts at nearby points that we have previously constructed, we need the following argument. Let $(S_{T}, u', l)$ be the pre-gluing of $(S_{T_{0}}, u_{0}, l_{0})$ and $(S_{T_{1}}, u_{1}, l_{1})$ for some small gluing parameter $\rho$. Here by pre-gluing, we mean gluing the two maps together on the glued pearly tree using cut-off functions to obtain a continuous map, which might not satisfy the desired equations on the glued region. For the underlying pearly trees, the process of gluing is easily described. Regarding the map $u'$, it is the result of pre-gluing gradient flow lines if the common asymptote of $u_{0}$ and $u_{1}$ is a critical point of $f$, or the result of pre-gluing solutions to Floer's equation if the common asymptote of $u_{0}$ and $u_{1}$ is a time-one Hamiltonian chord for $H$. \par
	Using the cut-off functions in the above gluing process, we may define maps
\begin{equation}
I_{i}: \mathcal{E}_{(S_{T_{i}}, u_{i}, l_{i})} \to \mathcal{E}_{(S_{T}, u', l)},
\end{equation}
by cutting off the domain of the one-forms on $S_{T_{i}}$ and including the resulting domain to $S_{T}$. Recall that we have chosen the obstruction spaces $(E_{i})_{(S_{T_{i}}, u_{i}, l_{i})}$ so that every vector has compact support. Therefore we may choose $E_{(S_{T}, u', l)}$ carefully so that the restriction of sum of $I_{0}$ and $I_{1}$:
\begin{equation}
I_{0} + I_{1}: (E_{0})_{(S_{T_{0}}, u_{0}, l_{0})} \oplus (E_{1})_{(S_{T_{1}}, u_{1}, l_{1})} \to E_{(S_{T}, u', l)}
\end{equation}
is injective. In this way we may regard
\begin{equation}\label{direct sum of obstruction spaces}
(E_{0})_{(S_{T_{0}}, u_{0}, l_{0})} \oplus (E_{1})_{(S_{T_{1}}, u_{1}, l_{1})} \subset E_{(S_{T}, u', l)}.
\end{equation}
Consider maps $u: S_{T} \to M$ in a $C^{0}$-neighborhood of $u'$ of class $W^{1, p}$, which satisfy the following condition
\begin{equation}\label{Cauchy-Riemann equation with direct sum obstruction space}
\bar{\partial}u \equiv 0 \mod (E_{0})_{(S_{T_{0}}, u_{0}, l_{0})} \oplus (E_{1})_{(S_{T_{1}}, u_{1}, l_{1})},
\end{equation}
i.e. it satisfies gradient flow equation on interval components, $J$-holomorphic curve equation on disk components of color $0$, and Floer's equation on disk components of color $1$, modulo the obstruction subspace \eqref{direct sum of obstruction spaces}. \par
	Then the strategy for solving this equation \eqref{Cauchy-Riemann equation with direct sum obstruction space} is standard in the theory of pseudoholomorphic curves (perhaps in more general setup of Fredholm problems). First, we may find an approximate solution $\tilde{u}$ by using a cut-off function to interpolate any given two tangent vectors $V_{i} \in T_{(S_{T_{i}}, u_{i}, l_{i})}\mathcal{B}_{i}, i = 0, 1$, and use exponential map to get an element in $\mathcal{B}_{(S_{T}, u', l)}$. Then a standard implicit theorem argument (finding a right inverse to the linearized operator), together with exponential decay estimates as explained in detail in \cite{FOOO2}, \cite{FOOO3} provides an exact solution $u$ to the equation \eqref{Cauchy-Riemann equation with direct sum obstruction space}. This part is a combination of the standard proofs of gluing theorems in various setups: for gradient flow equations, for $J$-holomorphic curve equations, and for Floer's equations. For complete details see section 7.1 of \cite{FOOO2}, though the conditions are slightly different because we have fixed a set of generators as asymptotic convergence conditions for the stable pearly tree maps in consideration. \par
	Now we define
\begin{equation}
U((S_{T_{0}}, u_{0}, l_{0}), (S_{T_{1}}, u_{1}, l_{1}))
\end{equation}
to be the space of $(S_{T}, u, l)$ where $u$ is an exact solution of \eqref{Cauchy-Riemann equation with direct sum obstruction space} obtained from perturbing the approximate solution $\tilde{u}$, which is close to the original pre-gluing $u'$. The Kuranishi map $s_{[(S_{T_{0}}, u_{0}, l_{0}), (S_{T_{1}}, u_{1}, l_{1})]}$ is simply the restriction of $\bar{\partial}$ to $U((S_{T_{0}}, u_{0}, l_{0}), (S_{T_{1}}, u_{1}, l_{1}))$. Again, the isotropy group is taken to be the trivial group, $\Gamma_{[(S_{T_{0}}, u_{0}, l_{0}), (S_{T_{1}}, u_{1}, l_{1})]} = \{1\}$. \par
	 This completes the construction of a Kuranishi chart at a point in the codimension-one strata of $\bar{\mathcal{M}}_{k+1}(\alpha, \beta; J, H; c_{0}, \cdots, c_{k})$. For elements in strata of higher codimension, the construction can be done in the same way. \par
\end{proof}

	Now we shall gather all these Kuranishi charts and globalize the construction to a Kuranishi structure on the moduli space $\bar{\mathcal{M}}_{k+1}(\alpha, \beta; J, H; c_{0}, \cdots, c_{k})$. Briefly denote by $\sigma$ a point in the moduli space $\bar{\mathcal{M}}_{k+1}(\alpha, \beta; J, H; c_{0}, \cdots, c_{k})$, i.e. a stable broken pearly tree map. Since the moduli space is compact, we may choose finitely many elements, $\sigma_{1}, \cdots, \sigma_{K}$, so that their Kuranishi neighborhoods $U_{\sigma}$ cover the moduli space
\begin{equation}
\cup_{i} Int(U_{\sigma}) \supset \bar{\mathcal{M}}_{k+1}(\alpha, \beta; J, H; c_{0}, \cdots, c_{k}),
\end{equation}
where we take the Kuranishi neighborhoods $U_{\sigma}$ to be closed neighborhoods of $\sigma_{i}$ inside the Banach manifold $\mathcal{B}$. Now given any point $\sigma \in \bar{\mathcal{M}}_{k+1}(\alpha, \beta; J, H; c_{0}, \cdots, c_{k})$, we modify the previously chosen Kuranishi neighborhood $U(\sigma)$ and obstruction space $E_{\sigma}$ so that the new ones fit well into a global structure. \par
	First consider the case where the underlying broken pearly tree $\Sigma = (S_{T_{0}}, \cdots, S_{T_{m}})$ of $\sigma$ is stable. We define a new obstruction space
\begin{equation}\label{new obstruction space from a finite covering}
E'_{\sigma} = \oplus_{i: \sigma \in U_{\sigma_{i}}} E_{\sigma_{i}}.
\end{equation}
Also, we modify the previously chosen $U(\sigma)$ to a new Kuranishi neighborhood $U'_{\sigma}$ to be the set of isomorphism classes of $(S'_{T'}, u', l')$, where $S'_{T'}$ is close to $S_{T}$ in the moduli space of stable pearly trees (which is a smooth manifold with corners), and $u'$ satisfies
\begin{equation}\label{Cauchy-Riemann equation with the new obstruction space}
\bar{\partial}u' \equiv 0 \mod E'_{\sigma},
\end{equation}
while $l'$ is still simply a lifting condition compatible with the restriction of $u'$ to $\partial S'_{T'}$. \par
	We shall choose these $U'_{\sigma}$ carefully such that the following condition is satisfied:
\begin{equation}
\sigma' \in U'_{\sigma} \implies E'_{\sigma'} \xhookrightarrow{} E'_{\sigma'}.
\end{equation}
Here the inclusion map is defined using parallel transport with respect to the metric on the moduli space of underlying pearly trees, as well as the metric on $M$ determined by $\omega$ and $J$. For such choices, we may assume that whenever $\sigma' \in U'_{\sigma}$, there is a neighborhood $U'_{\sigma', \sigma}$ of $\sigma'$ in $U'_{\sigma'}$, together with a natural inclusion map $U'_{\sigma', \sigma} \xhookrightarrow{} U'_{\sigma}$. Such 
\begin{equation}
(U'_{\sigma}, E'_{\sigma}, s_{\sigma}, \Gamma_{\sigma} = \{1\})
\end{equation}
will be the desired Kuranishi chart, so that there are natural coordinate changes whenever $\sigma' \in U'_{\sigma}$ described as above. The compatibility of these coordinate changes with the Kuranishi maps $s_{\sigma}, s_{\sigma'}$ holds automatically, because the Kuranishi map is simply the restriction of the section $\bar{\partial}$ on $U'_{\sigma}$, as a subspace of a Banach manifold or (finite) direct products of Banach manifolds. \par

	Now consider the case where the underlying broken pearly tree $\Sigma = (S_{T_{0}}, \cdots, S_{T_{m}})$ of $\sigma$ is unstable. In this case, "nearby" pearly trees are not unique - they occur in positive dimensional families parametrized by
\begin{equation}\label{positive-dimensional symmetries parametrized by neighborhood in the Lie algebra}
\prod_{i} Lie(Aut(\Sigma_{i}))_{0},
\end{equation}
where $Lie(Aut(\Sigma_{i}))_{0}$ is a neighborhood of $0$ in the Lie algebra of the group of automorphisms of the underlying broken pearly trees $\Sigma_{i}$ of $\sigma_{i}$, where the product is taken over all unstable underlying pearly trees $\Sigma_{i}$ of $\sigma_{i}$, among the chosen $\sigma_{1}, \cdots, \sigma_{K}$. However, the equation \eqref{Cauchy-Riemann equation with the new obstruction space} is not invariant under the pullback by the automorphisms given by (exponentiating) elements in \eqref{positive-dimensional symmetries parametrized by neighborhood in the Lie algebra}. \par
	There are several ways of resolving this issue, among which we choose the method of "canonical gauge fixing", introduced in \cite{Fukaya-Ono}. That method has the advantage that we only have to deal with the fixed $\sigma_{i}$'s, instead of performing some construction for each element $\sigma$, e.g. stabilizing the domain $\Sigma$ of $\sigma$. \par
	The argument is as follows. For $i = 1, \cdots, K$, let $\Sigma_{i, un}$ be the union of unstable components of $\Sigma_{i}$. Note that $\Sigma_{i, un}$ is not a broken pearly tree and might not be connected. Let $\Sigma_{i, un}^{0}$ be $\Sigma_{i, un}$ with intervals $[\rho/2, +\infty), (-\infty, -\rho/2]$ or half-strips $[\rho/2, +\infty) \times [0, 1], (-\infty, -\rho/2] \times [0, 1]$ cut off from some exterior edges, so that we can identify $\Sigma_{i, un}^{0}$ as a subset of $\Sigma$ whenever $\sigma$ is close to $\sigma_{i}$, so that $\Sigma$ is obtained from partial gluing of the broken pearly tree $\Sigma_{i}$, with suitable deformation of the almost complex structure. \par
	Suppose $\sigma$ is close to $\sigma_{i}$ so that there is some gluing and deformation parameter $\xi = (\vec{\rho}, \lambda)$ and an isomorphism $\phi: \Sigma \to \Sigma_{i, \xi}$, where $\Sigma_{i, \xi}$ is obtained from $\Sigma$ by partial gluing with respect to $\vec{\rho}$ and deformation of complex structure by $\lambda$. Let $d^{2}: M \times M \to [0, \infty)$ be the Riemannian distance square function. Define a function on the space of such $(\xi, \phi)$ by
\begin{equation}
F_{i}(\xi, \phi) = \int_{x \in \Sigma_{i, un}^{0}} d^{2}(u_{\xi}(x), u \circ \phi^{-1}(x))dx,
\end{equation}
where $u_{\xi}: \Sigma_{i, \xi} \to M$ is a map on $\Sigma_{i, \xi}$ obtained from deforming $u$ to an exact solution of \eqref{Cauchy-Riemann equation with the new obstruction space}. Using this $F_{i}$, we define a function on the neighborhood in the Lie algebra, $Lie(Aut(\Sigma_{i}))_{0}$, by
\begin{equation}
F_{i}(v \xi, \phi), v \in Lie(Aut(\Sigma_{i}))_{0}.
\end{equation}
Using the convexity of the Riemannian distance square, we can prove that $F_{i}$ is convex with respect to $v \in Lie(Aut(\Sigma_{i}))_{0}$, i.e. the Hessian matrix with respect to $v$ is positive definite. Moreover, there is a uniform constant $c > 0$ such that the Hessian matrix is bigger than $cI$ uniformly, where $I$ is the identity matrix. Now consider $v \in \partial Lie(Aut(\Sigma_{i}))_{0}$ on the boundary of this neighborhood. In that case it is not difficult to find a biholomorphic map $\phi$ such that
\begin{equation}
F_{i}(v \xi, \phi) - F_{i}(\xi, \phi) > c, \forall v \in \partial Lie(Aut(\Sigma_{i}))_{0}.
\end{equation}
Thus we may choose $U'_{\sigma_{i}}$ small enough such that there is a unique $v \in Lie(Aut(\Sigma_{i}))_{0}$ such that $F_{i}(v \xi, \phi)$ has a unique local minimum with respect to $\phi$. Thus the gauge fixing condition can be chosen as follows:
\begin{equation}
F_{i}(v \xi, \phi) \ge F_{i}(\xi, \phi), \forall v \in Lie(Aut(\Sigma_{i}))_{0}.
\end{equation}
Requiring this condition uniquely determines $(\xi, \phi)$. \par
	After imposing this condition, for any such $\sigma$, the equation \eqref{Cauchy-Riemann equation with the new obstruction space} has no ambiguity for elements close to $\sigma$, so that we can use it to define the Kuranishi neighborhood $U'_{\sigma}$. This completes the construction of the desired Kuranishi chart $(U'_{\sigma}, E'_{\sigma}, s_{\sigma}, \Gamma_{\sigma} = \{1\})$, which satisfies the desired properties as argued in the case where $\Sigma$ is stable. \par
	The statement that the induced Kuranishi structure on $S_{m}\bar{\mathcal{M}}_{k+1}(\alpha, \beta; J, H; c_{0}, \cdots, c_{k})$ agrees with the fiber product Kuranishi structures on the fiber products \eqref{boundary strata of moduli space of pearly trees in the form of fiber products} follows immediately from the way in which the Kuranishi charts are constructed. \par

\subsection{Perturbations by multisections}
	An important geometric assumption for us is that the ambient symplectic manifold $M$ is exact, which allows us to take a single-valued multisection to the obstruction bundle (system) of the Kuranishi space $\bar{\mathcal{M}}_{k+1}(\alpha, \beta; J, H; c_{0}, \cdots, c_{k})$ as in this case the isotropy group of every Kuranishi chart is trivial. Also see \cite{FOOO4}, in which notions of stacks and piecewise smooth global sections are introduced in order to deal with the more general case of spherically positive manifolds. \par
	Going back to the construction of Kuranishi structures in the previous subsection, recall that we demand that the isotropy group $\Gamma$ for any Kuranishi chart on the moduli space $\bar{\mathcal{M}}_{k+1}(\alpha, \beta; J, H; c_{0}, \cdots, c_{k})$ be trivial. \par
	The way of achieving transversality in the theory of Kuranishi structures is to perturb the Kuranishi map by multisections locally transverse to zero, then glue (the zero loci of) these multisections together according to the compatibility conditions with respect to the coordinate changes. In this way, we can obtain a virtual fundamental chain on the moduli space, as some kind of C\v{e}ch-type chain. Precisely it should be a dual of C\v{e}ch cochain, such discussions will not be made or used in this paper. \par
	The following proposition can be essentially derived from \cite{FOOO2}. Moreover, we remark that both the statement and the proof here are much simpler and indeed simplified, because in \cite{FOOO2} they consider singular chains on the Lagrangian submanifolds, in whose intersection theory there are serious issues with transversality at the diagonal, forcing one to introduce a hierarchy of perturbations and choices of appropriate geometric chains to make the whole chain-level argument work. However, in our case we have chosen a different chain model involving a discrete set of generators satisfying certain non-degeneracy conditions (at the price of making the construction less canonical), and we have constructed Kuranishi structures depending on these generators as inputs and outputs, therefore there is not much trouble in arranging the perturbations by multisections to achieve transversality. \par

\begin{proposition} \label{compatible system of single-valued multisections}
	For all $k, \alpha, \beta$ and for each $c_{0}, \cdots, c_{k}$, such that the moduli spaces $\bar{\mathcal{M}}_{k+1}(\alpha, \beta; J, H; c_{0}, \cdots, c_{k})$ are non-empty as topological spaces (not after perturbation), there exist (systems of) single-valued multisections $s_{\alpha, \beta; J, H; c_{0}, \cdots, c_{k}}$ for the given Kuranishi structure on $\bar{\mathcal{M}}_{k+1}(\alpha, \beta; J, H; c_{0}, \cdots, c_{k})$, such that they are transverse to zero. \par
	In addition, these multisections are compatible with the fiber product multisections on the fiber product \eqref{boundary strata of moduli space of pearly trees in the form of fiber products}. This means that there are isomorphisms of multisections
\begin{equation}\label{fiber product multisections}
s_{\alpha, \beta; J, H; c_{0}, \cdots, c_{k}} \cong
s_{\alpha_{0}, \beta_{0}; J, H; c_{0, 0}, \cdots, c_{0, k_{0}}} \times \cdots \times
s_{\alpha_{m}, \beta_{m}; J, H; c_{m, 0}, \cdots, c_{m, k_{m}}},
\end{equation}
when restricting the multisection $s_{\alpha, \beta; J, H; c_{0}, \cdots, c_{k}}$ to that codimension-$m$ stratum \eqref{boundary strata of moduli space of pearly trees in the form of fiber products}.
\end{proposition}

	Note that the choices of these multisections depend on the inputs and outputs $c_{0}, \cdots, c_{k}$. Therefore, it is not necessary to deal with transversality issues for various evaluation maps as in the setup using singular chains as in \cite{FOOO2}, so the procedure described here is much simpler. In fact, the construction only uses abstract theory of Kuranishi structures, once we have constructed the Kuranishi structures on the moduli spaces $\bar{\mathcal{M}}_{k+1}(\alpha, \beta; J, H; c_{0}, \cdots, c_{k})$ so that the Kuranishi structures are compatible at the boundary with the fiber product Kuranishi structures on \eqref{boundary strata of moduli space of pearly trees in the form of fiber products}. \par

\subsection{The curved $A_{\infty}$-algebra associated to an exact cylindrical Lagrangian immersion} \label{curved A-infinity algebra associated to exact cylindrical Lagrangian immersions}
	Based on the discussion presented above, we may extract from the moduli spaces $\bar{\mathcal{M}}_{k+1}(\alpha, \beta; J, H; c_{0}, \cdots, c_{k})$ a structure of a curve $A_{\infty}$-algebra on the underlying wrapped Floer cochain space $CW^{*}(L, \iota; H)$. \par

\begin{proposition}
	A coherent choice of single-valued multisections on $\bar{\mathcal{M}}_{k+1}(\alpha, \beta; J, H; c_{0}, \cdots, c_{k})$ defines a structure of a curved $A_{\infty}$-algebra on $CW^{*}(L, \iota; H)$. Moreover, this curved $A_{\infty}$-algebra is independent of the choice up to homotopy. 
\end{proposition}
\begin{proof}[Sketch of proof]
	We sketch the main steps of the proof while referring the reader to techniques developed in \cite{FOOO1}, \cite{FOOO2}. The discussion in the previous subsections provides coherent system of Kuranishi structures on all the moduli spaces $\bar{\mathcal{M}}_{k+1}(\alpha, \beta; J, H; c_{0}, \cdots, c_{k})$. Also, by Proposition \ref{compatible system of single-valued multisections}, there exist single-valued multisections, which are compatible at the boundary strata with fiber product multisections. Let us fix such a coherent choice of single-valued multisections. \par
	To define the structure maps of the curved $A_{\infty}$-algebra, we shall consider only the rigid cases, which means that we only consider generators $c_{i}$ of appropriate degrees such that the virtual dimension of the moduli space $\bar{\mathcal{M}}_{k+1}(\alpha, \beta; J, H; c_{0}, \cdots, c_{k})$ is zero. The zero-sets of the chosen single-valued multisections define an integral virtual fundamental chain, which therefore gives us an integer number
\begin{equation*}
a_{\alpha, \beta; J, H; c_{0}, \cdots, c_{k}} \in \mathbb{Z}.
\end{equation*}
Then we set
\begin{equation}\label{A-infinity operations}
m^{k}(c_{k}, \cdots, c_{1}) = \sum_{\substack{\alpha, \beta, c_{0}\\ \deg(c_{0}) - \deg(c_{1}) - \cdots - \deg(c_{k}) = 2 - k}} a_{\alpha, \beta; J, H; c_{0}, \cdots, c_{k}} c_{0}.
\end{equation}
To make sense of the formula \eqref{A-infinity operations}, we must show that this sum is in fact finite.
If $k = 0$, the zeroth order map
\begin{equation*}
m^{0}(1) = \sum_{\substack{\alpha, beta, c_{0}\\ \deg(c) = 2}} a_{\alpha, \beta; J, H; c_{0}} c_{0}
\end{equation*}
"counts" inhomogeneous pseudoholomorphic disks with one marked point, which necessarily pass through some self-intersection point. That is, $c_{0}$ is of the form $(p, w)$. Then, Lemma \ref{finiteness of pseudoholomorphic disks passing through a self-intersection point} implies that there can only be finitely many such homology classes $\beta$ appearing so that the moduli space $\bar{\mathcal{M}}_{1}(\alpha, \beta; J, H; c_{0})$ is non-empty.
For $k \ge 1$, we can again use the action-energy identity, which holds because the Lagrangian immersions $\iota: L \to M$ is exact, to show that if the inputs $c_{1}, \cdots, c_{k}$ are fixed, there are only finitely many possible $c_{0}$ and finitely many possible homology classes $\beta$, for which the moduli space $\bar{\mathcal{M}}_{k+1}(\alpha, \beta; J, H; c_{0}, \cdots, c_{k})$ is non-empty. This implies that the sum \eqref{A-infinity operations} is a finite sum. Thus, it gives rise to a well-defined multilinear map
\begin{equation*}
m^{k}: CW^{*}(L, \iota; H)^{\otimes k} \to CW^{*}(L, \iota; H).
\end{equation*} \par
	To verify that these maps satisfy the $A_{\infty}$-equations, we need to study boundary strata of moduli spaces $\bar{\mathcal{M}}_{k+1}(\alpha, \beta; J, H; c_{0}, \cdots, c_{k})$ of virtual dimension one. Recall that the multisection $s_{\alpha, \beta; J, H; c_{0}, \cdots, c_{k}}$ is isomorphic to the fiber product multisections \eqref{fiber product multisections} on \eqref{boundary strata of moduli space of pearly trees in the form of fiber products}. In particular, if the virtual dimension of $\bar{\mathcal{M}}_{k+1}(\alpha, \beta; J, H; c_{0}, \cdots, c_{k})$ is one, the only boundary strata that we have to consider are of codimension one, and have the form
\begin{equation}
\bar{\mathcal{M}}_{k_{0}+1}(\alpha_{0}, \beta_{0}; J, H; c_{0, 0}, \cdots, c_{0, k_{0}}) \times \bar{\mathcal{M}}_{k_{1}+1}(\alpha_{1}, \beta_{1}; J, H; c_{1, 0}, \cdots, c_{1, k_{1}}).
\end{equation}
The numbering of these generators is as follows. The root $e_{1, 0}$ of a stable pearly tree map in $\bar{\mathcal{M}}_{k_{1}+1}(\alpha_{1}, \beta_{1}; J, H; c_{1, 0}, \cdots, c_{1, k_{1}})$ is connected to some leaf $e_{0, l(1)}$ of a stable pearly tree map in $\bar{\mathcal{M}}_{k_{0}+1}(\alpha_{0}, \beta_{0}; J, H; c_{0, 0}, \cdots, c_{0, k_{0}})$. Thus $c_{1, 0} = c_{0, l(1)} = c_{new}$, while the other generators agree with the original ones:
\begin{align}
&c_{0, i} = c_{i}, i = 0, \cdots, l(1),\\
&c_{1, j} = c_{j + l(1) - 1}, j = 1, \cdots, k_{1},\\
&c_{0, i} = c_{i + k_{1} - 1}, i = l(1) + 1, \cdots, k_{0}.
\end{align}
So we rewrite the fiber product as
\begin{equation}
\begin{split}
&\bar{\mathcal{M}}_{k_{0}+1}(\alpha_{0}, \beta_{0}; J, H; c_{0}, \cdots, c_{l(1)-1}, c_{new}, c_{l(1) + k_{1}}, \cdots, c_{k_{0} + k_{1} -1})\\
&\times \bar{\mathcal{M}}_{k_{1}+1}(\alpha_{1}, \beta_{1}; J, H; c_{new}, c_{l(1)}, \cdots, c_{l(1) + k_{1} - 1}).
\end{split}
\end{equation}
Thus we have isomorphisms of multisections:
\begin{equation}
\begin{split}
&s_{\alpha, \beta; J, H; c_{0}, \cdots, c_{k}}\\
&\cong s_{\alpha_{0}, \beta_{0}; J, H; c_{0}, \cdots, c_{l(1)-1}, c_{new}, c_{l(1) + k_{1}}, \cdots, c_{k_{0} + k_{1} -1}} \times s_{\alpha_{1}, \beta_{1}; J, H; c_{new}, c_{l(1)}, \cdots, c_{l(1) + k_{1} - 1}},
\end{split}
\end{equation}
when restricting the multisection $s_{\alpha, \beta; J, H; c_{0}, \cdots, c_{k}}$ to the boundary strata. This implies that the operations $m^{k}$ defined in \eqref{A-infinity operations} satisfy the $A_{\infty}$-equations. \par
	The independence of choices of multisections up to homotopy is a consequence of general theory of Kuranishi structures. To prove that the curved $A_{\infty}$-algebra is independent of choice of almost complex structure $J$ up to homotopy, we may introduce parametrized moduli spaces of pearly tree maps with time-allocation, associated to such a homotopy $\{J_{t}\}_{t \in [0, 1]}$, and construct Kuranishi structures on these moduli spaces in a well-arranged way so that they define cobordisms between the Kuranishi structures on the moduli spaces of pearly tree maps $\bar{\mathcal{M}}_{k+1}(\alpha, \beta; J_{i}, H; c_{0}, \cdots, c_{k}), i = 0, 1$ with respect to different almost complex structures. Here recall that the generators $c_{0}, \cdots, c_{k}$ themselves do not depend on the almost complex structures used to define pearly tree maps, so both kinds of moduli spaces make sense. \par
\end{proof}

\subsection{Bounding cochains}
	Given a cylindrical Lagrangian immersion $\iota: L \to M$, we have constructed a curved $A_{\infty}$-algebra structure $(CW^{*}(L, \iota; H), m^{k})$ on the wrapped Floer cochain space. In general $m^{1}$ does not square to zero because of the curvature term $m^{0} \neq 0$. In order to get a differential, we want to deform the operations $m^{k}$ algebraically such that the deformed differential squares to zero, and moreover, the deformed structure maps of all orders satisfy the $A_{\infty}$-equations without curvature. This idea was introduced into Lagrangian Floer theory by Fukaya-Oh-Ohta-Ono \cite{FOOO1}, by finding solutions to the inhomogeneous Maurer-Cartan equation:
\begin{equation} \label{inhomogeneous Maurer-Cartan equation}
\sum_{k = 0}^{\infty} m^{k}(b, \cdots, b) = 0.
\end{equation}
In order for this equation to make sense, we shall further impose the condition that $b$ is nilpotent, so that the sum stops at a final stage: there exists $K$ such that $m^{k}(b, \cdots, b) = 0$ for all $k > K$ and 
\begin{equation*}
\sum_{k = 0}^{K} m^{k}(b, \cdots, b) = 0.
\end{equation*} \par
	Because of the presence of the inhomogeneous term $m^{0}(1)$, this equation might not have any solution in general. But if it does have one solution $b \in CW^{*}(L, \iota; H)$ which is nilpotent, we call $b$ a bounding cochain for $(L, \iota)$, and say $(L, \iota)$ is unobstructed in the sense of wrapped Floer theory. In this case, we can deform the curved $A_{\infty}$-algebra to an $A_{\infty}$-algebra with vanishing curvature:
\begin{equation} \label{deform the A-infinity algebra by MC elements}
m^{k; b}(c_{k}, \cdots, c_{1}) = \sum_{\substack{i \ge 0\\ i_{0} + \cdots + i_{k} = i}} m^{k+i}(\underbrace{b, \cdots, b}_{\text{$i_{k}$ times}}, c_{k}, \underbrace{b, \cdots, b}_{\text{$i_{k-1}$ times}}, \cdots, c_{1}, \underbrace{b, \cdots, b}_{\text{$i_{0}$ times}}).
\end{equation}
In particular, $m^{1; b}$ squares to zero, and we can define a cohomology group of the graded complex $CW^{*}(L, \iota; H)$ with respect to the differential $m^{1; b}$, which we denote by $HW^{*}(L, \iota, b; H)$. We call this the wrapped Floer cohomology group of the Lagrangian immersion $\iota: L \to M$ with respect to the bounding cochain $b$. \par

\subsection{Wrapped Floer cochain space of a pair of Lagrangian immersions}
	Given a pair of exact cylindrical Lagrangian immersions with transverse self-intersections, we want to construct a curved $A_{\infty}$-bimodule $CW^{*}((L_{0}, \iota_{0}), (L_{1}, \iota_{1}); H)$ over the curved $A_{\infty}$-algebras $CW^{*}(L_{i}, \iota_{i}; H)$ associated to the two Lagrangian immersions $\iota_{i}: L_{i} \to M, i = 0, 1$. Assume that these two Lagrangian immersions are not the same (the case where they are the same can be treated using the construction for a single exact cylindrical Lagrangian immersion). In a generic situation, we may assume that the constant $H$-chords from $\iota_{0}(L_{0})$ to $\iota_{1}(L_{1})$, which are identical to intersection points of $\iota_{0}(L_{0})$ with $\iota_{1}(L_{1})$, do not coincide with the self-intersection points of $\iota_{0}$ or those of $\iota_{1}$. Moreover, we may assume that this pair intersects transversely, in the following sense. \par

\begin{definition}
	Let $(L_{0}, \iota_{0}), (L_{1},  \iota_{1}))$ be a pair of exact cylindrical Lagrangian immersions with transverse self-intersections. They are said to intersect transversely, if the following conditions are satisfied:
\begin{enumerate}[label=(\roman*)]

\item the intersection points of $\iota_{0}(L_{0}) \cap \iota_{1}(L_{1})$ are different from the self-intersections of $\iota_{0}: L_{0} \to M$ and those of $\iota_{1}: L_{1} \to M$;

\item the images $\iota_{0}(L_{0})$ and $\iota_{1}(L_{1})$ intersect transversely;

\item they do not intersection over the cylindrical end $\partial M \times [1, +\infty)$.

\end{enumerate}
\end{definition}

	Now let us define the wrapped Floer cochain space for such a pair $(L_{0}, \iota_{0}), (L_{1},  \iota_{1}))$. \par
	
\begin{definition}\label{definition of wrapped Floer cochain space for a pair}
	We define the wrapped Floer cochain space $CW^{*}((L_{0}, \iota_{0}), (L_{1}, \iota_{1}); H)$ to be the free $\mathbb{Z}$-module generated by the following two kinds of generators
\begin{enumerate}[label=(\roman*)]

\item an intersection point $p$ of $\iota_{0}(L_{0})$ and $\iota_{1}(L_{1})$ which is contained in the interior part $M_{0}$, or equivalently a constant $H$-chord $x = x_{p}$ at $p$ from $\iota_{0}(L_{0})$ to $\iota_{1}(L_{1})$ (recall $H$ is constantly zero in the interior part);

\item non-constant time one $H$-chords $x$ from $\iota_{0}(L_{0})$ to $\iota_{1}(L_{1})$ which is contained in the cylindrical end, where both $\iota_{0}, \iota_{1}$ are assumed to be embeddings.

\end{enumerate}

\end{definition}

	Since every interior intersection point of $\iota_{0}(L_{0})$ and $\iota_{1}(L_{1})$ is assumed to be different from any self-intersections of $\iota_{0}$ or $\iota_{1}$, and since both Lagrangian immersions are exact, the above definition makes sense, without having to include capping half-disks for these intersection points. From now on we shall use the letter $c$ to denote any kind of generator. \par


\subsection{Moduli spaces of Floer trajectories}
	To study wrapped Floer theory for a pair of exact cylindrical Lagrangian immersions, we shall introduce moduli spaces of stable broken Floer trajectories. In our situation, because the Lagrangian submanifolds are immersed, we also need to add some refinement of the data to the moduli spaces, which are similar to those for a single Lagrangian immersion. \par
	Choose for each Lagrangian immersion $\iota_{i}: L_{i} \to M$ an admissible almost complex structure $J_{i}, i = 0, 1$ of contact type, using which the curved $A_{\infty}$-algebras $(CW^{*}(L_{i}, \iota_{i}; H), m^{k})$ are defined. Also, choose a path of admissible almost complex structures $J_{t}$ of contact type connecting $J_{0}$ and $J_{1}$. We shall consider the moduli spaces of the following kinds of maps. \par
	Let $I_{0} \subset \{1, \cdots, k\}$ and $I_{1} \subset \{1, \cdots, l\}$, and $\alpha_{i}: I_{i} \to S(L_{i}, \iota_{i})$ be maps, labeling those marked points which are mapped to some self-intersection points of $\iota_{0}$ and $\iota_{1}$ respectively. Also, let $\beta \in H_{2}(M, \iota_{0}(L_{0}) \cup \iota_{1}(L_{1}))$ be a relative homology class. \par

\begin{definition}
	A $(k, l)$-marked Floer trajectory is a quadruple $(\Sigma, \vec{s}, u, \vec{l})$ satisfying the following conditions:
\begin{enumerate}[label=(\roman*)]

\item $\Sigma = \mathbb{R} \times [0, 1]$ is the infinite strip.

\item $\vec{s} = (\vec{s}^{0}, \vec{s}^{1})$ with $\vec{s}^{0} = (s^{0}_{1}, \cdots, s^{0}_{k})$ and $\vec{s}^{1} = (s^{1}_{1}, \cdots, s^{1}_{l})$ are collections of real numbers, such that 
\begin{equation*}
s^{0}_{j} > s^{0}_{j+1}, j = 1, \cdots, k-1,
\end{equation*}
\begin{equation*}
s^{1}_{j} < s^{1}_{j+1}, j = 1, \cdots, l-1.
\end{equation*}

\item $u: \Sigma \to M$ is a continuous map.

\item $u$ satisfies the inhomogeneous Cauchy-Riemann equation (Floer's equation)
\begin{equation*}
\frac{\partial u}{\partial s} + J_{t}(\frac{\partial u}{\partial t} - X_{H}(u)) = 0.
\end{equation*}

\item $u(s, 0) \in \iota_{0}(L_{0}), u(s, 1) \in \iota_{1}(L_{1})$, for all $s \in \mathbb{R}$.

\item $u$ asymptotically converges to time-one $H$-chords $x_{-}, x_{+}$ from $\iota_{0}(L_{0})$ to $\iota_{1}(L_{1})$ as $s \to - \infty$ and respectively $+\infty$, where these $H$-chords might be constant or non-constant.

\item The homology class of $u$ is $\beta \in H_{2}(M, \iota_{0}(L_{0}) \cup \iota_{1}(L_{1}))$.

\item $\vec{l} = (l_{0}, l_{1})$ is a pair of smooth maps $l_{i}:  \mathbb{R} \times \{i\} \to L_{i} \times_{\iota_{i}} L_{i}$, such that $u(s, i) = \iota_{i} \circ l_{i}(s)$, for $s \in \mathbb{R} \times \{i\} \setminus \{s^{i}_{j}| j \in I_{i}\}$. $\vec{l}$ is the boundary lifting condition of $u$.

\item $(\lim\limits_{s \uparrow s^{i}_{j}} l_{i}(s), \lim\limits_{s \downarrow s^{i}_{j}} l_{i}(s)) = \alpha_{i}(j)$, for every $j \in I_{i}$.

\item $(\Sigma, \vec{s}, u, l)$ is stable.

\end{enumerate} \par

\end{definition}

	There is an obvious $\mathbb{R}$-action by translations on the set of all marked Floer trajectories $(\Sigma, \vec{s}, u, l)$. We denote by
\begin{equation*}
\mathcal{N}_{k, l}((L_{0}, \iota_{0}), (L_{1}, \iota_{1}); \vec{\alpha}, \beta; \{J_{t}\}_{t}, H; c_{-}, c_{+})
\end{equation*}
the set of equivalence classes of such $(\Sigma, \vec{s}, u, l)$. For convenience, we sometimes omit the notations $(L_{0}, \iota_{0}), (L_{1}, \iota_{1})$ in case no confusion can occur. \par
	The above moduli space has a natural compactification, called the moduli space of stable broken Floer trajectories, to be described in two steps: first, we need to add trees of disk bubbles to each boundary of a Floer trajectory; second, we need to add broken Floer trajectories. \par
	First consider stable unbroken Floer trajectories, which are marked Floer trajectories with stable pearly trees attached to the two boundary components of the strip $\mathbb{R} \times [0, 1]$. \par
	In order to describe these elements in detail, we add more data to marked Floer trajectories. Let $(\Sigma, \vec{s}, u, l)$ be a marked Floer trajectory. \par

\begin{definition}
	A decoration for $(\Sigma, \vec{s}, u, l)$ is an assignment of coloring $c: \vec{s} \to \{0, 1\}$. \par
\end{definition}

	For each $s^{i}_{j}$ with color $c(s^{i}_{j}) = 0$, we attach a half-infinite ray $I_{s^{i}_{j}} = [0, +\infty)$ to $\Sigma$ at $s^{i}_{j}$. For each $s^{i}_{j}$ with color $c(s^{i}_{j}) = 1$, we remove $s^{i}_{j}$ and add a half-infinite strip $Z_{s^{i}_{j}}^{+} = [0, +\infty) \times [0, 1]$ as a strip-like end near the puncture $s^{i}_{j}$. \par
	Now we extend the map $u$ to the decorated domain $\tilde{\Sigma}$. On each newly-added half-infinite ray $I_{s^{i}_{j}}$, we extend $u$ by a map $u_{s^{i}_{j}}: I_{s^{i}_{j}} \to L_{i}$, which comes with a preferred lift $\tilde{u_{s^{i}_{j}}}: I_{s^{i}_{j}} \to L_{i} \times_{\iota_{i}} L_{i}$, which satisfies the gradient flow equation:
\begin{equation}
\frac{du_{s^{i}_{j}}}{dt} + \nabla f_{i, s}(u_{s^{i}_{j}}) = 0,
\end{equation}
and asymptotically converges to a critical point of $f_{i}$, where $f_{i, s}$ a family of perturbations of $f_{i}$ parametrized by $s \in I_{s^{i}_{j}}$, so that $f_{i, s} = f_{i}$ for $s \gg 0$. On each newly-added half-infinite strip $Z_{s^{i}_{j}}^{+}$, we extend $u$ by a map $u_{s^{i}_{j}}: Z_{s^{i}_{j}}^{+} \to M$, such that $u_{s^{i}_{j}}$ maps the boundary to $\iota_{i}(L_{i})$, and satisfies the inhomogeneous Cauchy-Riemann equation:
\begin{equation}
\partial_{s}u_{s^{i}_{j}} + J_{i}(\partial_{t}u_{s^{i}_{j}} - X_{H}(u_{s^{i}_{j}})) = 0,
\end{equation}
and asymptotically converges to some time-one $H$-chord $x$. \par
	These critical points of $f_{i}$ or non-constant $H$-chords from $\iota_{i}(L_{i})$ to itself should come equipped with choices of capping half-disks, which make them into generators $c_{i}^{j}$ of the wrapped Floer cochain space $CW^{*}((L_{0}, \iota_{0}), (L_{1}, \iota_{1}); H)$. \par
	We also extend the maps $\vec{l} = (l_{0}, l_{1})$ to these newly-added components, so that the extended map $\tilde{l}_{i}$ satisfies $\iota_{i} \circ \tilde{l}_{i} = \tilde{u}$. In particular, the extension to a half-infinite ray $I_{s^{i}_{j}}$ is precisely the preferred lift $\tilde{u}_{s^{i}_{j}}$. \par

\begin{definition}
	We call the extended map $(\tilde{\Sigma}, \vec{s}, \tilde{u}, \vec{\tilde{l}})$ a decorated Floer trajectory. Denote by 
\begin{equation}
\mathcal{N}_{k, l}^{dec}(\vec{\alpha}, \beta; \{J_{t}\}_{t}, H; c_{-}, c_{+}; c^{0}_{1}, \cdots, c^{0}_{k}, c^{1}_{1}, \cdots, c^{1}_{l})
\end{equation}
the moduli space of decorated Floer trajectories.
\end{definition}

	For convenience, denote by $k = k_{0}$ and $l = k_{1}$. Let $n_{0}, n_{1}, m_{0}, m_{1}$ and $m_{0, 1}, \cdots, m_{0, n_{0}}$, and $m_{1, 1}, \cdots, m_{1, n_{1}}$ be non-negative integers such that
\begin{equation*}
k_{i} = m_{i} + \sum_{a=1}^{n_{i}}m_{i, a}.
\end{equation*}
Let $A_{i} \subset \{1, \cdots, m_{i}+n_{i}\}$ be a subset of $n_{i}$ elements and put
\begin{equation}
A_{i} = \{\sigma_{i}(1), \cdots, \sigma_{i}(n_{i})\},
\end{equation}
where $\sigma_{i}: \{1, \cdots, n_{i}\} \to \{1, \cdots, m_{i} + n_{i}\}$ is an injective map satisfying $\sigma_{i}(a) < \sigma_{i}(a+1)$.
Let $\pi_{i, 1}: (Crit(f_{i}) \coprod \mathcal{X}_{+}(\iota_{i}(L_{i}), H))^{k_{i}+n_{i}} \to (Crit(f_{i}) \coprod \mathcal{X}_{+}(\iota_{i}(L_{i}), H))^{n_{i}}$ be the projection
\begin{equation*}
\pi_{i, 1}(x_{1}, \cdots, x_{k_{i}+n_{i}}) = (x_{\sigma_{i}(1)}, \cdots, x_{\sigma_{i}(n_{i})}).
\end{equation*}
Let $\pi_{i, 2}: (Crit(f_{i}) \coprod \mathcal{X}_{+}(\iota_{i}(L_{i}), H))^{k_{i}+n_{i}} \to (Crit(f_{i}) \coprod \mathcal{X}_{+}(\iota_{i}(L_{i}), H))^{k_{i}}$ be the projection to the other factors. \par
	For simplicity, denote by $A$ the collection of data $n_{0}, n_{1}, m_{0}, m_{1}$, $m_{0, 1}, \cdots, m_{0, n_{0}}$, and $m_{1, 1}, \cdots, m_{1, n_{1}}$ as well as $A_{0}, A_{1}$. We put
\begin{equation} \label{moduli spaces of unbroken stable Floer trajectories}
\begin{split}
&\mathcal{N}^{A}_{k_{0}, k_{1}}(\vec{\alpha}, \beta; \{J_{t}\}_{t}, H; c_{-}, c_{+}; \{c^{0}_{a}, c^{0}_{a, 1}, \cdots, c^{0}_{a, m_{0, a}}\}_{a=1}^{n_{0}}, \{c^{1}_{b}, c^{1}_{b, 1}, \cdots, c^{1}_{b, m_{1, b}}\}_{b=1}^{n_{1}})\\
&= \bigcup_{\beta' \sharp \sum \beta_{0, a} \sharp \sum \beta_{1, b} = \beta} \coprod_{\substack{\alpha'_{0} \cup \bigcup_{a} \alpha_{0, a} = \alpha_{0}\\ \alpha'_{1} \cup \bigcup_{b} \alpha_{1, b} = \alpha_{1}}} \mathcal{N}_{n_{0}+m_{0}, n_{1}+m_{1}}^{dec}(\vec{\alpha}', \beta'; \{J_{t}\}_{t}, H; x_{-}, x_{+};\\
&c^{0}_{1}, \cdots, c^{0}_{n_{0}+m_{0}}, c^{1}_{1}, \cdots, c^{1}_{n_{1}+m_{1}})\\
& \times_{(\pi_{0, 1}, \pi_{1, 1}) \circ ev, \vec{ev}_{0}}
(\prod_{a=1}^{n_{0}} \bar{\mathcal{M}}_{m_{0, a}+1}(\alpha_{0, a}, \beta_{0, a}; J_{0}, H; c^{0}_{a}, c^{0}_{a, 1}, \cdots, c^{0}_{a, m_{0, a}})\\
&\times \prod_{b=1}^{n_{1}} \bar{\mathcal{M}}_{m_{1, b}+1}(\alpha_{1, b}, \beta_{1,b}; J_{1}, H; c^{1}_{b}, c^{1}_{b, 1}, \cdots, c^{1}_{b, m_{1, b}})),
\end{split}
\end{equation}
where $\vec{ev}^{0} = (ev^{0}, \cdots, ev^{0})$ are the evaluation maps at the $0$-th marked point, and the moduli spaces 
\begin{equation*}
\bar{\mathcal{M}}_{m_{0, a}+1}(\alpha_{0, a}, \beta_{0, a}; J_{0}, H; c^{0}_{a}, c^{0}_{a, 1}, \cdots, c^{0}_{a, m_{0, a}})
\end{equation*}
and
\begin{equation*}
\bar{\mathcal{M}}_{m_{1, b}+1}(\alpha_{1, b}, \beta_{1,b}; J_{1}, H; c^{1}_{b}, c^{1}_{b, 1}, \cdots, c^{1}_{b, m_{1, b}})
\end{equation*}
are moduli spaces of stable pearly tree maps. \par
	To simplify the notations, we denote by $\vec{c}^{0}_{A}$ the collection $\{c^{0}_{a}, c^{0}_{a, 1}, \cdots, c^{0}_{a, m_{0, a}}\}_{a=1}^{n_{0}}$ of generators and similarly $\vec{c}^{1}_{A}$ for the other collection. Also, denote by $\vec{c}^{0}, \vec{c}^{1}$ for any possible collections among $c^{0}_{A}, c^{1}_{A}$ for all type $A$. \par

\begin{definition}
	We put
\begin{equation}
\mathcal{N}^{unbr}_{k_{0}, k_{1}}(\alpha_{0}, \alpha_{1}; \beta; c_{-}, c_{+}; \vec{c}^{0}, \vec{c}^{1})
= \bigcup_{A} \mathcal{N}^{A}_{k_{0}, k_{1}}(\alpha_{0}, \alpha_{1}; \beta; c_{-}, c_{+}; \vec{c}^{0}_{A}, \vec{c}^{1}_{A}),
\end{equation}
and call it the moduli space of stable unbroken Floer trajectories.
\end{definition}

	In words, the elements in this moduli space are Floer trajectories with trees of disks attached to each boundary of $\mathbb{R} \times [0, 1]$. We do not have to include sphere bubbles because the ambient symplectic manifold $M$ is exact. \par


\subsection{Compactification: stable broken Floer trajectories}
	Next, as the limit of a sequence of stable unbroken Floer trajectories can break into several Floer trajectories, we introduce the moduli space of stable broken Floer trajectories. \par
	
\begin{definition}
	A stable broken Floer trajectory is a tuple
\begin{equation*}
((\tilde{\Sigma}(1), \vec{s}(1), \tilde{u}(1), \vec{\tilde{l}}(1)), \cdots, (\tilde{\Sigma}(K), \vec{s}(K), \tilde{u}(K), \vec{\tilde{l}}(K)),
\end{equation*}
where each $(\tilde{\Sigma}(a), \vec{s}(a), \tilde{u}(a), \vec{\tilde{l}}(a))$ is a stable unbroken Floer trajectory, such that the asymptotic convergence conditions match for successive stable unbroken Floer trajectories:
\begin{equation}
c_{+}(a) = c_{-}(a+1).
\end{equation} \par
	The moduli space of stable broken Floer trajectories is denoted by
\begin{equation}
\begin{split}
&\bar{\mathcal{N}}_{k_{0}, k_{1}}(\vec{\alpha}, \beta; \{J_{t}\}_{t}, H; c_{-}, c_{+}; \vec{c}^{0}, \vec{c}^{1})\\
&= \coprod_{K} \coprod_{\substack{\alpha_{0}(1), \cdots, \alpha_{0}(K)\\\alpha_{0}(1) \cup \cdots \cup \alpha_{0}(K) = \alpha_{0}}} \coprod_{\substack{\alpha_{1}(1), \cdots, \alpha_{1}(K)\\\alpha_{1}(1) \cup \cdots \cup \alpha_{1}(K) = \alpha_{0}}} \coprod_{\beta(1) \sharp \cdots \sharp \beta(K) = \beta}\\
&\coprod_{\substack{k_{0, 1}, \cdots, k_{0, K}\\k_{0, 1} + \cdots + k_{0, K} = k_{0}}} \coprod_{\substack{k_{1, 1}, \cdots, k_{1, K}\\k_{1, 1} + \cdots + k_{1, K} = k_{1}}} \prod_{a=1}^{K-1}\\
&\mathcal{N}^{unbr}_{k_{0, a}, k_{1, a}}(\alpha_{0}(a), \alpha_{1}(a), \beta(a); \{J_{t}\}_{t}, H; c_{-}(a), c_{+}(a); \vec{c}^{0}(a), \vec{c}^{1}(a)),
\end{split}
\end{equation}
where $c(1) = c_{-}, c(K) = c_{+}$. Here the disjoint union is taken over all $K, \vec{\alpha}, \vec{c}, k_{0, a}, k_{1, a}$. 

\end{definition}

	Now it is standard to prove: \par

\begin{lemma}
	The moduli space of stable broken Floer trajectories
\begin{equation*}
\bar{\mathcal{N}}_{k_{0}, k_{1}}(\vec{\alpha}, \beta; \{J_{t}\}_{t}, H; c_{-}, c_{+}; \vec{c}^{0}, \vec{c}^{1})
\end{equation*}
is compact.
\end{lemma}


\subsection{Kuranishi structures and single-valued multisections}

	The following propositions are essential for extracting algebraic structures from the moduli spaces of stable broken Floer trajectories. The proofs are similar to the case of moduli spaces of stable pearly tree maps in the previous section \ref{section: Kuranishi structure on moduli spaces of stable pearly tree maps}. \par

\begin{proposition}\label{Kuranishi structures on the moduli space of stable broken Floer trajectories}
	The moduli space of stable broken Floer trajectories
\begin{equation*}
\bar{\mathcal{N}}_{k_{0}, k_{1}}(\alpha_{0}, \alpha_{1}, \beta; \{J_{t}\}_{t}, H; c_{-}, c_{+}; \vec{c}^{0}, \vec{c}^{1})
\end{equation*}
has an oriented Kuranishi structure which is compatible with the fiber product Kuranishi structures at its codimension one boundary strata:
\begin{equation}\label{boundary stratum of the moduli space of broken stable Floer trajectories}
\begin{split}
&\partial \bar{\mathcal{N}}_{k_{0}, k_{1}}(\vec{\alpha}, \beta; \{J_{t}\}_{t}, H; c_{-}, c_{+}; \vec{c}^{0}, \vec{c}^{1})\\
&\cong \coprod_{\substack{k'_{0}+k''_{0}=k_{0}\\ 1 \le i \le k'_{0}\\ \alpha_{0}}} \coprod_{\substack{(\vec{c}'^{0}, \vec{c}'''^{0}, \vec{c}''^{0}) = \vec{c}^{0}\\ \text{$\vec{c}'^{0}$ is an $i$-tuple}}}
\bar{\mathcal{N}}_{k'_{0}+1, k_{1}}(\alpha'_{0}, \alpha_{1}, \beta'; \{J_{t}\}_{t}, H; c_{-}, c_{+}; (\vec{c}'^{0}, \vec{c}''^{0}), \vec{c}^{1})\\
&\times_{ev_{0}^{i}, ev^{0}} \bar{\mathcal{M}}_{k''_{0}+1}(\alpha''_{0}, \beta_{0}; J_{0}, H; \vec{c}'''^{0})\\
&\cup \coprod_{\substack{k'_{1}+k''_{1}=k_{1}\\ 1 \le i \le k'_{1}\\ \alpha_{1}}} \coprod_{\substack{(\vec{c}'^{1}, \vec{c}'''^{1}, \vec{c}''^{1}) = \vec{c}^{1}\\ \text{$\vec{c}'^{1}$ is an $i$-tuple}}}
\bar{\mathcal{N}}_{k_{0}, k'_{1}+1}(\alpha_{0}, \alpha'_{1}, \beta; \{J_{t}\}_{t}, H; c_{-}, c_{+}; \vec{c}^{0}, \vec{c}'^{1}, \vec{c}''^{1})\\
&\times_{ev_{1}^{i}, ev^{0}} \bar{\mathcal{M}}_{k''_{1}+1}(\alpha''_{1}, \beta_{1}; J_{1}, H; \vec{c}'''^{1})\\
&\cup \coprod_{c'_{+} = c'_{-}} \bar{\mathcal{N}}_{k'_{0}, k'_{1}}(\alpha'_{0}, \alpha'_{1}, \beta'; \{J_{t}\}_{t}, H; c_{-}, c'_{+}; \vec{c}'^{0}, \vec{c}'^{1})\\
&\times \bar{\mathcal{N}}_{k''_{0}, k''_{1}}(\alpha''_{0}, \alpha''_{1}, \beta''; \{J_{t}\}_{t}, H; c'_{-}, c_{+}; \vec{c}''^{0}, \vec{c}''^{1})
\end{split}
\end{equation}
\end{proposition}

\begin{proposition} \label{existence of single-valued multisections on the moduli space of stable broken Floer trajectories}
	On the moduli space of stable broken Floer trajectories
\begin{equation*}
\bar{\mathcal{N}}_{k_{0}, k_{1}}(\vec{\alpha}, \beta; \{J_{t}\}_{t}, H; c_{-}, c_{+}; \vec{c}^{0}, \vec{c}^{1}),
\end{equation*}
equipped with a Kuranishi structure as in Proposition \ref{Kuranishi structures on the moduli space of stable broken Floer trajectories},
there exists a single-valued multisection
\begin{equation*}
s_{k_{0}, k_{1}; \vec{\alpha}, \beta; c_{-}, c_{+}; \{J_{t}\}_{t}, H; \vec{c}^{0}, \vec{c}^{1}},
\end{equation*}
which is transverse to zero, and compatible with the fiber-product single-valued multisections on the Kuranishi spaces \eqref{boundary stratum of the moduli space of broken stable Floer trajectories} at the boundary with the fiber-product multisections on \eqref{boundary stratum of the moduli space of broken stable Floer trajectories}.
\end{proposition}

\subsection{The curved $A_{\infty}$-bimodule associated to a pair of exact cylindrical Lagrangian immersions with transverse self-intersections}

	For a pair of exact cylindrical Lagrangian immersions as above, we shall construct a curved $A_{\infty}$-bimodule structure on the wrapped Floer cochain space $CW^{*}((L_{0}, \iota_{0}), (L_{1}, \iota_{1}); H)$. The main result is stated as follows. \par

\begin{proposition}\label{prop: the curved A-infinity bimodule for a pair}
	There is a natural curved $A_{\infty}$-bimodule structure on the wrapped Floer cochain space $CW^{*}((L_{0}, \iota_{0}), (L_{1}, \iota_{1}); H)$ over the curved $A_{\infty}$-algebras for $\iota_{i}: L_{i} \to M$, 
\begin{equation*}
(CW^{*}(L_{0}, \iota_{0}; H), CW^{*}(L_{1}, \iota_{1}; H)).
\end{equation*}
The structure maps $n^{k, l}$ are defined by appropriate counts of broken stable Floer trajectories. \par
	Suppose both $\iota_{i}: L_{i} \to M$ are unobstructed with choices of bounding cochains $b_{i} \in CW^{*}(L_{i}, \iota_{i}; H)$. Then the $(b_{0}, b_{1})$-deformation $n^{k, l; b_{0}, b_{1}}$ defines a non-curved $A_{\infty}$-bimodule over the deformed $A_{\infty}$-algebras
\begin{equation*}
((CW^{*}((L_{0}, \iota_{0}; H), m^{k; b_{0}}), (CW^{*}((L_{1}, \iota_{1}; H), m^{k; b_{1}})).
\end{equation*}
\end{proposition}
\begin{proof}
	Consider moduli spaces of stable broken Floer trajectories which are of virtual dimension zero. Then the virtual fundamental chains associated to the chosen single-valued multisections $s_{k_{0}, k_{1}; \vec{\alpha}, \beta; c_{-}, c_{+}; \{J_{t}\}_{t}, H; \vec{c}^{0}, \vec{c}^{1}}$ gives rise to an integer number
\begin{equation}
a_{k, l; \vec{\alpha}, \beta; c_{-}, c_{+}; \{J_{t}\}_{t}, H; \vec{c}^{0}, \vec{c}^{1}} = (s_{k, l; \vec{\alpha}, \beta; c_{-}, c_{+}; \{J_{t}\}_{t}, H; \vec{c}^{0}, \vec{c}^{1}})^{-1}(0) \in \mathbb{Z}.
\end{equation}
We define multilinear maps
\begin{equation}
n^{k, l}: CW^{*}(L_{0}, \iota_{0})^{\otimes k} \otimes CW^{*}((L_{0}, \iota_{0}), (L_{1}, \iota_{1})) \otimes CW^{*}(L_{1}, \iota_{1})^{\otimes l} \to CW^{*}((L_{0}, \iota_{0}), (L_{1}, \iota_{1})),
\end{equation}
for $k = k_{0}, l = k_{1} \ge 0$, by the formula:
\begin{equation}
\begin{split}
&n^{k, l}(c^{0}_{k}, \cdots, c^{0}_{1}, c_{+}, c^{1}_{1}, \cdots, c^{1}_{l})\\
&= \sum_{\substack{\alpha, \beta, c_{-}\\ \dim \bar{\mathcal{N}}_{k_{0}, k_{1}}(\vec{\alpha}, \beta; c_{-}, c_{+}; \{J_{t}\}_{t}, H; \vec{c}^{0}, \vec{c}^{1}) = 0}}
a_{k, l; \vec{\alpha}, \beta; c_{-}, c_{+}; \{J_{t}\}_{t}, H; \vec{c}^{0}, \vec{c}^{1}} c_{-}.
\end{split}
\end{equation}
Again, this is a finite sum by the same argument as that for \eqref{A-infinity operations}, so that the multilinear map $n^{k, l}$ is well-defined. \par
	These maps satisfy the $A_{\infty}$-equations for a curved $A_{\infty}$-bimodule, by Proposition \eqref{existence of single-valued multisections on the moduli space of stable broken Floer trajectories}. \par
\end{proof}

	In general, $n^{0, 0}$ does not square to zero because of the contribution of $m^{0}$ from each $(L_{i}, \iota_{i})$. To obtain a differential, we need to assume that both $\iota_{0}: L_{0} \to M$ and $\iota_{1}: L_{1} \to M$ are unobstructed, in which case we can deform the curved $A_{\infty}$-bimodule structure by the chosen bounding cochains for $\iota_{0}: L_{0} \to M$ and $\iota_{1}: L_{1} \to M$ respectively:
\begin{equation}
\begin{split}
&n^{k, l; b_{0}, b_{1}}(c^{0}_{k}, \cdots, c^{0}_{1}, c_{+}, c^{1}_{1}, \cdots, c^{1}_{l})\\
&= \sum_{\substack{i \ge 0, j \ge 0\\ i_{0} + \cdots + i_{k} = i\\ j_{0} + \cdots + j_{l} = j}} n^{k+i, l+j}(\underbrace{b_{0}, \cdots, b_{0}}_{\text{$i_{k}$ times}}, c^{0}_{k}, \underbrace{b_{0}, \cdots, b_{0}}_{\text{$i_{k-1}$ times}}, \cdots, c^{0}_{1}, \underbrace{b_{0}, \cdots, b_{0}}_{\text{$i_{0}$ times}},\\
&c_{+}, \underbrace{b_{1}, \cdots, b_{1}}_{\text{$j_{0}$ times}}, c^{1}_{1}, \underbrace{b_{1}, \cdots, b_{1}}_{\text{$j_{1}$ times}}, \cdots, c^{1}_{l}, \underbrace{b_{1}, \cdots, b_{1}}_{\text{$j_{l}$ times}}).
\end{split}
\end{equation}
Because of the Maurer-Cartan equations that the bounding cochains satisfy, $n^{0, 0; b_{0}, b_{1}}$ squares to zero, and thus defines a differential on $CW^{*}((L_{0}, \iota_{0}), (L_{1}, \iota_{1}); H)$. We call the resulting cohomology group the wrapped Floer cohomology group of the pair of the exact cylindrical Lagrangian immersions $\iota_{0}: L_{0} \to M$ and $\iota_{1}: L_{1} \to M$, with respect to the bounding cochains $b_{0}$ and $b_{1}$, and denote it by $HW^{*}((L_{0}, \iota_{0}, b_{0}), (L_{1}, \iota_{1}, b_{1}); H)$. \par

\section{Clean intersections}\label{section: immersed wrapped Floer theory in the case of clean intersections}

\subsection{Clean self-intersections}

	This section extends wrapped Floer theory to a wider class of immersed Lagrangian submanifolds: exact cylindrical Lagrangian immersions with clean self-intersections. The construction of the moduli spaces follows exactly the same pattern, but there are additional data to be taken into account, which we introduce below. \par

\begin{definition}
	Let $\iota: L \to M$ be a Lagrangian immersion. We say that it has clean self-intersections, if the following conditions are satisfied:
\begin{enumerate} [label=(\roman*)]

\item The fiber product $L \times_{\iota} L = \{(p, q) \in L \times L| \iota(p) = \iota(q)\}$ is a smooth submanifold of $L \times L$;

\item At each point $(p, q) \in L \times_{\iota} L$, we have that
\begin{equation}
T_{(p, q)}(L \times_{\iota} L) = \{(V, W) \in T_{p}L \times T_{q}L| d_{p}\iota(V) = d_{q}\iota(V)\}
\end{equation}

\end{enumerate}
\end{definition}

	For such $\iota: L \to M$, we can decompose the fiber product $L \times_{\iota} L$ into connected components:
\begin{equation}
L \times_{\iota} L = \coprod_{a \in A} L_{a} = \Delta_{L} \cup \coprod_{a \in A \setminus \{0\}} L_{a},
\end{equation}
indexed by an indexing set $A$ containing a special element $0$. Here $\Delta_{L}$ is the diagonal component, which is identified with $L$ and labeled by $0 \in A$, and $L_{a}$'s are some other components which are smooth manifolds (possibly of different dimensions). These $L_{a}$'s are called switching components. \par
	In order to make wrapped Floer theory well behaved for such Lagrangian immersions, we must impose further conditions: being exact and cylindrical. Exactness is the same as before: there exists $f: L \to \mathbb{R}$ such that $df = \iota^{*}\lambda_{M}$, while being cylindrical is slightly different. \par

\begin{definition}
	Let $\iota: L \to M$ be an exact Lagrangian immersion with clean self-intersections. We say that it is exact, if there exists $f: L \to \mathbb{R}$ such that $df = \iota^{*}\lambda_{M}$. \par
	We say that $\iota: L \to M$ is cylindrical, if there exists an embedded closed Legendrian submanifold $l$ of $\partial M$ such that the geometric image $\iota(L)$ satisfies
\begin{equation}
\iota(L) \cap (\partial M \times [1, +\infty)) = l \times [1, +\infty),
\end{equation}
and moreover, the preimage $\iota^{-1}(l \times [1, +\infty)$ is a union of copies of $l \times [1, +\infty)$, so that the restriction of $\iota$ is a trivial discrete covering map.
\end{definition}

	For an exact cylindrical Lagrangian immersion $\iota: L \to M$ with clean self-intersections, we have a finer description of the decomposition of the fiber product $L \times_{\iota} L$ into its connected components. That is, we shall specify those connected components which are mapped to the cylindrical end $l \times [1, +\infty)$ of the image. We denote
\begin{equation}
L \times_{\iota} L = \coprod L_{a} \cup \coprod L_{b},
\end{equation}
where the $L_{b}$'s are the connected components part of which are mapped to the cylindrical end $l \times [1, +\infty)$, including the diagonal component $\Delta_{L} \cong L$, and the $L_{a}$'s are the ones which are not, and correspond to those self-intersections only contained in the interior part of $M$. \par
	Here is a more refined description of the components $L_{b}$. Suppose that $\iota^{-1}(l \times [1, +\infty)$ is a union of copies of $l \times [1, +\infty)$, labeled by $i \in I$ for some index set $I$, which can be infinite. Then the labels $b$ correspond to pairs $(i, j)$ for $i, j \in I$, distinguishing the copies $L_{b}$ in the fiber product. \par
	For wrapped Floer theory to be well-defined and to give desired $A_{\infty}$-structures, we shall from now on assume that the immersion $\iota: L \to M$ be proper. In particular, the covering of the cylindrical end $l \times [1, +\infty)$ has at most finite sheets, say $d$-fold. \par

\subsection{Local systems}
	Similar to the setup of Morse homology of a Morse-Bott function, we need to take into account certain local systems on the components $L_{a}$ and $L_{b}$ of the fiber product $L \times_{\iota} L$, in order to obtain canonical orientations of the moduli spaces of pseudoholomorphic disks considered in wrapped Floer theory. Most of the definitions follow from Chapter 8 of \cite{FOOO2}, so we just list the essential definitions that we need to fix notations, provide part of the proofs while leaving the full details to the reference. \par
	Let $L_{a}$ be any connected component different from the diagonal component $\Delta_{L} \cong L$. Since $\iota$ is a Lagrangian immersion, for each $x \in L_{a}$ we get two Lagrangian subspaces
\begin{align}\label{the left and right tangent subspaces at a self-intersection}
\lambda_{a, x, l} = d\iota_{x}(T_{pr_{1}(x)}L),\\
\lambda_{a, x, r} = d\iota_{x}(T_{pr_{2}(x)}L)
\end{align}
of $T_{\iota(x)}M$, where $pr_{1}, pr_{2}: L_{a} \subset L \times_{\iota} L \to L$ are induced by the projections to the two factors. \par
	Let $\mathcal{P}_{a, x}$ be the set of all smooth maps $\lambda_{a, x}: [0, 1] \to \mathcal{LAG}(T_{\iota(x)}M)$, such that $\lambda_{a, x}(0) = \lambda_{a, x, l}, \lambda_{a, x}(1) = \lambda_{a, x, r}$. Associated to each $\lambda_{a, x} \in \mathcal{P}_{a, x}$, there are two Cauchy-Riemann operators
\begin{equation}
\bar{\partial}_{\lambda_{a, x}}^{-}: L^{p}_{1, \delta}(Z_{-}; T_{\iota(x)}M, \lambda_{a, x}) \to L^{p}_{\delta}(Z_{-}; T_{\iota(x)}M)
\end{equation}
and
\begin{equation}
\bar{\partial}_{\lambda_{a, x}}^{+}: L^{p}_{1, \delta}(Z_{+}; T_{\iota(x)}M, \lambda_{a, x}) \to L^{p}_{\delta}(Z_{+}; T_{\iota(x)}M)
\end{equation}
on the weighted Sobolev spaces. Here
\begin{equation}
Z_{-} = \{z \in \mathbb{C}| Re(z) \le 0, |z| \le 1\} \cup \{z \in \mathbb{C}| Re(z) \ge 0, -1 \le Im(z) \le 1\},
\end{equation}
and
\begin{equation}
Z_{+} = \{z \in \mathbb{C}| Re(z) \ge 0, |z| \le 1\} \cup \{z \in \mathbb{C}| Re(z) \le 0, -1 \le Im(z) \le 1\}.
\end{equation}
The weighted Sobolev space $L^{p}_{1, \delta}(Z_{-}; T_{\iota(x)}M, \lambda_{a, x})$ is the $L^{p}_{1, \delta}$-completion of the space of smooth maps $u: Z_{-} \to T_{\iota(x)}M$ satisfying the following conditions:
\begin{enumerate}[label=(\roman*)]

\item $u(s+\sqrt{-1}) \in \lambda_{a, x, l}$, for all $s \ge 0$;

\item $u(s-\sqrt{-1}) \in \lambda_{a, x, r}$, for all $s \ge 0$;

\item $u(e^{\sqrt{-1}(\frac{\pi}{2}+\pi t)}) \in \lambda_{a, x}(t)$, for $t \in [0, 1]$;

\item $\int_{Z_{-}} e^{\delta |Re(z)|} ||\nabla u||^{p} dzd\bar{z} < \infty$.

\end{enumerate}
The other weighted Sobolev spaces are defined in similar fashion. \par
	These operators $\bar{\partial}_{\lambda_{a, x}}^{-}$ and $\bar{\partial}_{\lambda_{a, x}}^{+}$ are Fredholm. Consider their determinant lines:
\begin{equation}\label{determinant lines of the orientation operators}
\Theta_{\lambda_{a, x}}^{\pm} = \det(\bar{\partial}_{\lambda_{a, x}}^{\pm}).
\end{equation}
We wish to move $x$ as well as $\lambda_{a, x}$, so that $\bar{\partial}_{\lambda_{a, x}}^{\pm}$ form a family index bundle $Ind(D^{\pm})$, and the associated determinant line bundle $\underline{\det} Ind(D^{\pm})$ has fiber being \eqref{determinant lines of the orientation operators}. To make sense of this discussion, we shall first define the space over which the family index bundle is defined. \par

\begin{definition}
	Define a fiber bundle $\mathcal{I}_{a, x} \to \mathcal{P}_{a, x}$ in the following five steps:
\begin{enumerate}[label=(\roman*)]

\item First define $(I_{a, x})_{\lambda_{a, x}}$ to be the space of all smooth maps $\sigma_{a, x}: [0, 1] \times \mathbb{R}^{n} \to TM$ such that for each $t \in [0, 1]$, the map $\sigma_{a, x}(t, \cdot)$ is a linear isometry between $\mathbb{R}^{n}$ and $\lambda_{a, x}(t)$. That is, $(I_{a, x})_{\lambda_{a, x}}$ is the space of trivializations along the path $\lambda_{a, x}$ of Lagrangian subspaces in $T_{\iota(x)}M$.

\item Let $P_{SO}(L)$ be the principal $SO(n)$-bundle associated to the tangent bundle of $L$. For the given points $p_{\pm}$, let $P_{Spin}(L)_{p_{\pm}}$ be the double cover of the fiber $P_{SO}(L)_{p_{\pm}}$ of $P_{SO}(L)$ at $p_{\pm}$. If $x = (p_{-}, p_{+}) \in L_{a}$, we set
\begin{equation}\label{the spin fibers}
P_{x} = (P_{Spin}(L)_{p_{-}} \times P_{Spin}(L)_{p_{+}})/\{\pm 1\}.
\end{equation}

\item Define a map $(I_{a, x})_{\lambda_{a, x}} \to P_{SO}(L)_{p_{-}} \times P_{SO}(L)_{p_{+}}$ as follows. For each $\sigma_{a, x}$, consider its restriction to the endpoints $t = 0, 1$. By definition, $\sigma(0, \cdot)$ is an isometry between $\mathbb{R}^{n}$ and $d\iota_{x}(T_{p_{-}}L)$, hence is canonically identified as an element in the fiber $P_{SO}(L)_{p_{-}}$. A parallel argument applies to $t = 1$.

\item Then define $(\mathcal{I}_{a, x})_{\lambda_{a, x}}$ to be the fiber product
\begin{equation}
(\mathcal{I}_{a, x})_{\lambda_{a, x}} = (I_{a, x})_{\lambda_{a, x}} \times_{P_{SO}(L)_{p_{-}} \times P_{SO}(L)_{p_{+}}} P_{x}.
\end{equation}

\item Finally, we consider the union over all paths $\lambda_{a, x}$:
\begin{equation}
\mathcal{I}_{a, x} = \cup_{\lambda_{a, x} \in \mathcal{P}_{a, x}} (\mathcal{I}_{a, x})_{\lambda_{a, x}}.
\end{equation}
This is the desired fiber bundle $\mathcal{I}_{a, x} \to \mathcal{P}_{a, x}$.

\end{enumerate}
\end{definition}

\begin{lemma}
	Suppose $L$ is spin with a chosen spin structure. Then the union $\cup_{x \in L_{a}} \mathcal{I}_{a, x}$ restricts to a fiber bundle $\mathcal{I}_{a}$ over the $3$-skeleton of $L_{a}$.
\end{lemma}
\begin{proof}
	Since $L$ is spin, there is a globally defined fiberwise double cover $P_{Spin}(L)$ of $P_{SO}(L)$ over the $3$-skeleton $(L_{a})_{[3]}$ of $L_{a}$, determined by the spin structure. Thus when defining $\mathcal{I}_{a, x}$, the definition of $P_{x}$ as in \eqref{the spin fibers} can be made globally over $(L_{a})_{[3]}$.
\end{proof}

	From now on we shall always assume $L$ to be spin, with a chosen spin structure $v$. Returning to the concern about family index bundles, by moving $x$ and $\lambda_{a, x}$, the operators $\bar{\partial}_{\lambda_{a, x}}^{\pm}$ form a family index bundle $Ind(D^{\pm})$ over $\mathcal{I}_{a}$, whose determinant line bundle $\underline{\det} Ind(D^{\pm})$ is a real line bundle with fiber $\Theta_{\lambda_{a, x}}^{\pm}$. \par
	First observe that: \par

\begin{lemma}\label{triviality of determinant line bundle on the fiber}
	On each fiber $\mathcal{I}_{a, x}$ of the fiber bundle $\mathcal{I}_{a} \to (L_{a})_{[3]}$, the pullback of the determinant line bundle $\underline{\det} Ind(D^{\pm})$ is trivial.
\end{lemma}
\begin{proof}
	Fix a reference point $[\lambda_{a, x}, \sigma_{a, x}, s_{1}, s_{2}] \in \mathcal{I}_{a, x}$. Consider the family of operators $D'^{-} = \{\bar{\partial}_{\lambda'_{a, x}, Z_{-}}\}$ parametrized by $[\lambda'_{a, x}, \sigma'_{a, x}, s'_{1}, s'_{2}] \in \mathcal{I}_{a, x}$. By gluing $D'^{-}$ with $\bar{\partial}_{\lambda_{a, x}, Z_{+}}$, where the latter is a single operator regarded as a constant family, we obtain a family of Dolbeault operators with domain $D^{2}$, with boundary conditions given by the family of real sub-bundles parametrized by $S^{1} = \partial D^{2}$, specified by the union of the paths $\lambda_{a, x}$ and $\lambda'_{a, x}$. \par
	Since $[\lambda_{a, x}, \sigma_{a, x}, s_{1}, s_{2}]$ and $[\lambda'_{a, x}, \sigma'_{a, x}, s'_{1}, s'_{2}]$ determine the spin structures on the family of real sub-bundles consistently, the determinant line bundle of the family of Dolbeault operators is trivial. By definition, this family of Dolbeault operators is obtained from gluing $D'^{-}$ with a constant family of operators, so the determinant line bundle of the family $D'^{-}$ is also trivial, which completes the proof. \par
\end{proof}

\begin{lemma}
	There exist local systems $\Theta_{a}^{\pm}$ on $L_{a}$, such that their pullbacks to $\mathcal{I}_{a}$ are isomorphic to $Ind(D^{\pm})$. Moreover, there is an isomorphism
\begin{equation}
\Theta_{a}^{-} \otimes \Theta_{a}^{+} \cong \det TL_{a}.
\end{equation}
\end{lemma}
\begin{proof}
	Recall that $\mathcal{P}_{a, x}$ is the space of smooth paths in $\mathcal{LAG}(T_{\iota(x)}M)$ connecting the Lagrangian subspaces $\lambda_{a, x, l}$ and $\lambda_{a, x, r}$ in \eqref{the left and right tangent subspaces at a self-intersection}. Hence it is homotopy equivalent to the based loop-space of the Lagrangian Grassmannian $\mathcal{LAG}(n)$ of linear Lagrangian subspaces of $\mathbb{R}^{2n}$. Set $I_{a, x} = \cup_{\lambda_{a, x}} (I_{a, x})_{\lambda_{a, x}}$. The fiber bundle $I_{a, x} \to \mathcal{P}_{a, x}$ is homotopy equivalent to the free loop-space of the Lagrangian Grassmannian $\mathcal{LAG}(n)$.  Therefore, $\pi_{0}(I_{a, x}) = \mathbb{Z}$ so that $I_{a, x}$ has $\mathbb{Z}$-worth of connected components, labeled by $I_{a, x; k}$. \par
	Also recall that $\mathcal{I}_{a, x}$ is a double cover of $I_{a, x}$, which is non-trivial. Let $\mathcal{I}_{a, x; k}$ be the pullback of $I_{a, x; k}$ to $\mathcal{I}_{a, x}$, which is therefore connected as $I_{a, x; k}$ is connected. Let $\mathcal{I}_{a; k}$ be the union of $\mathcal{I}_{a, x; k}$ over $x \in (L_{a})_{[3]}$. From the fibration 
\begin{equation*}
\mathcal{I}_{a, x; k} \xhookrightarrow{} \mathcal{I}_{a; k} \to (L_{a})_{[3]},
\end{equation*}
we obtain a long exact sequence of homotopy groups:
\begin{equation}
\pi_{1}(\mathcal{I}_{a, x; k}) \to \pi_{1}(\mathcal{I}_{a; k}) \to \pi_{1}((L_{a})_{[3]}) \to \{1\}.
\end{equation}
This implies that $\pi_{1}(\mathcal{I}_{a; k}) \to \pi_{1}((L_{a})_{[3]})$ is surjective. Now we have a real line bundle $\Theta_{a}^{\pm}$ on $(L_{a})_{[3]}$, which is classified by a homomorphism $\pi_{1}((L_{a})_{[3]}) \to \mathbb{Z}/2$, whose pullback to the homomorphism $\pi_{1}(\mathcal{I}_{a; k}) \to \mathbb{Z}/2$ defined by the real line bundle $\underline{\det} Ind(D^{\pm})$. The homomorphism $\pi_{1}(\mathcal{I}_{a; k}) \to \mathbb{Z}/2$ is well-defined because of Lemma \ref{triviality of determinant line bundle on the fiber}. \par
	Since any line bundle on the $3$-skeleton has a unique extension to the whole space, we obtain the desired $\Theta_{a}^{\pm}$ on $L_{a}$.
\end{proof}

\begin{definition}
	The local system $\Theta_{a}^{-}$ is called the orientation local system for the cleanly self-intersecting component $L_{a}$.
\end{definition}

	The next lemma explains how the local systems change under different choices of spin structures. \par

\begin{lemma}
	Let $v_{1}, v_{2}$ be two spin structures on $L$. Let $\Theta_{a}^{-}(v_{1}), \Theta_{a}^{-}(v_{2})$ be the orientation local systems defined by $v_{1}$ and $v_{2}$, respectively. Then the local system
\begin{equation*}
\Theta_{a}^{-}(v_{1}) \otimes \Theta_{a}^{-}(v_{2})
\end{equation*}
is classified by the $\mathbb{Z}/2$-cohomology class
\begin{equation}
pr_{1}^{*}(v_{1} - v_{2}) - pr_{2}^{*}(v_{1} - v_{2}) \in H^{1}(L_{a}, \mathbb{Z}/2).
\end{equation}
\end{lemma}

	The proof of this lemma also follows from an argument by gluing families of elliptic operators. As we shall not quite use it, we refer the reader to Chapter 8 of \cite{FOOO2} for the detailed proof. \par

\subsection{The wrapped Floer cochain space in the presence of clean self-intersections}\label{section: wrapped Floer cochain space in case of clean self-intersections}
	Let $\iota: L \to M$ be a cylindrical proper Lagrangian immersion with clean self-intersections, for which we have fixed a grading and a spin structure for it. As in the case of a cylindrical Lagrangian immersion with transverse self-intersections, the wrapped Floer cochain space for $(L, \iota)$ should capture both the topology of the fiber product $L \times_{\iota} L$ as well as the non-constant time-one Hamiltonian chords in the cylindrical end of $M$ starting and ending on $\iota(L)$. And because $\iota$ is a covering to its cylindrical end, we should keep track of this information when studying the non-constant Hamiltonian chords. \par
	To pick a chain model for the wrapped Floer cochain space, we choose an auxiliary Morse function $f_{a}$ on each connected component $L_{a}$ of the fiber product $L \times_{\iota} L$. Let $p_{a, j}$ be the critical points of $f_{a}$. The corresponding Morse complex computes the cohomology of the fiber product $L \times_{\iota} L$. Let $\mathcal{X}_{+}(\iota(L); H)$ be the set of all non-constant time-one $H$-chords from $\iota(L)$ (the geometric image) to itself which are contained in the cylindrical end $M \times [1, +\infty)$. These non-constant time-one $H$-chords naturally correspond to Reeb chords of all times from $l$ to itself on the contact boundary $\partial M$. \par

\begin{definition}\label{definition of wrapped Floer cochain space for a cylindrical Lagrangian immersion with clean self-intersections}
	The wrapped Floer cochain space $CW^{*}(L, \iota; H)$ is defined in a similar way to Definition \ref{definition of wrapped Floer cochain space for a single Lagrangian immersion} as the free $\mathbb{Z}$-module generated by the following two kinds of generators
\begin{enumerate}[label=(\roman*)]

\item $(p, w) \otimes \theta_{p}$, where $p \in Crit(f_{a})$, $w$ is a $\Gamma$-equivalence class of capping half-disks for $p$, and $\theta_{p} \in (\Theta_{a}^{-})_{p}$;

\item $(x, b)$, where $x$ is a non-constant time-one $H$-chord from $\iota(L)$ to itself which is contained in the cylindrical end $\partial M \times [1, +\infty)$, and $b$ is a lifting index, corresponding to a pair $(i, j)$ where $i, j$ label the copies of the preimage of the covering $\iota$ when restricted to the cylindrical end.

\end{enumerate}
That is, 
\begin{equation}
CW^{*}(L, \iota; H) = (\bigoplus_{a} CM^{*}(L_{a}, f_{a}; \Theta_{a}^{-}) \oplus \bigoplus_{b = (i, j) = (1, 1)}^{(d, d)} \mathbb{Z}\mathcal{X}_{+}(\iota(L); H).
\end{equation}
Here $\bigoplus_{b} \mathbb{Z}\mathcal{X}_{+}(\iota(L); H)$ means the direct sum of several copies of $\mathbb{Z}\mathcal{X}_{+}(\iota(L); H)$, one for each index $b = (i, j)$.
\end{definition}

	In the definition of the wrapped Floer cochain space, we include several copies of the free module generated by non-constant $H$-chords, in order to keep track of which component of the preimage the boundary map is lifted to. \par
	In particular, when $\iota: L \to M$ is a proper embedding, we see that this cochain space is isomorphic to the usual Morse-Bott wrapped Floer cochain space. \par

\subsection{Moduli space of stable pearly trees and Kuranishi structures}
	To associate a curved $A_{\infty}$-algebra to an exact cylindrical Lagrangian immersion with clean self-intersections, we use similar kind of moduli spaces of stable pearly tree maps. \par
	Choosing a Morse cochain complex to compute the cohomology of the fiber product $L \times_{\iota} L$ means we have implicitly perturb the Morse-Bott submanifolds $L_{a}$ to isolated critical points. This implies that we can define stable pearly tree maps in a similar way to those in the case of a cylindrical Lagrangian immersion with transverse self-intersections, while there are two differences, which are to be discussed below. \par
	First, in Definition \ref{stable pearly tree map}, we need to change the definition of the map $\alpha: I \to S(L, \iota)$, as in the clean self-intersection case there is no longer a discrete set $S(L, \iota)$ of preimages of self-intersection points. We simply change it to a set-valued map
\begin{equation}\label{set-valued switching label}
\alpha: I \to \{L_{a}: a \in A \setminus \{0\} \},
\end{equation}
among all the self-intersection components different from the diagonal component $\Delta_{L} \cong L$. \par
	Second, the condition $(x)$ in Definition \ref{stable pearly tree map} should also be modified to the following condition:
\begin{equation}\label{modified switching condition}
\begin{split}
(x)': &(\lim\limits_{\theta \uparrow 0} l(e^{\sqrt{-1}\theta}\zeta_{i}), \lim\limits_{\theta \downarrow 0} l(e^{\sqrt{-1}\theta}\zeta_{i})) \in \alpha(i) \text{ for } i \in I.\\
& \text{In addition }, \tilde{u}_{e_{i}}(z_{v(e_{i})}) = p_{a, j(i)} \in \iota(\alpha(i)), \text{ for some critical point } p_{a, j(i)} \text{ on } \alpha(i).
\end{split}
\end{equation}\par

	Third, we have the orientation local systems $\Theta_{a}^{-}$ on the Morse-Bott submanifolds $L_{a}$ that needs to be taken into account. We shall explain the role of the orientation local systems in the rest of this subsection. \par
	Let $\bar{\mathcal{M}}_{k+1}(\alpha, \beta; J, H; c_{0}, \cdots, c_{k})$ be the moduli space of stable broken pearly tree maps defined in a similar way as that in section \ref{section: moduli space of disks bounded by immersed Lagrangian submanifolds}, with the map $\alpha$ replaced by \eqref{set-valued switching label}, and the condition $(x)$ replaced by $(x)'$ in \eqref{modified switching condition}. As in the case of a cylindrical Lagrangian immersion with transverse self-intersections, we have the following structure results on the moduli space. \par

\begin{proposition}
	The moduli space $\bar{\mathcal{M}}_{k+1}(\alpha, \beta; J, H; c_{0}, \cdots, c_{k})$ has a Kuranishi structure with corners, which is compatible at the boundary with the fiber product Kuranishi structures on
\begin{equation*}
\bar{\mathcal{M}}_{k_{0}+1}(\alpha_{0}, \beta_{0}; J, H; c_{0, 0}, \cdots, c_{0, k_{0}}) \times \cdots \times \bar{\mathcal{M}}_{k_{m}+1}(\alpha_{m}, \beta_{m}; J, H; c_{m, 0}, \cdots, c_{m, k_{m}})
\end{equation*}
as in \eqref{boundary strata of moduli space of pearly trees in the form of fiber products}.
\end{proposition}

\begin{remark}
	In the above proposition, we did not mention the word "oriented", so for the moment the statement holds for unoriented Kuranishi structures. In the rest of this subsection we shall describe the orientations on these Kuranishi structures.
\end{remark}

	Next, we shall orient $\bar{\mathcal{M}}_{k+1}(\alpha, \beta; J, H; c_{0}, \cdots, c_{k})$ using the orientation local systems $\Theta_{a}^{-}$ on $L_{a}$. First we introduce some notations. \par

\begin{definition}
	The type function $\tau$ is define by
\begin{equation}
\tau(i) =
\begin{cases}
a, &\text{if } c_{i} \text{ corresponds to a critical point of } f_{a} \text{ on } L_{a};\\
\infty, &\text{if } c_{i} \text{ corresponds to a non-constant $H$-chord, appearing in the $b$-th copy}.
\end{cases}
\end{equation}
\end{definition}

	By convention, we set $\Theta_{\infty}^{-} = \underline{\mathbb{R}}$ the trivial line bundle, and also $\det T^{*}L_{\infty} = \underline{\mathbb{R}}$. \par

\begin{proposition}
	The above Kuranishi structure is oriented. At each point 
\begin{equation*}
\sigma \in \bar{\mathcal{M}}_{k+1}(\alpha, \beta; J, H; c_{0}, \cdots, c_{k}),
\end{equation*}
if $(U_{\sigma}, E_{\sigma}, s_{\sigma}, \Gamma_{\sigma} = \{1\})$ is the Kuranishi chart, then we have a canonical isomorphism of the orientation line bundles
\begin{equation}\label{orientation on the moduli space of pearly tree maps in the case of clean self-intersection}
\begin{split}
&\det TU_{\sigma} \otimes \det E_{\sigma}^{*}\\ 
&\cong
o_{c_{0}} \otimes \Theta_{\tau(0)}^{-} \otimes (o_{c_{1}} \otimes \Theta_{\tau(1)}^{-} \otimes \det T^{*}L_{\tau(1)})^{-1} \otimes \cdots \otimes (o_{c_{k}} \otimes \Theta_{\tau(k)}^{-} \otimes \det T^{*}L_{\tau(k)})^{-1},
\end{split}
\end{equation}
where $o_{c_{i}}$ is the orientation line of the generator $c_{i}$, which is the orientation line for a non-degenerate time-one $H$-chord, if $c_{i}$ corresponds to a non-constant $H$-chord, or the orientation line for a critical point of $f_{a}$, if $c_{i}$ corresponds to a critical point of $f_{a}$, defined by a choice of coherent orientations on the unstable manifolds of all the critical points, twisted by the orientation local system $\Theta_{a}^{-}$.
\end{proposition}
\begin{proof}
	The proof is an application of the standard gluing theorem for Cauchy-Riemann operators. We glue the operation operators $D_{c_{0}}^{-}, D_{c_{1}}^{+}, \cdots, D_{c_{k}}^{+}$ to the Fredholm operator
\begin{equation*}
D_{\sigma}: W^{1, p} \to L^{p}
\end{equation*}
and obtain a Cauchy-Riemann operator $D$ with Lagrangian boundary conditions. The operator $D$ can be deformed through Fredholm operators to the standard Dolbeault operator $\bar{\partial}$, whose determinant line is canonically trivialized. On the one hand, we have by definition:
\begin{align*}
&\det D_{c_{0}}^{-} = o_{c_{0}} \otimes \Theta_{\tau(0)}^{-},\\
&\det D_{c_{i}}^{+} = o_{c_{0}}^{-1} \otimes \Theta_{\tau(i)}^{+}, i = 1, \cdots, k.
\end{align*}
On the other hand, at the point $\sigma$, there is a canonical isomorphism
\begin{equation*}
\det TU_{\sigma} \otimes \det E_{\sigma}^{*} \cong \det D_{\sigma},
\end{equation*}
because $TU_{\sigma} \to E_{\sigma}$ is a finite-dimensional reduction of the Fredholm complex for $D_{\sigma}$. At last, recall that $\Theta_{a}^{-} \otimes \Theta_{a}^{+} \cong \det TL_{a}$. Thus we deduce that
\begin{equation*}
\det D_{\sigma} \otimes (o_{c_{0}} \otimes \Theta_{\tau(0)}^{-}) \otimes (o_{c_{1}} \otimes \Theta_{\tau(1)}^{-} \otimes \det T^{*}L_{\tau(1)})^{-1} \otimes \cdots \otimes (o_{c_{k}} \otimes \Theta_{\tau(k)}^{-} \otimes \det T^{*}L_{\tau(k)})^{-1}
\end{equation*}
is canonically trivialized. This implies the desired isomorphism \eqref{orientation on the moduli space of pearly tree maps in the case of clean self-intersection}. \par
\end{proof}

\begin{corollary}
	The Kuranishi structure on $\bar{\mathcal{M}}_{k+1}(\alpha, \beta; J, H; c_{0}, \cdots, c_{k})$ is oriented. Moreover, the induced orientation at the boundary agrees with the fiber product orientation.
\end{corollary}

	From this point, the remaining part of the construction of a curved $A_{\infty}$-algebra associated to $(L, \iota)$ is the same as that in the case of a cylindrical Lagrangian immersion with transverse self-intersections. \par

\subsection{A cleanly-intersecting pair}\label{section: wrapped Floer cochain space for a pair with clean intersections}
	Now consider a pair $(\iota_{0}: L_{0} \to M, \iota_{1}: L_{1} \to M)$ of cylindrical proper Lagrangian immersions with clean self-intersections. The goal of this subsection is to define wrapped Floer cohomology for such a pair in the case where the two Lagrangian immersions intersect cleanly in the following sense. \par

\begin{definition}
	The pair $(\iota_{0}: L_{0} \to M, \iota_{1}: L_{1} \to M)$ is said to have clean intersections, if the following conditions are satisfied:
\begin{enumerate}[label=(\roman*)]

\item the fiber product
\begin{equation*}
L_{0} \times_{\iota_{0}, \iota_{1}} L_{1}
\end{equation*}
is a smooth manifold, possibly disconnected with different components having different dimensions,
\begin{equation}
L_{0} \times_{\iota_{0}, \iota_{1}} L_{1} = \coprod_{a} C_{a}.
\end{equation}

\item the tangent space of the fiber product at each point is given by
\begin{equation}
T_{(p_{0}, p_{1})}(L_{0} \times_{\iota_{0}, \iota_{1}} L_{1}) = \{(V_{0}, V_{1}) \in T_{p_{0}}L_{0} \times T_{p_{1}}L_{1}: d_{p_{0}}\iota_{0}(V_{0}) = d_{p_{1}}\iota_{1}(V_{1})\}.
\end{equation}

\end{enumerate}
\end{definition}

	In wrapped Floer theory, it is important to keep track of the geometry of the cylindrical ends. Recall that for $i = 0, 1$, $\iota_{i}: L_{i} \to M$ is assumed to be a discrete trivial covering of a cylindrical end $l_{i} \times [1, +\infty)$ over the cylindrical end $\partial M \times [1, +\infty)$ modeled on some Legendrian submanifold $l_{i}$ of $\partial M$. And by the properness assumption, the covering is finitely-sheeted, say $d_{i}$-fold. We shall make the following assumption on how this pair intersects, distinguishing from the case of a single Lagrangian immersion with clean self-intersections. \par

\begin{assumption}
	$\iota_{0}: L_{0} \to M$ and $\iota_{1}: L_{1} \to M$ do not intersect in the cylindrical end $\partial M \times [1, +\infty)$. That is, the Legendrian boundaries $l_{0}, l_{1}$ do not intersect in the contact boundary $\partial M$.
\end{assumption}

	This is essentially reduced to an assumption on Legendrian submanifolds of $\partial M$, which can be achieved generically, in case $l_{0} \neq l_{1}$. The case $l_{0} = l_{1}$ but $\iota_{0} \neq \iota_{1}$ is different and slightly more involved, which can be studied in a similar way but will not be discussed in this paper. \par
	To set up wrapped Floer theory for such a pair of Lagrangian immersions, we need to choose a chain model for the wrapped Floer cochain space. We shall for each component $C_{a}$ an auxiliary Morse function $f_{a}: C_{a} \to \mathbb{R}$, which is $C^{2}$-small and satisfies the Morse-Smale condition. Let $Crit(f_{a})$ be the set of critical points of $f_{a}$. The specific choice will be made below. \par

\begin{lemma}
	There exists a $C^{2}$-small generic perturbation $K$ of $H$ so that all the time-one $K$-chords that are contained in the interior part of $M$ are non-degenerate and constant. Moreover, these $K$-chords correspond bijectively to the critical points of the lift of $K$ to $C_{a}$.
\end{lemma}

	Thus it is natural to choose $f_{a}$ to be the lift of $K$ to $C_{a}$. \par
	As in the case of a single cylindrical Lagrangian immersion with clean self-intersections, we have rank-one $\mathbb{Z}/2$-local systems $\Theta_{a}^{\pm}$ on $C_{a}$. They satisfy $\Theta_{a}^{-} \otimes \Theta_{a}^{+} \cong TC_{a}$. We call $\Theta_{a}^{-}$ the orientation local system on $C_{a}$. \par
	Let $\mathcal{X}_{+}(\iota_{0}(L_{0}), \iota_{1}(L_{1}); H)$ be the set of non-constant time-one $H$-chords from $\iota_{0}(L_{0})$ to $\iota_{1}(L_{1})$. These $H$-chords are contained in the cylindrical end $\partial M \times [1, +\infty)$ and naturally correspond to Reeb chords on the contact manifold $\partial M$ from $l_{0}$ to $l_{1}$ of all lengths. \par
	The wrapped Floer cochain space $CW^{*}((L_{0}, \iota_{0}), (L_{1}, \iota_{1}); H)$ for a pair of exact cylindrical Lagrangian immersions with clean intersections is defined in a way similar to Definition  \ref{definition of wrapped Floer cochain space for a pair}. \par

\begin{definition}
	The wrapped Floer cochain space $CW^{*}((L_{0}, \iota_{0}), (L_{1}, \iota_{1}); H)$ for a pair of exact cylindrical Lagrangian immersions with clean intersections is the free $\mathbb{Z}$-module generated by the following two kinds of generators:
\begin{enumerate}

\item $(p, w) \otimes \theta_{p}$, where $p \in Crit(f_{a})$, $w$ is a $\Gamma$-equivalence class of capping half-disks for $p$, and $\theta_{p} \in (\Theta_{a}^{-})_{p}$;

\item $(x, b)$, where $x$ is a non-constant time-one $H$-chord from $\iota_{0}(L_{0})$ to $\iota_{1}(L_{1})$ which is contained in the cylindrical end $\partial M \times [1, +\infty)$, and $b$ is a lifting index, corresponding to a pair $(i, j)$ where $i$ labels the copy of the preimage of the covering $\iota_{0}$, and $j$ labels the copy of the preimage of the cover $\iota_{1}$, when restricted to the cylindrical end.

\end{enumerate} \par
That is,
\begin{equation}
CW^{*}((L_{0}, \iota_{0}), (L_{1}, \iota_{1}); H) = (\bigoplus_{a}CM^{*}(C_{a}, f_{a}; \Theta_{a}^{-}) \oplus \bigoplus_{b = (i, j) = (1, 1)}^{(d_{0}, d_{1})} \mathbb{Z}\mathcal{X}_{+}(\iota_{0}(L_{0}), \iota_{1}(L_{1}); H).
\end{equation}
\end{definition}

\subsection{Floer trajectories between a cleanly-intersecting pair}
	The moduli spaces of Floer trajectories are similar to those in the case of a transversely intersecting pair of cylindrical Lagrangian immersions with transverse self-intersections. So we just outline the definitions emphasizing the differences. \par
	Choosing a Morse cochain complex to compute the cohomology of the fiber product $L_{0} \times_{\iota_{0}, \iota_{1}} L_{1}$ implicitly perturbs the Morse-Bott submanifolds $C_{a}$ to isolated critical points. Thus the only difference is that in the clean intersecting case we have the orientation local systems $\Theta_{a}^{-}$ on the components $C_{a}$ of the fiber product $L_{0} \times_{\iota_{0}, \iota_{1}} L_{1}$. The discussion on the role of these orientation local systems is parallel to that in the case of a single cylindrical Lagrangian immersion with clean self-intersections. We simply state the results as the proofs are completely similar. \par

\begin{proposition}
	There exists a Kuranishi structure on the moduli space of stable broken Floer trajectories
\begin{equation*}
\bar{\mathcal{N}}_{k, l}(\vec{\alpha}, \beta; \{J_{t}\}_{t}, H; c_{-}, c_{+}; \vec{c}^{0}, \vec{c}^{1}),
\end{equation*}
which is compatible with the fiber product Kuranishi structure on the boundary \eqref{boundary stratum of the moduli space of broken stable Floer trajectories}.
\end{proposition}

\begin{proposition}
	The Kuranishi structure on the moduli space
\begin{equation*}
\bar{\mathcal{N}}_{k, l}(\vec{\alpha}, \beta; \{J_{t}\}_{t}, H; c_{-}, c_{+}; \vec{c}^{0}, \vec{c}^{1})
\end{equation*}
is oriented. The orientation is defined as follows. If $\sigma$ is any point in the moduli space with Kuranishi chart $(U_{\sigma}, E_{\sigma}, s_{\sigma}, \Gamma_{\sigma} = \{1\})$, then there is an isomorphism
\begin{equation}
\begin{split}
\det TU_{\sigma} \otimes \det E_{\sigma}^{*} \cong &o_{c_{-}} \otimes \Theta_{\tau(-)}^{-} \otimes (o_{c_{+}} \otimes \Theta_{\tau(+)}^{-} \otimes \det T^{*}L_{\tau(+)})^{-1}\\
&\otimes \bigotimes_{i=1}^{k} (o_{c^{0}_{i}} \otimes \Theta_{\tau(0, i)}^{-} \otimes \det T^{*}L_{\tau(0, i)})^{-1}\\
&\otimes \bigotimes_{j=1}^{l} (o_{c^{1}_{j}} \otimes \Theta_{\tau(1, j)}^{-} \otimes \det T^{*}L_{\tau(1, j)})^{-1}.
\end{split}
\end{equation}
\end{proposition}

\subsection{On the category level}
	Now let us summarize the results above and extend them to a categorical level. This yields to the definition of the (unobstructed) immersed wrapped Fukaya category $\mathcal{W}_{im}(M)$. The objects are pairs $((\iota: L \to M), b)$ (or sometimes denoted by $(L, \iota, b)$), where $\iota: L \to M$ is an exact cylindrical Lagrangian immersion with transverse or clean self-intersections, and $b$ is a bounding cochain for it. The morphism space between a pair of objects is the wrapped Floer cochain space
\begin{equation*}
\hom_{\mathcal{W}_{im}(M)}((L_{0}, \iota_{0}, b_{0}), (L_{1}, \iota_{1}, b_{1})) = CW^{*}((L_{0}, \iota_{0}), (L_{1}, \iota_{1}); H),
\end{equation*}
which is independent of the bounding cochains as a $\mathbb{Z}$-module. However, we sometimes also write the morphism space as $CW^{*}((L_{0}, \iota_{0}, b_{0}), (L_{1}, \iota_{1}, b_{1}))$ to remember the choices of bounding cochains. The first order structure map $m^{1}_{\mathcal{W}_{im}(M)}$ for the pair $((L_{0}, \iota_{0}, b_{0}), (L_{1}, \iota_{1}, b_{1}))$ is given by the $(b_{0}, b_{1})$-deformed differential
\begin{equation*}
n^{0, 0; b_{0}, b_{1}}: CW^{*}((L_{0}, \iota_{0}), (L_{1}, \iota_{1}); H) \to CW^{*}((L_{0}, \iota_{0}), (L_{1}, \iota_{1}); H)[1],
\end{equation*}
where $[1]$ means the map is of degree $+1$.
Higher order structure maps
\begin{equation}
\begin{split}
m^{k}_{\mathcal{W}_{im}(M)}&: CW^{*}((L_{k-1}, \iota_{k-1}, b_{k-1}), (L_{k}, \iota_{k}, b_{k}); H) \otimes \cdots \otimes CW^{*}((L_{0}, \iota_{0}, b_{0}), (L_{1}, \iota_{1}, b_{1}); H)\\
& \to CW^{*}((L_{0}, \iota_{0}, b_{0}), (L_{k}, \iota_{k}, b_{k}); H)
\end{split}
\end{equation}
are defined by appropriate counts of inhomogeneous pseudoholomorphic disks, with insertions of bounding cochains $b_{i}$ on the boundary components of the disks. These disks are similar to those used in the definition of the ordinary wrapped Fukaya category, but now they are modeled as stable pearly tree maps as in subsection \ref{section: moduli space of disks bounded by immersed Lagrangian submanifolds}, with boundary conditions replaced by multiple Lagrangian immersions $\iota_{j}: L_{j} \to M, j = 0, \cdots, k$. \par

\begin{proposition}
	With the bounding cochains $b_{i}$'s taken into account, the structure maps $m^{k}_{\mathcal{W}_{im}(M)}$ satisfy the equations for a non-curved $A_{\infty}$-category over $\mathbb{Z}$, i.e. $m^{0}_{\mathcal{W}_{im}(M)} = 0$.
\end{proposition}

	The more formal definitions are based on the construction of Kuranishi structures on the moduli spaces of these stable pearly tree maps, choices of single-valued multisections of the Kuranishi charts, for each tuple of Lagrangian immersions involved in the structure maps. Technically, we should work with a countable collection of Lagrangian immersions in order to be able to make consistent choices of (abstract) perturbations. In that way, we should obtain a curved $A_{\infty}$-category $\mathcal{W}_{ob, im}(M)$ whose unobstructed deformation gives $\mathcal{W}_{im}(M)$. Moreover, this curved $A_{\infty}$-category should allow us to define the weakly unobstructed wrapped Fukaya category with bulk deformations. We will not discuss that construction in detail as it is not needed in this paper. However, the perspective of deformation theory will be of interest and addressed in an upcoming work on generalizing the Viterbo functor \cite{Gao2}. \par

\subsection{A quasi-embedding}

	A basic but important property of the immersed wrapped Fukaya category is that the ordinary wrapped Fukaya category, consisting of properly embedded exact cylindrical Lagrangian submanifolds, embeds into the immersed wrapped Fukaya category as a full sub-category. \par

\begin{proposition}\label{prop: the ordinary wrapped Fukaya category quasi-embeds into the immersed wrapped Fukaya category}
	The wrapped Fukaya category $\mathcal{W}(M)$ quasi-embeds into the immersed wrapped Fukaya category $\mathcal{W}_{im}(M)$. That is, there is a canonical cohomologically fully faithful $A_{\infty}$-functor
\begin{equation}\label{quasi-embedding of the ordinary wrapped Fukaya category into the immersed wrapped Fukaya category}
j: \mathcal{W}(M) \to \mathcal{W}_{im}(M).
\end{equation}
\end{proposition}
\begin{proof}
	There are two definitions for the ordinary wrapped Fukaya category $\mathcal{W}(M)$: one is using a Morse-Bott Hamiltonian as in the case of immersed wrapped Fukaya category, the other is using a non-degenerate Hamiltonian as usual. \par
	If we use the first definition of $\mathcal{W}(M)$, this is obvious, as $\mathcal{W}(M)$ is naturally an $A_{\infty}$-sub-category of $\mathcal{W}_{im}(M)$, thus we may simply take $j$ to be the natural inclusion. \par
	Let us consider the second definition. Let $K$ be a non-degenerate time-independent Hamiltonian, which is used to define the wrapped Fukaya category $\mathcal{W}(M)$. We shall define an $A_{\infty}$-functor
\begin{equation*}
j: \mathcal{W}(M) \to \mathcal{W}_{im}(M),
\end{equation*}
which on the level of objects acts by $L \mapsto (L, 0)$. The zeroth order term vanishes, $j^{0} = 0$, and the first order term
\begin{equation*}
j^{1}: CW^{*}(L_{i}, L_{j}; K) \to CW^{*}((L_{i}, 0), (L_{j}, 0); H)
\end{equation*}
is defined as follows. On the one hand, $CW^{*}(L_{i}, L_{j}; K)$ is generated by time-one $K$-chords from $L_{i}$ to $L_{j}$: these are either Hamiltonian chords that are contained in the interior part $M_{0}$, or those that are contained in the cylindrical end $\partial M \times [1, +\infty)$. On the other hand, $CW^{*}((L_{i}, 0), (L_{j}, 0); H)$ is generated by two kinds of generators: critical points of auxiliary Morse functions on the components of the intersection $L_{i} \cap L_{j}$, as well as time-one $K$-chord that are contained in the cylindrical end $\partial M \times [1, +\infty)$. Note that $H$ and $K$ agree in the cylindrical end, so that the corresponding Hamiltonian chords which are contained in the cylindrical end agree. \par
	If $L_{i}$ and $L_{j}$ intersect transversely in the interior part $M_{0}$, these critical points are in one-to-one correspondence with the isolated intersection points of $L_{i} \cap L_{j}$. In this case, we define the map $j^{1}$ as follows. $j^{1}$ is identity for Hamiltonian chords that are contained in the cylindrical end. For time -one $K$-chords that are contained in the interior part $M_{0}$, we define the map $j^{1}$ by counting rigid elements in the following parametrized moduli space $\mathcal{M}(p, x)$ of inhomogeneous pseudoholomorphic maps
\begin{equation*}
w: S \to M,
\end{equation*}
where $S$ is a disk with one positive boundary puncture $z_{+}$ and one boundary marked point $z_{0}$, satisfying the equation:
\begin{equation*}
(dw - dt \otimes X_{H_{S}})^{0, 1} = 0,
\end{equation*}
for a domain-dependent family of Hamiltonians $H_{S}$ which is equal to $K$ near $z_{+}$ and $H$ near $z_{-}$ (identifying $S$ with an infinite strip by removing $z_{-}$, we have a natural coordinate $t \in [0, 1]$), as well as the condition:
\begin{equation*}
w(z_{0}) = p, \lim\limits_{s \to +\infty} w \circ \epsilon_{+}(s, \cdot) = x(\cdot),
\end{equation*}
where $p$ is an intersection point of $L_{i} \cap L_{j}$, and $x$ is a time-one $K$-chord from $L_{i}$ to $L_{j}$ which is contained in the interior part $M_{0}$. \par
	If $L_{i}$ and $L_{j}$ intersect cleanly in the interior part $M_{0}$, there is a decomposition
\begin{equation*}
L_{i} \cap L_{j} = \coprod_{a} C_{a}
\end{equation*}
into connected components, so that the wrapped Floer cochain space is generated by critical points of $f_{a}$ for some chosen Morse functions $f_{a}$ on $C_{a}$. There is an orientation local system $\Theta_{a}^{-}$ so that the $\Theta_{a}^{-}$-twisted Morse complex of $f_{a}$ contributes to the wrapped Floer cochain space $CW^{*}((L_{i}, 0), (L_{j}, 0); H)$. In this case, we define the map $j^{1}$ on generators of $CW^{*}(L_{i}, L_{j}; K)$ that come from interior time-one $K$-chords using slightly modified moduli spaces $\mathcal{M}_{C_{a}}((\xi_{a})_{p_{a}}, x)$, defined as follows. Let $\bar{S}$ be union of a disk $S$ with one positive boundary puncture $z_{+}$ and one boundary marked point $z_{0}$ and a half-infinite ray $T \cong (-\infty, 0]$ joint to $S$ at the point $z_{0}$. Let $w: \bar{S} \to M$ be a continuous map which is the union of two maps $w_{S}: S \to M$ and $w_{T}: T \to C_{a} \subset M$, such that $w_{S}$ is an inhomogeneous pseudoholomorphic map with two boundary components lying on $L_{i}$ and $L_{j}$ respectively, and $w_{T}$ is a gradient flow line for $f_{a}$ on $C_{a}$. Moreover, the map $w$ satisfies the following conditions:
\begin{align*}
\lim\limits_{s \to -\infty} w_{T}(s) = p_{a},\\
w_{T}(0) = w_{S}(z_{0}),\\
\lim\limits_{s \to +\infty} w_{S} \circ \epsilon_{+}(s, \cdot) = x(\cdot).
\end{align*}
In addition, we impose the following monodromy condition: the monodromy of $w$ around the boundary of $\bar{S}$ (starting from $-\infty$ of the infinite-half ray $l$ to $0$, going around $\partial S$, then returning back to $-\infty$) is $\xi_{a} \in \Theta_{a}^{-}$. Counting rigid elements in such a moduli space defines the map $j^{1}$ on generators of $CW^{*}(L_{i}, L_{j}; K)$ coming from interior $K$-chords. \par
	Higher order terms can be defined in a similar way.
\end{proof}

\begin{remark}
	In the above proof we have constructed several moduli spaces $\mathcal{M}(p, x)$ and $\mathcal{M}_{C_{a}}(p_{a}, x)$ using which we define the map $j^{1}$. Since the Lagrangian submanifolds $L_{i}, L_{j}$ are embedded and exact, we can use classical transversality argument to prove that these moduli spaces are smooth manifolds of expected dimension for generic perturbation, so that we can actually count rigid elements. 
\end{remark}
	
\section{Wrapped Floer theory in the product manifold} \label{section: product manifolds}

\subsection{Overview}
	In view of Lagrangian correspondences as Lagrangian submanifolds in the product symplectic manifold, we need to study wrapped Floer theory of the product manifold in order to understand the construction of functors between the wrapped Fukaya categories. Then a somewhat technical but essential problem arises: there are two natural models for the wrapped Fukaya category of the product manifold, which are obviously equivalent. This causes some trouble in understanding wrapped Floer theory of the product manifold. \par
	Denote by $H_{M, N} = \pi_{M}^{*}H_{M} + \pi_{N}^{*}H_{N}$ the split Hamiltonian and $J_{M, N} = J_{M} \times J_{N}$ the product almost complex structure. Using $H_{M, N}$ and $J_{M, N}$, we can define a version of the wrapped Fukaya category of the product Liouville manifold $M \times N$, which we call the split model of wrapped Fukaya category of $M \times N$ and denote it by $\mathcal{W}^{s}(M \times N)$. On the other hand, there is a natural choice of the cylindrical end $\Sigma \times [1, +\infty)$ for $M \times N$ (see \cite{Gao1}), which allows us to define Hamiltonian functions $K$ that are quadratic at infinity. Therefore the ordinary wrapped Fukaya category of $M \times N$ can also be defined as usual. \par
	However one important thing is not mentioned explicitly. That is, the classes of Lagrangian submanifolds for which we can define wrapped Floer cohomology and $A_{\infty}$-algebras are a priori different with respect to the two kinds of Hamiltonians and almost complex structures. For the the split Hamiltonian $H_{M, N}$ and product almost complex structure $J_{M, N}$, the natural class of Lagrangian submanifolds are products of objects in $\mathcal{W}(M)$ and those in $\mathcal{W}(N)$. For the quadratic Hamiltonian $K$ and almost complex structure $J$, the natural class of Lagrangian submanifolds are cylindrical Lagrangian submanifolds with respect to the cylindrical end $\Sigma \times [1, +\infty)$. If we want to identify these two versions of wrapped Floer theory, we must ensure that both classes of objects can be included in each wrapped Fukaya category. One of the main results in \cite{Gao1} confirms that this is possible, which we will review in subsection \ref{section: well-definedness of wrapped Floer theory in the product}. \par
	This section is devoted to proving Theorem \ref{two models of wrapped Fukaya categories of product manifolds are equivalent}, which says $\mathcal{W}^{s}(M \times N)$ is quasi-equivalent to $\mathcal{W}(M \times N)$. Thus the ambiguity in differentiating these two models is removed, which allows us to better understand the functoriality properties of wrapped Fukaya categories with respect to Lagrangian correspondences. \par

\begin{remark}
	In this section all the wrapped Fukaya categories are to consist of embedded Lagrangian submanifolds, but the result should continue to hold in the immersed case, following similar lines of argument. Nonetheless, note that the action filtration no longer exists, but instead we need to use energy filtration instead.
\end{remark}

	The strategy of proof goes as follows. In \cite{Gao1}, the action-restriction map is constructed. We would like to extend it to an $A_{\infty}$-functor, including it as the first-order term. \par
	Fix an ordering on the collection $\mathbb{L} = \{\mathcal{L}_{i}\}_{i=1}^{\infty}$ and consider the finite collection $\mathcal{L}_{1}, \cdots, \mathcal{L}_{d}$. In the next subsection we will introduce certain geometric data which we call action-restriction data, which are used to define a $A_{\infty}$-functor $R_{d}$ from the $A_{\infty}$-subcategory $\mathcal{W}^{s}_{d}$ of $\mathcal{W}^{s}(\mathbb{L})$ consisting of $\mathcal{L}_{1}, \cdots, \mathcal{L}_{d}$ as objects to the $A_{\infty}$-subcategory $\mathcal{W}_{d}$ of $\mathcal{W}(\mathbb{L})$ with the same objects, which acts by identity on the level of objects, and induces quasi-isomorphisms on all morphism spaces. Therefore this $A_{\infty}$-functor $R_{d}$ is a quasi-isomorphism $\mathcal{W}^{s}_{d} \to \mathcal{W}_{d}$. Moreover, the $A_{\infty}$-functor $R_{d}$ will be cohomologically unital. \par
	This would finish the proof of a weak version of Theorem \ref{two models of wrapped Fukaya categories of product manifolds are equivalent} when $M \times N$ is non-degenerate, which means that the two models of wrapped Fukaya categories are both split-generated by finitely many Lagrangian submanifolds.  \par
	To deal with the general case, we need an additional argument with the help of homological algebra discussed in subsection \ref{A-infinity homotopy direct limit}. We will return to this point later. \par
	
\begin{remark}
	Of course, the quasi-isomorphism will depend on choices of action-restriction data. But for the purpose of comparing the two versions of wrapped Fukaya categories of the product, this does not matter, since the Fukaya category itself is only well-defined up to quasi-isomorphism, because of the flexibility in the choice of auxilliary data in Floer theory. The choices themselves form a contractible space, so we might work harder to show that different functors constructed in this way are homotopic, though we will not attempt to do so in this paper.
\end{remark}

\subsection{Well-definedness}\label{section: well-definedness of wrapped Floer theory in the product}
	One thing that needs explanation is why the same class of objects in $\mathcal{W}(M \times N)$ can be included in $\mathcal{W}^{s}(M \times N)$, and vice versa. This is discussed in detail in \cite{Gao1}. We recall some relevant notions. \par

\begin{definition}\label{definition of admissible Lagrangian submanifolds in the product}
	A Lagrangian submanifold $\mathcal{L} \subset M \times N$ is called admissible, if it is exact and cylindrical, and moreover satisfies $c_{1}(M \times N, \mathcal{L}) = 0$. More specifically, $\mathcal{L}$ can be one of the following kinds:
\begin{enumerate}[label=(\roman*)]

\item a product $L \times L'$, where $L$ is an exact cylindrical Lagrangian submanifold of $M$, and $L'$ is one of $N$;

\item a cylindrical Lagrangian submanifold with respect to the natural cylindrical end $\Sigma \times [1, +\infty)$.

\end{enumerate}
\end{definition}

\begin{remark}
	Note that the notion of being cylindrical only depends on the Liouville structure, but not a particular choice of cylindrical end. In particular, class (i) belongs to class (ii). However, we distinguish them as it will be convenient when we discuss well-definedness of wrapped Floer theory.
\end{remark}

	One of the main results in \cite{Gao1} is about the well-definedness of the two versions of wrapped Floer cohomology for an admissible Lagrangian submanifold $\mathcal{L} \subset M \times N$. The argument can easily be generalized, using the same methods, to prove the following result: \par

\begin{proposition}
	Suppose $\mathcal{L} \subset M \times N$ is an admissible Lagrangian submanifold. Then wrapped Floer $A_{\infty}$-algebra associated to $\mathcal{L}$ is well-defined, with respect to either the split Hamiltonian $H_{M, N}$ and product almost complex structure $J_{M, N}$, or the quadratic Hamiltonian $K$ and almost complex structure $J$.
\end{proposition}
\begin{proof}[Sketch of proof]
	This is proved in \cite{Gao1}. We shall review some key arguments while referring to \cite{Gao1} for details. There are two key issues in proving well-definedness: transversality and compactness of the moduli spaces of pseudoholomorphic disks. \par
	Regarding transversality, perturbing the almost complex structure $J_{M, N}$ within the class of product almost complex structures $\mathcal{J}(M) \times \mathcal{J}(N)$ might not ensure enough genericity to make the moduli spaces regular. However, when defining the wrapped Fukaya category, we shall take domain-dependent perturbations of almost complex structures in a suitable way. Namely, if we allow perturbations of $J_{M, N}$ in a neighborhood in $\mathcal{J}(M \times N)$ instead of just product almost complex structures, then transversality can be achieved. On the other hand, if one of the Lagrangian submanifolds involved in $m^{k}$ is a product Lagrangian $L \times L'$, it is true that transversaliy of the moduli spaces involved can be achieved by perturbations within the class of product almost complex structures. The argument is similar to Wehrheim and Woodward's argument in quilted Floer cohomology (for compact monotone or exact Lagrangian submanifolds), combined with the transversality argument in wrapped Floer theory. \par
	Compactness is the substantial issue that we need to think carefully about. Consider for simplicity the case of a pair $(\mathcal{L}_{0}, \mathcal{L}_{1})$ for which we want to define wrapped Floer cohomology using the split Hamiltonian $H_{M, N}$ and product almost complex structure $J_{M, N}$. For a pair of $H_{M, N}$-chords $\underline{x}_{0}, \underline{x}_{1}$ from $\mathcal{L}_{0}$ to $\mathcal{L}_{1}$, the moduli space $\mathcal{M}(\underline{x}_{0}, \underline{x}_{1}) = \tilde{\mathcal{M}}(\underline{x}_{0}, \underline{x}_{1})/\mathbb{R}$ of inhomogeneous pseudoholomorphic strips from $x$ to $y$ has a natural Gromov bordification $\bar{\mathcal{M}}(\underline{x}_{0}, \underline{x}_{1})$ by adding broken strips. To ensure that $\bar{\mathcal{M}}(\underline{x}_{0}, \underline{x}_{1})$ is compact, the main ingredient in addition to Gromov compactness is the maximum principle, which prevents strips connecting $\underline{x}_{0}$ and $\underline{x}_{1}$ from escaping to infinity. By directly estimating the action of Hamiltonian chords, and using the action-energy equality to give an a priori estimate for the energy of inhomogeneous pseudoholomorphic disks. On the other hand, we need a $C^{0}$-bound, which can be achieved by comparing the split Hamiltonian $H_{M, N}$ with an admissible one. The argument is done when the action-restriction map is constructed and proved to be a cochain map. Let us recall that here. Given $\underline{x}_{0}$ and $\underline{x}_{1}$, we choose $b$ so that their action is greater than $-b$, and correspondingly a large compact subset of $M \times N$ containing all Hamiltonian chords whose action is greater than $-b$. Then we modify the Hamiltonian $H_{M, N}$ to another one $K_{b}$ which agrees with $H_{M, N}$ inside of this compact set, and is quadratic in the radial coordinate on $\Sigma \times [B, +\infty)$ for some $B$ large enough (depending on $b$). A corresponding almost complex structure is also constructed, which is of contact type in $\Sigma \times [B, +\infty)$. It is then shown that $(H_{M, N}, J_{M, N})$-pseudoholomorphic strips connecting $\underline{x}_{0}$ and $\underline{x}_{1}$ are in one-to-one correspondence with $(K_{b}, J_{b})$-pseudoholomorphic strips connecting $\underline{x}_{0}$ and $\underline{x}_{1}$. In particular, maximum principle can be applied to give $C^{0}$-bound for $(K_{b}, J_{b})$-pseudoholomorphic strips, and therefore also $(H_{M, N}, J_{M, N})$-pseudoholomorphic strips. Such analysis is done in details in \cite{Gao1}. \par
	The second part of the compactness result in wrapped Floer theory is that for fixed $\underline{x}_{1}$, the moduli space $\bar{\mathcal{M}}(\underline{x}_{0}, \underline{x}_{1})$ is empty for all but finitely many $\underline{x}_{0}$'s. This can be proved again using the action-energy equality. In the case of a single $M$, the statement follows from the fact that if there is an Hamiltonian chord $x$ lying on the level hypersurface $\partial M \times \{r\}$, then its action is roughly $-r^{2}$ (with a difference contributed from the primitives of the Lagrangian submanifolds, which are negligible for large $r$). \par
	Now in case of split Hamiltonian $H_{M, N}$ on the product $M \times N$, the Hamiltonian vector field $X_{H_{M, N}}$ no longer agrees with the Reeb vector field on the contact hypersurface we have chosen, because that hypersurface is a small deformation of the singular hypersurface $\partial (M \times N)$. As a consequence, if $\underline{x}=(x, x')$ is an $H_{M, N}$-chord, then $x$ is an $H_{M}$-chord and $x'$ is an $H_{N}$-chord. In the cylindrican end, such a chord is no longer exactly located on some smooth contact hypersurface $\Sigma \times \{r\}$. There are three places where such chords can occur. The first place is $\partial M \times [1, +\infty) \times \partial N \times [1, +\infty)$, where $x$ lies on some hypersurface $\partial M \times \{r_{1}\}$ and $x'$ on $\partial N \times \{r_{2}\}$. We can calculate its action to be $-r_{1}^{2} - r_{2}^{2} + c$, where $c$ is a uniformly bounded (i.e. independent of the radial coordinates $r_{1}, r_{2}$) constant coming from the primitives of $\mathcal{L}_{0}$ and $\mathcal{L}_{1}$. The second place is $int(M) \times \partial N \times [1, +\infty)$, where now $x$ is an $H_{M}$-chord in the interior of $M$ (there are finitely many such "interior" chords, all of which have small action), and $x'$ lies on some hypersurface $\partial N \times \{r_{2}\}$. The third place is $\partial M \times [1, +\infty) \times N$. This case is similar to the second one. \par
	 These action computations yield the desired compactness result: since the action of such chords must also go to negative infinity as the positions of them tend to infinity, there cannot be infinitely many inhomogeneous pseudoholomorphic disks with different outputs which tend to infinity, as the energy of any inhomogeneous pseudoholomorphic disk is always positive. \par
	 We have thus shown that the wrapped Floer cohomology for an admissible pair $(\mathcal{L}_{0}, \mathcal{L}_{1})$ is well-defined using the split Hamiltonian $H_{M, N}$ and product almost complex structure $J_{M, N}$ (but with generic domain-dependent perturbations). The same argument applies to the definition of higher structure maps $m^{k}$, by using action-energy equality for inhomogeneous pseudoholomorphic disks with more punctures. \par
\end{proof}

\subsection{Action-restriction data}
	As mentioned at the beginning of section \ref{section: product manifolds} main technical issue in studying Floer theory on the product manifold $M \times N$ and relating that to quilted Floer theory is with regard to the choice of Hamiltonian functions. Having found the cylindrical end $\Sigma \times [1, +\infty)$, one can immediately set up wrapped Floer theory using a Hamiltonian that depends only on the radial coordinate $r$ in the cylindrical end and has quadratic growth for $r$ large enough. On the other hand, what is more directly related to quilted Floer theory is the split Hamiltonian, i.e. the sum of the two Hamiltonians on both factors $M$ and $N$. However, the split Hamiltonian is not a priori admissible in the usual sense of wrapped Floer theory. Also, there is a similar issue with almost complex structures. Thus it is not immediately clear that the resulting two versions of wrapped Fukaya categories are equivalent. \par
	By the invariance nature of Floer cohomology, it is expected that these two versions are equivalent. In \cite{Gao1}, we showed that these two types of geometric data define isomorphic wrapped Floer cohomology, by constructing a cochain quasi-isomorphism between the two versions of wrapped Floer cochain complexes. Let us briefly summarize the idea here. The action functional defines a filtration on the wrapped Floer complex $CW^{*}(\mathcal{L}_{0}, \mathcal{L}_{1}; H_{M, N}, J_{M, N})$. Consider the truncated complex $CW^{*}_{(-b, a]}(\mathcal{L}_{0}, \mathcal{L}_{1}; H_{M, N}, J_{M, N})$. Starting from the split Hamiltonian, we may construct an admissible Hamiltonian $K_{b}$ which agrees with the split Hamiltonian inside some large compact set $C_{b}$ depending on $b$, and outside a bigger compact set $C'_{b}$ is quadratic in the radial coordinate $r$ on $\Sigma \times [1, +\infty)$. Note that the wrapped Floer complexes as graded $\mathbb{Z}$-modules do not depend on the choice of almost complex structure. Thus the generators of $CW^{*}_{(-b, a]}(\mathcal{L}_{0}, \mathcal{L}_{1}; H_{M, N}, J_{M, N})$ are automatically generators of $CW^{*}_{(-b, a]}(\mathcal{L}_{0}, \mathcal{L}_{1}; K_{b}, J_{b})$ for any admissible $J_{b}$. We therefore get a homomorphism of graded modules
\begin{equation} \label{action-restriction map}
\bar{R}^{1}_{b}: CW^{*}_{(-b, a]}(\mathcal{L}_{0}, \mathcal{L}_{1}; H_{M, N}, J_{M, N}) \to CW^{*}_{(-b, a]}(\mathcal{L}_{0}, \mathcal{L}_{1}; K_{b}, J_{b}).
\end{equation} \par

	We proved in \cite{Gao1} that this is an isomorphism of graded modules by showing that some extra chords of the Hamiltonian $K_{b}$ outside $C_{b}$ but inside $C'_{b}$ have sufficiently positive action, and consequently those chords do not lie in the truncated complex. By modifying $J_{M, N}$ to an admissible almost complex structure $J_{b}$, we proved that \eqref{action-restriction map} is a cochain map. \par
	To extend the above map to the whole wrapped Floer complex, we need to identify each truncated Floer complex $CW^{*}_{(-b, a]}(\mathcal{L}_{0}, \mathcal{L}_{1}; K_{b}, J_{b})$ with one defined with respect to a Hamiltonian and an almost complex structure independent of $b$. \par

\begin{lemma}
	There exists a Hamiltonian $K$ which is quadratic in the radial coordinate on $\Sigma \times [1, +\infty)$ away from a compact set, and an almost complex structure of contact type with respect to $\Sigma$, such that for any $b > 0$, there is a cochain homotopy equivalence
\begin{equation}
h_{b}: CW^{*}_{(-b, a]}(\mathcal{L}_{0}, \mathcal{L}_{1}; K_{b}, J_{b}) \to CW^{*}_{(-b, a]}(\mathcal{L}_{0}, \mathcal{L}_{1}; K, J).
\end{equation}
\end{lemma}
\begin{proof}
	This is proved in \cite{Gao1}. Let us briefly recall the argument here. The Hamiltonian $K_{b}$ is quadratic in the radial coordinate $r$ on $\Sigma \times [1, +\infty)$ for $r \ge B$ where $B$ is a suitable number depending quadratically on $b$. We may use the time-$-\log(B)$ Liouville flow to rescale the Hamiltonian so that the rescaled Hamiltonian is quadratic outside a neighborhood of $r \le 1$. The resulting Hamiltonian $K'_{b}$ still depends on $b$, but we can choose a particular Hamiltonian $K$ which agrees with $K'_{b}$ outside a neighborhood of $r \le 1$, such that there is a compactly-supported homotopy between $K'_{b}$ and $K$. This homotopy induces a cochain homotopy equivalence between the corresponding wrapped Floer complexes.
\end{proof}

	Then we compose the above cochain homotopy equivalence with \eqref{action-restriction map} (which we still denote by the same symbol) to obtain a cochain map:
\begin{equation}\label{action-restriction map, with a Hamiltonian independent of filtration}
R^{1}_{b}: CW^{*}_{(-b, a]}(\mathcal{L}_{0}, \mathcal{L}_{1}; H_{M, N}, J_{M, N}) \to CW^{*}_{(-b, a]}(\mathcal{L}_{0}, \mathcal{L}_{1}; K, J),
\end{equation}
\begin{equation}
R^{1}_{b} = h_{b} \circ \bar{R}^{1}_{b},
\end{equation}
for admissible $K$ and $J$ independent of $b$. These maps are shown to commute up to homotopy with natural inclusions of subcomplexes under action filtration, so that we are able to take the limit (homotopy direct limit) to obtain the desired cochain quasi-isomorphism on the whole wrapped Floer complex:
\begin{equation}\label{action-restriction map on the whole wrapped Floer complex}
R^{1}: CW^{*}(\mathcal{L}_{0}, \mathcal{L}_{1}; H_{M, N}, J_{M, N}) \to CW^{*}(\mathcal{L}_{0}, \mathcal{L}_{1}; K, J).
\end{equation}
See the last section of \cite{Gao1} for more details for the construction of this map. \par

\begin{remark}
	The way we define the action-restriction map is not via certain parametrized moduli space, but this map can alternatively be interpreted as the continuation map induced by a monotone homotopy from the split Hamiltonian to the admissible Hamiltonian, on the truncated wrapped Floer cochain complexes.
\end{remark}

	In the construction of the map $R^{1}_{b}$ above, there are several parameters involved in the procedure of modifying the Hamiltonian (e.g. the size of the compact set, the precise behavior of the Hamiltonian $K_{b}$, etc.), such that we can obtain the desired estimate on the action of the extra chords. We call a choice of these parameters, or more essentially the collection of geometric data these parameters determine, an action-restriction datum. A precise definition is given below. \par

\begin{definition} \label{definition of action-restriction datum for strips}
	An action-restriction datum for the strip $Z = \mathbb{R} \times [0, 1]$ and the pair $(\mathcal{L}_{0}, \mathcal{L}_{1})$ of Lagrangian submanifolds of $M \times N$ consists of the following data:
\begin{enumerate}[label=(\roman*)]

\item a truncation $(-b, a]$ of the wrapped Floer cochain space $CW^{*}(\mathcal{L}_{0}, \mathcal{L}_{1}; H_{M, N})$ using the action filtration;

\item a large number $A$ such that all $H_{M, N}$-chords from $\mathcal{L}_{0}$ to $\mathcal{L}_{1}$ of action in $(-b, a]$ are contained in the compact set $\{ r \le A \}$, and which satisfies an additional property to specified in the next condition;

\item an admissible Hamiltonian function $K_{b}$, which is of split type in the compact set $\{ r \le A \}$, and is quadratic outside a small neighborhood of a larger compact set $\{ r \le cA \}$; the number $A$ should be large enough such that all the extra chords of $K_{b}$ compared to $H_{M, N}$ have sufficiently positive action;

\item an admissible almost complex structure $J_{b}$, which agrees with the product almost complex structure $J_{M, N}$ in the compact set $\{ r \le A \}$, and is of contact type outside a neighborhood of the compact set $\{ r \ge cA \}$;

\item a homotopy $H_{s}$ between $H_{M, N}$ and $K_{b}$, parametrized by $s \in \mathbb{R}$, which has support within the compact set $\{r \le cA\}$, and constantly equals $H_{M, N}$ for $s$ sufficiently negative and $K_{b}$ for $s$ sufficiently positive;

\item a deformation $J_{s}$ from $J_{M, N}$ to $J_{b}$, parametrized by $s \in \mathbb{R}$, which has support within the compact set $\{r \le cA\}$, and constantly equals $H_{M, N}$ for $s$ sufficiently negative and $J_{b}$ for $s$ sufficiently positive.

\end{enumerate}
\end{definition}

	The existence of such action-restriction data was proved in \cite{Gao1}, which was in turn used to prove: \par
	
\begin{proposition}
	The map \eqref{action-restriction map} is a cochain isomorphism. Therefore, the maps \eqref{action-restriction map on the whole wrapped Floer complex} and \eqref{action-restriction map, with a Hamiltonian independent of filtration} are cochain homotopy equivalence.
\end{proposition}

	To define a higher-order analogue of the action restriction map \eqref{action-restriction map, with a Hamiltonian independent of filtration} so that we can extend it to an $A_{\infty}$-functor, we need an analogue of action-restriction datum for disks with more punctures. \par

\begin{definition} \label{definition of action-restriction datum for punctured disks}
	An action-restriction datum for a disk $S$ with $(k+1)$-boundary punctures and a $(k+1)$-tuple of Lagrangian submanifolds $(\mathcal{L}_{0}, \cdots, \mathcal{L}_{k})$ consists of the following data:
\begin{enumerate}[label=(\roman*)]

\item a collection of positive integers $w^{i}, i=0, \cdots, k$, called weights;

\item a smooth function $\psi_{S}: \partial S \to [1, cA]$ parametrizing moving Lagrangian boundary conditions $\phi_{M \times N}^{\psi_{S}(z)} \mathcal{L}_{i}$, when $z$ lies between $z_{i}$ and $z_{i+1}$, which is equal to $w_{i}$ near $z_{i}$;

\item truncations $(-b_{i}, a_{i}]$ of the action filtration on $CW^{*}(\mathcal{L}_{i-1}, \mathcal{L}_{i}; H_{M, N})$, for $i = 1, \cdots, k$, and a truncation $(-b_{0}, a_{0}]$ of $CW^{*}(\mathcal{L}_{0}, \mathcal{L}_{k}; H_{M, N})$, subject to the following condition:
\begin{equation} \label{action relation}
\sum_{i=1}^{k} b_{i} \le b_{0}, \text{ and }
\sum_{i=1}^{k} a_{i} \le a_{0};
\end{equation}

\item a sub-closed one-form $\alpha_{S}$ on $S$, which vanishes along the boundary $\partial S$, and agrees with $w^{i}dt$ over the $i$-th strip-like end, such that the differential $d\alpha_{S}$ also vanishes in a small neighborhood of the boundary;

\item a large number $A_{i}$ depending on $b_{i}$, such that all $H_{M, N}$-chords from $\mathcal{L}_{i-1}$ to $\mathcal{L}_{i}$ ($i = 1, \cdots, k$) and those from $\mathcal{L}_{0}$ to $\mathcal{L}_{k}$ within the action ranges $(-b_{i}, a_{i}]$ and $(-b_{0}, a_{0}]$ are contained in the compact set $\{ r \le A_{i} \}$, where $r$ is the radial coordinate on the cylindrical end $\Sigma \times [1, +\infty)$;

\item a modified Hamiltonian $K_{b_{i}}$ which agrees with the split Hamiltonian in the compact set $\{r \le A_{i}\}$, and is quadratic in $r$ outside $\{r \le cA_{i}\}$. The choice of $A_{i}$ above should satisfy the additional condition that the extra $K_{b_{i}}$-chords in the region $\{ r \le cA_{i} \}$ from $\mathcal{L}_{i}$ to $\mathcal{L}_{i+1}$ (respectively from $\mathcal{L}_{0}$ to $\mathcal{L}_{k}$ all have sufficiently positive action, much bigger than $a_{i}$ (respectively $a_{0}$);

\item an $S$-parametrized family of admissible Hamiltonians $K_{S, \vec{b}}$, such that over the $i$-th strip-like end ($i=1, \cdots, k$), $K_{S, \vec{b}}$ agrees with the split Hamiltonian $H_{M, N}$ rescaled by weight $w^{i}$, and over the $0$-th strip-like end, $K_{S, \vec{b}}$ agrees with the modified Hamiltonian $K_{b_{0}}$, rescaled by weight $w^{0}$;

\item a modified almost complex structure $J_{b_{i}}$, which is product-type in $\{r \le A_{i}\}$ and is generic of contact type outside $\{r \le cA_{i}\}$;

\item an $S$-parametrized family of admissible almost complex structures $J_{S, \vec{b}}$, such that over the $i$-th strip-like end ($i=1, \cdots k$), it agrees with the product almost complex structure $J_{M, N}$, and over the $0$-th strip-like end, it agrees with the modified almost complex structure $J_{b_{0}}$.

\end{enumerate}
In particular, when restricting to each strip-like end, the Hamiltonian and the almost complex structure should be independent of the $s$-coordinate, at least for sufficiently large $|s|$. Here $s$ is the coordinate on $\mathbb{R}_{\pm}$.
\end{definition}

	Later, we shall slightly modify the definition of action-restriction data, as the abstract moduli spaces underlying inhomogeneous pseudoholomorphic curves used to define $A_{\infty}$-functors consist of not only punctured disks but also an additional non-negative number associated to each punctured disk (see the next subsection). All the ingredients of an action-restriction datum will be the same, though. \par

\begin{remark}
	Although we need only to modify the Hamiltonian near the $0$-th strip-like end, we include all the relevant information near other strip-like ends because it will be convenient to extend the action-restriction data to nodal disks and state the gluing process for action-restriction data.
\end{remark}

	With these definitions, one of the main results of \cite{Gao1} can be interpreted as having established the existence, i.e. possibility of making a choice, of action-restriction data for the strip and any pair of admissible Lagrangian submanifolds, as well as those for the triangle and any triple of admissible Lagrangian submanifolds. As is expected, the existence in fact holds for any polygon and any tuple of admissible Lagrangian submanifolds, whose construction follows the same pattern as that in \cite{Gao1}. \par

\subsection{Choosing action-restriction data for all curves}
	The underlying operad controlling $A_{\infty}$-functors is commonly known as multiplihedra, introduced by Stasheff. Before constructing the action-restriction functor, we need a model for multiplihedra whose elements are the underlying domains for various inhomogeneous pseudoholomorphic maps used to define the action-restriction functor which relates the two versions of wrapped Fukaya categories of the product manifold. There are many equivalent models for the multiplihedra, among which the version used in \cite{Sylvan} to construct continuation functors of wrapped Fukaya categories seems mostly adaptable to our setup, because we have always used a single Hamiltonian to define the wrapped Fukaya category, instead of a cofinal family of linear Hamiltonians. \par
	The multiplihedra are constructed as compactifications of moduli spaces $\mathcal{N}_{k+1}$ of punctured disks equipped with a weight. We set $\mathcal{N}_{1+1}$ to be a single-point set, which consists of the strip $Z$ equiped with a positive strip-like end at $+\infty$ and a negative strip-like end at $-\infty$. For $k \ge 2$, define $\mathcal{N}_{k+1}=\mathcal{M}_{k+1} \times \mathbb{R}_{+}$. \par
	We are going to describe a compactification $\bar{\mathcal{N}}_{k+1}$ of $\mathcal{N}_{k+1}$, which serves as a model for the multiplihedra. These are constructed inductively in $k$. Set $\bar{\mathcal{N}}_{2} = \mathcal{N}_{2}$. Suppose we have constructed $\bar{\mathcal{N}}_{l+1}$ for $l<k$, which has boundary and generalized corners being products of associahedra and multiplihedra of lower dimensions, and that for each element $(S, w) \in \bar{\mathcal{N}}_{l+1}$ we have chosen strip-like ends near the punctures in a consistent way. Then we want to construct $\bar{\mathcal{N}}_{k+1}$ as well as choose strip-like ends for every element therein, such that the following two conditions are satisfied:
\begin{enumerate}[label=(\roman*)]

\item The boundary and generalized corners of $\bar{\mathcal{N}}_{k+1}$ are products of $\mathcal{M}_{i+1}$'s and $\mathcal{N}_{j+1}$'s for $i, j < k$. 

\item If an element of $\bar{\mathcal{N}}_{k+1}$ lies in the boundary strata, then the strip-like ends we choose for it should agree with the ones chosen for the irreducible components of $\Sigma$ regarded as elements in $\mathcal{M}_{i+1}$'s and $\mathcal{N}_{j+1}$'s for $i, j < k$.

\end{enumerate}

	A more accurate description of the boundary strata is given as follows. Some boundary strata appear as $w$ goes to infinity, which have the form
\begin{equation}
\mathcal{M}_{l+1}^{1} \times \prod_{i=1}^{l} \mathcal{N}_{m_{i}+1},
\end{equation}
where $\sum_{i=1}^{l}m_{i} = d$. Here $\mathcal{M}_{l+1}^{1}$ means a copy of $\mathcal{M}_{l+1}$ arising at $w=\infty$. Other boundary strata appear in the compactification of $\mathcal{M}_{k+1}$, one copy for $w=0$ and one for $w=\infty$. These products arise as boundary of the compactification via the gluing maps. For instance, a boundary chart in the case $w \to \infty$ is of the form
\begin{equation}
U \times \prod_{i=1}^{l} U_{i} \times [0, a),
\end{equation}
where $U \subset \mathcal{M}_{l+1}^{1}$, and $U_{i} \subset \mathcal{N}_{m_{i}+1}$. The last factor $[0, a)$ is the gluing parameter, where $0$ corresponds to  degenerate curves, i.e. tuples of curves in $U \times \prod_{i=1}^{l} U_{i}$. Given $(S_{l+1}, (S_{m_{1}+1}, w_{1}), \cdots, (S_{m_{l}+1}, w_{l})) \in U \times \prod_{i=1}^{l} U_{i}$, as well as a nonzero gluing parameter $\rho \in (0, a)$, we would like to produce a pair $(S, w) \in \mathcal{N}_{k+1}$ by gluing the surfaces together. For every $i$, glue the negative strip-like end of $S_{m_{i}+1}$ to the $i$-th positive strip-like end $\epsilon_{i}$ of $S_{l+1}$ simultaneously, with length $l_{i} = e^{\frac{1}{\rho}} - w_{i} - w_{i}(S_{l+1})$, where $w_{i}(S_{l+1})$ is the width of the $i$-th positive strip-like end $\epsilon_{i}$ of $S_{l+1}$. This procedure is uniform for all elements in $U \times \prod_{i=1}^{l} U_{i}$, and provides gluing maps
\begin{equation} \label{gluing for multiplihedra: first kind}
\sharp_{1}: \mathcal{M}_{l+1}^{1} \times \prod_{i=1}^{l} \mathcal{N}_{m_{i}+1} \times [0, 1) \to \bar{\mathcal{N}}_{k+1}.
\end{equation}
The other kind of boundary stratum has the form 
\begin{equation}
\mathcal{N}_{m+2} \times \mathcal{M}_{k-m+1}, 0 \le m \le k-2
\end{equation}
which also comes with gluing maps
\begin{equation} \label{gluing for multiplihedra: second kind}
\sharp_{2}: \mathcal{N}_{m+2} \times \mathcal{M}_{k-m+1} \times [0, 1) \to \bar{\mathcal{N}}_{k+1}.
\end{equation} \par
	The third type of boundary strata are obtained by compactifying all the $\mathcal{M}_{k+1}$ components. These $\mathcal{M}_{k+1}$ components come into the picture either in the first two type codimension-one boundary strata, as $w \to 0$ or as $w \to +\infty$, or in the place where $w$ remains finite as the compactification of $\mathcal{M}_{k+1} \times \{w\} \subset \mathcal{S}_{k+1} = \mathcal{M}_{k+1} \times \mathbb{R}_{+}$. In the latter case, consider a boundary stratum $\sigma \subset \partial \bar{\mathcal{N}}_{k+1}$ whose elements are modeled over some rooted tree $T$ with labeled leaves. A boundary chart for $\sigma$ takes the form
\begin{equation}\label{boundary chart for the third type of boundary stratum}
U \times \prod_{\text{interior vertices } v} U_{v} \times \prod_{\text{interior edges } e} [0, a_{e}),
\end{equation}
where $U$ is a small open subset of $\mathcal{N}_{l+1}$ where $l$ is the valency of the root, while others $U_{v}$ are small open subsets of $\mathcal{M}_{m_{v}+1}$. When all the gluing parameters in $\prod_{\text{interior edges } e} [0, a_{e})$ are nonzero, the above chart \eqref{boundary chart for the third type of boundary stratum} is identified with a subset of $\mathcal{N}_{k+1}$ via the gluing construction in the same way as that for $\bar{\mathcal{M}}_{k+1}$, with the extra $\mathbb{R}_{+}$ factor in $\mathcal{N}_{m+1} \supset U$ mapping to the $\mathbb{R}_{+}$ factor in $\mathcal{N}_{k+1}$ by identity. \par
	Let us slightly modify the definition of an action-restriction datum so that it is more adapted to the construction of $A_{\infty}$-functors. An action-restriction datum is now chosen for elements in $\mathcal{N}_{k+1}$, with a potential extension to the compactification $\bar{\mathcal{N}}_{k+1}$. \par

\begin{definition}
	Let $(S, w) \in \mathcal{N}_{k+1}$, with given Lagrangian labels. An action-restriction datum for $(S, w)$ and the Lagrangian labels is simply an action-restriction datum for $S$ as in Definition \ref{definition of action-restriction datum for punctured disks}, which depends on $w \in \mathbb{R}_{+}$.
\end{definition}

	Although an action-restriction datum for $(S, w)$ is the same as one for $S$, the dependence on $w$ matters when we make a choice of an action-restriction datum for every element in $\mathcal{N}_{k+1}$. In other words, it matters when we consider families of action-restriction data. \par

\begin{definition}
	A universal choice of action-restriction data over $\mathcal{N}_{k+1}$ is a choice of an action-restriction datum for each $(S, w) \in \mathcal{N}_{k+1}$, which depends smoothly on $(S, w)$ (note that $\mathcal{N}_{k+1} = \mathcal{M}_{k+1} \times \mathbb{R}_{+}$ has a natural structure of smooth manifold).
\end{definition}

	In the next subsection, we shall discuss how the action-restriction data are extended to the compactification $\bar{\mathcal{N}}_{k+1}$. \par

\subsection{Making choices of action-restriction data consistently}
	According to the previous summary of the results in \cite{Gao1}, the construction of a cochain homotopy equivalence of the two versions of wrapped Floer complexes works well for any pair among a fixed finite collection of Lagrangian submanifolds, as that only requires a choice of an action-restriction datum for the strip and that particular pair of Lagrangian submanifolds in consideration. Extending this to an $A_{\infty}$-functor on the categorical level requires us to choose action-restriction data for many Lagrangian submanifolds. \par
	The following observation is useful in arranging various action-restriction data: it is the parameters $a$'s, $b$'s, $A$'s that are crucial in the definition of an action-restriction datum, and they essentially depend on the initially given split Hamiltonian, product almost complex structure, and Lagrangian submanifolds only. Moreover, we can choose the same parameters for all pairs such that the action-restriction maps are defined, which is possible because of the following three reasons:
\begin{enumerate}[label=(\roman*)]

\item  we only need to fix $a$'s not too small so that the truncated wrapped Floer complexes $CW^{*}_{(-b, a]}(\mathcal{L}_{i}, \mathcal{L}_{j}; H_{M, N}, J_{M, N})$ include all Hamiltonian chords from $\mathcal{L}_{i}$ to $\mathcal{L}_{j}$ that are contained in the interior part of $M \times N$;

\item we have to consider all possible $b$'s satisfying the condition \eqref{action relation}, thus these are not matters of choices;

\item we only have to choose the numbers $A_{i}$'s, which determines the size of the compact set inside which the split Hamiltonian $H_{M, N}$ agrees with the admissible one $K_{b}$, and that of the compact set outside which $K_{b}$ is quadratic, to be sufficiently large in order for the action-restriction map to be defined.

\end{enumerate} \par

	However, it is practically difficult to directly construct the action-restriction functor on the whole wrapped Fukaya category. The reason is as follows. To define higher order maps and verify the $A_{\infty}$-equations, we need to choose action-restriction data for all punctured disks and all Lagrangian submanifolds in a consistent way. But there are infinitely many Lagrangian submanifolds in the category, making it unlikely possible to arrange the choices of action-restriction data for all Lagrangian submanifolds simultaneously. Keeping track of all these infinitely many choices at one time might be a painful job. \par
	To overcome this difficulty, we take the following maneuver. First, construct the desired $A_{\infty}$-functors on a chain of full $A_{\infty}$-subcategories of the wrapped Fukaya category, each of which consists of finitely many objects. We want the union of these subcategories to be the whole wrapped Fukaya category, which is possible because we have assumed that the wrapped Fukaya category is made up of a countable collection of objects. Second, prove that these $A_{\infty}$-functors are compatible with each other under natural inclusions of these subcategories in the chain. Third, take the limit in an appropriate sense to obtain the desired $A_{\infty}$-quasi-equivalence. \par
	Let us start from the first step: considering the full $A_{\infty}$-subcategory $\mathcal{W}_{d}^{s}$ (respectively $\mathcal{W}_{d}$) of $\mathcal{W}^{s}(M \times N)$ (respectively $\mathcal{W}(M \times N)$) consisting of $d$ Lagrangian submanifolds $\mathcal{L}_{1}, \cdots, \mathcal{L}_{d}$ from our collection. We may construct an $A_{\infty}$-functor
\begin{equation}
R_{d}: \mathcal{W}_{d}^{s} \to \mathcal{W}_{d}
\end{equation}
which is a quasi-isomorphism. As mentioned before, to verify the $A_{\infty}$-equations, we need to arrange the action-restriction data in a consistent way. \par
	To explain the precise meaning of this consistency, we first describe the way how the action-restriction data can be glued together when we glue the underlying curves. As in the case of Floer data for punctured disks, we introduce the concept of conformal equivalence for action-restriction data. To state this, we introduce some notations. If $H$ is a Hamiltonian, denote by $H_{C}$ the time-$\log C$ rescaling of $H$,
\begin{equation*}
H_{C} = \frac{1}{C} H \circ \psi^{C},
\end{equation*}
where $\psi^{C}$ is the time-$\log C$ Liouville flow. Similarly, if $J$ is an almost complex structure, denote by $J_{C}$ the time-$\log C$ rescaling of $J$,
\begin{equation*}
J_{C} = J \circ \psi^{C}.
\end{equation*}

\begin{definition}
	Two action-restriction data are said to be conformally equivalent, if there exist constants $C, W > 0$ such that:
\begin{enumerate}[label=(\roman*)]

\item the weights differ by scalar multiplication by $W$, i.e. $w^{i} = W w'^{i}$;

\item the action filtrations agree, i.e. $a_{i} = a'_{i}, b_{i} = b'_{i}$;

\item the Hamiltonians agree up to rescaling,
\begin{equation}
K_{S, \vec{b}} = \frac{1}{W}(K'_{S, \vec{b}})_{C};
\end{equation}

\item the almost complex structures agree up to rescaling,
\begin{equation}
J_{S, \vec{b}} = (J'_{S, \vec{b}})_{C}.
\end{equation}

\end{enumerate}
\end{definition}
	
	Fix a finite collection $\mathcal{L}_{1}, \cdots, \mathcal{L}_{d}$ in consideration. For example, suppose that we are given $S \in \mathcal{M}_{l+1}^{1}$ and $(S_{i}, w_{i}) \in \mathcal{N}_{m_{i}+1}, i = 1, \cdots, l$, all equipped with appropriate Lagrangian labels so that they match on the boundary components of $S$ and $S_{i}$ that are to be glued together, in the way that the $i$-th strip-like end of $S$ is glued with the $0$-th strip-like end of $S_{i}$. Suppose that we have chosen an action-restriction datum for each $(S_{i}, w_{i})$, as well as a Floer datum for $S$ (which is involved in the definition of the $A_{\infty}$-structure on $\mathcal{W}(M \times N)$). \par

\begin{definition}
	The action-restriction data on $(S_{i}, w_{i})$ and the Floer datum on $S$ are said to be conformally consistent, if there exist positive constants 
\begin{equation*}
C, W, C_{1}, \cdots, C_{l}, W_{1}, \cdots, W_{l} > 0,
\end{equation*}
such that the following conditions hold:
\begin{enumerate}[label=(\roman*)]

\item Let $w^{i}$ be the weight on the $i$-th strip-like end of $S$, and $w_{i}^{0}$ the weight on the $0$-th strip-like end of $S_{i}$. Then
\begin{equation}
C w^{i} = C_{i} w_{i}^{0}.
\end{equation}

\item Let $K_{S, i}$ be the asymptotic Hamiltonian of the $S$-dependent family of Hamiltonians $K_{S}$ over the $i$-th strip-like end of $S$, and $K_{S_{i}, b_{i}^{0}}$ the asymptotic Hamiltonian of the $S_{i}$-dependent family of Hamiltonians $K_{S_{i}, \vec{b}_{i}}$ over the $0$-th strip-like end of $S_{i}$. Then
\begin{equation}
\frac{1}{W}(K_{S, i})_{C} = \frac{1}{W_{j}}(K_{S_{i}, b_{i}^{0}})_{C_{i}}.
\end{equation}

\item Let $J_{S, i}$ be the asymptotic almost complex structure of the $S$-dependent family of almost complex structures $J_{S}$ over the $i$-th strip-like end of $S$, and $J_{S_{i}, b_{i}^{0}}$ the asymptotic almost complex structure of the $S_{i}$-dependent family of almost complex structures $J_{S_{i}, \vec{b}_{i}}$ over the $0$-th strip-like end of $S_{i}$. Then
\begin{equation}
(J_{S, i})_{C} = (J_{S_{i}, b_{i}^{0}})_{C_{i}}.
\end{equation}

\end{enumerate} 

\end{definition}

	Given every gluing parameter $\rho \in (0, a)$, we may glue $S$ and $(S_{i}, w_{i})$ together to get a smooth punctured disk $(S_{\rho}, w_{\rho})$. The glued action-restriction datum for $(S_{\rho}, w_{\rho})$ is defined as follows. The weights are simply those $C_{i} w_{i}^{j}$'s, and $C w^{0}$, removing those weights $w^{i}, w_{i}^{0}$ for the strip-like ends that are glued together. As for the truncations, over the strip-like ends of $S_{\rho}$ that are obtained from $S_{i}$'s, we simply take the existing truncations $(-b_{i}^{j}, a_{i}^{j}]$. But we have not specified a truncation for the $0$-th strip-like end of $S_{\rho}$ which comes from $S$, as the Floer datum on $S$ does not contain such information. Nonetheless, the previous condition on Hamiltonians give natural truncation over that end. We may simply take $b_{0} = \sum_{i=1}^{l} b^{i}_{0}$, and take $a_{0}$ not less than $\sum_{i=1}a^{i}_{0}$ such that the truncated wrapped Floer complex $CW^{*}_{(-b_{0}, a]}$ is independent of $a$ when $a \ge a_{0}$. \par
	Because of the above conditions on families of Hamiltonians and almost complex structures over $S$ and $S_{i}$'s that they agree up to rescaling over the strip-like ends which are to be glued together, we may rescale these families $K_{S}, K_{S_{i}, \vec{b}_{i}}$ and $J_{S}, J_{S_{i}, \vec{b}_{i}}$, and take the union of the rescaled families on the glued surface $\Sigma_{\rho}$ to obtain the families $K_{\Sigma_{\rho}}, J_{\Sigma_{\rho}}$, which have all the desired properties. \par
	There are more complicated gluings, which happen in higher-codimensional strata. But the corresponding gluing process for action-restriction data is the same. \par
	Now let us formalize all the above ideas in the following definition. \par

\begin{definition}
	A universal and conformally consistent choice of action-restriction data is a choice of an action-restriction datum for every $k \ge 1$ and for every (representative of) element $\bar{\mathcal{N}}_{k+1}$ and every $(k+1)$-tuple of Lagrangian submanifolds $(\mathcal{L}_{0}, \cdots, \mathcal{L}_{k})$ of $M \times N$, which varies smoothly on $\mathcal{N}_{k+1}$, and satisfies the following conditions:
\begin{enumerate}[label=(\roman*)]

\item For an element $(S, w) \in \mathcal{N}_{k+1}$ that is sufficiently close to the boundary strata $\partial \bar{\mathcal{N}}_{k+1}$, then the choice of action-restriction datum for $(S, w)$ is conformally equivalent to the action-restriction datum induced by gluing of action-restriction data and Floer data;

\item The following chart
\begin{equation}\label{deleted neighborhood of the boundary chart for the third type of boundary stratum}
U \times \prod_{\text{interior vertices } v} U_{v} \times \prod_{\text{interior edges } e} (0, a_{e}),
\end{equation}
for a deleted neighborhood of a boundary stratum $\sigma \subset \partial \bar{\mathcal{N}}_{k+1}$ which has a chart \eqref{boundary chart for the third type of boundary stratum}, the restriction of the action-restriction data to the main component $(S, w) \in \mathcal{S}_{l+1}$ induces a family of action-restriction data for $(S, w)$ parametrized by
\begin{equation*}
U \times \prod_{e \text{ adjacent to the root}} (0, a_{e}) \times E.
\end{equation*}
We require this family extends smoothly to
\begin{equation*}
U \times \prod_{e \text{ adjacent to the root}} [0, a_{e}) \times E,
\end{equation*}
and that it agrees on $U \times \prod_{e \text{ adjacent to the root}} \{0\} \times E$ with the family of action-restriction data that was chosen for $\mathcal{N}_{l+1}$, up to a family of conformal rescalings.

\end{enumerate}

\end{definition}

\begin{remark}
	The "universal and conformally consistent" condition is stated in a slightly more complicated way than that for Floer data as in \cite{Abouzaid1}, because the compactification $\bar{\mathcal{N}}_{k+1}$ does not have a structure of smooth manifold with corners. In fact it is a more generalized smooth space. This is the reason why we state the consistency condition in (ii) in such a complicated way, which in a concise language means the action-restriction data vary smoothly in each smooth chart of $\bar{\mathcal{N}}_{k+1}$.
\end{remark}

\begin{lemma}
	Fix a finite collection of Lagrangian submanifolds $\mathcal{L}_{1}, \cdots, \mathcal{L}_{d}$ of $M \times N$. Then universal and conformally consistent choices of action-restriction data exist, with Lagrangian labels chosen from this collection.
\end{lemma}
\begin{proof}
	The proof is an inductive argument based upon the inductive structure of the construction of the multiplihedra $\bar{\mathcal{N}}_{k+1}$, as well as the fact that the space of action-restriction data for any single element $(S, w) \in \mathcal{N}_{k+1}$ is contractible. \par
	First, fix a choice of an action-restriction datum over the strip for every pair $(\mathcal{L}_{i}, \mathcal{L}_{j})$ of Lagrangian submanifolds in this collection, which is used to define the action-restriction map
\begin{equation}
R^{1}: CW^{*}(\mathcal{L}_{i}, \mathcal{L}_{j}; H_{M, N}, J_{M, N}) \to CW^{*}(\mathcal{L}_{i}, \mathcal{L}_{j}; K, J).
\end{equation}
Such a choice is possible by the results of \cite{Gao1}.
Then we consider the $3$-pointed disk $S_{3}$, which is the unique element in $\mathcal{M}_{3}$. For each $w \in (0, +\infty)$, we consider the pair $(S_{3}, w) \in \mathcal{N}_{3}$. Let $\mathcal{L}_{j_{i}}, i = 0, 1, 2$ be the Lagrangian labels for the boundary condition. We choose an action-restriction datum for $S_{3}$ and these Lagrangian labels, such that over the $0$-th strip-like end, it agrees with the Floer datum for $(\mathcal{L}_{j_{0}}, \mathcal{L}_{j_{2}})$ in $\mathcal{W}_{d}$, and over the $i$-th strip-like ends ($i=1, 2$), it agrees with the Floer datum for $(\mathcal{L}_{j_{i-1}}, \mathcal{L}_{j_{i}})$ in $\mathcal{W}^{s}_{d}$. The existence of such an action-restriction datum is also proved in \cite{Gao1}. Moreover, we can choose action-restriction data which vary smoothly with respect to $w$, such that as $w \to 0$, the action-restriction datum converges to a Floer datum for $\mathcal{W}^{s}_{d}$, and as $w \to \infty$, the action-restriction datum converges to a Floer datum for $\mathcal{W}_{d}$. \par
	To proceed, note that the gluing maps \eqref{gluing for multiplihedra: first kind} and \eqref{gluing for multiplihedra: second kind} are the key for us to run an inductive argument. The initial step is to obtain action-restriction data for all elements in $\bar{\mathcal{N}}_{4}$ and all Lagrangian labels in the collection. Note that the choices of action-restriction data for $S_{3}$ and all possible triples of Lagrangian labels determine choices of action-restriction data for elements in a neighborhood of the boundary stratum of $\bar{\mathcal{N}}_{4}$ for the quadruple of Lagrangian labels $\mathcal{L}_{j_{i}}, i = 0, 1, 2, 3$. We then extend the choices of action-restriction data over the whole $\bar{\mathcal{N}}_{4}$. This is possible because there are no obstructions: $\bar{\mathcal{N}}_{4}$ is contractible, and the spaces of admissible Hamiltonians and of almost complex structures are contractible. \par
	For the inductive step, suppose that we have made consistent choices of action-restriction data for elements in $\bar{\mathcal{N}}_{m_{i}+1}, \bar{\mathcal{N}}_{m+2}$ as well as Floer data for elements in $\bar{\mathcal{M}}_{l+1}, \mathcal{M}_{k-m+1}$, such that the Hamiltonians, almost complex structures, etc. agree over the strip-like ends that are to be glued together. We then use the gluing maps \eqref{gluing for multiplihedra: first kind} and \eqref{gluing for multiplihedra: second kind} to obtain action-restriction data for elements in a neighborhood of the boundary strata $\partial \bar{\mathcal{N}}_{k+1}$ in $\bar{\mathcal{M}}_{k+1}$, in the way we described before Definitions \ref{definition of action-restriction datum for strips} and \ref{definition of action-restriction datum for punctured disks}. Since $\bar{\mathcal{N}}_{k+1}$ is contractible and the spaces of admissible Hamiltonians and almost complex structures are contractible, we may extend the choices for all elements therein, without changing those for elements in the boundary $\partial \bar{\mathcal{N}}_{k+1}$. This completes the inductive step and therefore finishes the proof. \par
\end{proof}

\begin{remark}
	Unlike the case of compact Lagrangian submanifolds, in general $CW^{*}(\mathcal{L}_{0}, \mathcal{L}_{1})$ is different from $CW^{*}(\mathcal{L}_{1}, \mathcal{L}_{0})$ even on the level of cohomology, and in general there is no Poincare duality between these two complexes. Thus there is no need to worry about compatibility for the choices of action-restriction datum for $(\mathcal{L}_{0}, \mathcal{L}_{1})$ and that for $(\mathcal{L}_{1}, \mathcal{L}_{0})$ - these are simply two different, independent choices.
\end{remark}

\subsection{The action-restriction functor: partial definition}
	Suppose we have made a universal and conformally consistent choice of action-restriction data for all $k \ge 1$ and all $\bar{\mathcal{N}}_{k+1}$, with the Lagrangian labels for the boundary components of punctured disks chosen from a fixed finite collection of Lagrangian submanifolds $\mathcal{L}_{1}, \cdots, \mathcal{L}_{d}$ of $M \times N$. Then we shall construct an $A_{\infty}$-functor
\begin{equation} \label{action-restriction functor for finitely many Lagrangians}
R_{d}: \mathcal{W}^{s}_{d} \to \mathcal{W}_{d},
\end{equation}
with the following properties. On the level of objects, it acts as the identity. The first order map $R_{d}^{1}$ is the action-restriction map \eqref{action-restriction map, with a Hamiltonian independent of filtration}, for each pair of Lagrangian submanifolds from the collection $\mathcal{L}_{1}, \cdots, \mathcal{L}_{d}$. \par
	As regard for higher order structure maps, consider $(k+1)$-tuple of Lagrangian submanifolds $\mathcal{L}_{j_{0}}, \cdots, \mathcal{L}_{j_{k}}$, where $j_{i} \in \{1, \cdots, d\}$. Suppose we have chosen a universal and conformally consistent choice of action-restriction data, where the Lagrangian labels are limited to the cyclically-ordered tuple $(\mathcal{L}_{j_{0}}, \cdots, \mathcal{L}_{j_{k}})$. For any $(S, w) \in \mathcal{N}_{k+1}$, the action-restriction datum from our choice sets up a moduli problem. Varying $(S, w)$ in the moduli space $\mathcal{N}_{k+1}$ gives rise to a family over $\mathcal{N}_{k+1}$, which is a parametrized moduli space $\mathcal{N}_{k+1}(\tilde{x}_{0}; x_{1}, \cdots, x_{k})$, provided appropriate asymptotic convergence conditions are given:
\begin{equation*}
x_{i} \in CW^{*}_{(-b_{i}, a_{i}]}(\mathcal{L}_{j_{i-1}}, \mathcal{L}_{j_{i}}; H_{M, N}, J_{M, N}),
\end{equation*}
and
\begin{equation*}
\tilde{x}_{0} \in CW^{*}_{(-b_{0}, a_{0}]}(\mathcal{L}_{j_{0}}, \mathcal{L}_{j_{k}}; K_{\vec{b}}, J_{\vec{b}}).
\end{equation*}
The moduli space $\mathcal{N}_{k+1}(\tilde{x}_{0}; x_{1}, \cdots, x_{k})$ has a compactification $\bar{\mathcal{N}}_{k+1}(\tilde{x}_{0}; x_{1}, \cdots, x_{k})$. The additional elements are either (equivalence classes of) broken inhomogeneous pseudoholomorphic maps from domains being elements in $\bar{\mathcal{N}}_{k+1}$, or broken inhomogeneous pseudoholomorphic maps which break out inhomogeneous pseudoholomorphic strips. \par
	Combined with the results in the last section of \cite{Gao1}, we can use standard transversality argument to prove: \par

\begin{proposition}
	For generic choice of action-restriction data, the moduli space $\bar{\mathcal{N}}_{k+1}(\tilde{x}_{0}; x_{1}, \cdots, x_{k})$ satisfies the following properties:
\begin{enumerate}[label=(\roman*)]

\item If the virtual dimension is zero, $\bar{\mathcal{N}}_{k+1}(\tilde{x}_{0}; x_{1}, \cdots, x_{k})$ is a zero-dimensional compact smooth manifold, hence consists of finitely many points;

\item If the virtual dimension is one, $\bar{\mathcal{N}}_{k+1}(\tilde{x}_{0}; x_{1}, \cdots, x_{k})$ is a one-dimensional compact topological manifold, hence is a disjoint union of finitely many circles and intervals.

\end{enumerate}
\end{proposition}

	 By counting rigid solutions in the moduli space $\bar{\mathcal{N}}_{k+1}(\tilde{x}_{0}; x_{1}, \cdots, x_{k})$, we define the following multilinear maps of degree $1-k$
\begin{equation} \label{higher order maps of the action-restriction functor, on truncated Floer complexes}
\begin{split}
\bar{R}_{d, \vec{b}}^{k}: &CW^{*}_{(-b_{k}, a_{k}]}(\mathcal{L}_{j_{k-1}}, \mathcal{L}_{j_{k}}; H_{M, N}, J_{M, N}) \otimes \cdots \otimes CW^{*}_{(-b_{1}, a_{1}]}(\mathcal{L}_{j_{0}}, \mathcal{L}_{j_{1}}; H_{M, N}, J_{M, N})\\
&\to CW^{*}_{(-b_{0}, a_{0}]}(\mathcal{L}_{j_{0}}, \mathcal{L}_{j_{k}}; K_{b_{0}}, J_{b_{0}}).
\end{split}
\end{equation}
Similar to the first-order map, we may compose this with a canonical cochain homotopy equivalence
\begin{equation*}
h_{b_{0}}: CW^{*}(\mathcal{L}_{j_{0}}, \mathcal{L}_{j_{k}}; K_{b_{0}}, J_{b_{0}}) \to CW^{*}(\mathcal{L}_{j_{0}}, \mathcal{L}_{j_{k}}; K, J),
\end{equation*}
to obtain the following map
\begin{equation}\label{higher order maps of the action-restriction functor, on truncated Floer complexes, with independent Hamiltonian and almost complex structure}
\begin{split}
R_{d, \vec{b}}^{k}: &CW^{*}_{(-b_{k}, a_{k}]}(\mathcal{L}_{j_{k-1}}, \mathcal{L}_{j_{k}}; H_{M, N}, J_{M, N}) \otimes \cdots \otimes CW^{*}_{(-b_{1}, a_{1}]}(\mathcal{L}_{j_{0}}, \mathcal{L}_{j_{1}}; H_{M, N}, J_{M, N})\\
&\to CW^{*}_{(-b_{0}, a_{0}]}(\mathcal{L}_{j_{0}}, \mathcal{L}_{j_{k}}; K, J).
\end{split}
\end{equation}

\begin{remark}
	An alternative interpretation of the action-restriction functor is helpful. Simply put, the action-restriction functor is a continuation functor associated to a particular kind of monotone homotopy from $(H_{M, N}, J_{M, N})$ to $(K, J)$. The reason we perform the construction in a slightly non-standard way as above is that it is not obviously clear that this kind of continuation functor is a homotopy equivalence, as $(H_{M, N}, J_{M, N})$ and $(K, J)$ behave somewhat differently at infinity. \par
\end{remark}

	To extend the above maps to the whole wrapped Floer complexes so that we can obtain the desired $A_{\infty}$-functor \eqref{action-restriction functor for finitely many Lagrangians}, we are faced with two potential problems. First, we must ensure that these maps are compatible with natural inclusions of wrapped Floer complexes under action filtration. Second, the target wrapped Floer complex is defined with respect to a Hamiltonian $K_{\vec{b}}$ and a one-parameter family of almost complex structures $J_{\vec{b}}$ which depend on the action filtration windows chosen in the action-restriction data. The solutions to these two problems will have to depend on the precise behavior of the choices of families of Hamiltonians and almost complex structures involved in the action-restriction data. We shall discuss these matters in greater detail in the next subsection. \par

\subsection{Arranging geometric data in a compatible system}
	We shall now explain how to obtain the families of Hamiltonians and almost complex structures involved in the action-restriction data. In Definition \eqref{definition of action-restriction datum for punctured disks}, the condition \eqref{action relation} implies that we can define a new $k$-th order "multiplication" $m^{k}_{\vec{b}}$ on the truncated wrapped Floer complexes:
\begin{equation}
\begin{split}
m^{k}_{\vec{b}}&: CW^{*}_{(-b_{k}, a_{k}]}(\mathcal{L}_{j_{k-1}}, \mathcal{L}_{j_{k}}; K_{b_{k}}, J_{b_{k}}) \otimes \cdots \otimes CW^{*}_{(-b_{1}, a_{1}]}(\mathcal{L}_{j_{0}}, \mathcal{L}_{j_{1}}; K_{b_{1}}, J_{b_{1}})\\
& \to CW^{*}_{(-b_{0}, a_{0}]}(\mathcal{L}_{J_{0}}, \mathcal{L}_{j_{k}}; K_{b_{0}}, J_{b_{0}}),
\end{split}
\end{equation}
whose definition is an obvious analogue of the operation $m^{2}_{b, 2b}$ introduced in \cite{Gao1}. Roughly speaking, this counts rigid inhomogeneous pseudoholomorphic polygons with possibly different asymptotic data (Hamiltonians and almost complex structures) over the strip-like ends. \par
	In \cite{Gao1}, we proved that $m^{2}_{b, 2b}$ is strictly compatible with the first order action-restriction maps $R_{3}^{1}$ (here $3$ is for a triple of Lagrangians). But the new operations $m^{k}_{\vec{b}}$ are not the $A_{\infty}$-operations for the $A_{\infty}$-subcategory $\mathcal{W}_{d}$ of the wrapped Fukaya category $\mathcal{W}(M \times N)$. In general, the maps $\bar{R}_{d}^{k}$ in \eqref{higher order maps of the action-restriction functor, on truncated Floer complexes} satisfy analogue of $A_{\infty}$-equations with the $A_{\infty}$-structure maps replaced by these new ones $m^{k}_{\vec{b}}$, but not the original $A_{\infty}$-structure maps. This is one of the main reasons that we need to introduce the higher order maps $R_{d}^{k}$ to adjust the failure of $R_{d}^{1}$ from being an $A_{\infty}$-functor with respect to the honest $A_{\infty}$-structures. \par
	Now let us start constructing the families of Hamiltonians and almost complex structures. The family of Hamiltonian functions is chosen as follows. We start from the families of Hamiltonians that are used to define $A_{\infty}$-structures of the wrapped Fukaya categories $\mathcal{W}^{s}_{d}$ and $\mathcal{W}_{d}$. For the first one, recall from \cite{Gao1} that for each $(k+1)$-punctured disk, we may choosen the family of split Hamiltonians $H_{S}$ to be "essentially trivial", meaning that the whole family is the rescaling by a (globally extended) time-shifting function of the single split Hamiltonian $H_{M, N}$, i.e.,
\begin{equation}
H_{S, s} = \frac{H_{M, N} \circ \phi^{\rho_{S}(s)}}{\rho(s)^{2}},
\end{equation}
where $\rho_{S}: S \to [1, +\infty)$ is an extension of the time-shifting function to the whole surface $S$. For the second, instead of considering a single family of Hamiltonians, we consider for each $b_{0}$ the $S$-family of admissible Hamiltonians $H_{S, b_{0}}$, which can also be chosen to be of the form
\begin{equation}
H_{S, b_{0}, s} = \frac{K_{b_{0}} \circ \phi^{\rho_{S}(s)}}{\rho_{S}(s)^{2}}.
\end{equation}
To find the family of Hamiltonians that is required in an action-restriction datum, our strategy is to first find a homotopy between the two Hamiltonian functions $H_{M, N}$ and $K_{\vec{b}}$. Note that after suitable compactly-supported homotopy, $K_{b_{0}}$ agrees with $H_{M, N}$ inside a compact set, and differs from $H_{M, N}$ by a product term $c r_{1} r_{2}$ in the region $\{ r_{1} \ge 1, r_{2} \ge 1 \}$, where $c$ is a universal constant (roughly $c = 2$) and $r_{1}, r_{2}$ are radial coordinates on the cylindrical ends of $M$ and $N$ respectively. We take a homotopy 
\begin{equation*}
K_{\vec{b}, t}, t \in [0, 1]
\end{equation*}
from $H_{M, N}$ to $K_{b_{0}}$ such that in the region $\partial M \times [B, +\infty) \times \partial N \times [B, +\infty)$, it takes the form $r_{1}^{2} + r_{2}^{2} + c \chi(s) r_{1}r_{2}$, where $\chi: \mathbb{R} \to [0, 1]$ is a smooth increasing function which takes the value $0$ for sufficiently negative $s \ll 0$, and takes the value $1$ for sufficiently positive $s \gg 0$. In other regions as part of the cylindrical end $\Sigma \times [1, +\infty)$, we require that the homotopy is also an increasing homotopy. Then we apply the rescaling to this homotopy to obtain the desired family of Hamiltonians, using the following trick. \par
	To get a family of Hamiltonians parametrized by the $(k+1)$-punctured disk $S$ from the homotopy, we shall realize $S$ as a subset of $\mathbb{R}^{2}$ to make use of $\chi$ to construct the family. Embed $S$ into $\mathbb{R}^{2}$ as a domain with boundary, such that the $0$-th strip-like end is mapped to the following region
\begin{equation*}
[A, +\infty) \times [-\frac{1}{2}, \frac{1}{2}],
\end{equation*}
for some large constant $A > 0$, and the $i$-th strip-like end is mapped to the following region
\begin{equation*}
(-\infty, -A] \times [C - \frac{1}{2}, C + \frac{1}{2}]
\end{equation*}
where the $C_{i}$'s are some constants satisfying $C_{i} < C_{i+1} - 2$. Now the function $\chi: \mathbb{R} \to [0, 1]$ induces a function $\chi_{S}: S \to [0, 1]$ by trivial extension to $\mathbb{R}^{2} = \mathbb{R} \times \mathbb{R}$ and restriction to $S$. Then we put
\begin{equation}
K_{S, \vec{b}, s} = \frac{K_{\vec{b}, \chi_{S}(s)} \circ \phi^{\rho_{S}(s)}}{\rho_{S}(s)^{2}}.
\end{equation}
This is the desired family of Hamiltonians $K_{S, \vec{b}}$. \par
	The construction of the family of almost complex structures $J_{S, \vec{b}}$ follows the same pattern. We choose a small infinitesimal deformation $Y$ of the almost complex structure in the space of admissible almost complex structures compatible, and add it (via the exponential map) to the product almost complex structure, such that the perturbed almost complex structure is generic, for the purpose of achieving transversality. Multiply $Y$ by $\chi_{S}(s)$ and use the exponential map to get a deformation of almost complex structures from $J_{M, N}$ to $J_{b_{0}}$, so that we can apply similar trick to define the family $J_{S, \vec{b}}$. \par
	With these families of Hamiltonians and almost complex structures, we obtain consistent choices of action-restriction data for all $k \ge 1$ and for all $\bar{\mathcal{N}}_{k+1}$, with Lagrangian labels from the fixed collection $\mathcal{L}_{1}, \cdots, \mathcal{L}_{d}$. Now we can prove: \par

\begin{lemma}
	The multilinear maps \eqref{higher order maps of the action-restriction functor, on truncated Floer complexes} satisfy an analogue of $A_{\infty}$-functor equations, with respect to the $m^{k}$'s for $\mathcal{W}(M)$ and and the modified operations $m^{k}_{\vec{b}}$ for $\mathcal{W}(N)$:
\begin{equation}\label{A-infinity functor equation for the action-restriction functor on truncated Floer complexes}
\begin{split}
& m^{1}_{b_{0}} \circ \bar{R}^{k}_{\vec{b}}(x_{k}, \cdots, x_{1})\\
= &\sum_{2 \le l \le k} \sum_{\substack{s_{1} \ge 1, \cdots, s_{l} \ge 1\\s_{1} + \cdots + s_{l} = k}}
m^{l}_{\vec{b}^{l, new}}(\bar{R}^{s_{l}}_{\vec{b}^{s_{l}}}(x_{s_{1} + \cdots + s_{l}}, \cdots, x_{s_{1} + \cdots + s_{l-1} + 1}), \cdots, \bar{R}^{s_{1}}_{\vec{b}^{s_{1}}}(x_{s_{1}}, \cdots, x_{1}))\\
&+ \sum_{0 \le s \le k-1} \sum_{i} \bar{R}^{s+1}_{\vec{b}^{s+1}}(x_{k}, \cdots, x_{i+k-s+1}, m^{k-s}(x_{i+k-s}, \cdots, x_{i+1}), x_{i}, \cdots, x_{1}).
\end{split}
\end{equation}
Here the filtration numbers $b$'s are determined by the following rule. In the first term on the right hand side of \eqref{A-infinity functor equation for the action-restriction functor on truncated Floer complexes}, we have $\vec{b}^{s_{i}} = (b^{s_{i}}_{0}, \cdots, b^{s_{i}}_{s_{i}})$, such that $\vec{b}^{l, new} = (b_{0}, b^{s_{1}}_{0}, \cdots, b^{s_{l}}_{0})$. On the left hand side, we have $\vec{b} = (b_{0}, \cdots, b_{k})$. 
Moreover, we require that when deleting those $b^{s_{i}}_{0}$'s for all $i$ and combining the rest together, $(b_{0}, b^{s_{1}}_{1}, \cdots, b^{s_{1}}_{s_{1}}, \cdots, b^{s_{l}}_{1}, \cdots, b^{s_{l}}_{s_{l}})$ should agree with $(b_{0}, b_{1}, \cdots, b_{k})$ from the left hand side.
In the second term, we have $\vec{b}^{s+1} = (b_{0}, \cdots, b_{i}, b_{i+k-s+1}, \cdots, b_{k})$.
\end{lemma}
\begin{proof}
	The compactification $\bar{\mathcal{N}}_{k+1}(\tilde{x}_{0}; x_{1}, \cdots, x_{k})$ has codimension one boundary strata consisting of the following four kinds of products of moduli spaces:
\begin{enumerate}[label=(\roman*)] 

\item \begin{equation*}
\begin{split}
&\mathcal{M}_{l+1}^{1}(\tilde{x}_{0}; \tilde{x}_{1}, \cdots, \tilde{x}_{l})\\ \times \prod_{i=1}^{l} &\mathcal{N}_{s_{i}+1}(\tilde{x}_{i}; x_{s_{1} + \cdots + s_{i-1}+1}, \cdots, x_{s_{1} + \cdots + s_{i}}), l \ge 2;
\end{split}
\end{equation*}

\item \begin{equation*}
\mathcal{N}_{k+1}(\tilde{x}'_{0}; x_{1}, \cdots, x_{k}) \times \mathcal{M}_{2}^{1}(\tilde{x}_{0}; \tilde{x}'_{0});
\end{equation*}

\item \begin{equation*}
\begin{split}
&\mathcal{N}_{s+2}(\tilde{x}_{0}; x_{1}, \cdots, x_{i}, x', x_{i+k-s+1}, \cdots, x_{k})\\ \times &\mathcal{M}_{k-s+1}^{0}(x'; x_{i+1}, \cdots, x_{i+k-s}), s \le k - 2;
\end{split}
\end{equation*}

\item \begin{equation*}
\mathcal{M}_{2}^{0}(x'_{i}; x_{i}) \times \mathcal{N}_{k+1}(\tilde{x}_{0}; x_{1}, \cdots, x_{i-1}, x'_{i}, x_{i+1}, \cdots, x_{k}).
\end{equation*}

\end{enumerate}
	Here the moduli spaces $\mathcal{M}_{l+1}^{1}(\cdots)$ with superscript $1$ consist of inhomogeneous pseudoholomorphic disks defined with respect to the Floer data $(K_{\vec{b}}, J_{\vec{b}})$. \par
	The first two types appear when the domains degenerate, which occur in the compactification of $\mathcal{N}_{k+1}$. The last two types appear when a sequence of pseudoholomorphic maps breaks off pseudoholomorphic strips at one of the strip-like ends. Strata of higher codimensions correspond to further degenerations of the domains, and breaking off more pseudoholomorphic strips. Since for our purpose only codimension one strata need to be considered, we will not spell out the details for strata of higher codimensions. \par
	Considering various operations defined by the moduli spaces appearing in the boundary strata of $\bar{\mathcal{N}}_{k+1}(\tilde{x}_{0}; x_{1}, \cdots, x_{k})$, we get the desired $A_{\infty}$-equations \eqref{A-infinity functor equation for the action-restriction functor on truncated Floer complexes}. \par
	We remark that in the above formula \eqref{A-infinity functor equation for the action-restriction functor on truncated Floer complexes} there are have two kinds of terms, as we have combined the contributions from types (i), (ii) and the types (iii), (iv). \par
\end{proof}

	In order to extend the maps \eqref{higher order maps of the action-restriction functor, on truncated Floer complexes, with independent Hamiltonian and almost complex structure} over the whole wrapped Floer complexes, we need to check that these maps are compatible with each other (for different values of $\vec{b} = (b_{0}, \cdots, b_{k})$) with respect to the natural inclusions
\begin{equation*}
\kappa_{i}: CW^{*}_{(-b_{i}, a_{i}]}(\mathcal{L}_{j_{i-1}}, \mathcal{L}_{j_{i}}; H_{M, N}, J_{M, N}) \to CW^{*}_{(-b'_{i}, a'_{i}]}(\mathcal{L}_{j_{i-1}}, \mathcal{L}_{j_{i}}; H_{M, N}, J_{M, N}),
\end{equation*}
and
\begin{equation*}
\kappa'_{0}: CW^{*}_{(-b_{0}, a_{0}]}(\mathcal{L}_{j_{0}}, \mathcal{L}_{j_{k}}; K, J) \to CW^{*}_{(-b'_{0}, a'_{0}]}(\mathcal{L}_{j_{0}}, \mathcal{L}_{j_{k}}; K, J)
\end{equation*}
of the truncated wrapped Floer complexes, whenever $b'_{i} \ge b_{i}, a'_{i} \le a_{i}$. In fact, we may fix once-for-all the $a$'s at the beginning for all the truncated wrapped Floer complexes, as changing these numbers do not affect the wrapped Floer complexes as soon as they are chosen so that all the interior chords are included. Thus keeping track of $a$'s is unnecessary. \par
	Let us describe in more detail the compatibility conditions. Consider the two composed maps $R_{d, \vec{b}'}^{k} \circ (\kappa_{k} \otimes \cdots \otimes \kappa_{1})$ and $\kappa_{0} \circ R_{d, \vec{b}}^{k}$. If they strictly agreed, then the maps $R_{d, \vec{b}}^{k}$ would be the restriction of the single map \eqref{higher order maps of the action-restriction functor, on truncated Floer complexes, with independent Hamiltonian and almost complex structure} to the truncated wrapped Floer complexes. However, by the nature of our construction, we might have chosen different families of Hamiltonians and almost complex structures when constructing the maps $R_{d, \vec{b}}^{k}$ on truncated wrapped Floer complexes for different values of $\vec{b}$. Thus these two compositions in general differ from each other. The compatibility condition we require should therefore be phrased that the sequence of maps $R_{d, \vec{b}'}^{k} \circ (\kappa_{k} \otimes \cdots \otimes \kappa_{1})$ is homotopic to $\kappa_{0} \circ R_{d, \vec{b}}^{k}$, for $b'_{i} \ge b_{i}$. \par

\begin{proposition}\label{compatibility of action-restriction maps with inclusions}
	Consider the maps $R_{d, \vec{b}}^{k}$ defined in \eqref{higher order maps of the action-restriction functor, on truncated Floer complexes, with independent Hamiltonian and almost complex structure}. Then there are multilinear maps $T_{\vec{b}, \vec{b}'}^{k}$, forming a homotopy between the two sequences of maps 
\begin{equation*}
R_{d, \vec{b}'}^{k} \circ (\kappa_{k} \otimes \cdots \otimes \kappa_{1})
\end{equation*}
and
\begin{equation*}
\kappa_{0} \circ R_{d, \vec{b}}^{k},
\end{equation*}
in the sense of homotopy between $A_{\infty}$-functors. That is, these homotopies satisfy the following analogue of $A_{\infty}$-equations:
\begin{equation}
\begin{split}
& R_{d, \vec{b}'}^{k} \circ (\kappa_{k} \otimes \cdots \otimes \kappa_{1})(x_{k}, \cdots, x_{1}) - \kappa_{0} \circ R_{d, \vec{b}}^{k}(x_{k}, \cdots, x_{1}) \\
= & \sum_{r, i} \sum_{s_{1}, \cdots, s_{r}} (-1)^{*} 
m^{r}(\kappa_{0} \circ R_{d, \vec{b}}^{s_{r}}(x_{k}, \cdots, x_{k-s_{r}+1}), \cdots,\\
& \kappa_{0} \circ R_{d, \vec{b}}^{s_{i+1}}(x_{s_{1} + \cdots + s_{i+1}}, \cdots, x_{s_{1} + \cdots + s_{i} + 1}), 
T_{\vec{b}, \vec{b}'}^{s_{i}}(x_{s_{1} + \cdots + s_{i}}, \cdots, x_{s_{1} + \cdots + s_{i-1} + 1}),\\
& R_{d, \vec{b}'}^{s_{i}} \circ (\kappa_{s_{1} + \cdots + s_{i-1}} \otimes \cdots \otimes \kappa_{_{s_{1} + \cdots + s_{i-2} + 1}})(x_{s_{1} + \cdots + s_{i-1}}, \cdots, x_{s_{1} + \cdots + s_{i-2} + 1}),\\
& \cdots, R_{d, \vec{b}'}^{s_{1}} \circ (\kappa_{s_{1}} \otimes \cdots \otimes \kappa_{1})) \\
+ & \sum_{m, l} (-1)^{**} T_{\vec{b}, \vec{b}'}^{k-m+1}(x_{k}, \cdots, x_{m+l+1}, \mu^{m}(x_{m+l}, \cdots, x_{l+1}), x_{l}, \cdots, x_{1}).
\end{split}
\end{equation}
Here the symbols $\mu^{k}$ (temporarily) denote the $A_{\infty}$-structure maps in the wrapped Fukaya category defined with respect to $(H_{M, N}, J_{M, N})$, and $m^{k}$ denote those in the wrapped Fukaya category defined with respect to $(K, J)$. The signs are
\begin{equation*}
* = s_{1} + \cdots + s_{i-1} - \deg(x_{1}) - \cdots - \deg(x_{s_{1} + \cdots s_{i-1}}),
\end{equation*}
and
\begin{equation*}
** = \deg(x_{1}) + \cdots + \deg(x_{l}) - l - 1.
\end{equation*}
\end{proposition}
\begin{proof}[Sketch of proof]
	Checking this kind of compatibility is essentially repetition of the argument in the last section of \cite{Gao1}, now inductively on $k$. The key reasoning is that the spaces of admissible Hamiltonians and of almost complex structures are contractible. 
\end{proof}

	This compatibility condition stated in Proposition \ref{compatibility of action-restriction maps with inclusions} then implies that the homotopy direct limit of $R^{k}_{d, \vec{b}}$ exists as $b_{i} \to +\infty$ for all $i$. \par

\begin{corollary}
	There exist multilinear maps
\begin{equation}
\begin{split}
R^{k}_{d}: &CW^{*}(\mathcal{L}_{j_{k-1}}, \mathcal{L}_{j_{k}}; H_{M, N}, J_{M, N}) \otimes \cdots \otimes CW^{*}(\mathcal{L}_{j_{0}}, \mathcal{L}_{j_{1}}; H_{M, N}, J_{M, N})\\
&\to CW^{*}(\mathcal{L}_{j_{0}}, \mathcal{L}_{j_{k}}; K, J),
\end{split}
\end{equation}
such that when restricted to any truncated wrapped Floer complex, it is homotopic to the maps \eqref{higher order maps of the action-restriction functor, on truncated Floer complexes}.
\end{corollary}
\begin{proof}
	We can modify the maps $R_{d, \vec{b}}^{k}$ in \eqref{higher order maps of the action-restriction functor, on truncated Floer complexes} by composing them with self-homotopy equivalences on the truncated wrapped Floer complexes 
\begin{equation*}
CW^{*}_{(-b_{i}, a_{i}]}(\mathcal{L}_{j_{i-1}}, \mathcal{L}_{j_{i}}; H_{M, N}, J_{M, N})
\end{equation*}
and also
\begin{equation*}
CW^{*}_{(-b_{0}, a_{0}]}(\mathcal{L}_{j_{0}}, \mathcal{L}_{j_{k}}; K, J),
\end{equation*}
so that $R_{d, \vec{b}'}^{k} \circ (\kappa_{k} \otimes \cdots \otimes \kappa_{1})$ and $\kappa_{0} \circ R_{d, \vec{b}}^{k}$ strictly agree after such modification. Thus the direct limit of the modified maps exists. The homotopy direct limit 
\begin{equation}
R^{k}_{d} = \varinjlim_{b_{i} \to +\infty}R^{k}_{d, \vec{b}}
\end{equation}
is defined to be the direct limit of the modified maps. 
\end{proof}

	The $A_{\infty}$-functor equations for the sequence of multilinear maps $\{R_{d}^{k}\}_{k=1}^{\infty}$ follow from the universal and conformally consistent choice of action-restriction data. Recall that we have analogues of $A_{\infty}$-functor equations for the maps \eqref{higher order maps of the action-restriction functor, on truncated Floer complexes}. By taking homotopy direct limit as above, we obtain the desired $A_{\infty}$-equations for the maps $R_{d}^{k}$. \par
	 We have thus completed the construction of the $A_{\infty}$-quasi-isomorphism \eqref{action-restriction functor for finitely many Lagrangians}. \par

\subsection{Extension to the whole wrapped Fukaya categories}
	To finish the proof of Theorem \ref{two models of wrapped Fukaya categories of product manifolds are equivalent}, we need to be able to extend the functors \eqref{action-restriction functor for finitely many Lagrangians} to the whole wrapped Fukaya categories. The approach we take is to apply the homotopy direct limit procedure described in section \ref{A-infinity homotopy direct limit} to obtain a homotopy direct limit functor
\begin{equation} \label{the full action-restriction functor}
R: \mathcal{W}^{s}(\mathbb{L}) \to \tilde{\mathcal{W}}(\mathbb{L}),
\end{equation}
from the sequence of $A_{\infty}$-functors $R_{d}: \mathcal{W}_{d}^{s} \to \mathcal{W}_{d}$. Here $\tilde{\mathcal{W}}(\mathbb{L})$ is an $A_{\infty}$-category quasi-equivalent to the wrapped Fukaya category $\mathcal{W}(\mathbb{L})$, consisting of the same objects, i.e. those Lagrangians in the collection $\mathbb{L}$. Then it follows immediately that $R$ is a quasi-equivalence, since each $R_{d}$ is a quasi-isomorphism. \par
	It remains to check the two assumptions in section \ref{A-infinity homotopy direct limit}. Assumption \eqref{direct limit assumption on subcategories} automatically holds by the definition of these $A_{\infty}$-subcategories $\mathcal{W}^{s}_{d}$ of $\mathcal{W}^{s}(\mathbb{L})$ and $\mathcal{W}_{d}$ of $\mathcal{W}(\mathbb{L})$, because of the inductive nature of the choices of Floer data used in the definition of the wrapped Fukaya category. Therefore, to get the homotopy direct limit functor \eqref{the full action-restriction functor}, it suffices to check Assumption \ref{direct limit assumption on functors}. \par
	We restrict to the case of wrapped Fukaya categories: $\mathcal{A} = \mathcal{W}^{s}(\mathbb{L}), \mathcal{A}_{d} = \mathcal{W}^{s}_{d}, \mathcal{B} = \mathcal{W}(\mathbb{L}), \mathcal{B}_{d} = \mathcal{W}_{d}$, and the action-restriction functors $\mathcal{F}_{d} = R_{d}$ that we define on these subcategories. Note that Assumption \eqref{direct limit assumption on functors} is rather weak, in the sense that it does not ask for a specific homotopy between the two $A_{\infty}$-functors $j_{d, d'}^{-1} \circ R_{d'} \circ i_{d, d'}$ and $R_{d}$. Thus the freedom in the choice of such a homotopy makes the process quite flexible. \par
	Since there are at most $d'$ Lagrangian submanifolds $\mathcal{L}_{1}, \cdots, \mathcal{L}_{d'}$ to consider at each time, the $A_{\infty}$-functors $R_{d}$ and $R_{d'}$ are quite concrete: they are determined by our choices of action-restriction data. There are two families of such choices - one for the collection of Lagrangian labels $\mathcal{L}_{1}, \cdots, \mathcal{L}_{d}$, denoted by $D_{0}$, the other for $\mathcal{L}_{1}, \cdots, \mathcal{L}_{d'}$, denoted by $D_{1}$ (meaning choices for all elements in $\bar{\mathcal{N}}_{k+1}$ in a consistent way. We may also restrict $D_{1}$ to the collection $\mathcal{L}_{1}, \cdots, \mathcal{L}_{d}$ to obtain consistent choices of action-restriction data for $\mathcal{L}_{1}, \cdots, \mathcal{L}_{d}$, still denoted by $D_{1}$. The idea of proving that $R_{d}$ and $j_{d, d'}^{-1} \circ R_{d'} \circ i_{d, d'}$ are homotopic as $A_{\infty}$-functors from $\mathcal{W}^{s}_{d}$ to $\mathcal{W}_{d}$ is to choose a one-parameter family $D_{t}$ of action-restriction data interpolating these two, and then use the resulting parametrized moduli spaces to construct the desired homotopy between the two action-restriction functors determined by $D_{0}$ and $D_{1}$ respectively. The existence of such one-parameter family is because the spaces of such Hamiltonians/almost complex structures are contractible. \par
	To construct a homotopy between the functors associated to these two data $D_{0}, D_{1}$, we need a one-dimensional higher analogue of the multiplihedra, which we call the homotopehedra and denote by $\bar{\mathcal{P}}_{k+1}$. Define $\mathcal{P}_{k+1} = \mathcal{N}_{k+1} \times [0, +\infty)$. The compactification $\bar{\mathcal{P}}_{k+1}$ has boundary strata made of products of copies of $\mathcal{M}_{i+1}, \mathcal{N}_{j+1}$ as well as $\mathcal{P}_{l+1}$. \par
	Now we consider moduli space $\mathcal{P}_{k+1}(\tilde{x}_{0}; x_{1}, \cdots, x_{k})$ of elements $(S, w, t, u)$ where $(S, w, t) \in \mathcal{P}_{k+1}$, and $u: (S, w) \to M$ is a continuation disk satisfying the Cauchy-Riemann equation with respect to the Hamiltonian and almost complex structure from the action-restriction datum $D_{t}$, which converges to some $K$-chord $\tilde{x}_{0}$ over the $0$-th strip-like end, and to some $H_{M, N}$-chord $x_{i}$ over the $i$-th strip-like end. \par
	There is a natural bordification $\bar{\mathcal{P}}_{k+1}(\tilde{x}_{0}; x_{1}, \cdots, x_{k})$, whose codimension one boundary strata consist of a union of product moduli spaces of the following form:
\begin{equation}\label{boundary strata of moduli space defining homotopy of action-restriction functors}
\begin{split}
&\partial \bar{\mathcal{P}}_{k+1}(\tilde{x}_{0}; x_{1}, \cdots, x_{k})\\
&= \coprod_{m, n} \mathcal{P}_{k-m}(\tilde{x}_{0}; x_{1}, \cdots, x_{n}, x', x_{n+m+1}, \cdots, x_{k}) \times \mathcal{M}_{m+1}(x'; x_{n+1}, \cdots, x_{n+m})\\
&\cup \coprod_{\substack{r, s\\i_{1} + \cdots i_{r} = k}} \coprod_{\tilde{x}'_{1}, \cdots, \tilde{x}'_{r}} \mathcal{M}_{r+1}(\tilde{x}_{0}; \tilde{x}'_{1}, \cdots \tilde{x}'_{r}) \times (\mathcal{N}_{i_{1}+1}(\tilde{x}'_{1}; x_{1}, \cdots, x_{i_{1}}; D_{0})\\
&\times \cdots \times \mathcal{N}_{i_{s-1}+1}(\tilde{x}'_{s-1}; x_{i_{1} + \cdots + i_{s-2} + 1}, \cdots, x_{i_{1} + \cdots + i_{s-1}}; D_{0})\\
&\times \mathcal{P}_{i_{s}+1}(\tilde{x}'_{s}; x_{i_{1} + \cdots + i_{s-1} + 1}, \cdots, x_{i_{1} + \cdots + i_{s}}; D_{t})\\
&\times \mathcal{N}_{i_{s+1}+1}(\tilde{x}'_{s+1}; x_{i_{1} + \cdots + i_{s} + 1}, \cdots, x_{i_{1} + \cdots + i_{s+1}}; D_{1})\\
&\times \cdots \times \mathcal{N}_{i_{r}+1}(\tilde{x}'_{r}; x_{i_{1} + \cdots + i_{r-1} + 1}, \cdots, x_{k}; D_{1}))\\
&\cup \mathcal{N}_{k+1}(\tilde{x}_{0}; x_{1}, \cdots, x_{k}; D_{1})\\
&\cup \mathcal{N}_{k+1}(\tilde{x}_{0}; x_{1}, \cdots, x_{k}; D_{0}).
\end{split}
\end{equation}
Here the $x$'s without tilde denote $H_{M, N}$-chords, while the $\tilde{x}$'s denote $K$-chords. \par
	By counting rigid elements in the moduli spaces $\bar{\mathcal{P}}_{k+1}(\tilde{x}_{0}; x_{1}, \cdots, x_{k})$, we construct from this one-parameter family of action-restriction data $D_{t}$ a multilinear map
\begin{equation}
T^{k}: CW^{*}(\mathcal{L}_{0}, \mathcal{L}_{1}; H_{M, N}) \otimes \cdots \otimes CW^{*}(\mathcal{L}_{k-1}, \mathcal{L}_{k}; H_{M, N}) \to CW^{*}(\mathcal{L}_{0}, \mathcal{L}_{k}; K)
\end{equation}
of degree $-k$, whose $(\tilde{x}_{0}; x_{1}, \cdots, x_{k})$ component is the count of rigid elements in the moduli space $\bar{\mathcal{P}}_{k+1}(\tilde{x}_{0}; x_{1}, \cdots, x_{k})$. \par
	Setting $T^{0} = 0$, we claim that $T = \{T^{k}\}_{k=0}^{\infty}$ is a homotopy between the functors $R_{d}$ and $j_{d, d'}^{-1} \circ R_{d'} \circ i_{d, d'}$, where the former is defined by the action-restriction datum $D_{0}$, and the latter by $D_{1}$. Verifying the relation $m^{1}_{\mathcal{Q}}(T) = j_{d, d'}^{-1} \circ R_{d'} \circ i_{d, d'} - R_{d}$ amounts to looking at the boundary strata of the bordification $\bar{\mathcal{P}}_{k+1}(x'_{0}; x_{1}, \cdots, x_{k})$ as above \eqref{boundary strata of moduli space defining homotopy of action-restriction functors}. This relation is precisely the condition for $R_{d}$ and $j_{d, d'}^{-1} \circ R_{d'} \circ i_{d, d'}$ to be homotopic as $A_{\infty}$-functors. Thus we have verified Assumption \ref{direct limit assumption on functors}, and thus completed the proof of Theorem \ref{two models of wrapped Fukaya categories of product manifolds are equivalent}. \par

\begin{remark}
	Here is a technical remark regarding the general definition of the wrapped Fukaya category that we use here. In practice, one starts with an at most countable collection of closed exact or cylindrical Lagrangian submanifolds, and declare them to be the objects of the wrapped Fukaya category. It is only in this sense our perturbation framework makes sense inductively, and the extension of the action-restriction functor to the whole wrapped Fukaya category is valid via the procedure we described in section \ref{A-infinity homotopy direct limit}. \par
	In principle, one could allow all (not just a countable collection of) admissible Lagrangian submanifolds as objects of the wrapped Fukaya category, but should pay close attension to the transversality argument: one must carefully make choices of Floer data each time, and use Axiom of Choice uncountably many times to define the $A_{\infty}$-structure on the wrapped Fukaya category. However, it seems technically difficult to extend the action-restriction functor in that setup. We will not discuss this technical point since it is not needed for most applications. \par
\end{remark}

\section{$A_{\infty}$-functors associated to Lagrangian correspondences}
\label{A-infinity functors associated to Lagrangian correspondence}

\subsection{Extension of quilted wrapped Floer cohomology to Lagrangian immersions}\label{section: quilted wrapped Floer theory for Lagrangian immersions}
	This section provides chain-level refinements of the construction of cohomological functors in \cite{Gao1}. That is, we are going to prove that admissible Lagrangian correspondences give rise to functors between appropriate versions of wrapped Fukaya categories. For simplicity, we let the source of the functors be the wrapped Fukaya category of $M$ consisting of embedded exact cylindrical Lagrangian submanifolds. All $A_{\infty}$-functors are to be understood as cohomologically unital $A_{\infty}$-functors. \par
	As an introductory part of the main construction, we first give a naive attempt in extending quilted wrapped Floer cohomology to exact cylindrical Lagrangian immersions with transverse or clean self-intersections. For our purpose of constructing functors from Lagrangian correspondences, we shall only consider the case where $L \subset M$ and $\mathcal{L} \subset M^{-} \times N$ are properly embedded, while $L' \subset N$ is replaced by an exact cylindrical Lagrangian immersion $\iota: L' \to N$. \par
	One short-cut definition for the quilted wrapped Floer cochain space is given as follows. As the underlying $\mathbb{Z}$-module, the quilted wrapped Floer cochain space
\begin{equation*}
CW^{*}(L, \mathcal{L}, (L', \iota))
\end{equation*}
is defined as the wrapped Floer cochain space
\begin{equation*}
CW^{*}(\mathcal{L}, L \times (L', \iota)),
\end{equation*}
for the pair of exact cylindrical Lagrangian immersions in $M^{-} \times N$, where $L \times (L', \iota)$ is the obvious product Lagrangian immersion with clean self-intersections. This pair has clean intersections, whose wrapped Floer cochain space is defined in section \ref{section: wrapped Floer cochain space for a pair with clean intersections}. The quilted Floer "differential" $n^{0}$ is defined as the zeroth-order curved $A_{\infty}$-structure map on the above wrapped Floer cochain space $CW^{*}(\mathcal{L}, L \times (L', \iota))$. Here we put the quotation mark because $n^{0}$ might not square to zero in general. Alternatively, there is another straightforward definition, using moduli spaces of inhomogeneous pseudoholomorphic quilted strips, following the standard setup of quilted wrapped Floer theory. In fact, these inhomogeneous pseudoholomorphic quilted strips are in natural bijection to inhomogeneous pseudoholomorphic strips in the product manifold, and we can choose the same perturbations (by multisections) for both moduli spaces. Thus the second definition is equivalent to the first one. \par
	The second definition is more suitable for discussing the $A_{\infty}$-bimodule structure on the quilted wrapped Floer cochain space $CW^{*}(L, \mathcal{L}, (L', \iota))$. The details of the construction will be discussed in subsection \ref{section: module-valued functors}. \par

\subsection{The module-valued functors associated to Lagrangian correspondences}\label{section: module-valued functors}
	The guiding principle for constructing $A_{\infty}$-functors from Lagrangian correspondences is to use moduli spaces of inhomogeneous pseudoholomorphic quilted maps. For these $A_{\infty}$-functors to be defined over $\mathbb{Z}$, we must ensure that these moduli spaces carry coherent orientations. The reader is referred to the Appendix of \cite{Gao1} for the discussion on orientations on the relevant moduli spaces of inhomogeneous pseudoholomorphic quilts, where the discussion focused on one particular kind of moduli space but can be easily generalized to all the other ones which we actually use in this paper. For a more general discussion on orientations of moduli spaces of pseudoholomorphic quilts, we refer the reader to \cite{Wehrheim-Woodward5}, but remark that our approach is independent because some of the quilted surfaces we use are not included there. \par
	The starting point is to to relate the wrapped Fukaya category $\mathcal{W}(M^{-} \times N)$ of the product manifold, to the dg-category of $A_{\infty}$-bimodules over $(\mathcal{W}(M), \mathcal{W}(N))$. \par

\begin{proposition}\label{prop: construction of bimodule-valued functor}
	There is a canonical $A_{\infty}$-functor
\begin{equation} \label{bimodule-valued functor}
\Phi: \mathcal{W}(M^{-} \times N) \to (\mathcal{W}(M), \mathcal{W}(N))^{bimod},
\end{equation}
satisfying the following properties:
\begin{enumerate}[label=(\roman*)]

\item $\Phi$ is non-trivial for any non-trivial $\mathcal{W}(M^{-} \times N)$;

\item If either $M$ or $N$ is a point, $\Phi$ is the Yoneda functor for $N$ or $M$;

\item If $\mathcal{L} = L \times L'$ is a product Lagrangian correspondence, then the $A_{\infty}$-bimodule $\Phi(\mathcal{L})$ splits. That is, there is an isomorphism of $A_{\infty}$-bimodules
\begin{equation}
\Phi(\mathcal{L}) \cong \mathfrak{y}_{r}(L) \otimes \mathfrak{y}_{l}(L'),
\end{equation}
where $\mathfrak{y}_{r}$ and $\mathfrak{y}_{l}$ are the right and left Yoneda functors.

\end{enumerate}
\end{proposition}

	The idea of proof is to develop a quilted version of wrapped Floer theory, extending that in \cite{Gao1}. The natural construction should yield an $A_{\infty}$-functor from the split model of the wrapped Fukaya category $\mathcal{W}^{s}(M^{-} \times N)$. However, based on the results of section \ref{section: product manifolds}, that is equivalent to the ordinary wrapped Fukaya category $\mathcal{W}(M^{-} \times N)$. Thus it does not matter which version of wrapped Fukaya category of the product manifold we use as far as algebraic structures are concerned, at least up to quasi-equivalence. \par

\begin{definition}
	A Lagrangian correspondence $\mathcal{L} \subset M^{-} \times N$ from $M$ to $N$ is said to be admissible, if it is admissible for wrapped Floer theory in the product manifold $M^{-} \times N$ in the sense of Definition \ref{definition of admissible Lagrangian submanifolds in the product}, i.e. is an object of the wrapped Fukaya category $\mathcal{W}(M^{-} \times N)$.
\end{definition}

	As a result, by evaluating the above $A_{\infty}$-functor at each given object $\mathcal{L}$ of $\mathcal{W}(M^{-} \times N)$, i.e. an admissible Lagrangian correspondence, we then get an $A_{\infty}$-bimodule over $(\mathcal{W}(M), \mathcal{W}(N))$. By purely algebraic consideration involving the Yoneda embedding, we have the following $A_{\infty}$-functor associated to $\mathcal{L}$: \par

\begin{corollary}
	For any admissible Lagrangian correspondence $\mathcal{L} \subset M^{-} \times N$, there is an associated $A_{\infty}$-functor:
\begin{equation} \label{functor to modules}
\Phi_{\mathcal{L}}: \mathcal{W}(M) \to \mathcal{W}(N)^{l-mod},
\end{equation}
to the dg-category of left $A_{\infty}$-modules over $\mathcal{W}(N)$. 
\end{corollary}

	In particular, we remark that such an $A_{\infty}$-functor is defined for any admissible Lagrangian correspondence $\mathcal{L} \subset M^{-} \times N$, without any properness assumption. However, as we shall see in the next subsection, a suitable properness assumption is needed in order to prove that this module-valued functor is representable, thus can be improved to a filtered $A_{\infty}$-functor to the immersed wrapped Fukaya category $\mathcal{W}_{im}(N)$. \par
	Now let us discuss in detail the construction of the bimodule-valued functor \eqref{bimodule-valued functor}. On the level of objects, the functor $\Phi$ should assign an $A_{\infty}$-bimodule $\Phi(\mathcal{L})$ over $(\mathcal{W}(M), \mathcal{W}(N))$ to an admissible Lagrangian correspondence $\mathcal{L} \subset M^{-} \times N$. The first order bimodule structure map of $\Phi(\mathcal{L})$ has already been constructed in \cite{Gao1}. We now give an extension of that, defining the $A_{\infty}$-bimodule structure maps of all orders in a uniform treatment. \par
	Consider the quilted surface $\underline{S}^{k, l}$ consisting of two patches $S^{k}_{0}, S^{l}_{1}$, where $S^{k}_{0}$ is a disk with $(k+2)$ boundary punctures $z_{0}^{-}, z_{0}^{1}, \cdots, z_{0}^{k}, z_{0}^{+}$, and $S^{l}_{1}$ is a disk with $(l+2)$ boundary punctures $z_{1}^{-}, z_{1}^{1}, \cdots, z_{1}^{l}, z_{1}^{+}$. Let $I_{0}^{\pm}$ be the boundary component of $S^{k}_{0}$ between $z_{0}^{+}$ and $z_{0}^{-}$, and $I_{1}^{\pm}$ the boundary component of $S^{l}_{1}$ between $z_{1}^{+}$ and $z_{1}^{-}$. The quilted surface is obtained by seaming the two patches along these two boundary components. After seaming the two patches, the quilted surface $\underline{S}^{k, l}$ has $(k+l)$ positive strip-like ends $\epsilon_{0}^{1}, \cdots, \epsilon_{0}^{k}$ and $\epsilon_{1}^{1}, \cdots, \epsilon_{1}^{l}$ as well as two quilted ends (one positive and one negative), each of which consists of two strip-like ends. See the picture below. \par

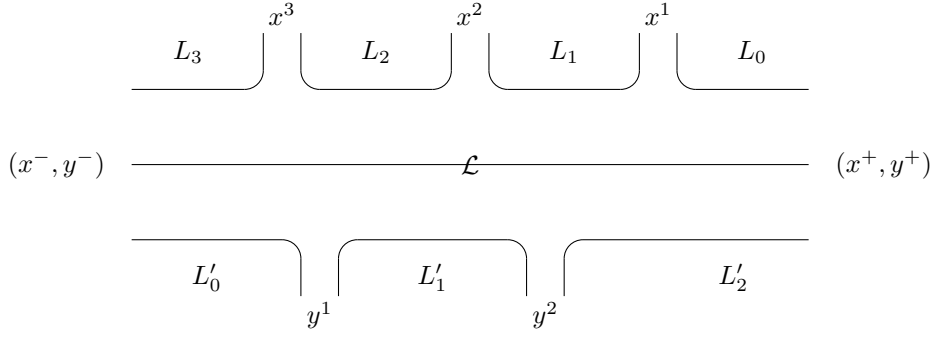
\begin{figure}
\begin{tikzpicture}
	\draw (-6, 1) -- (-4.5, 1);
	\draw (-4.5, 1) arc (270:360:0.25cm);
	\draw (-4.25, 1.25) -- (-4.25, 1.75);
	\draw (-3.5, 1) arc (270:180:0.25cm);
	\draw (-3.75, 1.25) -- (-3.75, 1.75);
	\draw (-3.5, 1) -- (-2, 1);
	\draw (-2, 1) arc (270:360:0.25cm);
	\draw (-1.75, 1.25) -- (-1.75, 1.75);
	\draw (-1, 1) arc (270:180:0.25cm);
	\draw (-1.25, 1.25) -- (-1.25, 1.75);
	\draw (-1, 1) -- (0.5, 1);
	\draw (0.5, 1) arc (270:360:0.25cm);
	\draw (0.75, 1.25) -- (0.75, 1.75);
	\draw (1.5, 1) arc (270:180:0.25cm);
	\draw (1.25, 1.25) -- (1.25, 1.75);
	\draw (1.5, 1) -- (3, 1);

	\draw (-6, 0) -- (3, 0);

	\draw (-6, -1) -- (-4, -1);
	\draw (-4, -1) arc (90:0:0.25cm);
	\draw (-3.75, -1.25) -- (-3.75, -1.75);
	\draw (-3, -1) arc (90:180:0.25cm);
	\draw (-3.25, -1.25) -- (-3.25, -1.75);
	\draw (-3, -1) -- (-1, -1);
	\draw (-1, -1) arc (90:0:0.25cm);
	\draw (-0.75, -1.25) -- (-0.75, -1.75);
	\draw (0, -1) arc (90:180:0.25cm);
	\draw (-0.25, -1.25) -- (-0.25, -1.75);
	\draw (0, -1) -- (3, -1);

	\draw (-5.25, 1.5) node {$L_{3}$};
	\draw (-2.75, 1.5) node {$L_{2}$};
	\draw (-0.25, 1.5) node {$L_{1}$};
	\draw (2.25, 1.5) node {$L_{0}$};
	\draw (-1.5, 0) node {$\mathcal{L}$};
	\draw (-5, -1.5) node {$L'_{0}$};
	\draw (-2, -1.5) node {$L'_{1}$};
	\draw (2, -1.5) node {$L'_{2}$};

	\draw (-4, 2) node {$x^{3}$};
	\draw (-1.5, 2) node {$x^{2}$};
	\draw (1, 2) node {$x^{1}$};
	
	\draw (-7, 0) node {$(x^{-}, y^{-})$};
	\draw (4, 0) node {$(x^{+}, y^{+})$};
	
	\draw (-3.5, -2) node {$y^{1}$};
	\draw (-0.5, -2) node {$y^{2}$};

\end{tikzpicture}
\caption{The quilted map defining the bimodule} \label{fig: the quilted map defining the bimodule}
\end{figure}

	Choosing a Floer datum for $\underline{S}^{k, l}$ allows us to define inhomogeneous pseudoholomorphic quilted maps from $\underline{S}^{k, l}$ to the pair $(M, N)$ with appropriate moving Lagrangian boundary conditions and asymptotic convergence conditions over the various ends. We shall choose Floer data for all (representatives of) such quilted surfaces in the moduli spaces, and extend the choices by automorphism-invariant Floer data on unstable components, i.e. quilted strips, of semistable quilted surfaces, such that they are compatible under gluing maps with the universal and consistent choices we made for disks. \par
	Let $\mathcal{R}^{k, l}((x^{-}, y^{-}); (x^{+}, y^{+}), \vec{x}, \vec{y})$ be the moduli space, namely the set of isomorphism classes of inhomogeneous pseudoholomorphic quilted maps $(\underline{S}^{k, l}, \underline{u})$ as pictured in Figure \ref{fig: the quilted map defining the bimodule}. There is a natural Gromov bordification
\begin{equation}\label{moduli space for defining the bimodule structure maps}
\bar{\mathcal{R}}^{k, l}((x^{-}, y^{-}); (x^{+}, y^{+}), \vec{x}, \vec{y}),
\end{equation}
which compactifies $\mathcal{R}^{k, l}((x^{-}, y^{-}); (x^{+}, y^{+}), \vec{x}, \vec{y})$. The codimension one boundary strata of $\bar{\mathcal{R}}^{k, l}((x^{-}, y^{-}); (x^{+}, y^{+}), \vec{x}, \vec{y})$ are covered by a union of products of moduli spaces of the following form:
\begin{equation} \label{boundary stratum of the moduli space defining the bimodule structure}
\begin{split}
\coprod_{\substack{1 \le i \le k\\ (\vec{x}', \vec{x}'') = \vec{x}}} \coprod_{x_{new}} 
&\bar{\mathcal{R}}^{k-i+1, l}((x^{-}, y^{-}); (x^{+}, y^{+}), \vec{x}', x_{new}, \vec{y}) \times \bar{\mathcal{M}}_{i+1}(x_{new}, \vec{x}'')\\
\cup \coprod_{\substack{1 \le j \le l\\ (\vec{y}', \vec{y}'') = \vec{y}}} \coprod_{y_{new}} 
&\bar{\mathcal{R}}^{k, l-j+1}((x^{-}, y^{-}); (x^{+}, y^{+}), \vec{x}, \vec{y}', y_{new}) \times \bar{\mathcal{M}}_{j+1}(y_{new}, \vec{y}'')\\
\cup \coprod_{\substack{k'+k''=k, l'+l''=l\\ (x^{+}_{1}, y^{+}_{1})}}
&\bar{\mathcal{R}}^{k', l'}((x^{-}, y^{-}); (x^{+}_{1}, y^{+}_{1}), \vec{x}', \vec{y}')\\
&\times \bar{\mathcal{R}}^{k'', l''}((x^{+}_{1}, y^{+}_{1}); (x^{+}, y^{+}), \vec{x}'', \vec{y}'')
\end{split}
\end{equation}
Here by the notation $\bar{\mathcal{R}}^{k-i+1, l}((x^{-}, y^{-}); (x^{+}, y^{+}), \vec{x}', x_{new}, \vec{y})$ we mean to insert the new Hamiltonian chord $x_{new}$ in every possible place that splits the tuple $\vec{x}$ of Hamiltonian chords to the two tuples $\vec{x}'$ and $\vec{x}''$, as long as the cyclic order is preserved. Similar remarks apply to the $y$'s. \par
	In \cite{Wehrheim-Woodward4}, it is demonstrated that this kind of moduli space is locally modeled on a Fredholm complex. Since there are no disk bubbles, we can use the standard transversality argument to prove that for generic universal and consistent choices of Floer data, the Gromov bordification $\bar{\mathcal{R}}^{k, l}((x^{-}, y^{-}); (x^{+}, y^{+}), \vec{x}, \vec{y})$ is a compact smooth manifold with corners of expected dimension 
\begin{equation}
	k - 2 + l - 2 + \deg((x^{-}, y^{-})) - \deg((x^{+}, y^{+})) - \sum \deg(x^{i}) - \sum \deg(y^{j}).
\end{equation}
And moreover, we can arrange the perturbations so that every stratum is regular. \par
	A finiteness result, which says that given inputs $(x^{+}, y^{+}), \vec{x}, \vec{y}$, the moduli spaces $\bar{\mathcal{R}}^{k, l}((x^{-}, y^{-}); (x^{+}, y^{+}), \vec{x}, \vec{y})$ are empty for all but finitely many outputs $(x^{-}, y^{-})$, can be deduced from the action-energy equality, which plays a crucial role in the well-definedness of various maps in wrapped Floer theory (\cite{Abouzaid-Seidel}, \cite{Abouzaid1}, also see \cite{Gao1} in which we used the quilted version in special cases $k \le 1, l \le 1$). This ensures that the count of rigid elements of all moduli spaces $\bar{\mathcal{R}}^{k, l}((x^{-}, y^{-}); (x^{+}, y^{+}), \vec{x}, \vec{y})$ for fixed inputs $(x^{+}, y^{+}), \vec{x}, \vec{y}$ is finite, which gives rise to a map
\begin{equation}
\begin{split}
n^{k|0|l}&: CW^{*}(L_{k-1}, L_{k}) \otimes \cdots \otimes CW^{*}(L_{0}, L_{1}) \otimes CW^{*}(L_{k}, \mathcal{L}, L'_{0})\\
&\otimes CW^{*}(L'_{l-1}, L'_{l}) \otimes \cdots \otimes CW^{*}(L'_{0}, L'_{1}) \to CW^{*}(L_{0}, \mathcal{L}, L'_{l}).
\end{split}
\end{equation}
By analyzing the boundary strata of one dimensional moduli spaces as described in \eqref{boundary stratum of the moduli space defining the bimodule structure}, we find that the operation $n^{k|0|l}$ satisfies the following equation
\begin{equation} \label{bimodule equation}
\begin{split}
&m^{1} \circ n^{k|0|l}([\vec{x}], [x^{+}, y^{+}], [\vec{y}])\\
&= \sum n^{k-i+1|0|l}([\vec{x}'], m^{i}([\vec{x}'']), [x^{+}, y^{+}], [\vec{y}])\\
&+ \sum n^{k|0|l-j+1}([\vec{x}], [x^{+}, y^{+}], [\vec{y}'], m^{j}([\vec{y}'']))\\
&+ \sum_{\substack{k'+k''=k, l'+l''=l\\ (\vec{x}', \vec{x}'') = \vec{x}}} n^{k'|0|l'}([\vec{x}'], n^{k''|0|l''}([\vec{x}''], [x^{+}, y^{+}], [\vec{y}'']), [\vec{y}']).
\end{split}
\end{equation}
This precisely means that the operations $n^{k|0|l}$ define an $A_{\infty}$-bimodule structure on $\Phi(\mathcal{L})$ over $(\mathcal{W}(M), \mathcal{W}(N))$. \par

	Next we want to study what the $A_{\infty}$-functor \eqref{bimodule-valued functor} does to morphisms, and how this is related with the $A_{\infty}$-structure maps of $\mathcal{W}(M)$ and $\mathcal{W}(N)$. A Floer cochain $[\gamma] \in CW^{*}(\mathcal{L}_{0}, \mathcal{L}_{1})$ should give rise to an $A_{\infty}$-bimodule homomorphism
\begin{equation} \label{bimodule homomorphism associated to Floer cochain}
\Phi_{\gamma}: \Phi_{\mathcal{L}_{0}} \to \Phi_{\mathcal{L}_{1}}.
\end{equation}
Moreover, this should be functorial in the wrapped Fukaya category of the product manifold $M^{-} \times N$ as stated in Proposition \ref{prop: construction  of bimodule-valued functor}.
More specifically, it means that there are multilinear maps
\begin{equation}\label{bimodule homomorphism: first order}
\begin{split}
n^{k|1|l}&: CW^{*}(\mathcal{L}_{0}, \mathcal{L}_{1}) \to \hom(CW^{*}(L_{k-1}, L_{k}) \otimes \cdots \otimes CW^{*}(L_{0}, L_{1})\\
&\otimes CW^{*}(L_{0}, \mathcal{L}_{0}, L'_{l}) \otimes CW^{*}(L'_{l-1}, L'_{l})\\
&\otimes \cdots \otimes CW^{*}(L'_{0}, L'_{1}), CW^{*}(L_{k}, \mathcal{L}_{1}, L'_{0})),
\end{split}
\end{equation} 
such that when evaluated on $[\gamma] \in CW^{*}(\mathcal{L}_{0}, \mathcal{L}_{1})$, the resulting maps form the desired $A_{\infty}$-bimodule homomorphism \eqref{bimodule homomorphism associated to Floer cochain}. \par
	To define the maps \eqref{bimodule homomorphism: first order}, we shall study moduli spaces of inhomogeneous pseudoholomorphic maps from another kind of quilted surface $\underline{S}^{1, k, l}$, which we describe as follows. It has two patches $S^{1, k}_{0}, S^{1, l}_{1}$, where $S^{1, k}_{0}$ is a disk with $(k+3)$ boundary punctures $z_{0}^{+}, z_{0}^{-}, z_{0}^{p}, z_{0}^{1}, \cdots, z_{0}^{k}$, and $S^{1, l}_{1}$ is a disk with $(l+3)$ boundary punctures $z_{1}^{+}, z_{1}^{-}, z_{1}^{p}, z_{1}^{1}, \cdots, z_{1}^{k}$. We denote by $I_{0, +}$ the boundary component of $S^{1, k}_{0}$ between $z_{0}^{+}$ and $z_{0}^{p}$, and by $I_{0, -}$ that between $z_{0}^{p}$ and $z_{0}^{-}$. We use similar notations for $S^{1, l}$. The quilted surface is obtained by seaming the two patches along the two pairs of boundary components $(I_{0, +}, I_{1, +})$ and $(I_{0, -}, I_{1, -})$. \par
	Choose a Floer datum for $\underline{S}^{1, k, l}$ so that we can write down the inhomogeneous Cauchy-Riemann equation for quilted maps $\underline{u}: \underline{S}^{1, k, l} \to (M, N)$:
\begin{equation}
\begin{cases}
(du_{0} - \alpha_{S^{1,k}_{0}} \otimes X_{H_{S^{1,k}_{0}}})^{0, 1} = 0\\
(du_{1} - \alpha_{S^{1,l}_{1}} \otimes X_{H_{S^{1,l}_{1}}})^{0, 1} = 0\\
u_{0}(z) \in \phi_{M}^{\rho_{S^{1, k}_{0}}(z)}L_{i}, \text{ if $z \in \partial S^{1, k}_{0}$ lies between $z_{0}^{i}$ and $z_{0}^{i+1}$}\\
u_{0}(z) \in \phi_{M}^{\rho_{S^{1, k}_{0}}(z)}L_{0}, \text{ if $z \in \partial S^{1, k}_{0}$ lies between $z_{0}^{+}$ and $z_{0}^{1}$ }\\
u_{0}(z) \in \phi_{M}^{\rho_{S^{1, k}_{0}}(z)}L_{k}, \text{ if $z \in \partial S^{1, l}_{0}$ lies between $z_{0}^{k}$ and $z_{0}^{-}$ }\\
u_{1}(z) \in \phi_{N}^{\rho_{S^{1, l}_{1}}(z)}L'_{j}, \text{ if $z \in  \partial S^{1, l}_{1}$ lies between $z_{1}^{j}$ and $z_{1}^{j+1}$}\\
u_{1}(z) \in \phi_{N}^{\rho_{S^{1, l}_{1}}(z)}L'_{0}, \text{ if $z \in \partial S^{1, l}_{1}$ lies between $z_{1}^{+}$ and $z_{1}^{1}$ }\\
u_{1}(z) \in \phi_{N}^{\rho_{S^{1, l}_{1}}(z)}L'_{l}, \text{ if $z \in \partial S^{1, l}_{1}$ lies between $z_{1}^{l}$ and $z_{1}^{-}$ }\\
(u_{0}(z), u_{1}(z)) \in (\phi_{M}^{\rho_{S^{1, k}_{0}}(z)} \times \phi_{N}^{\rho_{S^{1, l}_{1}}(z)})\mathcal{L}_{1}, \text{ if $z \in \partial S^{1, k}_{0}$ lies between $z_{0}^{-}$ and $z_{0}^{p}$}\\
(u_{0}(z), u_{1}(z)) \in (\phi_{M}^{\rho_{S^{1, k}_{0}}(z)} \times \phi_{N}^{\rho_{S^{1, l}_{1}}(z)})\mathcal{L}_{0}, \text{ if $z \in \partial S^{1, k}_{0}$ lies between $z_{0}^{p}$ and $z_{0}^{+}$}\\
\lim\limits_{s \to -\infty} (u_{0} \circ \epsilon_{0}^{-}(s, \cdot), u_{1} \circ \epsilon_{1}^{-}(s, \cdot)) = (x^{-}(\cdot), y^{-}(\cdot))\\
\lim\limits_{s \to +\infty} (u_{0} \circ \epsilon_{0}^{+}(s, \cdot), u_{1} \circ \epsilon_{1}^{+}(s, \cdot)) = (x^{+}(\cdot), y^{+}(\cdot))\\
\lim\limits_{s \to +\infty} (u_{0} \circ \epsilon_{0}^{p}(s, \cdot), u_{1} \circ \epsilon_{1}^{p}(s, \cdot)) = \gamma(\cdot)\\
\lim\limits_{s \to +\infty} u_{0} \circ \epsilon_{0}^{i}(s, \cdot) = x^{i}(\cdot), i = 1, \cdots, k\\
\lim\limits_{s \to +\infty} u_{1} \circ \epsilon_{1}^{j}(s, \cdot) = y^{j}(\cdot), j = 1, \cdots, l
\end{cases}
\end{equation} 
Here $[x^{i}] \in CW^{*}(L_{i-1}, L_{i}), [y^{i}] \in CW^{*}(L'_{j-1}, L'_{j})$ are Hamiltonian chords in $M$ and $N$ respectively, $[\gamma] \in CW^{*}(\mathcal{L}_{0}, \mathcal{L}_{1})$ is a Hamiltonian chord in $M^{-} \times N$ with respect to the split Hamiltonian, and $[(x^{-}, y^{-})] \in CW^{*}(L_{k}, \mathcal{L}_{1}, L'_{0})$, $[(x^{+}, y^{-})] \in CW^{*}(L_{0}, \mathcal{L}_{0}, L'_{l})$ are generalized chords for the corresponding Lagrangian boundary and seaming conditions. We omit suitable rescalings of the asymptotic Hamiltonian chords by the Liouville flow for the purpose of simplifying notation, but shall keep in mind that these can be chosen and have been chosen in a consistent way. The Lagrangian boundary conditions are ordered as follows: on the boundary of the first patch, $L_{k}, \cdots, L_{0}$ are in order from the negative quilted puncture to the positive quilted puncture; on the boundary of the second patch, $L'_{0}, \cdots, L'_{l}$ are in order from the negative quilted puncture to the positive quilted puncture; on the seam, $\mathcal{L}_{d}, \cdots, \mathcal{L}_{0}$ are in order from the negative quilted puncture to the positive quilted puncture. \par
	Let $\mathcal{R}^{1, k, l}((x^{-}, y^{-}); \vec{x}, \gamma, (x^{+}, y^{+}), \vec{y})$ be the moduli space of solutions $(\underline{S}^{1, k, l}, \underline{u})$ to the above equation. Here by $\underline{S}^{1, k, l}$ in the triple we mean a complex structure on $\underline{S}^{1, k, l}$ up to isomorphism. The Gromov bordification $\bar{\mathcal{R}}^{1, k, l}((x^{-}, y^{-}); \vec{x}, \gamma, (x^{+}, y^{+}), \vec{y})$ is in fact a compactification, with its codimension one stratum covered by the following union of fiber products of moduli spaces
\begin{equation} \label{boundary stratum of the moduli space defining the bimodule homomorphism}
\begin{split}
&\coprod_{\substack{1 \le i \le k\\ (\vec{x}', \vec{x}'') = \vec{x}}} \coprod_{x_{new}} \bar{\mathcal{R}}^{1, k-i+1, l}((x^{-}, y^{-}); \vec{x}', x_{new}, \gamma, (x^{+}, y^{+}), \vec{y}) \times \bar{\mathcal{M}}_{i+1}(x_{new}, \vec{x}'')\\
&\cup \coprod_{\substack{1 \le j \le l\\ (\vec{y}', \vec{y}'') = \vec{y}}} \coprod_{y_{new}} \bar{\mathcal{R}}^{1, l, l-j+1}((x^{-}, y^{-}); \vec{x}, \gamma, (x^{+}, y^{+}), \vec{y}') \times \bar{\mathcal{M}}_{j+1}(y_{new}, \vec{y}'')\\
&\cup \coprod_{\gamma_{1}} \bar{\mathcal{R}}^{1, k, l}((x^{-}, y^{-}); \vec{x}, \gamma_{1}, (x^{+}, y^{+}), \vec{y}) \times \bar{\mathcal{M}}(\gamma_{1}, \gamma)\\
&\cup \coprod_{\substack{k'+k''=k, l'+l''=l\\ (\vec{x}', \vec{x}'') = \vec{x}, (\vec{y}', \vec{y}'') = \vec{y}}} \coprod_{(x^{+}_{1}, y^{+}_{1})} \bar{\mathcal{R}}^{1, k', l'}((x^{-}, y^{-}); \vec{x}', \gamma, (x^{+}_{1}, y^{+}_{1}), \vec{y}')\\
&\times \bar{\mathcal{R}}^{0, k'', l''}((x^{+}_{1}, y^{+}_{1}); \vec{x}'', (x^{+}, y^{+}), \vec{y}'')\\
&\cup \coprod_{\substack{k'+k''=k, l'+l''=l\\ (\vec{x}', \vec{x}'') = \vec{x}, (\vec{y}', \vec{y}'') = \vec{y}}} \coprod_{(x^{+}_{1}, y^{+}_{1})} \bar{\mathcal{R}}^{0, k', l'}((x^{-}, y^{-}); \vec{x}', (x^{+}_{1}, y^{+}_{1}), \vec{y}')\\
&\times \bar{\mathcal{R}}^{1, k'', l''}((x^{+}_{1}, y^{+}_{1}); \vec{x}'', \gamma, (x^{+}, y^{+}), \vec{y}'').
\end{split}
\end{equation}\par

	In this situation, the underlying quilted surface is not obtained by gluing patches tangentially, thus the limit of a sequence of inhomogeneous pseudoholomorphic quilted maps does not create a figure-eight bubble (\cite{Wehrheim-Woodward3}). Therefore, the usual Sard-Smale theorem can be used to prove transversality. The upshot is that for generic choices of Floer data compatible with the choices made for puncture disks involved in the definition of wrapped Fukaya categories, these moduli spaces $\bar{\mathcal{R}}^{1, k, l}((x^{-}, y^{-}); \vec{x}, \gamma, (x^{+}, y^{+}), \vec{y})$ are compact smooth manifolds with corners of expected dimension 
\begin{equation}
	k - 1 + l - 1 + \deg((x^{-}, y^{-})) - \deg((x^{+}, y^{+})) - \deg(\alpha) - \sum \deg(x^{i}) - \sum \deg(y^{j}),
\end{equation}
and moreover each stratum is regular. Counting rigid elements in the zero dimensional moduli space $\bar{\mathcal{R}}^{1, k, l}((x^{-}, y^{-}); \vec{x}, \gamma, (x^{+}, y^{+}), \vec{y})$ gives rise to the desired map \eqref{bimodule homomorphism: first order}. \par
	We then extend the construction to higher orders. For this purpose, we consider the quilted surface $\underline{S}^{d, k, l}$ which consists of two patches $S^{d, k}_{0}, S^{d, l}_{1}$, where $S^{d, k}_{0}$ is a disk with $(k+d+2)$ boundary punctures $z_{0}^{+}, z_{0}^{-}, z_{0}^{p_{1}}, \cdots, z_{0}^{p_{d}}, z_{0}^{1}, \cdots, z_{0}^{k}$, and $S^{d, l}_{1}$ is a disk with $(l+d+2)$ boundary punctures $z_{1}^{+}, z_{1}^{-}, z_{1}^{p_{1}}, \cdots, z_{1}^{p_{d}}, z_{1}^{1}, \cdots, z_{1}^{k}$. After seaming these two patches together, the strip-like end near $z_{0}^{p_{i}}$ and the one near $z_{1}^{p_{i}}$ together form a quilted cylindrical end. \par
	Consider the moduli spaces $\mathcal{R}^{d, k, l}((x^{-}, y^{-}); \vec{x}, \gamma^{d}, \cdots, \gamma^{1}, (x^{+}, y^{+}), \vec{y})$ of inhomogeneous pseudoholomorphic quilted maps with appropriate boundary conditions and asymptotic convergence conditions. These are similar to that in Figure \ref{fig: the quilted map defining the bimodule}, but now there are also punctures on the seam which have appropriate asymptotic convergence conditions to generalized chords $\gamma$'s. \par
	By a standard argument using Gromov compactness theorem and the maximum principle, we may prove that the Gromov bordification 
\begin{equation} \label{moduli space for defining higher order bimodule structure maps}
\bar{\mathcal{R}}^{d, k, l}((x^{-}, y^{-}); \vec{x}, \gamma^{d}, \cdots, \gamma^{1}, (x^{+}, y^{+}), \vec{y})
\end{equation}
is compact. Thus it is possible to count rigid elements therein, which gives rise to multilinear maps
\begin{equation}\label{bimodule homomorphism: higher orders}
\begin{split}
n^{k|d|l}&: CW^{*}(\mathcal{L}_{d-1}, \mathcal{L}_{d}) \otimes \cdots \otimes CW^{*}(\mathcal{L}_{0}, \mathcal{L}_{1}) \to\\
&\hom(CW^{*}(L_{k-1}, L_{k}) \otimes \cdots \otimes CW^{*}(L_{0}, L_{1}) \otimes CW^{*}(L_{0}, \mathcal{L}_{0}, L'_{l})\\
&\otimes CW^{*}(L'_{l-1}, L'_{l}) \otimes \cdots CW^{*}(L'_{0}, L'_{1}), CW^{*}(L_{k}, \mathcal{L}_{d}, L'_{0}))
\end{split}
\end{equation} \par

\begin{lemma}
	The multilinear maps $\{n^{k|d|l}\}$ satisfy the $A_{\infty}$-functor equations for the $A_{\infty}$-functor \eqref{bimodule-valued functor}. More concretely, for varying $k, l$ and testing objects $L_{i}$ and $L'_{j}$, the multilinear maps $n^{k|d|l}$ define for each $d$-tuple of composable Floer cochains in $\mathcal{W}(M^{-} \times N)$ a pre-bimodule homomorphism
$\Phi_{\mathcal{L}_{0}} \to \Phi_{\mathcal{L}_{d}}$;
moreover, the assignments of pre-bimodule homomorphisms for $d$-tuples of composable Floer cochains satisfy the $A_{\infty}$-functor equations.
\end{lemma}
\begin{proof}
	To verify that the multilinear maps $n^{k|d|l}$ satisfy the desired $A_{\infty}$-equations, we look at the codimension one boundary strata of the moduli space \eqref{moduli space for defining higher order bimodule structure maps}. It is covered by a union of the following products of moduli spaces:
\begin{equation} \label{boundary stratum of the moduli space defining higher order bimodule homomorphism}
\begin{split}
&\partial \bar{\mathcal{R}}^{d, k, l}((x^{-}, y^{-}); \vec{x}, \gamma^{d}, \cdots, \gamma^{1}, (x^{+}, y^{+}), \vec{y})\\
\cong &\coprod_{\substack{1 \le i \le k\\ (\vec{x}', \vec{x}''', \vec{x}'') = \vec{x}}} \coprod_{x_{new}}
\bar{\mathcal{M}}_{i+1}(x_{new}, \vec{x}''')\\
&\times \bar{\mathcal{R}}^{d, k-i+1, l}((x^{-}, y^{-}); \vec{x}', x_{new}, \vec{x}'' \gamma^{d}, \cdots, \gamma^{1}, (x^{+}, y^{+}), \vec{y})\\
&\cup \coprod_{\substack{1 \le j \le l\\ (\vec{y}', \vec{y}''', \vec{y}'') = \vec{y}}} \coprod_{y_{new}}
\bar{\mathcal{M}}_{j+1}(y_{new}, \vec{y}''')\\
& \times \bar{\mathcal{R}}^{d, k, l-j+1}((x^{-}, y^{-}); \vec{x}, \gamma^{d}, \cdots, \gamma^{1}, (x^{+}, y^{+}), \vec{y}', y_{new}, \vec{y}'')\\
&\cup \coprod_{0 \le d_{1} \le d} \coprod_{\substack{k'+k''=k, l'+l''=l\\ (\vec{x}', \vec{x}'') = \vec{x}, (\vec{y}', \vec{y}'') = \vec{y}}} 
\bar{\mathcal{R}}^{d_{2}, k'', l''}((x^{-}, y^{-}); \vec{x}'', \gamma^{d}, \cdots, \gamma^{d_{1}+1}, (x^{+}_{1}, y^{+}_{1}), \vec{y}'')\\
& \times \bar{\mathcal{R}}^{d_{1}, k', l'}((x^{+}_{1}, y^{+}_{1}); \vec{x}', \gamma^{d_{1}}, \cdots, \gamma^{1}, (x^{+}, y^{+}), \vec{y}')\\
&\cup \coprod_{\substack{d_{1}+d_{2}=d+1\\ 0 \le s \le d_{1}}} \coprod_{\gamma_{new}}
\bar{\mathcal{M}}_{d_{2}+1}(\gamma^{s+d_{2}}, \cdots, \gamma^{s+1}, \gamma_{new})\\
&\times \bar{\mathcal{R}}^{d_{1}, k, l}((x^{-}, y^{-}); \vec{x}, \gamma^{d}, \cdots, \gamma^{s+d_{2}+1}, \gamma_{new}, \gamma^{s}, \cdots, \gamma^{1}, (x^{+}, y^{+})).
\end{split}
\end{equation}

	The above description of the codimension-one boundary strata of 
\begin{equation*}
\bar{\mathcal{R}}^{d, k, l}((x^{-}, y^{-}); \vec{x}, \gamma^{d}, \cdots, \gamma^{1}, (x^{+}, y^{+}), \vec{y}),
\end{equation*}
similar to that in \eqref{boundary stratum of the moduli space defining higher order bimodule homomorphism}, implies the following series of identities which the operations $n^{k|d|l}$ satisfy:
\begin{equation}
\begin{split}
&m^{1} \circ n^{k|d|l}([\vec{x}], [\gamma^{d}], \cdots, [\gamma^{1}], [x^{+}, y^{+}], [\vec{y}])\\
&= \sum n^{k-i+1|d|l}([\vec{x}'], m^{i}([\vec{x}''']), [\vec{x}''], [\gamma^{d}], \cdots, [\gamma^{1}], [x^{+}, y^{+}], [\vec{y}])\\
&+ \sum n^{k|d|l-j+1}([\vec{x}], [\gamma^{d}], \cdots, [\gamma^{1}], [x^{+}, y^{+}], [\vec{y}'], m^{j}([\vec{y}''']), [\vec{y}''])\\
&+ \sum n^{k''|d_{2}|l''}([\vec{x}''], [\gamma^{d}], \cdots, [\gamma_{d_{1}+1}],\\
& n^{k'|d_{1}|l'}([\vec{x}'], [\gamma^{d_{1}}], \cdots, [\gamma^{1}], [x^{+}, y^{+}], [\vec{y}']), [\vec{y}''])\\
&+ \sum n^{k|d_{1}|l}([\vec{x}, [\gamma^{d}], \cdots, [\gamma^{s+d_{2}+1}], \\
&m^{d_{2}}([\gamma^{s+d_{2}}], \cdots, [\gamma^{s+1}]), [\gamma^{s}], \cdots, [\gamma^{1}], [x^{+}, y^{+}])
\end{split}
\end{equation}
Note in particular that the term $n^{k''|d_{2}|l''}(\cdots, n^{k'|d_{1}|l'}(\cdots), \cdots)$ accounts for the second order structure map $m^{2}$ in the dg-category $(\mathcal{W}(M), \mathcal{W}(N))^{bimod}$ of $A_{\infty}$-bimodules over $(\mathcal{W}(M), \mathcal{W}(N))$. Rewriting the above identity in a suitable way, combining the quilted Floer differentials and Floer differentials in both $\mathcal{W}(M)$ and $\mathcal{W}(N)$ into the first order structure map in $(\mathcal{W}(M), \mathcal{W}(N))^{bimod}$, we obtain the desired $A_{\infty}$-functor equations for the assignment $\mathcal{L} \mapsto \Phi_{\mathcal{L}}$, from $\mathcal{W}(M^{-} \times N)$ to the dg-category of $A_{\infty}$-bimodules over $(\mathcal{W}(M), \mathcal{W}(N))$. \par
\end{proof}

	We have thus completed the construction of the $A_{\infty}$-functor \eqref{bimodule-valued functor}. As mentioned before, by evaluation and the Yoneda embedding we obtain the module-valued functor \eqref{functor to modules} for each admissible Lagrangian correspondence $\mathcal{L} \subset M^{-} \times N$. \par

\begin{remark}
	Note that in our construction, a Lagrangian correspondence $\mathcal{L} \subset M^{-} \times N$ gives rise to a bimodule over $(\mathcal{W}(M), \mathcal{W}(N))$, rather than $(\mathcal{W}(M^{-}), \mathcal{W}(N))$. The sign is important and is due to the fact that the quilted inhomogeneous pseudoholomorphic maps are defined with respect to the almost complex structure with the correct sign, forcing the boundary conditions to be ordered in the desired way demanded by the structure of a bimodule over $(\mathcal{W}(M), \mathcal{W}(N))$.
\end{remark}

\subsection{The quilted Floer bimodule for Lagrangian immersions}
	Now we would like to extend the $A_{\infty}$-bimodule $\Phi(\mathcal{L})$ over $(\mathcal{W}(M), \mathcal{W}(N))$ to an $A_{\infty}$-bimodule over $(\mathcal{W}(M), \mathcal{W}_{im}(N))$. The construction can be viewed as a generalization of that in subsection \ref{section: quilted wrapped Floer theory for Lagrangian immersions}. \par

\begin{proposition}\label{extension of the quilted Floer bimodule to Lagrangian immersions}
	The $A_{\infty}$-bimodule $\Phi(\mathcal{L})$ over $(\mathcal{W}(M), \mathcal{W}(N))$ extends to an $A_{\infty}$-bimodule over $(\mathcal{W}(M), \mathcal{W}_{im}(N))$. That is, there is a $A_{\infty}$-bimodule over $(\mathcal{W}(M), \mathcal{W}_{im}(N))$, which composed with the pullback
\begin{equation*}
j^{*}: \mathcal{W}_{im}(N)^{l-mod} \to \mathcal{W}(N)^{l-mod}
\end{equation*}
agrees with $\Phi(\mathcal{L})$, up to homotopy equivalence of bimodules. Here $j: \mathcal{W}(N) \to \mathcal{W}_{im}(N)$ is the quasi-embedding \eqref{quasi-embedding of the ordinary wrapped Fukaya category into the immersed wrapped Fukaya category}.
\end{proposition}

	The construction of this $A_{\infty}$-bimodule structure involves moduli spaces similar to $\mathcal{R}^{k, l}((x^{-}, y^{-}); (x^{+}, y^{+}), \vec{x}, \vec{y})$ as \eqref{moduli space for defining the bimodule structure maps}, where now some of the conditions for the inhomogeneous pseudoholomorphic quilted maps are slightly modified. First, the Lagrangian submanifolds $L'_{j}$ as boundary conditions are replaced by the images of Lagrangian immersions $\iota_{j}: L'_{j} \to N$. Second, we need to include some additional information: the switching labels $\alpha_{j}$ for the Lagrangian immersions $\iota_{j}: L'_{j} \to N$, and the relative homotopy class $\beta$ of the map. Third, the Hamiltonian chords $y_{j}$ from $L'_{j-1}$ to $L'_{j}$ are now replaced by appropriate generators for the wrapped Floer cochain space $CW^{*}((L'_{j-1}, \iota_{j-1}), (L'_{j}, \iota_{j}); H)$, which are either critical points of an auxiliary Morse function on the intersection components, or non-constant time-one Hamiltonian chords contained in the cylindrical end of $N$. Denoting these generators by the same letters $\vec{y} = (y_{1}, \cdots, y_{l})$, we write the corresponding moduli space by
\begin{equation*}
\mathcal{R}^{k, l}(\vec{\alpha}, \beta; (x^{-}, y^{-}); (x^{+}, y^{+}), \vec{x}, \vec{y}),
\end{equation*}
with the additional information included. \par
	There is a natural compactification
\begin{equation*}
\bar{\mathcal{R}}^{k, l}(\vec{\alpha}, \beta; (x^{-}, y^{-}); (x^{+}, y^{+}), \vec{x}, \vec{y})
\end{equation*}
which is obtained from $\mathcal{R}^{k, l}(\vec{\alpha}, \beta; (x^{-}, y^{-}); (x^{+}, y^{+}), \vec{x}, \vec{y})$ by adding stable broken inhomogeneous pseudoholomorphic disks in $M$ and stable broken pearly tree maps in $N$ with boundary on the image of one of the Lagrangian immersions $\iota_{j}: L'_{j} \to N$, as well as broken quilted maps. As an analogue to \eqref{boundary stratum of the moduli space defining the bimodule structure}, the codimension-one boundary strata can thus be described as follows: 
\begin{equation}\label{boundary strata of the moduli space of quilted maps defining the bimodule structure maps for Lagrangian immersions}
\begin{split}
&\coprod_{\substack{1 \le i \le k\\ (\vec{x}', \vec{x}'') = \vec{x}}} \coprod_{x_{new}}
\bar{\mathcal{R}}^{1, k-i+1, l}((x^{-}, y^{-}); \vec{x}', x_{new}, \gamma, (x^{+}, y^{+}), \vec{y}) \times \bar{\mathcal{M}}_{i+1}(x_{new}, \vec{x}'')\\
&\cup \coprod_{\substack{1 \le j \le l\\ (\vec{y}', \vec{y}'') = \vec{y}}} \coprod_{y_{new}} \bar{\mathcal{R}}^{1, l, l-j+1}((x^{-}, y^{-}); \vec{x}, \gamma, (x^{+}, y^{+}), \vec{y}') \times \bar{\mathcal{M}}_{j+1}(y_{new}, \vec{y}'')\\
&\cup \coprod_{\gamma_{1}} \bar{\mathcal{R}}^{1, k, l}((x^{-}, y^{-}); \vec{x}, \gamma_{1}, (x^{+}, y^{+}), \vec{y}) \times \bar{\mathcal{M}}(\gamma_{1}, \gamma)\\
&\cup \coprod_{\substack{k'+k''=k, l'+l''=l\\ (\vec{x}', \vec{x}'') = \vec{x}, (\vec{y}', \vec{y}'') = \vec{y}}} \coprod_{(x^{+}_{1}, y^{+}_{1})} \bar{\mathcal{R}}^{1, k', l'}((x^{-}, y^{-}); \vec{x}', \gamma, (x^{+}_{1}, y^{+}_{1}), \vec{y}')\\
&\times \bar{\mathcal{R}}^{0, k'', l''}((x^{+}_{1}, y^{+}_{1}); \vec{x}'', (x^{+}, y^{+}), \vec{y}'')\\
&\cup \coprod_{\substack{k'+k''=k, l'+l''=l\\ (\vec{x}', \vec{x}'') = \vec{x}, (\vec{y}', \vec{y}'') = \vec{y}}} \coprod_{(x^{+}_{1}, y^{+}_{1})} \bar{\mathcal{R}}^{0, k', l'}((x^{-}, y^{-}); \vec{x}', (x^{+}_{1}, y^{+}_{1}), \vec{y}')\\
&\times \bar{\mathcal{R}}^{1, k'', l''}((x^{+}_{1}, y^{+}_{1}); \vec{x}'', \gamma, (x^{+}, y^{+}), \vec{y}'').
\end{split}
\end{equation}
\par
	Following the same lines as in the construction of Kuranishi structures on the moduli space of inhomogeneous pseudoholomorphic disks that are used to constructe curved $A_{\infty}$-structures for the immersed wrapped Fukaya category introduced in section \ref{the immersed wrapped Fukaya category}, we can also construct Kuranishi structures on these moduli spaces. This compactification is obtained from $\mathcal{R}^{k, l}(\vec{\alpha}, \beta; (x^{-}, y^{-}); (x^{+}, y^{+}), \vec{x}, \vec{y})$ by adding broken inhomogeneous pseudoholomorphic punctured disks in both $M$ and $N$, as well as broken quilted maps. In particular, those broken inhomogeneous pseudoholomorphic disks in $N$ with boundary conditions given by the Lagrangian immersions $\iota_{j}: L'_{j} \to N$ form moduli spaces which carry Kuranishi structures as discussed before. Thus, by the inductive nature of the construction of Kuranishi structures, it remains to build Kuranishi charts on codimension zero strata of various moduli spaces $\mathcal{R}^{k', l'}(\vec{\alpha}', \beta'; (x^{-}, y^{-}); (x^{+}, y^{+}), \vec{x}', \vec{y}')$ for $k' \le k, l' \le l$. That is, we need to build Kuranishi charts over the locus in the moduli spaces whose elements have smooth domains. \par
	Suppose we are given an inhomogeneous pseudoholomorphic quilted map $\sigma = (\underline{S}, (u, v))$ in $\mathcal{R}^{k, l}(\vec{\alpha}, \beta; (x^{-}, y^{-}); (x^{+}, y^{+}), \vec{x}, \vec{y})$, where $\underline{S}$ is the moduli parameter of the underlying quilted surface, and $(u, v)$ is map to the manifold pair $(M, N)$. The linearized operator $D_{\sigma}\bar{\partial}$ of the inhomogeneous Cauchy-Riemann equations for $\sigma$ is Fredholm by non-degeneracy assumption. Thus the Fredholm complex, defined with respect to appropriate Sobolev norm $W^{1, p}$, has finite-dimensional reductions. We choose an obstruction space $E_{\sigma}$, which is a finite-dimensional subspace of $\Omega^{0, 1}(\underline{S}; u^{*}TM \times v^{*}TN)$ such that for each $V \in E_{\sigma}$, the support of $V$ is contained in a closed subset of the domain $\underline{S}$ away from the boundary components and the seam. Then, following the lines in section \ref{section: Kuranishi structure on moduli spaces of stable pearly tree maps}, we can build a Kuranishi chart $(U_{\sigma}, E_{\sigma}, s_{\sigma}, \psi_{\sigma}, \Gamma_{\sigma} = \{1\})$ at this point $\sigma$. \par
	In order to modify these charts for all elements in the moduli space so that they together define a Kuranishi structure, we need to make sure that these charts glue well with Kuranishi charts for inhomogeneous pseudoholomorphic disks in $M$ and $N$. To formulate this, consider the following moduli spaces
\begin{enumerate}[label=(\roman*)]

\item $\mathcal{R}^{k', l'}(\vec{\alpha}', \beta'; (x^{-}, y^{-}); (x^{+}, y^{+}), \vec{x}', \vec{y}')$,

\item $\mathcal{M}_{k''+1}(x_{new}, \vec{x}'')$,

\item $\mathcal{M}_{l''+1}(\alpha'', \beta''; y_{new}, \vec{y}'')$,

\end{enumerate}
where $\vec{x} = (x_{1}, \cdots, x_{k})$, and $\vec{x}' = (x_{1}, \cdots, x_{i}, x_{new}, x_{i+k''+1}, \cdots, x_{k})$, and $\vec{x}'' = (x_{i+1}, \cdots, x_{i+k''})$; similarly for the $y$'s. The union of the product moduli spaces
\begin{equation}
\coprod_{\substack{i\\x_{new}}}
\coprod_{\substack{j\\y_{new}}}
\mathcal{R}^{k', l'}(\vec{\alpha}', \beta'; (x^{-}, y^{-}); (x^{+}, y^{+}), \vec{x}', \vec{y}') \times \mathcal{M}_{k''+1}(x_{new}, \vec{x}'') \times \mathcal{M}_{l''+1}(y_{new}, \vec{y}'')
\end{equation}
is a boundary stratum of the compactification $\bar{\mathcal{R}}^{k, l}(\vec{\alpha}, \beta; (x^{-}, y^{-}); (x^{+}, y^{+}), \vec{x}, \vec{y})$, and under the gluing maps, it can be thickened to a neighborhood of the boundary in the compactification. The gluing happens near the ends for the quilted surface and respectively the punctured disks, over which the quilted map and respectively the inhomogeneous pseudoholomorphic map converge to $x_{new}$; similarly for the ends with convergence condition $y_{new}$. Since the gluing construction is local, the process is the same as gluing inhomogeneous pseudoholomorphic disks along strip-like ends. Since the various obstruction spaces are chosen such that the vectors have compact support away from the boundary of the disks and quilted surfaces, the obstruction spaces also glue well under the gluing map. Thus we may apply the process in section \ref{section: Kuranishi structure on moduli spaces of stable pearly tree maps} to modify the Kuranishi charts so that they form a Kuranishi structure on the moduli space $\bar{\mathcal{R}}^{k, l}(\vec{\alpha}, \beta; (x^{-}, y^{-}); (x^{+}, y^{+}), \vec{x}, \vec{y})$. \par
	Moreover, such construction extends to boundary strata of higher codimension, by an inductive argument. This implies that we can  construct fiber product Kuranishi structures on \eqref{boundary strata of the moduli space of quilted maps defining the bimodule structure maps for Lagrangian immersions}, which are are compatible with the Kuranishi structure on $\bar{\mathcal{R}}^{k, l}(\vec{\alpha}, \beta; (x^{-}, y^{-}); (x^{+}, y^{+}), \vec{x}, \vec{y})$. This proves: \par

\begin{proposition}
	There exists an oriented Kuranishi structure on every such moduli space $\bar{\mathcal{R}}^{k, l}(\vec{\alpha}, \beta; (x^{-}, y^{-}); (x^{+}, y^{+}), \vec{x}, \vec{y})$, which is compatible with the fiber product Kuranishi structures on \eqref{boundary strata of the moduli space of quilted maps defining the bimodule structure maps for Lagrangian immersions}. That is, the restriction of the Kuranishi structure on $\bar{\mathcal{R}}^{k, l}(\vec{\alpha}, \beta; (x^{-}, y^{-}); (x^{+}, y^{+}), \vec{x}, \vec{y})$ to \eqref{boundary strata of the moduli space of quilted maps defining the bimodule structure maps for Lagrangian immersions} agrees with the fiber product Kuranishi structure. 
\end{proposition}

	Choosing single-valued multisections for these Kuranishi structures gives rise to virtual fundamental chains on these moduli spaces, using which the desired $A_{\infty}$-bimodule structure maps are defined. \par

\begin{corollary}
	The $A_{\infty}$-bimodule $\Phi(\mathcal{L})$ over $(\mathcal{W}(M), \mathcal{W}(N))$ extends to an $A_{\infty}$-bimodule over $(\mathcal{W}(M), \mathcal{W}_{im}(N))$. That is, there is an $A_{\infty}$-bimodule over $(\mathcal{W}(M), \mathcal{W}_{im}(N))$, which composed with the pullback
\begin{equation*}
j^{*}: \mathcal{W}_{im}(N)^{l-mod} \to \mathcal{W}(N)^{l-mod}
\end{equation*} agrees with $\Phi(\mathcal{L})$, up to homotopy of bimodules. Here $j: \mathcal{W}(N) \to \mathcal{W}_{im}(N)$ is the quasi-embedding \eqref{quasi-embedding of the ordinary wrapped Fukaya category into the immersed wrapped Fukaya category}.
\end{corollary}
\begin{proof}
	Such extension is presented above. Thus the only thing that we need to prove is that the pullback by $j^{*}$ agrees with $\Phi(\mathcal{L})$ up to homotopy. One way to prove this is that we apply virtual techniques to construction virtual fundamental chains on the moduli space of inhomogeneous pseudoholomorphic quilted maps \eqref{moduli space for defining the bimodule structure maps}, regarding all Lagrangian submanifolds as Lagrangian immersions. If we stick with classical transversality methods for those moduli spaces, the other way is a straightforward analogue of the proof of Proposition \ref{prop: the ordinary wrapped Fukaya category quasi-embeds into the immersed wrapped Fukaya category}. As there is nothing essentially new, we leave the details to the interested reader. \par
\end{proof}

	By a parallel argument, we can also extend the the module-valued functor \eqref{functor to modules} to one with values in category of modules over the immersed wrapped Fukaya category: \par

\begin{proposition}
	There is a canonical extension of the $A_{\infty}$-functor \eqref{functor to modules} to an $A_{\infty}$-functor
\begin{equation} \label{functor to modules over the immersed category}
\Phi_{\mathcal{L}}: \mathcal{W}(M) \to \mathcal{W}_{im}(N)^{l-mod}
\end{equation}
to the $A_{\infty}$-category of left $A_{\infty}$-modules over $\mathcal{W}_{im}(N)$. That is, the composition by the pullback
\begin{equation*}
j^{*}: \mathcal{W}_{im}(N)^{l-mod} \to \mathcal{W}(N)^{l-mod}
\end{equation*}
agrees with \eqref{functor to modules}, up to homotopy.
\end{proposition}

	Of course, this also follows from the process of converting a bimodule to a module-valued functor. \par

\subsection{Geometric composition of Lagrangian correspondences}\label{geometric composition admissible for wrapped Floer theory}
	In order to obtain a functor that takes value in the actual immersed wrapped Fukaya category $\mathcal{W}_{im}(N)$ instead of the category of modules over it, we need to prove that the $A_{\infty}$-functor \eqref{functor to modules over the immersed category} is representable, in the sense of \cite{Fukaya1}. It has been long noted that the geometric compositions of Lagrangian correspondences are natural candidates for the objects representing the modules defined above. A good reference for the definition and basic properties of geometric compositions is \cite{Wehrheim-Woodward4}, in the case where the Hamiltonian perturbation is not present. \par 
	There are two issues in proving representability of \eqref{functor to modules over the immersed category}. First, representability does not always hold if we only consider embedded Lagrangian submanifolds, as geometric compositions of Lagrangian correspondences are generally Lagrangian immersions. This is why we must go to the immersed wrapped Fukaya category and consider the $A_{\infty}$-functor \eqref{functor to modules over the immersed category} instead of \eqref{functor to modules}. Second, the geometric composition of Lagrangian correspondences might not have good geometric properties, thus not a priori admissible for wrapped Floer theory on the nose. For this, we must imposed further conditions so that they are admissible in the immersed wrapped Fukaya category. \par
	Let us first recall the definition of geometric composition of Lagrangian correspondences. Given an admissible Lagrangian submanifold $L \subset M$ and an admissible Lagrangian correspondence $\mathcal{L} \subset M^{-} \times N$, in a generic situation the fiber product over the graph $\Gamma(\psi_{H_{M}}) \subset M^{-} \times M$ over the Hamiltonian symplectomorphism $\phi_{H_{M}}$
\begin{equation*}
L' = L \times_{\Gamma(\phi_{H_{M}})} \mathcal{L}
\end{equation*}
is a smooth submanifold of $M \times M \times N$. The composition of the embedding $L' \subset M \times M \times N$ with the projection $M \times M \times N \to N$ is a Lagrangian immersion:
\begin{equation*}
\iota: L \times_{\Gamma(\phi_{H_{M}})} \mathcal{L} \to N.
\end{equation*}
We call this Lagrangian immersion the geometric composition of $L$ with $\mathcal{L}$, under the perturbation by the Hamiltonian flow of $H_{M}$. By abuse of name, we sometimes also call the image of $\iota$ the geometric composition for simplicity, and often denote it by $L \circ_{H_{M}} \mathcal{L}$. \par
	If the geometric composition happens to be (properly) embedded, it comes with a natural choice of a primitive, which makes it an exact Lagrangian submanifold of $N$. Because of the presence of the Hamiltonian perturbation, this primitive is slightly different from the naive sum, but instead takes the following form:
\begin{equation}
g = f + F \circ (\phi_{H_{M}} \times id_{N}) + i_{X_{H_{M}}}\lambda_{M}
\end{equation}
where $X_{H_{M}}$ is the Hamiltonian vector field of $H_{M}$, thought of as a Hamiltonian pulled back to $M^{-} \times N$. This formula is calculated in \cite{Gao1}, which follows directly from the formula for the change of the primitive for an exact Lagrangian submanifold under a Hamiltonian isotopy, Lemma \ref{changing primitive under Hamiltonian isotopy}. \par
	Depending on the geometry of $L$ and $\mathcal{L}$, the geometric composition might be or not be cylindrical, even if it is embedded. Therefore a proof of its admissibility in wrapped Floer theory is completely necessary. This is done in \cite{Gao1} in the case where the geometric composition is properly embedded, with the above choice of the primitive. Technically, there we only proved that the wrapped Floer differential converges (and is in fact finite), but the same argument can be utilized to prove that higher order structure maps are also finite. \par
	In general, the geometric composition $\iota: L \circ_{H_{M}} \mathcal{L} \to N$ is not an embedding, but we still expect it to have some favorable properties. This primitive $g$ still makes $\iota$ an "exact" Lagrangian immersion, in the sense that $\iota^{*}\lambda_{N} = dg$. \par
	There is also the notion of geometric composition in the usual sense. Instead of taking the fiber product over the graph $\Gamma(\phi_{H_{M}})$, we take the fiber product over the diagonal $\Delta_{M}$. The geometric composition $L \circ \mathcal{L}$ is the map
\begin{equation*}
\iota: L \times_{\Delta_{M}} \mathcal{L} \to N.
\end{equation*}
Generically this is a Lagrangian immersion, and is also exact in the generalized sense, where the primitive is
\begin{equation*}
h = f + F.
\end{equation*} \par

	In \cite{Gao1}, we proved well-definedness of wrapped Floer cohomology in the case where the geometric composition is properly embedded. The argument can be generalized to well-definedness of $A_{\infty}$-structure maps of all orders. However, in this paper we shall take a slightly different point of view, in order to make our construction more functorial and canonical. In \cite{Gao1}, we considered the geometric composition under the large perturbation $L \circ_{H_{M}} \mathcal{L}$, and proved an isomorphism of Floer cohomology groups. While that definition is natural as one can find a natural one-to-one correspondence between the generators, we find it better to work with the geometric composition in the usual sense, when the whole categorical structure is in concern. Of course, the left-module structures associated to the geometric composition under large perturbation and the geometric composition in the usual sense are homotopy equivalent, so the essential difference is minor. \par
	
	For the geometric composition of Lagrangian correspondences $L \circ \mathcal{L}$ to have well-defined wrapped Floer theory in general, we must have good control of the behavior of inhomogeneous pseudoholomorphic disks(modeled as stable pearly tree maps) bounded by the image of the geometric composition. For this purpose, we need to make sure that the geometry of $L \circ \mathcal{L}$ at infinity does not behave too badly. Thus it is natural to introduce the following assumption. \par

\begin{assumption}\label{assumption on the geometric composition}
	For the Lagrangian submanifold $L$ in consideration, the geometric composition $L \circ \mathcal{L}$ is a proper Lagrangian immersion with transverse or clean self-intersections, which is cylindrical in the generalized sense for a Lagrangian immersion.
\end{assumption}

	When defining the wrapped Fukaya category $\mathcal{W}(M)$, we have to specify a class of Lagrangian submanifolds as objects. Since Assumption \ref{assumption on the geometric composition} is generic, it is possible for us to choose a countable collection of Lagrangian submanifolds as objects of the wrapped Fukaya category $\mathcal{W}(M)$, such that for every $L$ in this collection, Assumption \ref{assumption on the geometric composition} is satisfied. Then it follows almost from the definition that: \par

\begin{proposition}\label{prop: geometric composition is admissible}
	Under Assumption \ref{assumption on the geometric composition}, there is a well-defined curved $A_{\infty}$-algebra for the geometric composition $\iota: L \times_{\Delta_{M}} \mathcal{L} \to N$, in the sense of immersed wrapped Floer theory discussed in sections \ref{the immersed wrapped Fukaya category} and \ref{section: immersed wrapped Floer theory in the case of clean intersections}, for every $L$ from the collection of objects of $\mathcal{W}(M)$.
\end{proposition}

	Compared to the cohomological result in \cite{Gao1}, which works with the geometric composition under large perturbation, the assumption of Proposition \ref{prop: geometric composition is admissible} is in fact simpler as we assume the geometric composition $L \circ \mathcal{L}$ to be cylindrical. As discussed before, it is automatically exact, so the results of sections \ref{the immersed wrapped Fukaya category} and \ref{section: immersed wrapped Floer theory in the case of clean intersections}. \par

\subsection{Unobstructedness of the geometric composition}\label{section: unobstructedness of the geometric composition}
	We have shown that if $\mathcal{L} \to N$ is proper and if Assumption \ref{assumption on the geometric composition} is satisfied, there is a curved $A_{\infty}$-algebra structure on the wrapped Floer cochain space $CW^{*}(L \circ \mathcal{L}; H_{N})$ for the geometric composition $L \circ \mathcal{L}$.
To make it into an object of the immersed wrapped Fukaya category, we must also prove that this curved $A_{\infty}$-algebra is unobstructed.
The main result of this subsection says that the geometric composition $L \circ \mathcal{L}$ is unobstructed with a canonical and unique choice of bounding cochain $b$ satisfying a distinguished property. \par

\begin{theorem} \label{unobstructedness of geometric composition}
	Suppose $L$ is a properly embedded exact cylindrical Lagrangian submanifold of $M$, and $\mathcal{L} \subset M^{-} \times N$ is a properly embedded exact cylindrical Lagrangian correspondence between $M$ and $N$, such that the projection $\mathcal{L} \to N$ is proper. Let $L \circ \mathcal{L}$ be their geometric composition. 
Then $L \circ \mathcal{L}$ is unobstructed in the sense of wrapped Floer theory, with a canonical and unique choice of bounding cochain $b$ determined by $L$ and $\mathcal{L}$, with the property that $b$ gives rise to non-curved deformations for both the quilted Floer module $CW^{*}(L, \mathcal{L}, L \circ \mathcal{L})$ and the curved $A_{\infty}$-algebra $CW^{*}(L \circ \mathcal{L})$. \par
	In particular, $(L \circ \mathcal{L}, b)$ becomes an object in the immersed wrapped Fukaya category $\mathcal{W}_{im}(N)$.
\end{theorem}

	The idea to prove Theorem \ref{unobstructedness of geometric composition} is to use Lemma \ref{cyclic element and bounding cochain} to provide an algebraic argument for the existence and uniqueness of such a bounding cochain. For that purpose, we shall first equip the quilted wrapped Floer cochain space $CW^{*}(L, \mathcal{L}, L \circ \mathcal{L})$ with a curved $A_{\infty}$-module structure. \par

\begin{lemma}
	There is a natural curved left $A_{\infty}$-module structure on $CW^{*}(L, \mathcal{L}, L \circ \mathcal{L})$ over the curved $A_{\infty}$-algebra $(CW^{*}(L \circ \mathcal{L}), m^{k})$.
\end{lemma}

	The construction involves moduli spaces of inhomogeneous pseudoholomorphic quilted maps of the following kind. \par
	Consider the quilted surface $\underline{S}^{mod}$ consisting of two patches, $S_{0}^{mod}, S_{1}^{mod}$, each of which is a disk with one negative puncture $z_{i}^{-}$ and one positive puncture $z_{i}^{+}$. Given generalized chords $(x^{0}, y^{0}), (x^{1}, y^{1})$ for $(L, \mathcal{L}, L \circ \mathcal{L})$, consider quilted maps $\underline{u}: \underline{S}^{mod} \to (M, N)$ satisfying the following conditions
\begin{equation}
\begin{cases}
\frac{\partial u_{0}}{\partial s} + J_{M} \circ (\frac{\partial u_{0}}{\partial t} - X_{H_{M}}(u)) = 0\\
\frac{\partial u_{0}}{\partial s} + J_{M} \circ (\frac{\partial u_{0}}{\partial t} - X_{H_{M}}(u)) = 0\\
u_{0}(s, 0) \in L\\
u_{1}(s, 1) \in L \circ \mathcal{L}\\
(u_{0}(s, 1), u_{1}(s, 0)) \in \mathcal{L}\\
\lim_{s \to -\infty}u_{0}(s, \cdot) = x^{0}(\cdot)\\
\lim_{s \to +\infty}u_{0}(s, \cdot) = x^{1}(\cdot)\\
\lim_{s \to -\infty}u_{1}(s, \cdot) = y^{0}(\cdot)\\
\lim_{s \to +\infty}u_{1}(s, \cdot) = y^{1}(\cdot)\\
\end{cases}
\end{equation}
We also need to specify some additional data such as the switching conditions, the lifting conditions and the homology classes that are introduced in section \ref{section: moduli space of disks bounded by immersed Lagrangian submanifolds}. Quotienting by translations, we obtain moduli spaces of solutions to the above equation with these additional requirements. \par
	To define a curved $A_{\infty}$-module structure, we also need to study inhomogeneous pseudoholomorphic quilted maps of the same kind, but with more punctures. Break the boundary of $S_{1}^{mod}$ that gets mapped to $L \circ \mathcal{L}$ into several pieces, namely replace $S_{1}^{mod}$ by $S_{1}^{mod, k}$, a disk with $(k+2)$-boundary punctures, $z_{1}^{-}, z_{1}^{1}, \cdots, z_{1}^{k}, z_{1}^{+}$. We denote this new quilted surface by $\underline{S}^{mod, k}$. \par
	Since the map $u_{1}$ on the second patch of the quilted surface has boundary condition being an immersed Lagrangian submanifold, it is necessary to include additional data $\alpha, \beta, l$, etc. that indicate switching conditions of the boundary lifting, homology classes of the maps and so on. These data are analogous to those discussed in sections \ref{the immersed wrapped Fukaya category} and \ref{section: immersed wrapped Floer theory in the case of clean intersections}, so we omit the details here. \par
	There is a natural compactification of the moduli space of these quilted maps, which we denote by
\begin{equation*}
\bar{\mathcal{Q}}_{k}(\alpha, \beta; (x^{-}, y^{-}); y^{1}, \cdots, y^{k}, (x^{+}, y^{+})).
\end{equation*}
A typical element in the compactified moduli space is a broken inhomogeneous pseudoholomorphic quilted map $\{(u_{i}, v_{i})\}$ with trees of inhomogeneous pseudoholomorphic disks in $N$ attached to the lower boundary components of the second patches $v_{i}$'s of the broken quilted map.
In general, we have the following description of the codimension-one boundary strata of the above moduli space:
\begin{equation}\label{boundary strata of the moduli space defining the left-module structure on the quilted Floer complex}
\begin{split}
&\partial \bar{\mathcal{Q}}_{k}(\alpha, \beta; (x^{-}, y^{-}); y^{1}, \cdots, y^{k}, (x^{+}, y^{+}))\\
\cong &\coprod_{0 \le i \le k} \coprod_{\substack{\alpha_{1} \cup \alpha_{2} = \alpha\\\beta_{1} \sharp \beta_{2} = \beta}} \coprod_{(x^{+}_{1}, y^{+}_{1})}
\bar{\mathcal{Q}}_{i}(\alpha_{1}, \beta_{1}; (x^{-}, y^{-}); y^{1}, \cdots, y^{i}, (x^{+}_{1}, y^{+}_{1}))\\
&\times \bar{\mathcal{Q}}_{k-i}(\alpha_{2}, \beta_{2}; (x^{+}_{1}, y^{+}_{1}); y^{i+1}, \cdots, y^{k}, (x^{+}, y^{+}))\\
\cup &\coprod_{\substack{k_{1} + k_{2} = k+1\\ 1 \le i \le k_{1} }} \coprod_{\substack{\alpha_{1} \cup \alpha_{2} = \alpha\\ \beta_{1} \sharp \beta_{2} = \beta}} \coprod_{y_{new}}
\bar{\mathcal{Q}}_{k_{1}}(\alpha_{1}, \beta_{1}; (x^{-}, y^{-}); y^{1}, \cdots, y^{k_{1}},\\
&y_{new}, y^{i+k_{2}+1}, \cdots, y^{k}, (x^{+}, y^{+}))\\
&\times \bar{\mathcal{M}}_{i+1}(\alpha_{2}, \beta_{2}; y_{new}, y^{k_{1}+1}, \cdots, y^{i+k_{2}}).
\end{split}
\end{equation} 
Since $L$ and $\mathcal{L}$ are exact, there are no pseudoholomorphic disks bubbling off $L$ or $\mathcal{L}$. Thus the above fiber products \eqref{boundary strata of the moduli space defining the left-module structure on the quilted Floer complex} cover all the boundary strata of the compactification. \par

	As usual, we can construct Kuranishi structures on these moduli spaces compatibly and use single-valued multisections to define virtual fundamental chains. \par

\begin{proposition}
	There exists an oriented Kuranishi structure on the moduli space \begin{equation*}
\bar{\mathcal{Q}}_{k}(\alpha, \beta; (x^{-}, y^{-}); y^{1}, \cdots, y^{k}, (x^{+}, y^{+})),
\end{equation*}
such that the induced Kuranishi structures on the boundary strata are isomorphic to the fiber product Kuranishi structures on \eqref{boundary strata of the moduli space defining the left-module structure on the quilted Floer complex}. \par
	In addition, we may choose single-valued multisections on these moduli spaces, such that the multisection on $\bar{\mathcal{Q}}_{k}(\alpha, \beta; (x^{-}, y^{-}); y^{1}, \cdots, y^{k}, (x^{+}, y^{+}))$ is compatible with the fiber product multisections at the boundary strata \eqref{boundary strata of the moduli space defining the left-module structure on the quilted Floer complex}.
\end{proposition}

	The virtual fundamental chains on these moduli spaces associated to a coherent choice of multisections then define a curved $A_{\infty}$-module structure on the quilted wrapped Floer cochain space $CW^{*}(L, \mathcal{L}, L \circ \mathcal{L})$ over the curved $A_{\infty}$-algebra $CW^{*}(L \circ \mathcal{L})$. \par

	To apply the result in section \ref{section: cyclic element and bounding cochain}, we also need to find natural filtrations for the curved $A_{\infty}$-algebra $CW^{*}(L \circ \mathcal{L})$ and the curved $A_{\infty}$-module $CW^{*}(L, \mathcal{L}, L \circ \mathcal{L})$ over it. These filtrations are given by the symplectic action functional. \par

\begin{lemma}
	The action filtration on $CW^{*}(L \circ \mathcal{L})$ defines a discrete filtration for the curved $A_{\infty}$-algebra $(CW^{*}(L \circ \mathcal{L}), m^{k})$. \par
	The action filtration on $CW^{*}(L, \mathcal{L}, L \circ \mathcal{L})$ defines a discrete filtration for the curved $A_{\infty}$-module $(CW^{*}(L, \mathcal{L}, L \circ \mathcal{L}), n^{k})$ is compatible with the action filtration for $(CW^{*}(L \circ \mathcal{L}), m^{k})$. \par
	Moreover, these filtrations are bounded above. \par
\end{lemma}
\begin{proof}
	The proof of the fact that action filtrations are compatible with the curved $A_{\infty}$-algebra structure and the curved $A_{\infty}$-module structure follows immediately from the action-energy relation. \par
	To prove that the action filtration on  $CW^{*}(L \circ \mathcal{L})$ is discrete, we recall that the generators of a wrapped Floer cochain space consist of two kinds: first, critical points of auxiliary Morse functions on components of the self fiber product of the preimage of the immersion; second, non-constant Hamiltonian chords in the cylindrical ends together with lifting indices. There are finitely many critical points, and we can arrange the primitive and choose the auxiliary Morse functions carefully so that their actions are different. On the other hand, ecause the Hamiltonian is non-degenerate in the cylindrical end, these non-constant Hamiltonian chords are non-degenerate and have a discrete action spectrum. A similar argument applies to show that the action filtration on $CW^{*}(L, \mathcal{L}, L \circ \mathcal{L})$ is discrete. \par
	Compatibility follows from the action-energy relation applied to inhomogeneous pseudoholomorphic quilted maps in the moduli spaces $\bar{\mathcal{Q}}_{k}(\alpha, \beta; (x^{-}, y^{-}); y^{1}, \cdots, y^{k}, (x^{+}, y^{+}))$, which are used to define this curved $A_{\infty}$-module structure. \par
	The fact that these filtrations are bounded above follows from the definition of the wrapped Floer cochain space: there are only finitely many free generators which have positive action, as those infinitely many generators, the non-constant Hamiltonian chords in the cylindrical end, all have negative action. \par
\end{proof}

	To finish the proof of Theorem \ref{unobstructedness of geometric composition}, we need to find a cyclic element for the curved $A_{\infty}$-module $CW^{*}(L, \mathcal{L}, L \circ \mathcal{L})$ over the curved $A_{\infty}$-algebra $CW^{*}(L \circ \mathcal{L})$. Recall from sections \ref{section: wrapped Floer cochain space for a cylindrical Lagrangian immersion} and \ref{section: wrapped Floer cochain space in case of clean self-intersections} that in the wrapped Floer cochain space $CW^{*}(L \circ \mathcal{L})$, we have a distinguished generator - the minimum of chosen Morse function $f$ on the diagonal component $\Delta_{L \times_{\Delta_{M}} \mathcal{L}}$ of the self fiber product $(L \times_{\Delta_{M}} \mathcal{L}) \times_{\iota} (L \times_{\Delta_{M}} \mathcal{L})$.
This corresponds to the fundamental chain of the manifold $L \times_{\Delta_{M}} \mathcal{L}$ in the singular chain model. This generator is the homotopy unit for the curved $A_{\infty}$-algebra $(CW^{*}(L \circ \mathcal{L}), m^{k})$. \par
	As graded $\mathbb{Z}$-modules, we have that $CW^{*}(L, \mathcal{L}, L \circ \mathcal{L}) \cong CW^{*}(L \circ \mathcal{L})$, given by a natural bijective correspondence between the sets of generators. This follows directly from the definition of the geometric composition. Under this correspondence, we get a distinguished element $e_{L \circ \mathcal{L}} \in CW^{*}(L, \mathcal{L}, L \circ \mathcal{L})$, corresponding to the homotopy unit of $CW^{*}(L \circ \mathcal{L})$. \par

\begin{lemma}
	The element $e_{L \circ \mathcal{L}} \in CW^{*}(L, \mathcal{L}, L \circ \mathcal{L})$ defined above is a cyclic element.
\end{lemma}
\begin{proof}
	Recall that a cylic element has to satisfy two conditions. First, the map
\begin{equation*}
CW^{*}(L \circ \mathcal{L}) \to CW^{*}(L, \mathcal{L}, L \circ \mathcal{L})
\end{equation*}
defined by
\begin{equation*}
x \mapsto n^{1}(x; e_{L \circ \mathcal{L}})
\end{equation*}
is an isomorphism of $\mathbb{Z}$-modules.
Second, $e_{L \circ \mathcal{L}}$ lies in $F^{0}$, and applying $n^{0}$ to $e_{L \circ \mathcal{L}}$ should strictly increase the action filtration. \par
	The first condition follows from the fact that $e_{L \circ \mathcal{L}}$ corresponds to the homotopy unit of $CW^{*}(L \circ \mathcal{L})$. Multiplication with the homotopy unit yields a self map on $CW^{*}(L \circ \mathcal{L})$, which can be written as an upper-triangular matrix with respect to a basis for $CW^{*}(L \circ \mathcal{L})$ ordered in increasing action. Moreover, the diagonal entries of this upper-triangular matrix are all equal to the identity. Now we consider the basis for $CW^{*}(L, \mathcal{L}, L \circ \mathcal{L})$ which corresponds to the chosen basis for $CW^{*}(L \circ \mathcal{L})$ under the natural one-to-one correspondence between generators: each generalized chord for $(L, \mathcal{L}, L \circ \mathcal{L})$ corresponds to a unique Hamiltonian chord from $L \circ \mathcal{L}$ to itself (the same applies to critical points). This basis is also ordered in increasing action, so that the map $x \mapsto n^{1}(x; e_{L \circ \mathcal{L}})$ can be written as an upper-triangular matrix whose diagonal entries are all equal to the "identity", where this "identity" means the natural one-to-one correspondence between generators. \par
	
	Now let us check the second condition. First, the element $e_{L \circ \mathcal{L}}$ itself is a free generator of the quilted wrapped Floer cochain space $CW^{*}(L, \mathcal{L}, L \circ \mathcal{L})$, as it corresponds to the homotopy unit of $CW^{*}(L \circ \mathcal{L})$ under the natural one-to-one correspondence between generators. Moreover, we can choose the primitive carefully such that $e_{L \circ \mathcal{L}}$ has zero action. To prove that $n^{0}$ applied $e_{L \circ \mathcal{L}}$ strictly increases the action, it suffices to prove that there are no constant inhomogeneous pseudoholomorphic quilted strips with input being $e_{L \circ \mathcal{L}}$. Such a constant quilted strip, if existed, would correspond to a constant inhomogeneous pseudoholomorphic strip with boundary on the image of $L \circ \mathcal{L}$ with input being the homotopy unit of $CW^{*}(L \circ \mathcal{L})$. But there are no such constant strips, because the homotopy unit is the minimum of the chosen Morse function on the diagonal component $\Delta_{L \times_{\Delta_{M}} \mathcal{L}}$ of the self fiber product. As a consequence, $n^{0}(e_{L \circ \mathcal{L}})$ can be written as a linear combination of some generators of $CW^{*}(L, \mathcal{L}, L \circ \mathcal{L})$, none of which is any critical point on the diagonal component of the self fiber product, or any non-constant Hamiltonian chord in the cylindrical end. Therefore $n^{0}(e_{L \circ \mathcal{L}})$ can be written as a linear combination of generators, which correspond to critical points on the switching components of the self fiber product of the Lagrangian immersion $L \circ \mathcal{L}$ under the natural one-to-one correspondence between the generators of $CW^{*}(L \circ \mathcal{L})$ and those of $CW^{*}(L, \mathcal{L}, L \circ \mathcal{L})$.
In particular, applying $n^{0}$ to $e_{L \circ \mathcal{L}}$ must strictly increase the action, because it is defined by counting non-constant inhomogeneous pseudoholomorphic quilted maps, which have strictly positive energy. \par

\end{proof}

	By a purely algebraic argument using Lemma \ref{cyclic element and bounding cochain}, the existence of a cyclic element implies the unobstructedness of the geometric composition, with a unique bounding cochain $b$ satisfying the following property: \par

\begin{corollary}
	There exists a unique (nilpotent) bounding chain $b \in CW^{*}(L \circ \mathcal{L})$ such that $b \in F^{\epsilon}$ for some $\epsilon > 0$, and
\begin{equation*}
d^{b}(e_{L \circ \mathcal{L}}) = 0,
\end{equation*}
where $d^{b}(\cdot) = \sum_{k=0}^{\infty} n^{k}(b, \cdots, b; \cdot)$.
\end{corollary}

	Thus the proof of Theorem \ref{unobstructedness of geometric composition} is complete. \par

\begin{remark}
	Note that $CW^{*}(L \circ \mathcal{L})$ with respect to the undeformed structure maps is generally not a curved $A_{\infty}$-module over itself, because of the non-vanishing of the curvature term $m^{0}$. However, the $b$-deformed structure maps $m^{k; b}$ make $CW^{*}(L \circ \mathcal{L}, b)$ an $A_{\infty}$-module over itself, as the $b$-deformed $A_{\infty}$-algebra is non-curved. \par
	Nonetheless, the curved $A_{\infty}$-module structure on the quilted wrapped Floer cochain space $CW^{*}(L, \mathcal{L}, L \circ \mathcal{L})$ is essentially different from that on $CW^{*}(L \circ \mathcal{L}, b)$ as a $A_{\infty}$-module over itself, although the underlying $\mathbb{Z}$-modules are isomorphic. The quilted Floer-theoretic setup is essential for this curved $A_{\infty}$-module structure to exist. \par
\end{remark}

	Concerning the wrapped Fukaya category, we have the following vanishing result of the bounding cochain $b$ for the geometric composition, in the case where it is in fact a proper exact Lagrangian embedding. \par

\begin{proposition}\label{vanishing of the bounding cochain}
	If the geometric composition $\iota: L \circ \mathcal{L} \to N$ is a proper exact cylindrical Lagrangian embedding, whose primitive (coming from the primitive for $L$ and that for $\mathcal{L}$) extends to a function on $N$ which is locally constant in the cylindrical end of $N$, then the bounding cochain $b$ from Theorem \ref{unobstructedness of geometric composition} vanishes.
\end{proposition}
\begin{proof}[Sketch of proof]
	Recall that the bounding cochain $b$ is the unique solution to the equation
\begin{equation}
n^{0; b}(e_{L \circ \mathcal{L}}) = \sum_{k} n^{k}(b, \cdots, b; e_{L \circ \mathcal{L}}) = 0,
\end{equation}
with the property that $b \in F^{>0}$.
This implies that for this choice of bounding cochain $b$, the map
\begin{equation}
gc^{1}: CW^{*}(L, \mathcal{L}, (L \circ \mathcal{L}, b)) \to CW^{*}(L \circ \mathcal{L}, b)
\end{equation}
is a cochain map with respect to the deformed differentials on both sides. 
Recall that the deformed differential on the quilted wrapped Floer cochain space is
\begin{equation*}
n^{0; b}(x) = \sum_{k \ge 0} n^{k}(\underbrace{b, \cdots, b}_{k \text{ times}}; x),
\end{equation*}
while the deformed differential on the wrapped Floer cochain space $CW^{*}(L \circ \mathcal{L}, b)$ is
\begin{equation*}
m^{1; b}(x) = \sum_{k_{0}, k_{1} \ge 0} m^{k}(\underbrace{b, \cdots, b}_{k_{0} \text{ times}}, x, \underbrace{b, \cdots, b}_{k_{1} \text{ times}}).
\end{equation*}
Because of the assumption that $\iota$ is a proper embedding, $0$ is a bounding cochain for $L \circ \mathcal{L}$.
We want to prove that, if we choose $0$ as the bounding cochain for $L \circ \mathcal{L}$, the cyclic element is closed under the undeformed differential on $CW^{*}(L, \mathcal{L}, L \circ \mathcal{L})$. The proof will be separated in the following three lemmas. \par

\end{proof}

	The first lemma is a general statement about the bounding cochain $b$, where we do not assume the geometric composition $L \circ \mathcal{L}$ is an embedding. \par

\begin{lemma}
	The bounding cochain $b$ from Theorem \ref{unobstructedness of geometric composition} is supported only at the critical points, or non-constant Hamiltonian chords in the cylindrical end which have positive action. \par
\end{lemma}
\begin{proof}
	The statement follows immediately from the condition that $b \in F^{>0}$. \par
\end{proof}

	The second lemma lists some equivalent conditions for $0$ to be the desired bounding cochain under the assumption that the geometric composition is an embedding. \par

\begin{lemma}\label{equivalent conditions for closedness of the cyclic element}
	Suppose that the geometric composition $\iota: L \circ \mathcal{L} \to N$ is a proper exact cylindrical Lagrangian embedding. Then the following three conditions are equivalent:
\begin{enumerate}[label=(\roman*)]

\item The cyclic element $e_{L \circ \mathcal{L}}$ is closed under the undeformed differential on $CW^{*}(L, \mathcal{L}, L \circ \mathcal{L})$.

\item There does not exist an inhomogeneous pseudoholomorphic quilted strip with boundary condition $(L, \mathcal{L}, L \circ \mathcal{L})$, which converges to $e_{L \circ \mathcal{L}}$ over the positive quilted end (as input).

\item There is no figure eight bubble which asymptotically converges to a generalized chord for $(L, \mathcal{L}, \mathcal{L}, L)$ over the negative quilted end.

\end{enumerate}
\end{lemma}
\begin{proof}
	This lemma is proved in \cite{Gao1}. Let us briefly recall it here. \par
	(ii) implies (i) by the definition of the quilted Floer module structure map on $CW^{*}(L, \mathcal{L}, L \circ \mathcal{L})$. \par
	(iii) implies (ii) by a strip-shrinking argument. Suppose that there is no figure eight bubble as in (iii). If there is an inhomogeneous pseudoholomorphic quilted strip converging to $e_{L \circ \mathcal{L}}$ over the positive quilted end, then by shrinking it we get an inhomogeneous pseudoholomorphic strip in $N$ with boundary on $L \circ \mathcal{L}$, which asymptotically converges to the unit of $CW^{*}(L \circ \mathcal{L})$ over the positive end. This is certainly not possible. \par
	Thus it remains to prove that (i) implies (iii). Suppose on the contrary that $e_{L \circ \mathcal{L}}$ is not closed under the undeformed quilted Floer differential. Then there exists a non-trivial inhomogeneous pseudoholomorphic strip which asymptotically converges to $e_{L \circ \mathcal{L}}$.
By shrinking it we get two possible kinds of configurations: first, an inhomogeneous pseudoholomorphic strip in $N$ with boundary on $L \circ \mathcal{L}$, which asymptotically converges to the unit of $CW^{*}(L \circ \mathcal{L})$;
second, a figure eight bubble asymptotic to a generalized chord for $(L, \mathcal{L}, \mathcal{L}, L)$. The first kind does not exist as previously argued, and the second kind does not exist by the assumption of (iii). \par
\end{proof}

\begin{remark}
	Note that we have used a strip-shrinking argument, but not for the purpose of establishing an isomorphism between the moduli space of inhomogeneous pseudoholomorphic quilted strips and the moduli space of inhomogeneous pseudoholomophic strips. Rather, we use it to prove the non-existence of certain inhomogeneous pseudoholomorphic quilted strips by the given condition on the primitive for the geometric composition $L \circ \mathcal{L}$. 
That is, the moduli space is empty because of a geometric constraint; as a consequence, it is isomorphic to the empty set as Kuranishi spaces.
\end{remark}

	Thus, to prove that the cyclic element $e_{L \circ \mathcal{L}}$ is closed under the undeformed quilted wrapped Floer differential, it suffices to prove condition (iii). That uses the third lemma stated as below, which adds the assumption that the primitive is locally constant. \par

\begin{lemma}
	Suppose that $L \circ \mathcal{L}$ is a proper exact cylindrical Lagrangian embedding and that the primitive for $L \circ \mathcal{L}$ extends to a locally constant function in the cylindrical end of $N$. Then any Hamiltonian chord from $L \circ \mathcal{L}$ which lies outside the compact domain $N_{0}$ has negative action.
\end{lemma}
\begin{proof}
	 Since the primitive is locally constant in the cylindrical end, the action of any Hamiltonian chord $y$ from $L \circ \mathcal{L}$ to itself contained in the cylindrical end is
\begin{equation*}
\mathcal{A}(y) = - \int y^{*} \lambda_{N} + H_{N}(y(t))dt = -r^{2} < 0,
\end{equation*}
if $y$ lies on the level set $\partial N \times \{r\}$.
\end{proof}

	Suppose condition (iii) of Lemma \eqref{equivalent conditions for closedness of the cyclic element} does not hold under the assumption of Proposition \ref{vanishing of the bounding cochain}. Note that there is a uniform positive lower bound for the energy of any figure eight bubble. Thus if there were such a figure eight bubble, it would converge to some generalized chord for $(L, \mathcal{L}, \mathcal{L}, L)$ of positive action $\ge \epsilon > 0$, for some uniform constant $\epsilon > 0$ that is independent of individual pseudoholomorphic quilted maps but depends only on the background geometry - Liouville structures, Lagrangian submanifolds, Hamiltonians and almost complex structures.
Then this generalized chord corresponds to some generator $y$ of $CW^{*}(L \circ \mathcal{L})$ which has positive action.
By the lemma above, this generator cannot be any non-constant Hamiltonian chord contained in the cylindrical end. \par
	Thus, it suffices to prove that this generator $y$ cannot be a critical point either. The strategy is to consider Lagrangian Floer theory without wrapping. 
To carry out this idea, we consider the sub-complex $CW^{*}_{0}(L \circ \mathcal{L})$ of $CW^{*}(L \circ \mathcal{L})$, generated by only critical points. Up to $A_{\infty}$-homotopy equivalence, this is just the Morse complex of $H'|_{L \circ \mathcal{L}}$ with its Morse $A_{\infty}$-structure defined by counting gradient flow trees.
There is a similar subspace $CW^{*}_{0}(L, \mathcal{L}, L \circ \mathcal{L})$ of the quilted wrapped Floer cochain space, generated by generalized chords of low action (in absolute value), i.e. those generalized chords which correspond to Hamiltonian chords that are not contained in the cylindrical end of the product $M^{-} \times N$.
Since $L \circ \mathcal{L}$ is assumed to be a proper exact cylindrical Lagrangian embedding, $CW^{*}_{0}(L, \mathcal{L}, L \circ \mathcal{L})$ is indeed a cochain complex equipped with the quilted Floer differential, which can be alternatively defined with respect to a pair of Hamiltonians on $(M, N)$ that are $C^{2}$-small in a compact set and linear at infinity of small slope less than the minimal length of a Reeb chord. 
Then the map \eqref{geometric composition map} restricted to $CW^{*}_{0}(L, \mathcal{L}, L \circ \mathcal{L})$ has image  contained in the sub-complex $CW^{*}_{0}(L \circ \mathcal{L})$. Since these complexes can be identified with Floer complexes without Hamiltonian perturbations (or with small Hamiltonian perturbations if transversality is demanded) up to chain homotopy equivalences, this map \eqref{geometric composition map on low action part of Floer complexes} in fact becomes a cochain map
\begin{equation}\label{geometric composition map on low action part of Floer complexes}
gc: CW^{*}_{0}(L, \mathcal{L}, L \circ \mathcal{L}) \to CW^{*}_{0}(L \circ \mathcal{L}).
\end{equation}
after restriction without correction by any bounding cochain, by the argument of \cite{Lekili-Lipyanskiy}. Moreover, it induces an isomorphism on cohomology groups.
Note that the cyclic element $e_{L \circ \mathcal{L}}$ in fact lies in the subspace $CW^{*}_{0}(L, \mathcal{L}, L \circ \mathcal{L})$. In particular, it follows that $e_{L \circ \mathcal{L}}$ is closed under the undeformed quilted Floer differential on $CW^{*}_{0}(L, \mathcal{L}, L \circ \mathcal{L})$.
Thus we can argue by an analogue of Lemma \eqref{equivalent conditions for closedness of the cyclic element} in the setup of quilted Floer theory without wrapping, and conclude that the generator $y$ in question cannot be any nonzero element in $CW^{*}_{0}(L \circ \mathcal{L})$. 
Therefore, $y = 0$ and there cannot be a figure eight bubble, which contradicts our assumption.
The proof of Proposition \ref{vanishing of the bounding cochain} is now complete. \par

\subsection{Representability}
	The previously constructed functor $\Phi_{\mathcal{L}}$ \eqref{functor to modules over the immersed category} is not good enough for understanding the functoriality properties of wrapped Fukaya categories, as modules over a non-proper $A_{\infty}$-category can be very complicated. Thus we must find a more geometric replacement. In the case of compact monotone Lagrangian submanifolds in compact monotone symplectic manifolds, there are results from \cite{Wehrheim-Woodward4}, \cite{Lekili-Lipyanskiy} on the level of cohomology, which establish an isomorphism between the quilted Floer cohomology group and the Floer cohomology group of the geometric composition:
\begin{equation*}
HF^{*}(L, \mathcal{L}, L') \cong HF^{*}(L \circ \mathcal{L}, L').
\end{equation*}
Now we would like to generalize this statement on the categorical level, aiming to prove that the Yoneda module associated to the geometric composition is homotopy equivalent to the module $\Phi_{\mathcal{L}}(L)$ defined in terms of quilted wrapped Floer theory, and moreover that such homotopy equivalences are functorial in the wrapped Fukaya category of $M$. Such a result can be improved to the statement that the module-valued functor $\Phi_{\mathcal{L}}$ \eqref{functor to modules over the immersed category} is representable. \par
	In the previous subsection \ref{section: unobstructedness of the geometric composition}, we have shown that if the natural map $\mathcal{L} \to N$ is proper and if Assumption \ref{assumption on the geometric composition} holds, the geometric composition $L \circ \mathcal{L}$ is always unobstructed, and that there is a canonical choice of a bounding cochain $b$ for it, determined by $L$ and $\mathcal{L}$. Then, via the (left) Yoneda embedding
\begin{equation*}
\mathfrak{y}_{l}: \mathcal{W}_{im}(N) \to \mathcal{W}_{im}(N)^{l-mod},
\end{equation*}
the distinguished object $(L \circ \mathcal{L}, b)$ defines a left $A_{\infty}$-module over $\mathcal{W}_{im}(N)$. The main result of this subsection claims that this $A_{\infty}$-module is homotopy to the module $\Phi_{\mathcal{L}}(L)$, which therefore yields the representability of the functor $\Phi_{\mathcal{L}}$ \eqref{functor to modules over the immersed category}. \par

\begin{theorem} \label{representability of Lagrangian correspondence functor}
	Suppose that $\mathcal{L} \subset M^{-} \times N$ is an admissible Lagrangian correspondence such that the map $\mathcal{L} \to N$ is proper, and Assumption \ref{assumption on the geometric composition} is satisfied.
Then the $A_{\infty}$-functor $\Phi_{\mathcal{L}}$ \eqref{functor to modules over the immersed category} is representable. That is, there exists a canonical $A_{\infty}$-functor
\begin{equation}\label{the new module-valued functor}
\Psi_{\mathcal{L}}: \mathcal{W}(M) \to \mathcal{W}_{im}(N)^{rep-l-mod}
\end{equation}
such that $i \circ \Psi_{\mathcal{L}}$ is  homotopic to $\Phi_{\mathcal{L}}$ as $A_{\infty}$-functors, where
\begin{equation*}
i: \mathcal{W}_{im}(N)^{rep-l-mod} \to \mathcal{W}_{im}(N)^{l-mod}
\end{equation*}
is the obvious inclusion of the sub-category of representable modules to the category of all modules.
\end{theorem}

	The proof of this theorem will occupy the rest of this subsection. An immediate consequence of this theorem is that we get a functor to the immersed wrapped Fukaya category $\mathcal{W}_{im}(N)$: \par

\begin{corollary}
	There is an $A_{\infty}$-functor
\begin{equation}\label{functor to the immersed category}
\Theta_{\mathcal{L}}: \mathcal{W}(M) \to \mathcal{W}_{im}(N),
\end{equation}
which represents the module-valued functor \eqref{functor to modules over the immersed category}, in the sense that $\mathfrak{y}_{l} \circ \Theta_{\mathcal{L}}$ is homotopic to $\Phi_{\mathcal{L}}$.
\end{corollary}
\begin{proof}
	The Yoneda lemma says that the left Yoneda functor is a homotopy equivalence onto its image, i.e.,
\begin{equation*}
\mathfrak{y}_{l}: \mathcal{W}_{im}(N) \to \mathcal{W}_{im}(N)^{rep-l-mod}
\end{equation*}
is a homotopy equivalence.
Thus we may choose a homotopy inverse
\begin{equation*}
\lambda_{l}: \mathcal{W}_{im}(N)^{rep-l-mod} \to \mathcal{W}_{im}(N),
\end{equation*}
and compose the functor $\Psi_{\mathcal{L}}$ with $\lambda_{l}$ to obtain the desired functor $\Theta_{\mathcal{L}}$ \eqref{functor to the immersed category}. \par
\end{proof}

\begin{figure}
\centering
\begin{tikzpicture} 
	\draw (-4, 1) -- (-1.5, 1);
	\draw (-1.5, 1) arc (270:360:0.5cm);
	\draw (-1, 1.5) -- (-1, 2);
	\draw (1.5, 1) arc (270:180:0.5cm);
	\draw (1, 1.5) -- (1, 2);
	\draw (1.5, 1) -- (4, 1);
	
	\draw (0.75, 0) -- (4, 0);
	\draw (0.75, 0) arc (270:180:0.75cm);
	\draw (0, 0.75) -- (0, 2);
	
	\draw (-4, -1) -- (4, -1);

	\draw (-3, 1.25) node {$L \circ \mathcal{L}$};
	\draw (3, 1.25) node {$L$};
	\draw (2, 0) node {$\mathcal{L}$};
	\draw (0, -1.25) node {$L'_{0}$};

	\draw (0, 2.5) node {$e$};
	\draw (-4.5, 0) node {$x'$};
	\draw (4.5, 0) node {$(x, y)$};

\end{tikzpicture}
\caption{the quilted map defining the map $gc$}
 \label{fig: the quilted map defining the geometric composition map}
\end{figure}
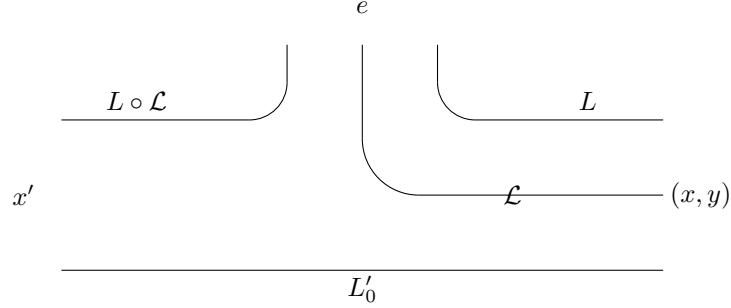

	Our proof of Theorem \ref{representability of Lagrangian correspondence functor} is a generalization of the proof of isomorphism of Floer cohomology groups under the geometric composition, as discussed in \cite{Gao1}. For any properly embedded exact cylindrical Lagrangian submanifold $L'_{0} \subset N$, there is a canonical isomorphism
\begin{equation} \label{quasi-isomorphism under geometric composition}
gc: CW^{*}(L, \mathcal{L}, L'_{0}) \to CW^{*}(L \circ \mathcal{L}, L'_{0})
\end{equation}
of $\mathbb{Z}$-modules, which become a chain quasi-isomorphism if we equip the latter cochain space $CW^{*}(L \circ \mathcal{L}, L'_{0})$ with the differential deformed by the bounding cochain $b$ for the geometric composition $L \circ \mathcal{L}$ provided by Theorem \ref{unobstructedness of geometric composition}. We call this map the geometric composition quasi-isomorphism. Recall that in \cite{Gao1} we already defined this map in a slightly different form:
\begin{equation*}
gc': CW^{*}(L, \mathcal{L}, L'_{0}) \to CW^{*}(L \circ_{H_{M}} \mathcal{L}, L'_{0})
\end{equation*}
in case the geometric composisition is properly embedded, and proved that it is a quasi-isomorphism. This map is related to \eqref{quasi-isomorphism under geometric composition} via the homotopy equivalence of left modules associated to $L \circ \mathcal{L}$ and $L \circ_{H_{M}} \mathcal{L}$ - these are modules over the curved $A_{\infty}$-category $\mathcal{W}_{ob, im}(M)$ whose objects are proper exact cylindrical Lagrangian immersions (possibly without bounding cochains). This curved $A_{\infty}$-category has been introduced in a somewhat implicit way when we defined the immersed wrapped Fukaya category, before the contributions of the bounding cochains to the structure maps are included. 
In fact, the same proof applies to the current setup: the map \eqref{quasi-isomorphism under geometric composition} is defined using moduli spaces
\begin{equation}\label{moduli space of quilted maps defining the first order term of the geometric composition map}
\mathcal{C}((x, y); x'; e)
\end{equation}
of appropriate inhomogeneous pseudoholomorphic quilted maps, as pictured in Figure \ref{fig: the quilted map defining the geometric composition map}. \par

	The asymptotic condition at the top quilted end (in Figure \ref{fig: the quilted map defining the geometric composition map}) is that the quilted map asymptotically converges to the generalized Hamiltonian chord $e$ for the triple $(L, \mathcal{L}, L \circ \mathcal{L})$ representing the cyclic element $e = e_{L \circ \mathcal{L}} \in CW^{*}(L, \mathcal{L}, L \circ \mathcal{L})$, which in turn corresponds to the homotopy unit of $CW^{*}(L \circ \mathcal{L})$ under the $\mathbb{Z}$-module isomorphism \eqref{quasi-isomorphism under geometric composition}. This cyclic element is discussed in the previous subsection, \ref{section: unobstructedness of the geometric composition}. \par
	In order to see the contributions from the bounding cochain $b$ for $L \circ \mathcal{L}$, we must modify these moduli spaces appropriately. Instead of looking at a single moduli space like \eqref{moduli space of quilted maps defining the first order term of the geometric composition map}, we consider a sequence of moduli spaces
\begin{equation}\label{moduli space of quilted maps defining the first order term of the geometric composition map with bounding cochains inserted at the boundary}
\mathcal{C}_{k}((x, y); \underbrace{b, \cdots, b}_{k \text{ times }}; x'; e),
\end{equation}
where we add $k$ punctures to the boundary component of a quilted map as in Figure \ref{fig: the quilted map defining the geometric composition map} which is mapped to $L \circ \mathcal{L}$, and impose the asymptotic convergence conditions at these punctures to be given by the bounding cochain $b$ for $L \circ \mathcal{L}$. By counting elements in these moduli spaces of virtual dimension zero, we get a map
\begin{equation}\label{quasi-isomorphism under geometric composition with deformation given by the bounding cochain}
gc: CW^{*}(L, \mathcal{L}, L'_{0}) \to CW^{*}((L \circ \mathcal{L}, b), L'_{0}).
\end{equation}

	The "count" requires careful treatment. As the geometric composition is in general no longer an embedding, we cannot use domain-dependent perturbations of Hamiltonians and almost complex structures to achieve transversality of the moduli spaces. The count is instead given by virtual fundamental chains associated to a coherent choice of single-valued multisections for Kuranishi structures on the moduli spaces. Such constructions have been discussed several times and should be routine by now, so we leave the details to the interested reader. \par

\begin{lemma}
	The map \eqref{quasi-isomorphism under geometric composition with deformation given by the bounding cochain} is a cochain map, where the differential on $CW^{*}((L \circ \mathcal{L}, b), L'_{0})$ is given by the $b$-deformed structure map. That is,
\begin{equation*}
gc \circ n^{0} = m^{1; b} \circ gc,
\end{equation*}
where $n^{0}$ denotes the quilted Floer differential on $CW^{*}(L, \mathcal{L}, L'_{0})$, and $m^{1; b}$ is the $b$-deformed Floer differential on $CW^{*}(L \circ \mathcal{L}, L'_{0})$.
\end{lemma}
\begin{proof}
	By looking at the codimension-one boundary strata of the moduli spaces $\mathcal{C}_{k}((x, y); \underbrace{b, \cdots, b}_{k \text{ times }}; x'; e)$ \eqref{moduli space of quilted maps defining the first order term of the geometric composition map with bounding cochains inserted at the boundary}, we find that the map $gc$ as in \eqref{quasi-isomorphism under geometric composition with deformation given by the bounding cochain} satisfies the following equation:
\begin{equation}
gc \circ n^{0}((x, y)) = m^{1; b} \circ gc((x, y)) + d^{b}(e).
\end{equation}
Because the cyclic element $e$ satisfies the condition that $d^{b}(e) = 0$, the last term vanishes, so the map \eqref{quasi-isomorphism under geometric composition with deformation given by the bounding cochain} is a cochain map. \par
	With the capping half-disks taken into account, there is a well-defined single-valued action of the generators, so that we can use an action-filtration argument to prove that \eqref{quasi-isomorphism under geometric composition with deformation given by the bounding cochain} is a cochain isomorphism, as follows. If we truncate the Floer complex using the action filtration, then this map can be written as an upper-triangular matrix with all diagonal entries equal to the "identity", as counting non-trivial inhomogeneous pseudoholomorphic quilted maps as above necessarily increases the action. Here by "identity", we mean the natural one-to-one correspondence between the set of generators for $CW^{*}(L, \mathcal{L}, L'_{0})$ and that for $CW^{*}(L \circ \mathcal{L}, L'_{0})$. \par
\end{proof}

	In the general case where $L'_{0}$ is an exact cylindrical Lagrangian immersion, the proof of the map \eqref{quasi-isomorphism under geometric composition with deformation given by the bounding cochain} being a cochain homotopy equivalence is in fact quite similar. The only difference is that the homotopy unit may no longer be closed under the previously-mentioned quilted Floer differential, because there are pseudoholomorphic disks with one marked point in $N$ with boundary on the image of $L'_{0}$. However, $L'_{0}$ itself comes with a bounding cochain $b'_{0}$ for it to be an object of the immersed wrapped Fukaya category, and as long as we use the bounding cochain $b'_{0}$ on $L'_{0}$ to cancel the contribution of those disks, the homotopy unit of $L'_{0}$ becomes closed under the deformed quilted Floer differential. In this case, the map \eqref{quasi-isomorphism under geometric composition with deformation given by the bounding cochain} takes the following form
\begin{equation}\label{quasi-isomorphism under geometric composition for immersed Lagrangians}
gc: CW^{*}(L, \mathcal{L}, (L'_{0}, b'_{0})) \to CW^{*}((L \circ \mathcal{L}, b), (L'_{0}, b'_{0})).
\end{equation} \par
	Next, we shall construct a $A_{\infty}$-pre-module homomorphism extending the map \eqref{quasi-isomorphism under geometric composition for immersed Lagrangians} as its first order term. We set $gc^{0} = 0$. For $d \ge 2$, we define multilinear maps $gc^{d}$, for all possible $(d-1)$-tuple of testing objects $(L'_{0}, b'_{0}), \cdots, (L'_{d-1}, b'_{d-1})$ of $\mathcal{W}_{im}(N)$, as follows:
\begin{equation}
\begin{split}
gc^{d}: & CW^{*}(L, \mathcal{L}, (L'_{d-1}, b'_{d-1})) \otimes CW^{*}((L'_{d-2}, b'_{d-2}), (L'_{d-1}, b'_{d-1}))\\
&\otimes \cdots \otimes CW^{*}((L'_{0}, b'_{0}), (L'_{1}, b'_{1}))
\to CW^{*}((L \circ \mathcal{L}, b), (L'_{0}, b'_{0}))
\end{split}
\end{equation}
defined by appropriate "count" of elements in the moduli spaces 
\begin{equation}\label{moduli space of quilted maps defining the higher order terms of the geometric composition map}
\bar{\mathcal{C}}_{k, d, l_{0}, \cdots, l_{d-1}}((x, y); \underbrace{b'_{0}, \cdots, b'_{0}}_{l_{0} \text{ times }}, x'_{1}, \underbrace{b'_{1}, \cdots, b'_{1}}_{l_{1} \text{ times }} \cdots, x'_{d-1}, \underbrace{b'_{d-1}, \cdots, b'_{d-1}}_{l_{d-1} \text{ times }}; \underbrace{b, \cdots, b}_{k \text{ times }}; x'; e)
\end{equation}
of quilted maps, defined in a way similar to \eqref{moduli space of quilted maps defining the first order term of the geometric composition map with bounding cochains inserted at the boundary}, but with multiple punctures and Lagrangian labels on the boundary components of the second patch of the quilted surface. The corresponding asymptotic convergence conditions are given by generators $x'_{i} \in CW^{*}((L'_{i-1}, b'_{i-1}), (L'_{i}, b'_{i}))$, and also those given by bounding cochains $b'_{i} \in CW^{*}((L'_{i}, b'_{i})), i = 0, \cdots, d-1$. \par

\begin{lemma}
	The maps $\{gc^{d}\}$ form an $A_{\infty}$-pre-module homomorphism
\begin{equation}\label{geometric composition map}
gc: \Phi_{\mathcal{L}}(L) \to \mathfrak{y}_{l}((L \circ \mathcal{L}, b))
\end{equation}
from the left $A_{\infty}$-module $\Phi_{\mathcal{L}}(L)$ over $\mathcal{W}_{im}(N)$, to the left Yoneda module $\mathfrak{y}_{l}((L \circ \mathcal{L}, b))$ over $\mathcal{W}_{im}(N)$. Moreover, this $A_{\infty}$-pre-module homomorphism is in fact an $A_{\infty}$-module homomorphism.
\end{lemma}
\begin{proof}
	The verification of the $A_{\infty}$-equations for pre-module homomorphisms can be done by looking at the boundary of the above-mentioned moduli spaces
\begin{equation*}
\bar{\mathcal{C}}_{k, d, l_{0}, \cdots, l_{d-1}}((x, y); \underbrace{b'_{0}, \cdots, b'_{0}}_{l_{0} \text{ times }}, x'_{1}, \underbrace{b'_{1}, \cdots, b'_{1}}_{l_{1} \text{ times }} \cdots, x'_{d-1}, \underbrace{b'_{d-1}, \cdots, b'_{d-1}}_{l_{d-1} \text{ times }}; \underbrace{b, \cdots, b}_{k \text{ times }}; x'; e).
\end{equation*}
For this, simply recall that the $A_{\infty}$-module structure on $CW^{*}(L, \mathcal{L}, \cdot)$ over $\mathcal{W}_{im}(N)$ is defined via suitable moduli spaces of inhomogeneous pseudoholomorphic quilted maps, while that on $CW^{*}((L \circ \mathcal{L}, b), \cdot)$ is defined via suitable moduli spaces of stable broken Floer trajectories in $N$, with decorations by the bounding cochain $b$. These are compatible with the compactification 
\begin{equation*}
\bar{\mathcal{C}}_{k, d, l_{0}, \cdots, l_{d-1}}((x, y); \underbrace{b'_{0}, \cdots, b'_{0}}_{l_{0} \text{ times }}, x'_{1}, \underbrace{b'_{1}, \cdots, b'_{1}}_{l_{1} \text{ times }} \cdots, x'_{d-1}, \underbrace{b'_{d-1}, \cdots, b'_{d-1}}_{l_{d-1} \text{ times }}; \underbrace{b, \cdots, b}_{k \text{ times }}; x'; e),
\end{equation*}
meaning that the two kinds of moduli spaces arise in the boundary strata of this moduli space. This fact can be used to show that the maps defined above satisfy the $A_{\infty}$-equations for pre-module homomorphisms. \par
	Now if we deform the structure maps on $CW^{*}(L \circ \mathcal{L}, \cdot)$ by the bounding cochain $b$, the cyclic element $e_{L \circ \mathcal{L}} \in CW^{*}(L, \mathcal{L}, (L \circ \mathcal{L}, b))$ becomes closed under the $b$-deformed Floer differential, which is the module differential for the left Yoneda module $\mathfrak{y}_{l}((L \circ \mathcal{L}, b))$. Thus the above $A_{\infty}$-pre-module homomorphism is a cocycle in the functor category, i.e. an $A_{\infty}$-module homomorphism. \par
\end{proof}

\begin{corollary}
	The $A_{\infty}$-module homomorphism $gc$ \eqref{geometric composition map} is a quasi-isomorphism of $A_{\infty}$-modules.
\end{corollary}
\begin{proof}
	This follows from the fact that the first-order map $gc^{1}$ is an isomorphism of cochain complexes. \par
\end{proof}

	This module homomorphism $gc$ is more than just a quasi-isomorphism, but indeed a homotopy equivalence. This statement is rather important because we are working over the integers. The proof is in fact very simple, based on our construction of $gc$. \par

\begin{proposition}
	The $A_{\infty}$-module homomorphism
\begin{equation*}
gc: \Phi_{\mathcal{L}}(L) \to \mathfrak{y}_{l}((L \circ \mathcal{L}, b))
\end{equation*}
is a homotopy equivalence of $A_{\infty}$-modules.
\end{proposition}
\begin{proof}
	A homotopy inverse can be constructed using moduli spaces similar to \eqref{moduli space of quilted maps defining the higher order terms of the geometric composition map}, but we interchange the inputs and the outputs. That is, we regard $x'$ as the input and $(x, y)$ as the output, and construct a sequence of multilinear maps of the form
\begin{equation}
\begin{split}
op^{d}: & CW^{*}((L \circ \mathcal{L}, b), (L'_{d-1}, b'_{d-1})) \otimes CW^{*}((L'_{d-2}, b'_{d-2}), (L'_{d-1}, b'_{d-1})) \otimes \cdots \\
&\otimes CW^{*}((L'_{0}, b'_{0}), (L'_{1}, b'_{1})) \to CW^{*}(L, \mathcal{L}, (L'_{0}, b'_{0})).
\end{split}
\end{equation}
These form an $A_{\infty}$-module homomorphism
\begin{equation*}
op: \mathfrak{y}_{l}((L \circ \mathcal{L}, b)) \to \Phi_{\mathcal{L}}(L).
\end{equation*}
Standard gluing argument in Floer theory implies that $gc^{1} \circ op^{1}$ and $op^{1} \circ gc^{1}$ are both chain homotopic to the identity, which implies that $op$ is the a homotopy inverse of $gc$. \par
\end{proof}

	Thus, we set
\begin{equation}
\Psi_{\mathcal{L}}(L) = \mathfrak{y}_{l}((L \circ \mathcal{L}, b)),
\end{equation}
the left Yoneda module of $(L \circ \mathcal{L}, b) \in Ob \mathcal{W}_{im}(N)$, for every object $L \in Ob \mathcal{W}(M)$.
	The next step is to prove that such an $A_{\infty}$-module homotopy equivalence of $A_{\infty}$-modules is functorial in $\mathcal{W}(M)$. For this purpose, we shall define multilinear maps
\begin{equation}
T^{d}: CW^{*}(L_{d-1}, L_{d}) \otimes \cdots \otimes CW^{*}(L_{0}, L_{1}) \to \hom_{\mathcal{W}_{im}(N)^{l-mod}}(\Phi_{\mathcal{L}}(L_{0}), \Psi_{\mathcal{L}}(L_{d}))[-d]
\end{equation}
of degree $-d$, which satisfy the equations for $A_{\infty}$-pre-natural transformations.
In more concrete terms, we shall define a multilinear map for all possible Floer cochains $x_{1} \in CW^{*}(L_{0}, L_{1}), \cdots, x_{d} \in CW^{*}(L_{d-1}, L_{d})$ as well as cylindrical Lagrangian immersions in $N$ equipped with bounding cochains $(L'_{0}, b'_{0}), \cdots, (L'_{k-1}, b'_{k-1})$:
\begin{equation}
\begin{split}
&(T^{d}(x_{d}, \cdots, x_{1}))^{k}: CW^{*}((L_{0} \circ \mathcal{L}, b_{0}), (L'_{k-1}, b'_{k-1}))\\
& \otimes CW^{*}((L'_{k-2}, b'_{k-1}), (L'_{k-1}, b'_{k-1}))
\otimes \cdots \otimes CW^{*}((L'_{0}, b'_{0}), (L'_{1}, b'_{1}))\\
&\to CW^{*}((L_{d} \circ \mathcal{L}, b_{d}), (L'_{0}, b'_{0})),
\end{split}
\end{equation}
which is linear with respect to each $x_{i}$, and satisfies the following equation:
\begin{equation} \label{A-infinity equation for the new module-valued functor}
\begin{split}
&m^{1}_{\mathcal{W}_{im}(N)^{l-mod}}(\Psi_{\mathcal{L}}^{d}(x_{d}, \cdots, x_{1}))\\
+ &\sum_{s} m^{2}_{\mathcal{W}_{im}(N)^{l-mod}}(\Psi_{\mathcal{L}}^{d-s}(x_{d}, \cdots, x_{s+1}), \Psi_{\mathcal{L}}^{s}(x_{s}, \cdots, x_{1}))\\
= &\sum_{n, k} (-1)^{*} \Psi_{\mathcal{L}}^{d-k+1}(x_{d}, \cdots, x_{n+k+1}, m^{k}_{\mathcal{W}(M)}(x_{n+k}, \cdots, x_{n+1}), x_{n}, \cdots, x_{1}),
\end{split}
\end{equation}
where (only here) $* = |x_{1}| + \cdots + |x_{n}| - n$. \par
	We consider the following quilted surfaces $\underline{S}^{nf}$ consisting of two patches $S^{nf}_{0}, S^{nf}_{1}$ where $S^{nf}_{0}$ is a disk with $(d+1)$ positive boundary punctures $z_{0}^{+, 1}, z_{0}^{1}, \cdots, z_{0}^{d}$, and one negative puncture $z_{0}^{-, 2}$, and $S^{nf}_{0}$ is a disk with $(k+1)$ positive boundary punctures $z_{1}^{+, 1}, z_{1}^{1}, \cdots, z_{1}^{k-1}$, and one negative boundary puncture $z_{1}^{-}, z_{1}^{-,2}$. We denote by $I_{0}^{+}$ the boundary component of $S^{nf}_{0}$ between $z_{0}^{+, 1}$ and $z_{0}^{+, 2}$, and by $I_{1}^{\pm}$ the boundary component of $S^{nf}_{1}$ between $z_{1}^{-}$ and $z_{1}^{+}$. $\underline{S}^{nf}$ is obtained by seaming together the two patches along the pair $(I_{0}^{+}, I_{1}^{\pm})$ of boundary components. We need to consider semi-stable nodal quilted surfaces arising as domains of limits of stable maps from such quilted surfaces, and we denote them by the same symbol. \par
	Suppose strip-like ends and quilted ends for all semistable nodal quilted surfaces $\underline{S}^{nf}$ have been chosen. Make consistent choices of Floer data for all such $\underline{S}^{nf}$, requiring the choices to automorphism-invariant Floer data for semistable nodal quilted surfaces that are domains of stable maps to $M$ with Lagrangian boundary conditions that are to be specified below. Let $e_{d} = e_{L_{d} \circ \mathcal{L}}$ be the generator of $CW^{*}(L_{d}, \mathcal{L}, (L_{d} \circ \mathcal{L}, b_{d}))$ corresponding to the fundamental chain of $L_{d} \circ \mathcal{L}$. Consider the moduli space 
\begin{equation}
\mathcal{T}^{nf}_{d, k-1}(\alpha, \beta; y^{-}; x^{1}, \cdots, x^{d}; y^{1}, \cdots, y^{k-1}; (x^{+}, y^{+}); e_{d})
\end{equation}
of triples $(\underline{S}^{nf}, \underline{u}, l_{1})$, satisfying the following conditions:
\begin{enumerate}[label=(\roman*)]

\item $\underline{u}: \underline{S}^{nf} \to (M, N)$ is a quilted map with marked points or punctures $\underline{\vec{z}}$ satisfying the following equations:
\begin{equation}
\begin{cases}
(du_{0} - \alpha_{S^{nf}_{0}} \otimes X_{H_{S^{nf}_{0}}})^{0, 1} = 0\\
(du_{1} - \alpha_{S^{nf}_{1}} \otimes X_{H_{S^{nf}_{1}}})^{0, 1} = 0\\
u_{0}(z) \in \psi_{M}^{\rho_{S^{nf}_{0}}(z)} L_{0}, \text{ if $z$ lies between $z_{0}^{+,1}$ and $z_{0}^{1}$ }\\
u_{0}(z) \in \psi_{M}^{\rho_{S^{nf}_{0}}(z)} L_{i}, \text{ if $z$ lies between $z_{0}^{i}$ and $z_{0}^{i+1}$ }\\
u_{0}(z) \in \psi_{M}^{\rho_{S^{nf}_{0}}(z)} L_{d}, \text{ if $z$ lies between $z_{0}^{d}$ and $z_{0}^{+, 2}$ }\\
u_{1}(z) \in \psi_{N}^{\rho_{S^{nf}_{1}}(z)} \iota_{0}(L'_{0}), \text{ if $z$ lies between $z_{1}^{+, 1}$ and $z_{1}^{1}$ } \\
u_{1}(z) \in \psi_{N}^{\rho_{S^{nf}_{1}}(z)} \iota_{j}(L'_{j}), \text{ if $z$ lies between $z_{1}^{j}$ and $z_{1}^{j+1}$ }\\
u_{1}(z) \in \psi_{N}^{\rho_{S^{nf}_{1}}(z)} \iota_{k-1}(L'_{k-1}), \text{ if $z$ lies between $z_{1}^{k-1}$ and $z_{1}^{-}$ }\\
u_{1}(z) \in \psi_{N}^{\rho_{S^{nf}_{1}}(z)} L_{d} \circ_{H_{M}} \mathcal{L}, \text{ if $z$ lies between $z_{1}^{-}$ and $z_{1}^{+, 2}$ }\\
(u_{0}(z), u_{1}(z)) \in (\psi_{M}^{\rho_{S^{nf}_{0}}(z)} \times \psi_{N}^{\rho_{S^{nf}_{1}}(z)}) \mathcal{L}, \text{ if $z$ lies on the seam }\\
\lim\limits_{s \to -\infty} u_{1} \circ \epsilon_{1}^{-}(s, \cdot) = \psi_{N}^{w_{1}^{-}} y^{-}(\cdot)\\
\lim\limits_{s \to +\infty} (u_{0} \circ \epsilon_{0}^{+,1}(s, \cdot), u_{1} \circ \epsilon_{1}^{+, 1}(s, \cdot)) = \psi_{M \times N}^{w^{+}}(x^{+}(\cdot), y^{+}(\cdot))\\
\lim\limits_{s \to +\infty} u_{0} \circ \epsilon_{0}^{i}(s, \cdot) = \psi_{M}^{w_{0}^{i}}x^{i}(\cdot)\\
\lim\limits_{s \to +\infty} (u_{0} \circ \epsilon_{0}^{-, 2}(s, \cdot), u_{1} \circ \epsilon_{1}^{-, 2}(s, \cdot)) = \psi_{M \times N}^{w^{e}} e_{d}(\cdot)\\
\lim\limits_{s \to +\infty} u_{1} \circ \epsilon_{1}^{j}(s, \cdot) = \psi_{N}^{w_{1}^{j}}y^{j}(\cdot)
\end{cases}
\end{equation}

\item $l_{1}: \partial S_{1}^{nf} \setminus \{z_{1}^{i}: i \in I\} \to (L'_{0} \times_{\iota_{0}} L'_{0}) \cup \cdots (L'_{k-1} \times_{\iota_{k-1}} L'_{k-1}) \cup ( (L_{d} \circ \mathcal{L}) \times_{\iota} (L_{d} \circ \mathcal{L}))$ is a smooth map.

\item $\phi_{N}^{\rho_{S^{nf}_{1}(l_{1}(z))}} \iota_{i} = u_{1} \circ l_{1}(z)$, when $z \in \partial S^{nf}_{1}$ lies between $z_{1}^{i}$ and $z_{1}^{i+1}$, for every $1 \le i \le k-2$. If $z$ lies between $z_{1}^{+, 1}$ and $z_{1}$, the corresponding Lagrangian immersion should be replaced by $\iota_{0}: L'_{0} \to N$. If $z$ lies between $z_{1}^{+, 2}$, the corresponding Lagrangian immersion should be replaced by $\iota_{k-1}: L'_{k-1} \to N$. If $z$ lies between $z_{1}^{+, 2}$ and $z_{1}^{+, 1}$, the corresponding Lagrangian immersion should be replaced by $\iota: L_{d} \circ \mathcal{L} \to N$.

\item the relative homology class of $\underline{u}$ is $\beta$.

\item the triple $(\underline{S}^{nf}, \underline{u}, l_{1})$ is stable, meaning that it has finite automorphism group.

\end{enumerate}

	The above conditions are analogous to those in the case of a single Lagrangian immersion for which we defined the moduli space of inhomogeneous pseudoholomorphic disks. Recall the relevant notations in section \ref{section: moduli space of disks bounded by immersed Lagrangian submanifolds}. \par
	There is a natural compactification $\bar{\mathcal{T}}^{nf}_{d, k-1}(y^{-}; x^{1}, \cdots, x^{d}; y^{1}, \cdots, y^{k-1}; (x^{+}, y^{+}))$ of this moduli space $\mathcal{T}^{nf}_{d, k-1}(y^{-}; x^{1}, \cdots, x^{d}; y^{1}, \cdots, y^{k-1}; (x^{+}, y^{+}))
$, which consists of broken quilted maps of the same type. In particular, the codimension-one boundary strata consist of union of fiber products
\begin{equation}
\begin{split}
&\partial \bar{\mathcal{T}}^{nf}_{d, k-1}(y^{-}; x^{1}, \cdots, x^{d}; y^{1}, \cdots, y^{k-1}; (x^{+}, y^{+}))\\
& \cong \coprod \bar{\mathcal{T}}^{nf}_{d_{1}, k_{1}-1}(y^{-}; x^{1}, \cdots, x^{i}, x^{new}, x^{i+k_{1}+1}, \cdots, x^{d}; \\
& y^{1}, \cdots, y^{j}, y^{new}, y^{j+k_{2}+1}, \cdots, y^{k-1}; (x^{+}, y^{+}))\\
& \times \bar{\mathcal{M}}_{d_{2}+1}(x^{new}, x^{i+1}, \cdots, x^{i+d_{2}}) \times \bar{\mathcal{M}}_{k_{2}+1}(y^{new}, y^{j+1}, \cdots, y^{j+k_{2}})\\
& \cup \coprod \bar{\mathcal{T}}_{d_{1}, k_{1}-1}((x^{-}, y^{-}); x^{1}, \cdots, x^{d_{1}}; y^{1}, \cdots, y^{k_{1}-1}; (x^{+}_{1}, y^{+}_{1}))\\
& \times \bar{\mathcal{T}}_{d_{2}, k_{2}-1}((x^{+}_{1}, y^{+}_{1}); x^{d_{1}+1}, \cdots, x^{d}; y^{k_{1}}, \cdots, y^{k-1}; (x^{+}, y^{+}))
\end{split}
\end{equation}
where the compactified moduli space $\bar{\mathcal{T}}^{nf}_{d', k'-1}(\cdots)$ for $d' < d, k' < k$ are built inductively in this way.
The stability condition ensures that the moduli space $\bar{\mathcal{T}}^{nf}_{d, k-1}(y^{-}; x^{1}, \cdots, x^{d}; y^{1}, \cdots, y^{k-1}; (x^{+}, y^{+}))
$ is compact and Hausdorff. \par
	To include the contribution from the bounding cochains $b'_{j}$ for $\iota_{j}: L'_{j} \to N$, we modify elements in the moduli space $\bar{\mathcal{T}}^{nf}_{d, k-1}(y^{-}; x^{1}, \cdots, x^{d}; y^{1}, \cdots, y^{k-1}; (x^{+}, y^{+}))$ by adding more punctures on each boundary component of $S^{nf}_{1}$ that is mapped to the image of one of the Lagrangian immersions $\iota_{j}: L'_{j} \to N$ and $L_{d} \circ \mathcal{L}$, and imposing the asymptotic convergence conditions at these additional puctures given by the bounding cochains $b'_{j}$. The resulting moduli space is denoted by
\begin{equation}\label{moduli space defining homotopy between the module-valued functors}
\bar{\mathcal{T}}^{nf}_{d, k-1,s_{0}, \cdots, s_{k-1}, s}(y^{-}; x^{1}, \cdots, x^{d}; y^{1}, \cdots, y^{k-1}; b'_{0}, \cdots, b'_{0}; \cdots; b'_{k-1}, \cdots, b'_{k-1}; (x^{+}, y^{+})).
\end{equation}
The picture of such a quilted map is shown in Figure \ref{fig: the quilted map defining the homotopy between two module-valued functors}, where we have omitted the bounding cochains, but shall remember that they are also included as suitable asymptotic convergence conditions on the boundary components of the second patch, with the prescribed number of punctures added. \par

\begin{figure}
\begin{tikzpicture}
	\draw (-5.5, 1) -- (-4.5, 1);
	\draw (-4.5, 1) arc (270:360:0.25cm);
	\draw (-4.25, 1.25) -- (-4.25, 1.75);
	\draw (-3, 1) arc (270:180:0.25cm);
	\draw (-3.25, 1.25) -- (-3.25, 1.75);
	\draw (-3, 1) -- (-1.75, 1);
	\draw (-1.75, 1) arc (270:360:0.25cm);
	\draw (-1.5, 1.25) -- (-1.5, 1.75);
	\draw (-0.75, 1) arc (270:180:0.25cm);
	\draw (-1, 1.25) -- (-1, 1.75);
	\draw (-0.75, 1) -- (0.75, 1);
	\draw (0.75, 1) arc (270:360:0.25cm);
	\draw (1, 1.25) -- (1, 1.75);
	\draw (1.75, 1) arc (270:180:0.25cm);
	\draw (1.5, 1.25) -- (1.5, 1.75);
	\draw (1.75, 1) -- (3, 1);
	\draw (3, 1) arc (270:360:0.25cm);
	\draw (3.25, 1.25) -- (3.25, 1.75);
	\draw (4, 1) arc (270:180:0.25cm);
	\draw (3.75, 1.25) -- (3.75, 1.75);
	\draw (4, 1) -- (5.5, 1);
	
	\draw (-3.25, 0) -- (5.5, 0);
	\draw (-3.25, 0) arc (270:180:0.5cm);
	\draw (-3.75, 0.5) -- (-3.75, 1.75);
	
	\draw (-5.5, -1) -- (-3.5, -1);
	\draw (-3.5, -1) arc (90:0:0.25cm);
	\draw (-3.25, -1.25) -- (-3.25, -1.75);
	\draw (-2.5, -1) arc (90:180:0.25cm);
	\draw (-2.75, -1.25) -- (-2.75, -1.75);
	\draw (-2.5, -1) -- (-0.5, -1);
	\draw (-0.5, -1) arc (90:0:0.25cm);
	\draw (-0.25, -1.25) -- (-0.25, -1.75);
	\draw (0.5, -1) arc (90:180:0.25cm);
	\draw (0.25, -1.25) -- (0.25, -1.75);
	\draw (0.5, -1) -- (2.5, -1);
	\draw (2.5, -1) arc (90:0:0.25cm);
	\draw (2.75, -1.25) -- (2.75, -1.75);
	\draw (3.5, -1) arc (90:180:0.25cm);
	\draw (3.25, -1.25) -- (3.25, -1.75);
	\draw (3.5, -1) -- (5.5, -1);

	\draw (-5.5, 1.5) node {$L_{3} \circ \mathcal{L}$};
	\draw (-2.5, 1.5) node {$L_{3}$};
	\draw (0, 1.5) node {$L_{2}$};
	\draw (2.5, 1.5) node {$L_{1}$};
	\draw (4.75, 1.5) node {$L_{0}$};
	
	\draw (0, 0) node {$\mathcal{L}$};
	
	\draw (-4.25, -1.5) node {$L'_{0}$};
	\draw (-1.5, -1.5) node {$L'_{1}$};
	\draw (1.5, -1.5) node {$L'_{2}$};
	\draw (4.25, -1.5) node {$L'_{3}$};

	\draw (-6, 0) node {$y^{-}$};
	\draw (-3.75, 2) node {$(x^{3, e}, y^{3, e})$};
	\draw (-1.25, 2) node {$x^{3}$};
	\draw (1.25, 2) node {$x^{2}$};
	\draw (3.5, 2) node {$x^{1}$};

	\draw (6, 0) node {$(x^{+}, y^{+})$};
	
	\draw (-3, -2) node {$y^{1}$};
	\draw (0, -2) node {$y^{2}$};
	\draw (3, -2) node {$y^{3}$};

\end{tikzpicture}
\caption{the quilted map defining the homotopy between two module-valued functors} \label{fig: the quilted map defining the homotopy between two module-valued functors}
\end{figure}
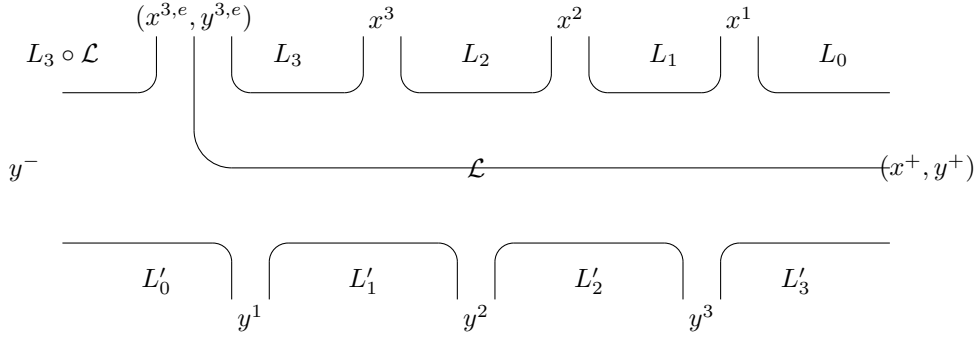

	Following the argument in section \ref{section: Kuranishi structure on moduli spaces of stable pearly tree maps}, we can construct Kuranishi structures on these moduli spaces, which are compatible with the fiber product Kuranishi structures at the boundary. By making a coherent choice of multisections on these Kuranishi spaces, we obtain the virtual fundamental chains, which give rise to the desired homotopy equivalence between the $A_{\infty}$-modules $i \circ \Psi_{\mathcal{L}}$ and $\Phi_{\mathcal{L}}$. \par 

	It is also possible to extend $\Phi_{\mathcal{L}}, \Psi_{\mathcal{L}}$ as well as $\Theta_{\mathcal{L}}$ to the immersed wrapped Fukaya category $\mathcal{W}_{im}(M)$, which we omit as it is not needed for our current purposes. \par

\subsection{Categorification of the functors}
	So far we have discussed the representability for the $A_{\infty}$-functor \eqref{functor to modules} associated to a single Lagrangian correspondence $\mathcal{L} \subset M^{-} \times N$. It is straightforward to generalize this functorially in the wrapped Fukaya category of the product manifold $M^{-} \times N$. That is, we ask whether \eqref{bimodule-valued functor} is representable. The answer is affirmative, stated in Theorem \ref{functoriality for Lagrangian correspondences}. \par

\begin{proof}[Proof of Theorem \ref{functoriality for Lagrangian correspondences}]
	The $A_{\infty}$-functor \eqref{bimodule-valued functor} defines in a natural way an $A_{\infty}$-functor
\begin{equation}
\Phi: \mathcal{W}(M^{-} \times N) \to func(\mathcal{W}(M), \mathcal{W}(N)^{l-mod}).
\end{equation}
This statement is proved in a purely algebraic way in section \ref{section: bimodules and functors}. Summarizing the argument, we compose the $A_{\infty}$-functor \eqref{bimodule-valued functor}  with the algebraically-defined $A_{\infty}$-functor 
\begin{equation}
(\mathcal{W}(M), \mathcal{W}(N))^{bimod} \to func(\mathcal{W}(M), \mathcal{W}(N)^{l-mod})
\end{equation}
to obtain the desired functor. \par
	In the previous two subsections, we have proved that if the projection $\mathcal{L} \to N$ is proper and if Assumption \ref{assumption on the geometric composition} holds for every $L \in Ob \mathcal{W}(M)$, then the filtered module-valued functor $\Phi_{\mathcal{L}}$ is representable. Therefore, $\Phi$ is representable over the full subcategory $\mathcal{A}(M^{-} \times N) \subset \mathcal{W}(M^{-} \times N)$, in the sense of Definition \ref{definition of representability over a subcategory}. Thus we may rewrite the above $A_{\infty}$-functor as
\begin{equation}
\Psi: \mathcal{A}(M^{-} \times N) \to func(\mathcal{W}(M), \mathcal{W}_{im}(N)^{rep-l-mod}).
\end{equation}
Composing this with a homotopy inverse
\begin{equation*}
\lambda_{l}: \mathcal{W}_{im}(N)^{rep-l-mod} \to \mathcal{W}_{im}(N)
\end{equation*}
of the left Yoneda functor
\begin{equation*}
\mathfrak{y}_{l}: \mathcal{W}_{im}(N) \to \mathcal{W}_{im}(N)^{rep-l-mod},
\end{equation*}
we obtain the desired $A_{\infty}$-functor $\Theta$ \eqref{A-infinity functor from product to functor category}. Technically speaking, we shall require that Assumption \ref{assumption on the geometric composition} hold for every Lagrangian correspondence $\mathcal{L}$ in the sub-category $\mathcal{A}(M^{-} \times N)$, which again is a generic condition on the class of objects of the wrapped Fukaya category of the product manifold $M^{-} \times N$. \par
\end{proof}

\subsection{A geometric realization of the cochain map for the correspondence functor} \label{section: a different geometric realization of the correspondence functor}
	In practice, it is helpful to have a more direct and geometric construction of the functor \eqref{functor to the immersed category}, without referring to the implicit construction with the help of the Yoneda lemma. At this time there are still some technical issues in fully realizing this, but it is possible to construct a cochain map, which is homotopic to the first order map of \eqref{functor to the immersed category}. This construction is useful in some applications, for example when studying the relation to the Viterbo restriction functor. \par
	Fix an admissible Lagrangian correspondence $\mathcal{L} \subset M^{-} \times N$ such that the projection $\mathcal{L} \to N$ is proper. Suppose $L \subset M$ is an admissible Lagrangian submanifold, which can be made as an object of $\mathcal{W}(M)$. Recall that the geometric composition $\iota: L \circ \mathcal{L} \to N$ comes with a canonical and unique bounding cochain $b$. 
Given a pair $(L_{0}, L_{1})$, we define a map
\begin{equation} \label{a different realization of the correspondence functor}
\Pi_{\mathcal{L}}: CW^{*}(L_{0}, L_{1}; H_{M}) \to CW^{*}((L_{0} \circ \mathcal{L}, b_{0}), (L_{1} \circ \mathcal{L}, b_{1}); H_{U})
\end{equation}
in the following way. Consider the moduli spaces
\begin{equation*}
\mathcal{U}_{l_{0}, l_{1}}(\alpha, \beta; x; y; y_{0, 1}, \cdots, y_{0, l_{0}}; y_{1, 1}, \cdots, y_{1, l_{1}}; e_{0}, e_{1})
\end{equation*}
of quilted inhomogeneous pseudoholomorphic maps $(\underline{S}, (u, v))$ in $(M, N)$, with the following properties:
\begin{enumerate}[label=(\roman*)]

\item The quilted surface $\underline{S} = (S_{0}, S_{1}$ has two patches. $S_{0}$ is a disk with $3$ punctures $z_{0}^{1}, z_{0}^{1, -}, z_{0}^{2, -}$, where $z_{0}^{1, -}, z_{0}^{2, -}$ are special punctures. $S_{1}^{k}$ is a disk with $3 + l_{0} + l_{1}$ punctures $z_{1}^{1, -}, z_{1}^{0, 1}, \cdots, z_{1}^{0, l_{0}}, z_{1}^{0}, z_{1}^{1, 1}, \cdots, z_{1}^{1, l_{1}}, z_{1}^{2, -}$, which are ordered in a counterclockwise order on the boundary, where $z_{1}^{1, -}, z_{1}^{2, -}$ are special punctures. 
The quilted surface is obtained by seaming the two patches along the boundary component $I_{0}^{-}$ of $S_{0}$ between $z_{0}^{1, -}, z_{0}^{2, -}$ and the boundary component $I_{1}^{-}$ of $S_{1}$ between $z_{1}^{1, -}, z_{1}^{2, -}$. Here, we regard the punctures $z_{0}^{1}$ and $z_{1}^{0}$ as being fixed, while $z_{1}^{0, 1}, \cdots, z_{1}^{0, l_{0}}$ and $z_{1}^{1, 1}, \cdots, z_{1}^{1, l_{1}}$ are allowed to move.

\item $u: S_{0} \to M$ is inhomogeneous pseudoholomorphic with respect to $(H_{S_{0}}, J_{S_{0}})$, for a family of Hamiltonians $H_{S_{0}}$ on $M$ parametrized by $S_{0}$, which agrees with $H_{M}$ near $z_{0}^{1}$, and a family of almost complex structures $J_{S_{0}}$ parametrized by $S_{0}$, which agrees with $J_{M}$ near $z_{0}^{1}$.

\item $v: S_{1} \to N$ is inhomogeneous pseudoholomorphic with respect to $(H_{S_{1}}, J_{S_{1}})$, for a family of Hamiltonians $H_{S_{1}}$, which agrees with $H_{U}$ near each of the punctures $z_{1}^{0, 1}, \cdots, z_{1}^{0, l_{0}}, z_{1}^{0}, z_{1}^{1, 1}, \cdots, z_{1}^{1, l_{1}}$, and a family of almost complex structures $J_{S_{1}}$, which agrees with $J_{U}$ each of the punctures $$z_{1}^{0, 1}, \cdots, z_{1}^{0, l_{0}}, z_{1}^{0}, z_{1}^{1, 1}, \cdots, z_{1}^{1, l_{1}}.$$

\item $u$ maps the boundary component of $S_{0}$ between $z_{0}^{1, -}$ and $z_{0}^{1}$ to $L_{0}$, the boundary component between $z_{0}^{j}$ and $z_{0}^{j+1}$ to $L_{j}$ (for $j = 1, \cdots, k-1$), and the boundary component between $z_{0}^{k}$ and $z_{0}^{2, -}$ to $L_{k}$.

\item $v$ maps the boundary component of $S_{1}$ between $z_{1}^{1, -}$ and $z_{1}^{0, 1}$, the boundary component between $z_{1}^{0, j}$ and $z_{1}^{0, j+1}$ (for $j = 1, \cdots, l_{0}-1$) as well as the boundary component between $z_{1}^{0, l_{0}}$ and $z_{1}^{0}$ to the image of the geometric composition $L_{0} \circ_{H_{M}} \mathcal{L}$. $v$ maps the boundary component between $z_{1}^{0}$ and $z_{1}^{1, 1}$, the boundary component between $z_{1}^{1, j}$ and $z_{1}^{1, j+1}$ (for $j = 1, \cdots, l_{1}-1$) as well as the boundary component between $z_{1}^{1, l_{1}}$ and $z_{1}^{2, -}$ to the image of the geometric composition $L_{1} \circ_{H_{M}} \mathcal{L}$.

\item On the seam, the matching condition for $(u, v)$ is given by the Lagrangian correspondence $\mathcal{L}$.

\item $u$ asymptotically converges to some time-one $H_{M}$-chord $x$ at $z_{0}^{1}$.

\item $v$ asymptotically converges to some generator $y$ for $CW^{*}(L_{0} \circ \mathcal{L}, L_{1} \circ \mathcal{L})$ at $z_{1}^{0}$. In the case where $y$ is a time-one $H_{N}$-chord from the image of $L_{0} \circ \mathcal{L}$ to that of $L_{1} \circ \mathcal{L}$, this condition is the same as those for $u$. In the case where $y$ is a critical point (this happens only when $L_{0} = L_{1}$), the domain $S_{1}$ and the map $v$ have to be slightly modified, to be described later on.

\item $v$ asymptotically converges to some generator $y_{0, j}$ of $CW^{*}(L_{0} \circ \mathcal{L}; H_{N})$ at $z_{1}^{0, j}$ for $j = 1, \cdots, l_{0}$, and to some generator $y_{1, j}$ of $CW^{*}(L_{1} \circ \mathcal{L}; H_{N})$ at $z_{1}^{1, j}$ for $j = 1, \cdots, l_{1}$.

\item Over the first quilted end, the quilted map $(u, v)$ asymptotically converges to the cyclic element $e_{0}$ for $(L_{0}, \mathcal{L}, L_{0} \circ \mathcal{L})$. Over the second quilted end, the quilted map $(u, v)$ asymptotically converges to the cyclic element $e_{1}$ for $(L_{1}, \mathcal{L}, L_{1} \circ \mathcal{L})$.

\end{enumerate}
Now let us describe the necessary modification when $L_{0} = L_{1}$ and $y$ is a critical point of the chosen Morse function on the self fiber product of $L_{0} \circ \mathcal{L}$. In this case, we require that the family of Hamiltonians $H_{S_{1}}$ is chosen so that it vanishes near the strip-like end of $z_{1}^{0}$, so that the map $v$ converges to a point on the image of $L_{1} \circ \mathcal{L}$. The domain $S_{1}$ should also be further modified, by attaching an infinite half ray $(-\infty, 0]$ to it at the negative puncture (which now becomes a marked point as the Hamiltonian vanishes near there). Then we require that the lift of the map on the infinite half ray, which is a gradient flow, converges to $y$ at $-\infty$. To make the statement concise and unified, in both cases we shall briefly say that the map $v$ asymptotically converges to $y$. \par
	There is a natural stable map compactification of this moduli space, denoted by
\begin{equation}\label{moduli space of quilted maps for another realization of the correspondence functor}
\bar{\mathcal{U}}_{l_{0}, l_{1}}(\alpha, \beta; x; y; y_{0, 1}, \cdots, y_{0, l_{0}}; y_{1, 1}, \cdots, y_{1, l_{1}}; e_{0}, e_{1}),
\end{equation}
which is constructed in an inductive nature. This compactification is obtained by adding all possible broken inhomogeneous pseudoholomorphic quilted maps. These broken quilted maps arise when energy escapes over the strip-like ends near the punctures (this phenomenon is often called strip breaking), or when the domains degenerate. There are several cases:
\begin{enumerate}[label=(\roman*)]

\item Inhomogeneous pseudoholomorphic disks bubbling off the boundary of the image of $L_{0} \circ \mathcal{L}$. The resulting broken quilted map has a main component which is similar to such a quilted map $(\underline{S}, (u, v))$, with possibly less punctures $l'_{0} \le l_{0}$, and some other components consisting of trees of inhomogeneous pseudoholomorphic disks with boundary on the image of $L_{0} \circ \mathcal{L}$.

\item Inhomogeneous pseudoholomorphic disks bubbling off the boundary of the image of $L_{1} \circ \mathcal{L}$. The resulting broken quilted map has a main component which is similar to such a quilted map $(\underline{S}, (u, v))$, with possibly less punctures $l'_{1} \le l_{1}$, and some other components consisting of trees of inhomogeneous pseudoholomorphic disks with boundary on the image of $L_{1} \circ \mathcal{L}$.

\item Inhomogeneous pseudoholomorphic strips breaking out at the strip-like end near $z_{0}^{1}$.

\item Inhomogeneous pseudoholomorphic strips breaking out at the strip-like end near $z_{1}^{0}$.

\item Inhomogeneous pseudoholomorphic quilted strips breaking out at the quilted ends.

\end{enumerate}
Thus, there is an isomorphism of the codimension-one boundary strata of the compactified moduli space:
\begin{equation}\label{boundary strata of the moduli space of quilted maps defining a different realization of the cochain map}
\begin{split}
& \partial \bar{\mathcal{U}}_{l_{0}, l_{1}}(\alpha, \beta; x; y; y_{0, 1}, \cdots, y_{0, l_{0}}; y_{1, 1}, \cdots, y_{1, l_{1}}; e_{0}, e_{1})\\
\cong & \coprod_{\substack{x_{1} \\ \deg(x_{1}) = \deg(x) + 1}}
\bar{\mathcal{M}}(x_{1}, x)\\
& \times \bar{\mathcal{U}}_{l_{0}, l_{1}}(\alpha, \beta; x_{1}; y; y_{0, 1}, \cdots, y_{0, l_{0}}; y_{1, 1}, \cdots, y_{1, l_{1}}; e_{0}, e_{1})\\
& \cup \coprod_{\substack{\alpha' \sharp \alpha'' = \alpha \\ \beta'' \sharp \beta'' = \beta}} \coprod_{\substack{y_{1} \\ \deg(y_{1}) = \deg(y) - 1}}
\bar{\mathcal{M}}(\alpha'', \beta''; y, y_{1})\\
& \times \bar{\mathcal{U}}_{l_{0}, l_{1}}(\alpha', \beta'; x; y_{1}; y_{0, 1}, \cdots, y_{0, l_{0}}; y_{1, 1}, \cdots, y_{1, l_{1}}; e_{0}, e_{1})\\
& \cup \coprod_{\substack{1 \le l'_{0} \le l_{0} \\ l'_{0} + l''_{0} = l_{0} + 1}} \coprod_{0 \le i_{0} \le l'_{0}} \coprod_{\substack{\alpha' \sharp \alpha'' = \alpha \\ \beta'' \sharp \beta'' = \beta}} \coprod_{\substack{y_{0, new}\\ \deg(y_{0, new}) = \deg(y_{0, i_{0}+1}) + \cdots + \deg(y_{0, i_{0} + l''_{0}}) + 2 - l''_{0}}}\\
& \bar{\mathcal{U}}_{l'_{0}, l_{1}}(\alpha', \beta'; x; y; y_{0, 1}, \cdots, y_{0, i_{0}}, y_{0, new}, y_{0, i_{0} + l''_{0} + 1}, \cdots, y_{0, l_{0}}; y_{1, 1}, \cdots, y_{1, l_{1}}; e_{0}, e_{1})\\
& \times \bar{\mathcal{M}}_{l''_{0}+1}(\alpha'', \beta''; y_{0, new}; y_{0, i_{0} + 1}, \cdots, y_{0, i_{0} + l''_{0}})\\
& \cup \coprod_{\substack{1 \le l'_{1} \le l_{1} \\ l'_{1}+l''_{1} = l_{1} + 1}} \coprod_{1 \le i_{1} \le l'_{1}} \coprod_{\substack{\alpha' \sharp \alpha'' = \alpha \\ \beta'' \sharp \beta'' = \beta}} \coprod_{\substack{y_{1, new}\\ \deg(y_{1, new}) = \deg(y_{1, i_{1}+1}) + \cdots + \deg(y_{1, i_{1}+l''_{1}}) + 2 - l''_{1}}}\\
& \bar{\mathcal{U}}_{l_{0}, l'_{1}}(\alpha', \beta'; x; y; y_{0, 1}, \cdots, y_{0, l_{0}}; y_{1, 1}, \cdots, y_{1, i_{1}}, y_{1, new}, y_{1, i_{1} + l''_{1} + 1}, \cdots, y_{1, l_{1}})\\
& \times \bar{\mathcal{M}}_{l''_{1}+1}(\alpha'', \beta''; y_{1, new}; y_{1, i_{1}+1}, \cdots, y_{1, i_{1}+l''_{1}})\\
& \cup \coprod_{\substack{0 \le l'_{0} \le l_{0} \\ l'_{0} + l''_{0} = l_{0}}} \coprod_{\substack{\alpha' \sharp \alpha'' = \alpha \\ \beta'' \sharp \beta'' = \beta}} \coprod_{\substack{(x_{0, new}, y_{0, new})\\ \deg((x_{0, new}, y_{0, new})) = \deg(e_{0}) + 1}}\\
& \bar{\mathcal{N}}_{l''_{0}}(\alpha'', \beta''; y_{0, 1}, \cdots, y_{0, l''_{0}}; (x_{0, new}, y_{0, new}), e_{0})\\
& \times \bar{\mathcal{U}}_{l'_{0}, l_{1}}(\alpha', \beta'; x; y; y_{0, l''_{0}+1}, \cdots, y_{0, l_{0}}; y_{1, 1}, \cdots, y_{1, l_{1}}; (x_{0, new}, y_{0, new}), e_{1})\\
& \cup \coprod_{\substack{0 \le l'_{1} \le l_{1} \\ l'_{1}+l''_{1} = l_{1}}} \coprod_{\substack{\alpha' \sharp \alpha'' = \alpha \\ \beta'' \sharp \beta'' = \beta}} \coprod_{\substack{(x_{1, new}, y_{1, new})\\ \deg((x_{1, new}, y_{1, new})) = \deg(e_{1}) + 1}}\\
& \bar{\mathcal{U}}_{l_{0}, l'_{1}}(\alpha', \beta'; x; y; y_{0, 1}, \cdots, y_{0, l_{0}}; y_{1, 1}, \cdots, y_{1, l'_{1}}; e_{0}, (x_{1, new}, y_{1, new}))\\
& \times \bar{\mathcal{N}}_{l''_{1}}(\alpha'', \beta''; y_{1, l'_{1}+1}, \cdots, y_{1, l_{1}}; (x_{1, new}, y_{1, new}), e_{1}).
\end{split}
\end{equation}
Some notations need to be explained. Here 
\begin{equation*}
\bar{\mathcal{U}}_{l'_{0}, l_{1}}(\alpha', \beta'; x; y; y_{0, l''_{0}+1}, \cdots, y_{0, l_{0}}; y_{1, 1}, \cdots, y_{1, l_{1}}; (x_{0, new}, y_{0, new}), e_{1})
\end{equation*}
is the moduli space of quilted maps of the same kind, except that the asymptotic convergence condition at the quilted end $(z_{0}^{1, -}, z_{1}^{1, -})$ is replaced by a new generalized chord $(x_{0, new}, y_{0, new})$ for $(L_{0}, \mathcal{L}, L_{0} \circ_{H_{M}} \mathcal{L})$; similarly for 
\begin{equation*}
\bar{\mathcal{U}}_{l_{0}, l'_{1}}(\alpha', \beta'; x; y; y_{0, 1}, \cdots, y_{0, l_{0}}; y_{1, 1}, \cdots, y_{1, l'_{1}}; e_{0}, (x_{1, new}, y_{1, new})).
\end{equation*}
And 
\begin{equation*}
\bar{\mathcal{N}}_{l''_{0}}(\alpha'', \beta''; y_{0, 1}, \cdots, y_{0, l''_{0}}; (x_{0, new}, y_{0, new}), e_{0})
\end{equation*}
is the moduli space of broken decorated inhomogeneous pseudoholomorphic quilted strips connecting $(x_{0, new}, y_{0, new})$ and $e_{0}$, with punctures on the boundary of the first patch; similarly for 
\begin{equation*}
\bar{\mathcal{N}}_{l''_{1}}(\alpha'', \beta''; y_{1, l'_{1}+1}, \cdots, y_{1, l_{1}}; (x_{1, new}, y_{1, new}), e_{1}).
\end{equation*} \par
	When the virtual dimension is zero, the virtual fundamental chains of the moduli spaces $\bar{\mathcal{U}}_{l_{0}, l_{1}}(\alpha, \beta; x; y; y_{0, 1}, \cdots, y_{0, l_{0}}; y_{1, 1}, \cdots, y_{1, l_{1}}; e_{0}, e_{1})$ give rise to multilinear maps
\begin{equation}
\begin{split}
a_{l_{0}, l_{1}}: & CW^{*}(L_{0} \circ \mathcal{L}; H_{N})^{\otimes l_{0}} \otimes CW^{*}(L_{0}, L_{1}; H_{M}) \otimes CW^{*}(L_{1} \circ \mathcal{L}; H_{N})^{\otimes l_{1}}\\
& \to CW^{*}(L_{0} \circ \mathcal{L}, L_{1} \circ \mathcal{L}; H_{N})).
\end{split}
\end{equation}
By specializing $y_{0, j} = b_{0}$ and $y_{1, j} = b_{1}$, we define
\begin{equation}
\Pi_{\mathcal{L}}(x) = \sum_{l_{0}, l_{1} = 0}^{\infty} a_{l_{0}, l_{1}}(\underbrace{b_{0}, \cdots, b_{0}}_{l_{0} \text{ times }}; x; \underbrace{b_{1}, \cdots, b_{1}}_{l_{1} \text{ times }}).
\end{equation} 
This is the definition of the map $\Pi_{\mathcal{L}}$. \par
	The main observation is that the map $\Pi_{\mathcal{L}}$ is homotopic to the cochain map $\Theta_{\mathcal{L}}^{1}$. \par

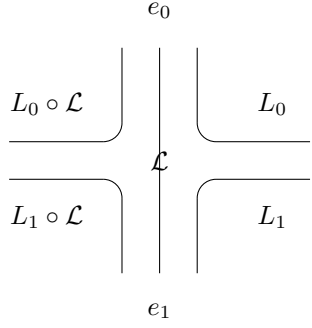
\begin{figure}
\begin{tikzpicture}
	\draw (-2, 0.25) -- (-0.75, 0.25);
	\draw (-0.75, 0.25) arc (270:360:0.25cm);
	\draw (-0.5, 0.5) -- (-0.5, 1.5);
	
	\draw (-2, -0.25) -- (-0.75, -0.25);
	\draw (-0.75, -0.25) arc (90:0:0.25cm);
	\draw (-0.5, -0.5) -- (-0.5, -1.5);
	
	\draw (0, 1.5) -- (0, -1.5);
	
	\draw (0.5, 1.5) -- (0.5, 0.5);
	\draw (0.5, 0.5) arc (180:270:0.25cm);
	\draw (0.75, 0.25) -- (2, 0.25);
	
	\draw (0.5, -1.5) -- (0.5, -0.5);
	\draw (0.5, -0.5) arc (180:90:0.25cm);
	\draw (0.75, -0.25) -- (2, -0.25);

	\draw (-1.5, 0.75) node {$L_{0} \circ \mathcal{L}$};
	\draw (-1.5, -0.75) node {$L_{1} \circ \mathcal{L}$};
	
	\draw (0, 0) node {$\mathcal{L}$};
	
	\draw (1.5, 0.75) node {$L_{0}$};
	\draw (1.5, -0.75) node {$L_{1}$};

	\draw (0, 2) node {$e_{0}$};
	\draw (0, -2) node {$e_{1}$};

\end{tikzpicture}
\centering
\caption{the quilted map defining the cochain map}\label{fig: the quilted map defining the cochain map}
\end{figure}

\begin{figure}
\begin{tikzpicture}
	
	\draw (-2.75, 3) -- (-2.75, 2);
	\draw (-2.75, 2) arc (360:270:0.25cm);
	\draw (-4, 1.75) -- (-3, 1.75);
	
	\draw (-4, 1.25) -- (-3, 1.25);
	\draw (-3, 1.25) arc (90:0:0.25cm);
	\draw (-2.75, 1) -- (-2.75, 0.5);
	\draw (-2.75, 0.5) arc (360:270:0.25cm);
	\draw (-3, 0.25) -- (-4, 0.25);
	
	\draw (-4, -0.25) -- (-3, -0.25);
	\draw (-3, -0.25) arc (90:0:0.25cm);
	\draw (-2.75, -0.5) -- (-2.75, -1);
	\draw (-2.75, -1) arc (360:270:0.25cm);
	\draw (-3, -1.25) -- (-4, -1.25);
	
	\draw (-4, -1.75) -- (-3, -1.75);
	\draw (-3, -1.75) arc (90:0:0.25cm);
	\draw (-2.75, -2) -- (-2.75, -3);
	
	\draw (-2.25, 3) -- (-2.25, -3);
	
	\draw (-1.75, 3) -- (-1.75, 0.5);
	\draw (-1.75, 0.5) arc (180:270:0.25cm);
	\draw (-1.5, 0.25) -- (0, 0.25);
	
	\draw (-1.75, -3) -- (-1.75, -0.5);
	\draw (-1.75, -0.5) arc (180:90:0.25cm);
	\draw (-1.5, -0.25) -- (0, -0.25);

	\draw (-3.5, 2.5) node {$L'(0)$};
	\draw (-3.5, 0.75) node {$L'(1)$};
	\draw (-3.5, -0.75) node {$L'(2)$};
	\draw (-3.5, -2.5) node {$L'(3)$};
	
	\draw (-2.25, 0) node {$\mathcal{L}$};
	
	\draw (-1, 1) node {$L_{0}$};
	\draw (-1, -1) node {$L_{1}$};
	
\end{tikzpicture}
\centering
\caption{the quilted map defining the linear term of the module-valued functor} \label{fig: the quilted map defining the linear term of the module-valued functor}
\end{figure}
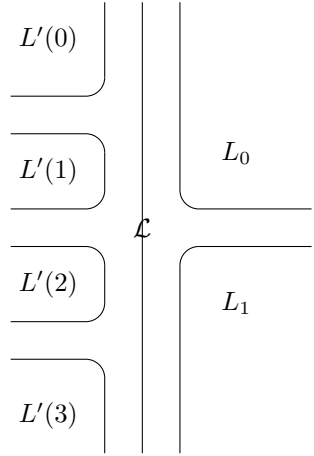

\begin{figure}
\begin{tikzpicture}

	\draw (-2.75, 3) -- (-2.75, 2);
	\draw (-2.75, 2) arc (360:270:0.25cm);
	\draw (-4, 1.75) -- (-3, 1.75);
	
	\draw (-4, 1.25) -- (-3, 1.25);
	\draw (-3, 1.25) arc (90:0:0.25cm);
	\draw (-2.75, 1) -- (-2.75, 0.5);
	\draw (-2.75, 0.5) arc (360:270:0.25cm);
	\draw (-3, 0.25) -- (-4, 0.25);
	
	\draw (-4, -0.25) -- (-3, -0.25);
	\draw (-3, -0.25) arc (90:0:0.25cm);
	\draw (-2.75, -0.5) -- (-2.75, -1);
	\draw (-2.75, -1) arc (360:270:0.25cm);
	\draw (-3, -1.25) -- (-4, -1.25);
	
	\draw (-4, -1.75) -- (-3, -1.75);
	\draw (-3, -1.75) arc (90:0:0.25cm);
	\draw (-2.75, -2) -- (-2.75, -3);
	
	\draw (-2.25, 3) -- (-2.25, 0.75);
	\draw (-2.25, 0.75) arc (180:270:0.25cm);
	\draw (-2, 0.5) -- (0, 0.5);
	
	\draw (-2.25, -3) -- (-2.25, -0.75);
	\draw (-2.25, -0.75) arc (180:90:0.25cm);
	\draw (-2, -0.5) -- (0, -0.5);
	
	\draw (-3.5, 2.5) node {$L'(0)$};
	\draw (-3.5, 0.75) node {$L'(1)$};
	\draw (-3.5, -0.75) node {$L'(2)$};
	\draw (-3.5, -2.5) node {$L'(3)$};
	
	\draw (-1, 0.75) node {$L_{0} \circ \mathcal{L}$};
	\draw (-1, -0.75) node {$L_{1} \circ \mathcal{L}$};
	
\end{tikzpicture}
\centering
\caption{inhomogeneous pseudoholomorphic disk defining the module structure on the Yoneda module} \label{fig: inhomogeneous pseudoholomorphic disk defining the module homomorphism of the Yoneda modules}
\end{figure}
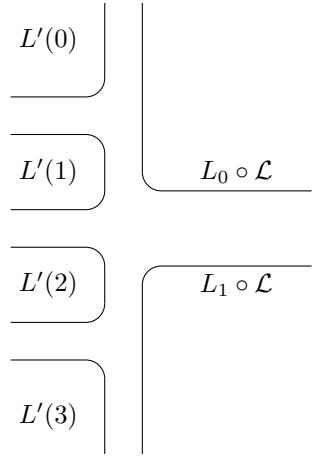

\begin{figure}
\begin{tikzpicture}
	
	\draw (-2.75, 3) -- (-2.75, 2);
	\draw (-2.75, 2) arc (360:270:0.25cm);
	\draw (-4, 1.75) -- (-3, 1.75);
	
	\draw (-4, 1.25) -- (-3, 1.25);
	\draw (-3, 1.25) arc (90:0:0.25cm);
	\draw (-2.75, 1) -- (-2.75, 0.5);
	\draw (-2.75, 0.5) arc (360:270:0.25cm);
	\draw (-3, 0.25) -- (-4, 0.25);
	
	\draw (-4, -0.25) -- (-3, -0.25);
	\draw (-3, -0.25) arc (90:0:0.25cm);
	\draw (-2.75, -0.5) -- (-2.75, -1);
	\draw (-2.75, -1) arc (360:270:0.25cm);
	\draw (-3, -1.25) -- (-4, -1.25);
	
	\draw (-4, -1.75) -- (-3, -1.75);
	\draw (-3, -1.75) arc (90:0:0.25cm);
	\draw (-2.75, -2) -- (-2.75, -3);
	
	\draw (-2.25, 3) -- (-2.25, 1);
	\draw (-2.25, 1) arc (180:270:0.25cm);
	\draw (-2, 0.75) -- (0, 0.75);
	\draw (0, 0.75) arc (270:360:0.25cm);
	\draw (0.25, 1) -- (0.25, 3);
	
	\draw (-2.25, -3) -- (-2.25, -1);
	\draw (-2.25, -1) arc (180:90:0.25cm);
	\draw (-2, -0.75) -- (0, -0.75);
	\draw (0, -0.75) arc (90:0:0.25cm);
	\draw (0.25, -1) -- (0.25, -3);
	
	\draw (0.75, 3) -- (0.75, -3);
	
	\draw (1.25, 3) -- (1.25, 0.5);
	\draw (1.25, 0.5) arc (180:270:0.25cm);
	\draw (1.5, 0.25) -- (2.5, 0.25);
	
	\draw (1.25, -3) -- (1.25, -0.5);
	\draw (1.25, -0.5) arc (180:90:0.25cm);
	\draw (1.5, -0.25) -- (2.5, -0.25);
	
	\draw (-3.5, 2.5) node {$L'(0)$};
	\draw (-3.5, 0.75) node {$L'(1)$};
	\draw (-3.5, -0.75) node {$L'(2)$};
	\draw (-3.5, -2.5) node {$L'(3)$};
	
	\draw (-1, 1.25) node {$L_{0} \circ \mathcal{L}$};
	\draw (-1, -1.25) node {$L_{1} \circ \mathcal{L}$};

	\draw (0.75, 0) node {$\mathcal{L}$};
	
	\draw (2, 1) node {$L_{0}$};
	\draw (2, -1) node {$L_{1}$};

	\draw (0.75, 3.5) node {$e_{0}$};
	\draw (0.75, -3.5) node {$e_{1}$};
	
\end{tikzpicture}
\centering
\caption{the quilted map for the composition of the cochain map with the Yoneda functor} \label{fig: the quilted map for the composition of the cochain map with the Yoneda functor}
\end{figure}
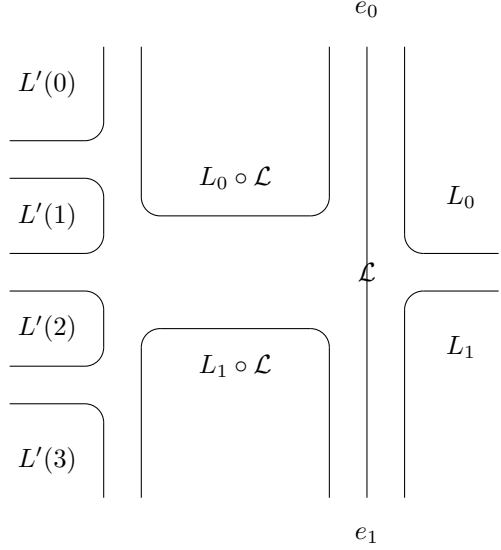

\begin{proposition}
	The map $\Pi_{\mathcal{L}}$ is a cochain map with respect to the usual Floer differential on $CW^{*}(L_{0}, L_{1}; H_{M})$ and the $(b_{0}, b_{1})$-deformed differential on $CW^{*}(L_{0} \circ \mathcal{L}, L_{1} \circ \mathcal{L}; H_{N})$. \par	
	Moreover, this map is chain homotopic to the first order map $\Theta_{\mathcal{L}}^{1}$ of the $A_{\infty}$-functor $\Theta_{\mathcal{L}}$.
\end{proposition}
\begin{proof}
	The first statement follows from the facts that the $b_{0}$-deformed and $b_{1}$-deformed $A_{\infty}$-structures have vanishing zeroth order terms, as well as the fact that the cyclic elements $e_{0}, e_{1}$ are closed under the $b_{0}$-deformed and respectively the $b_{1}$-deformed quilted Floer differentials. The underlying geometric idea is as described below. In order to get the correct "count" of elements, we shall impose the restricted asymptotic convergence conditions $y_{0, j} = b_{0}$ at $z_{1}^{0, j}$ and $y_{1, j} = b_{1}$ at $z_{1}^{1, j}$. The meaning of this is as follows. If $b_{0}$ (resp. $b_{1}$) is a formal linear combination of generators of $CW^{*}(L_{0} \circ \mathcal{L})$ (resp. $CW^{*}(L_{1} \circ \mathcal{L})$), then we consider these moduli spaces whose asymptotic convergence conditions are given by all such generators, and take the corresponding formal linear combination of the virtual fundamental chains with the same coefficients as those for $b_{0}$ and $b_{1}$. The numerical effect is that when inserting the bounding cochains $b_{0}$ and $b_{1}$ for the corresponding asymptotic convergence conditions, the total contributions from any kind of the following moduli spaces $\bar{\mathcal{M}}_{l''_{0}+1}(\alpha'', \beta''; y_{0, new}; y_{0, i_{0} + 1}, \cdots, y_{0, i_{0} + l''_{0}})$, or $\bar{\mathcal{M}}_{l''_{1}+1}(\alpha'', \beta''; y_{1, new}; y_{1, i_{1}+1}, \cdots, y_{1, i_{1}+l''_{1}})$, or $\bar{\mathcal{N}}_{l''_{0}}(\alpha'', \beta''; y_{0, 1}, \cdots, y_{0, l''_{0}}; (x_{0, new}, y_{0, new}), e_{0})$, or $\bar{\mathcal{N}}_{l''_{1}}(\alpha'', \beta''; y_{1, l'_{1}+1}, \cdots, y_{1, l_{1}}; (x_{1, new}, y_{1, new}), e_{1})$ are all zero. In other words, no quilted strip breaking can occur numerically. Now, by looking at \eqref{boundary strata of the moduli space of quilted maps defining a different realization of the cochain map}, we note that the only non-trivial contributions are from strip breaking at the strip-like ends near $z_{0}^{1}$ or $z_{1}^{0}$, and therefore conclude that $\Pi_{\mathcal{L}}$ is a cochain map. \par
	The proof of the second statement is based upon the proof of representability of the module-valued functor $\Phi_{\mathcal{L}}$. We shall compose this map $\Pi_{\mathcal{L}}$ with the Yoneda functor and compare the resulting map to the linear term of module-valued functor $\Phi_{\mathcal{L}}$. This will be done by analyzing how the relevant inhomogeneous pseudoholomorphic quilted maps can be related. Recall that the module-valued functor $\Phi_{\mathcal{L}}$ defines for each object $L$ of $\mathcal{W}(M)$ an $A_{\infty}$-module $\Phi_{\mathcal{L}}(L)$ over $\mathcal{W}_{im}(N)$, with the module structure maps defined by "counting" elements moduli spaces of quilted inhomogeneous pseudoholomorphic maps $(u^{0}, v^{0})$, whose first component $u^{0}$ in $M$ is a strip with one boundary component mapped to $L$, and second component $v^{0}$ is a disk with several punctures, and the boundary components are mapped to objects of $\mathcal{W}_{im}(N)$, which are testing objects for the $A_{\infty}$-module $\Phi_{\mathcal{L}}$. Floer cochains in $CW^{*}(L_{0}, L_{1}; H_{M})$ give rise to pre-module homomorphisms from $\Phi_{\mathcal{L}}(L_{0})$ to $\Phi_{\mathcal{L}}(L_{1})$, which are defined by using moduli spaces of quilted inhomogeneous pseudoholomophic maps $(u^{1}, v^{1})$, as shown in Figure \ref{fig: the quilted map defining the linear term of the module-valued functor}. These quilted maps are similar to the previous ones $(u^{0}, v^{0})$, but satisfying somewhat different conditions:
\begin{enumerate}[label=(\roman*)]

\item One boundary component of the domain of $u^{1}$ is mapped to $L_{0}$ and the other to $L_{1}$, with the prescribed Floer cochain in $CW^{*}(L_{0}, L_{1}; H_{M})$ as the asymptotic convergence condition at the puncture in between the two boundary components.

\item The boundary components of the domain of $v^{1}$ are mapped to testing objects of $\mathcal{W}_{im}(N)$. For instance, when defining the $d$-th order map
\begin{equation}
\begin{split}
(\Phi^{1}_{\mathcal{L}})^{d}: & CW^{*}(L_{0}, L_{1}; H_{M})\\
& \to \hom(CW^{*}(L_{0}, \mathcal{L}, (L'(0), b(0))) \otimes CW^{*}((L'(0), b(0)), (L'(1), b(1))) \otimes \cdots \\
& \otimes CW^{*}((L'(d-2), b(d-2)), (L'(d-1), b(d-1))),\\
& CW^{*}(L_{1}, \mathcal{L}, (L'(d-1), b(d-1))),
\end{split}
\end{equation}
the boundary conditions for $v^{1}$ are given by $(L'(0), b(0)), \cdots, (L'(d-1), b(d-1))$. With bounding cochains included, these are understood as follows. There are additional several punctures on these boundary components, whose asymptotic convergence conditions are given by the bounding cochains $b(j)$.
Here $(L'(j), b(j))$ are general testing objects, which do not have to be the geometric compositions of $L_{0}$ or $L_{1}$ with $\mathcal{L}$.

\item The asymptotic convergence conditions over the two quilted ends are the input and output for the modules $\Phi_{\mathcal{L}}(L_{0})$ and $\Phi_{\mathcal{L}}(L_{1})$.

\end{enumerate}	\par

	On the other hand, recall that the proof of the representability of the module-valued functor $\Phi_{\mathcal{L}}$ uses the moduli spaces $\bar{\mathcal{C}}(x'; (x, y); e)$ \eqref{moduli space of quilted maps defining the first order term of the geometric composition map} and $\bar{\mathcal{C}}_{d}(x'; x'_{1}, \cdots, x'_{d-1}; (x, y); e)$ \eqref{moduli space of quilted maps defining the higher order terms of the geometric composition map}, whose count yields the homotopy equivalence \eqref{geometric composition map}. Such a quilted map in $\bar{\mathcal{C}}_{d}(x'; x'_{1}, \cdots, x'_{d-1}; (x, y); e)$ converges to the cyclic element $e$ of the corresponding triple $(L, \mathcal{L}, L \circ_{H_{M}} \mathcal{L})$ over one quilted end. Algebraically, the virtual fundamental chains of these moduli spaces give rise to a homotopy equivalence from the module $\Phi_{\mathcal{L}}(L)$ to the Yoneda module $\mathfrak{y}_{l}((L \circ \mathcal{L}, b))$ for every $L$.
Now glue the other quilted end of two copies of such quilted maps $(u_{0}, v_{0})$ and $(u_{1}, v_{1})$ in the moduli spaces $\bar{\mathcal{C}}(x'; (x, y); e)$ (one for the triple $(L_{0}, \mathcal{L}, L_{0} \circ \mathcal{L})$ and the other for $(L_{1}, \mathcal{L}, L_{1} \circ_{H_{M}} \mathcal{L})$) to one quilted inhomogeneous pseudoholomorphic map $(u^{1}, v^{1})$, along the quilted ends over which $(u_{i}, v_{i})$ have a common asymptotic convergence condition with $(u^{1}, v^{1})$. The new quilted map $(\tilde{u}, \tilde{v})$ has two quilted ends, over which the asymptotic convergence conditions are given by the cyclic elements $e_{i}$ for the triple $(L_{i}, \mathcal{L}, L_{i} \circ \mathcal{L})$. Also, there are two new punctures on the domain of $\tilde{v}$, and correspondingly two new boundary components which are mapped to the image of the geometric composition $L_{0} \circ \mathcal{L}$ and $L_{1} \circ \mathcal{L}$ respectively. Over the strip-like ends of these two new punctures, the map $\tilde{v}$ asymptotically converges to some generator of $CW^{*}((L_{0} \circ \mathcal{L}, b_{0}), (L'(0), b(0)))$ and one of $CW^{*}((L_{1} \circ \mathcal{L}, b_{1}), (L'(d-1), b(d-1)))$., respectively. \par
	Now let us look at the composition of $\Pi_{\mathcal{L}}$ with the linear term of Yoneda functor
\begin{equation}
\begin{split}
\mathfrak{y}_{l} \circ \Pi_{\mathcal{L}}: & CW^{*}(L_{0}, L_{1}; H_{M})\\
& \to CW^{*}((L_{0} \circ \mathcal{L}, b_{0}), (L_{1} \circ \mathcal{L}, b_{1}); H_{N})\\
& \to \hom(\mathfrak{y}_{l}((L_{0} \circ \mathcal{L}, b_{0})), \mathfrak{y}_{l}((L_{1} \circ \mathcal{L}, b_{1}))),
\end{split}
\end{equation}
which concretely consists of multilinear maps
\begin{equation}
\begin{split}
\mathfrak{y}_{l} \circ \Pi_{\mathcal{L}}: & CW^{*}(L_{0}, L_{1}; H_{M})\\
& \to CW^{*}((L_{0} \circ \mathcal{L}, b_{0}), (L_{1} \circ \mathcal{L}, b_{1}); H_{N})\\
& \to \hom(CW^{*}(L_{0}, \mathcal{L}, (L'(0), b(0))) \otimes CW^{*}((L'(0), b(0)), (L'(1), b(1))) \otimes \cdots\\
& \otimes CW^{*}((L'(d-2), b(d-2)), (L'(d-1), b(d-1))),\\
& CW^{*}(L_{1}, \mathcal{L}, (L'(d-1), b(d-1))),
\end{split}
\end{equation}
Such maps are defined using moduli spaces of inhomogeneous pseudoholomorphic quilted maps which are obtained from gluing those quilted maps $(u, v)$ in the moduli spaces
\begin{equation*}
\bar{\mathcal{U}}_{l_{0}, l_{1}}(\alpha, \beta; x; y; y_{0, 1}, \cdots, y_{0, l_{0}}; y_{1, 1}, \cdots, y_{1, l_{1}}; e_{0}, e_{1}),
\end{equation*}
with those inhomogeneous pseudoholomorphic maps which are used to define pre-module homomorphisms between the Yoneda modules. The inhomogeneous pseudoholomorphic maps of the latter kind are just ordinary $A_{\infty}$-disks used to define the $A_{\infty}$-structure on $\mathcal{W}_{im}(N)$, because the module structure maps and the pre-module homomoprhisms for Yoneda modules are precisely given by the original $A_{\infty}$-structure maps for the corresponding objects in $\mathcal{W}_{im}(N)$, as shown in Figure \ref{fig: inhomogeneous pseudoholomorphic disk defining the module homomorphism of the Yoneda modules}. After gluing, we get a quilted map as shown in Figure \ref{fig: the quilted map for the composition of the cochain map with the Yoneda functor}. \par
	It is then an immediate observation that such a quilted map is of the same type as the previously constructed one $(\tilde{u}, \tilde{v})$. Recall that the quilted maps $(u_{0}, v_{0})$ and $(u_{1}, v_{1})$ are shown in Figure \ref{fig: the quilted map defining the geometric composition map}.
Gluing the two quilted maps $(u_{0}, v_{0})$ and $(u_{1}, v_{1})$ to $(u^{1}, v^{1})$ would replace the two quilted ends of $(u^{1}, v^{1})$ by two new quilted ends, over which the new map $(\tilde{u}, \tilde{v})$ asymptotically converges to the cyclic elements $e_{0}$ and $e_{1}$ respectively, and moreover create two new punctures on the component $\tilde{v}$, over which $\tilde{v}$ asymptotically converges to generators of $CW^{*}((L_{0} \circ \mathcal{L}, b_{0}), (L'(0), b(0)))$ and $CW^{*}((L_{1} \circ \mathcal{L}, b_{1}), (L'(d-1), b(d-1)))$. Comparing this to Figure \ref{fig: the quilted map for the composition of the cochain map with the Yoneda functor}, we see that they are of the same type. \par
	Such an argument can be generalized to the compactified moduli spaces in a straightforward way, by repeating the same process for broken quilted maps. It is clear that this process respects the structure of the boundary strata as described in \eqref{boundary strata of the moduli space of quilted maps defining a different realization of the cochain map}. Translating the result algebraically, we conclude that the map $\mathfrak{y}_{l} \circ \Pi_{\mathcal{L}}$ is chain homotopic to $\Phi_{\mathcal{L}}^{1}$.
Since $\Theta_{\mathcal{L}}$ represents the module-valued functor $\Phi_{\mathcal{L}}$, we conclude that the map $\Pi_{\mathcal{L}}$ is chain homotopic to $\Theta_{\mathcal{L}}^{1}$. Therefore the proof of the second statement is complete. \par
\end{proof}

\subsection{A K\"{u}nneth formula for the wrapped Fukaya category} \label{section: Kunneth formula}
	As in classical Floer cohomology theory, it is natural to expect that the wrapped Fukaya category of the product manifold $\mathcal{W}(M \times N)$ can be expressed as a tensor product of $\mathcal{W}(M)$ and $\mathcal{W}(N)$. This is some kind of K\"{u}nneth formula in wrapped Floer theory. As an application of the construction of functors discussed before, in particular the bimodule-valued functors, we can phrase this in a precise way, where we take the statement in \cite{Ganatra-Pardon-Shende}. \par
	There are some issues regarding the statement of the K\"{u}nneth formula. First, $\mathcal{W}(M \times N)$ generally has more objects than product Lagrangian submanifolds $L \times L'$, so in general we cannot expect this to be equivalent to the tensor product. However, passing to the split-closure gives some hope of establishing this kind of equivalence, provided that $\mathcal{W}(M \times N)$ is split-generated by product Lagrangian submanifolds. \par
	Second, the definition of $A_{\infty}$-tensor product of $A_{\infty}$-categories is delicate. Rather than giving a definition by certain "universal property", we shall take a particular model of the $A_{\infty}$-tensor product $\mathcal{W}(M) \otimes \mathcal{W}(N)$ so that the construction of bimodule-valued functors associated to Lagrangian correspondences can be untilized in the proof. For the definition of the $A_{\infty}$-tensor product, we follow the construction of \cite{Saneblidze-Umble}, in which a diagonal for the Stasheff associahedra is constructed. In particular, in this model of $A_{\infty}$-tensor product $\mathcal{W}(M) \otimes \mathcal{W}(N)$, the objects are pairs $(L, L')$ of Lagrangian submanifolds of $M$ and $N$ respectively (or formal products $L \times L'$, the underlying morphism spaces are the usual tensor products of wrapped Floer complexes
\begin{equation}
\hom((L_{0}, L'_{0}), (L_{1}, L'_{1})) = CW^{*}(L_{0}, L_{1}) \otimes CW^{*}(L'_{0}, L'_{1}),
\end{equation}
and the first order structure map is the tensor product wrapped Floer differential:
\begin{equation}
m^{1} = m^{1}_{L_{0}, L_{1}} \otimes id + id \otimes m^{1}_{L'_{0}, L'_{1}}.
\end{equation}
However, we remark that all the different constructions are quasi-isomorphic (see \cite{Markl-Shnider}). \par

\begin{proposition} \label{Kunneth formula under non-degeneracy assumption}
	Assume both $M$ and $N$ are non-degenerate, i.e. their wrapped Fukaya categories satisfy Abouzaid's generation criterion for finite collections of Lagrangian submanifolds. Then there is a canonical quasi-equivalence:
\begin{equation} \label{inverse Kunneth functor}
\mathcal{W}(M \times N)^{perf} \to (\mathcal{W}(M) \otimes \mathcal{W}(N))^{perf}.
\end{equation}
\end{proposition}

	The outline of the proof goes as follows. First, non-degeneracy implies that there are finite collections of Lagrangian submanifolds $L_{1}, \cdots, L_{k}$ of $M$ and $L'_{1}, \cdots, L'_{l}$ of $N$ which split-generate their wrapped Fukaya categories, and moreover, the products $L_{i} \times L'_{j}$ split-generate $\mathcal{W}(M \times N)$. Thus it will be enough to consider the full $A_{\infty}$-subcategory $\mathcal{P}$ of $\mathcal{W}(M \times N)$ whose objects are products $L_{i} \times L'_{j}$. 
Second, the framework of Lagrangian correspondence gives us a bimodule-valued functor
\begin{equation}\label{bimodule-valued functor on the subcategory of product Lagrangians}
\mathcal{P} \to (\mathcal{W}(M^{-}), \mathcal{W}(N))^{bimod},
\end{equation}
which is defined at the beginning of this section. 
Third, there is a canonical algebraically defined Yoneda-type $A_{\infty}$-functor
\begin{equation}\label{a Yoneda-type bimodule-valued functor}
\mathcal{W}(M) \otimes \mathcal{W}(N) \to (\mathcal{W}(M^{-}), \mathcal{W}(N))^{bimod},
\end{equation}
to the dg-category of $A_{\infty}$-bimodules over $(\mathcal{W}(M^{-}), \mathcal{W}(N))$. An appropriate version of Yoneda lemma says that this is cohomologically full and faithful, thus is a quasi-isomorphism onto the image. 
Fourth, note that the image of $\mathcal{P}$ under the $A_{\infty}$-functor \eqref{bimodule-valued functor on the subcategory of product Lagrangians} lands in the image of $\mathcal{W}(M) \otimes \mathcal{W}(N)$ under the $A_{\infty}$-functor \eqref{a Yoneda-type bimodule-valued functor} in the dg-category of bimodules. Inverting the functor \eqref{a Yoneda-type bimodule-valued functor} on the image gives us an $A_{\infty}$-functor
\begin{equation}
\mathcal{P} \to \mathcal{W}(M) \otimes \mathcal{W}(N).
\end{equation}
Finally, using a direct argument by analyzing the moduli space of inhomogeneous pseudoholomorphic quilted strips, we can easily prove the K\"{u}nneth formula on the level of cohomology.
This isomorphism on cohomology can be rephrased as the statement that the above $A_{\infty}$-functor is a quasi-isomorphism, which allows us to compare the images of $\mathcal{P}$ and $\mathcal{W}(M) \otimes \mathcal{W}(N)$ in the dg-category of bimodules $(\mathcal{W}(M^{-}), \mathcal{W}(N))^{bimod}$, and show that they are quasi-isomorphic.
Passing to the split-closure, we get a quasi-equivalence:
\begin{equation}
\mathcal{W}(M \times N)^{perf} \to (\mathcal{W}(M) \otimes \mathcal{W}(N))^{perf}.
\end{equation} \par
	While this is an outline, a complete detailed proof only requires careful writing out the formulas for the algebraic structures involving $A_{\infty}$-tensor products that we use here. The reader is referred to \cite{Saneblidze-Umble} and \cite{Markl-Shnider} for detailed account of $A_{\infty}$-tensor products and related algebraic results. \par

\begin{remark}
	In fact, a chain-level $A_{\infty}$-quasi-equivalence between (the split-closure of) the split wrapped Fukaya category $\mathcal{W}_{s}(M \times N)$ and (the split-closure of) $A_{\infty}$-tensor product of (suitable dg-replacements of) $\mathcal{W}(M)$ and $\mathcal{W}(N)$ was already established in \cite{Ganatra}. That is, Proposition \ref{Kunneth formula under non-degeneracy assumption} was proved there for the split wrapped Fukaya category $\mathcal{W}^{s}(M \times N)$ of the product manifold, with objects being products of exact cylindrical Lagrangian submanifolds of individual factors.
\end{remark}

	Without assuming non-degeneracy of $M$ or $N$, there is a general version of K\"{u}nneth formula for the wrapped Fukaya category. \par

\begin{proposition} \label{Kunneth formula in general}
	There is a natural cohomologically fully faithful $A_{\infty}$-bifunctor (or say $A_{\infty}$-bilinear functor)
\begin{equation} \label{Kunneth bifunctor}
\mathfrak{Kun}: \mathcal{W}(M) \times \mathcal{W}(N) \to \mathcal{W}(M \times N).
\end{equation}
This determines a unique cohomologically fully faithful $A_{\infty}$-functor
\begin{equation} \label{Kunneth functor}
\mathfrak{Kun}': \mathcal{W}(M) \otimes \mathcal{W}(N) \to \mathcal{W}(M \times N),
\end{equation}
up to homotopy.
\end{proposition}
\begin{proof}
	Define $\mathcal{W}(M \times N)$ using the split model $\mathcal{W}^{s}(M \times N)$. Then the $A_{\infty}$-bifunctor \eqref{Kunneth bifunctor} can be defined in a straightforward way. On the level of objects, the bifunctor simply takes a pair $(L, L')$ to their direct product,
\begin{equation*}
(L, L') \mapsto L \times L'.
\end{equation*}
On the level of morphism spaces, the bifunctor
\begin{equation*}
CW^{*}(L_{0}, L_{1}) \times CW^{*}(L'_{0}, L'_{1}) \to CW^{*}(L_{0} \times L'_{0}, L_{1} \times L'_{1})
\end{equation*}
is defined as follows. For any basic Floer cochains $x \in CW^{*}(L_{0}, L_{1})$ and $x' \in CW^{*}(L'_{0}, L'_{1})$ (here basic means that the Floer cochain is represented by a single Hamiltonian chord), the product $x \times x'$ is naturally a Floer cochain in $CW^{*}(L_{0} \times L'_{0}, L_{1} \times L'_{1})$ since the latter is defined with respect to the split Hamiltonian.
Higher-order terms are of the form
\begin{equation} \label{higher order terms of the Kunneth bifunctor}
\begin{split}
\mathfrak{Kun}^{k, l}: &(CW^{*}(L_{k-1}, L_{k}) \otimes \cdots \otimes CW^{*}(L_{0}, L_{1}))\\
& \times (CW^{*}(L'_{l-1}, L'_{l}) \otimes \cdots \otimes CW^{*}(L'_{0}, L'_{1})) \to CW^{*}(L_{0} \times L'_{0}, L_{k} \times L'_{l}),
\end{split}
\end{equation}
which are defined as follows. Let $x_{i} \in CW^{*}(L_{i-1}, L_{i})$ and $x'_{j} \in CW^{*}(L'_{j-1}, L'_{j})$ be basic Floer cochains. The image of $(x_{k} \otimes \cdots \otimes x_{1}) \times (x'_{l} \otimes \cdots \otimes x'_{1})$ under the map \eqref{higher order terms of the Kunneth bifunctor} is defined by counting inhomogeneous pseudoholomorphic (generalized) quilted maps of the following kind:

\begin{figure}
\begin{tikzpicture}
	\draw (-5, 0.5) -- (-3.75, 0.5);
	\draw (-3.75, 0.5) arc (270:360:0.25cm);
	\draw (-3.5, 0.75) -- (-3.5, 1.25);
	\draw (-3, 1.25) -- (-3, 0.75);
	\draw (-3, 0.75) arc (180:270:0.25cm);
	\draw (-2.75, 0.5) -- (-1.5, 0.5);
	\draw (-1.5, 0.5) arc (270:360:0.25cm);
	\draw (-1.25, 0.75) -- (-1.25, 1.25);
	\draw (-0.75, 1.25) -- (-0.75, 0.75);
	\draw (-0.75, 0.75) arc (180:270:0.25cm);
	\draw (-0.5, 0.5) -- (0.75, 0.5);
	\draw (0.75, 0.5) arc (270:360:0.25cm);
	\draw (1, 0.75) -- (1, 1.25);
	\draw (1.5, 1.25) -- (1.5, 0.75);
	\draw (1.5, 0.75) arc (180:270:0.25cm);
	\draw (1.75, 0.5) -- (3, 0.5);
	\draw (3, 0.5) arc (270:360:0.25cm);
	\draw (3.25, 0.75) -- (3.25, 1.25);
	\draw (3.75, 1.25) -- (3.75, 0.75);
	\draw (3.75, 0.75) arc (180:270:0.25cm);
	\draw (4, 0.5) -- (5.25, 0.5);

	\draw (-5, -0.5) -- (-3.75, -0.5);
	\draw (-3.75, -0.5) arc (90:0:0.25cm);
	\draw (-3.5, -0.75) -- (-3.5, -1.25);
	\draw (-3, -1.25) -- (-3, -0.75);
	\draw (-3, -0.75) arc (180:90:0.25cm);
	\draw (-2.75, -0.5) -- (-1.5, -0.5);
	\draw (-1.5, -0.5) arc (90:0:0.25cm);
	\draw (-1.25, -0.75) -- (-1.25, -1.25);
	\draw (-0.75, -1.25) -- (-0.75, -0.75);
	\draw (-0.75, -0.75) arc (180:90:0.25cm);
	\draw (-0.5, -0.5) -- (0.75, -0.5);
	\draw (0.75, -0.5) arc (90:0:0.25cm);
	\draw (1, -0.75) -- (1, -1.25);
	\draw (1.5, -1.25) -- (1.5, -0.75);
	\draw (1.5, -0.75) arc (180:90:0.25cm);
	\draw (1.75, -0.5) -- (3, -0.5);
	\draw (3, -0.5) arc (90:0:0.25cm);
	\draw (3.25, -0.75) -- (3.25, -1.25);
	\draw (3.75, -1.25) -- (3.75, -0.75);
	\draw (3.75, -0.75) arc (180:90:0.25cm);
	\draw (4, -0.5) -- (5.25, -0.5);

	\draw (-5, 0) -- (-0.25, 0);
	\draw (0.25, 0) -- (5.5, 0);

	\draw (-4.5, 0.75) node {$L_{4}$};
	\draw (-2.25, 0.75) node {$L_{3}$};
	\draw (0, 0.75) node {$L_{2}$};
	\draw (2.25, 0.75) node {$L_{1}$};
	\draw (4.5, 0.75) node {$L_{0}$};
	\draw (-4.5, -0.75) node {$L'_{4}$};
	\draw (-2.25, -0.75) node {$L'_{3}$};
	\draw (0, -0.75) node {$L'_{2}$};
	\draw (2.25, -0.75) node {$L'_{1}$};
	\draw (4.5, -0.75) node {$L'_{0}$};

	\draw (-3.25, 1.5) node {$x_{4}$};
	\draw (-1, 1.5) node {$x_{3}$};
	\draw (1.25, 1.5) node {$x_{2}$};
	\draw (3.5, 1.5) node {$x_{1}$};
	
	\draw (-3.25, -1.5) node {$x'_{4}$};
	\draw (-1, -1.5) node {$x'_{3}$};
	\draw (1.25, -1.5) node {$x'_{2}$};
	\draw (3.5, -1.5) node {$x'_{1}$}; 
	
	\draw (-5.5, 0) node {$e_{L_{4} \times L'_{4}}$};
	\draw (6, 0) node {$e_{L_{0} \times L'_{0}}$};
	
	\draw (0, 0) node {$\alpha$};
	
\end{tikzpicture}
\centering
\caption{the quilted map defining the K\"{u}nneth bifunctor}
\end{figure}
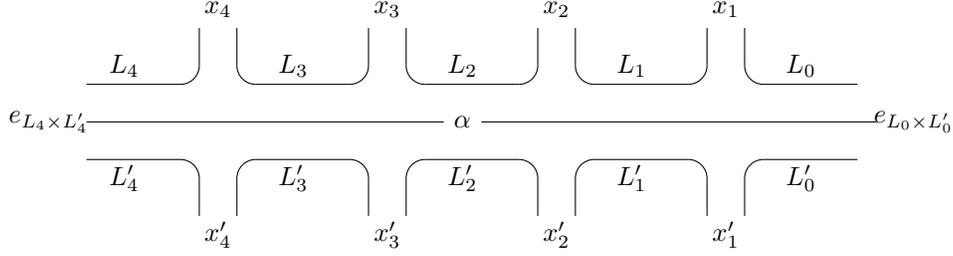

Here the output is $\alpha \in CW^{*}(L_{0} \times L'_{0}, L_{k} \times L'_{l})$, and the asymptotic conditions near the quilted ends are given by the elements
\begin{equation*}
e_{L_{0} \times L'_{0}} \in CW^{*}(L_{0}, L_{0} \times L'_{0}, L'_{0}),
\end{equation*}
and
\begin{equation*}
e_{L_{k} \times L'_{l}} \in CW^{*}(L_{k}, L_{k} \times L'_{l}, L'_{l})
\end{equation*}
which correspond to the homotopy unit
\begin{equation*}
1_{L_{0} \times L'_{0}} \in CW^{*}(L_{0} \times L'_{0}),
\end{equation*}
and respectively
\begin{equation*}
1_{L_{k} \times L'_{l}} \in CW^{*}(L_{k} \times L'_{l}).
\end{equation*}
This construction includes $k = l = 1$ as a special case, as such a quilted map is necessarily constant when $k = l = 1$. \par
	It remains to prove that $\mathfrak{Kun}$ \eqref{Kunneth bifunctor} as defined above is cohomologically fully faithful. This follows immediately from the definition of $\mathfrak{Kun}$, as the first order map is the "identity": sending $(x, x')$ to the product $x \times x'$, which induces an isomorphism on wrapped Floer cohomology groups, proven in \cite{Gao1} as a special case of the results in section \ref{section: product manifolds}. \par
	The statement regarding $\mathfrak{Kun}'$ \eqref{Kunneth functor} follows from the universal property of the $A_{\infty}$-tensor product. \par
\end{proof}

	Now let us go back to the situation where both $M$ and $N$ are non-degenerate and compare the statement of Proposition \ref{Kunneth formula under non-degeneracy assumption} and that of Proposition \ref{Kunneth formula in general}. Non-degeneracy implies that the Kunneth $A_{\infty}$-functor \eqref{Kunneth functor} is invertible after passing to the idempotent completion, and the inverse is given by the quasi-equivalence \eqref{inverse Kunneth functor}. \par


\section{Liouville sub-domains and restriction functors}\label{section: sub-domains and the restriction functors}

\subsection{Sub-domains and restrictions of exact Lagrangian submanifolds}
	Let $M_{0}$ be a Liouville domain, and $M$ its completion. A Liouville sub-domain $U_{0}$ is a codimension zero exact symplectic submanifold of $M_{0}$ with smooth boundary, such that the Liouville vector field on $M_{0}$ points outward the boundary $\partial U_{0}$ transversely. $U_{0}$ is itself a Liouville domain with the induced Liouville structure. We may attach to $U_{0}$ an infinite half-cylinder $\partial U_{0} \times [1, +\infty)$ to the boundary to get a complete Liouville manifold $U$. In this paper, we shall only consider Liouville sub-domains, and sometimes call them sub-domains for short. \par
	Denote by $M^{-}$ the Liouville manifold with opposite symplectic form $-\omega$, with the Liouville form also being the opposite $-\lambda_{M}$, similarly for $U^{-}$. There is a natural Lagrangian correspondence between $M$ and $U$, defined as follows. Because $U_{0}$ is a Liouville sub-domain, the graph of the embedding $U_{0} \subset M_{0}$ is naturally an exact Lagrangian submanifold (with corners) of $U_{0}^{-} \times M_{0}$ with respect to the product exact one-form $-\lambda_{U_{0}} \times \lambda_{M_{0}}$. The natural completion of that in $U^{-} \times M$, is an exact cylindrical Lagrangian submanifold of $U^{-} \times M$, with respect to the contact hypersurface $\Sigma$. We denote this Lagrangian correspondence by $\Gamma$. Alternatively, consider the natural embedding
\begin{equation}
i: U \to M
\end{equation}
induced by the inclusion $U_{0} \subset M_{0}$, whose is definition is given by the following formula:
\begin{equation} \label{the natural embedding of the completion}
i(x)=
\begin{cases}
x, & \text{ if } x \in U_{0},\\
\psi_{M}^{r}(i(y)), & \text{ if } x = (y, r) \in \partial U \times [1, +\infty).
\end{cases}
\end{equation}
Then $\Gamma$ is the graph of $i$. \par
	
	 Equipped with the opposite primitive, it can also be regarded as an exact Lagrangian submanifold of $M^{-} \times U$. To make $\Gamma$ into an admissible Lagrangian correspondence in Floer-theoretic sense, one needs a spin structure on it. A natural spin structure exists because we assume both $M$ and $U$ have vanishing first Chern classes. We shall give this Lagrangian correspondence a special name. \par

\begin{definition}
	The Lagrangian submanifold $\Gamma \subset M^{-} \times U$ equipped with the natural spin structure is called the graph correspondence.
\end{definition}

	Given an exact Lagrangian submanifold $L_{0}$ of $M_{0}$, possibly with non-empty boundary $\partial L_{0}$, the intersection $L'_{0} = L_{0} \cap U_{0}$ is naturally an exact Lagrangian submanifold of $U_{0}$, possibly with non-empty boundary $\partial L'_{0}$, even if $L_{0}$ has empty boundary. \par
	If $L_{0}$ is either a closed exact Lagrangian submanifold, or an exact cylindrical Lagrangian submanifold, it can be naturally completed to an object $L$ of $\mathcal{W}(M)$. However, in general this is not true for $L'_{0}$, as the boundary $\partial L'_{0}$ might not behave nicely. In the next subsection we shall seek geometric assumptions such that $L'_{0}$ can be completed into an object of $\mathcal{W}(U)$, so that we may attempt to define a restriction functor. \par

\subsection{Restriction and the associated functor}
	First, note that the graph correspondence $\Gamma$ is admissible for wrapped Floer theory in the product $M^{-} \times U$. This implies that the module-valued functor $\Phi_{\Gamma}$ is well-defined. More importantly, note that the projection $\Gamma \to U$ is proper, thus the module-valued functor $\Phi_{\Gamma}$ is representable, and represented by the $A_{\infty}$-functor
\begin{equation}\label{the restriction functor associated to the graph}
\Theta_{\Gamma}: \mathcal{W}(M) \to \mathcal{W}_{im}(U).
\end{equation}
In this and next subsections, we shall study this functor in a more specific way, compare it with the Viterbo restriction functor defined in \cite{Abouzaid-Seidel}, and prove Theorem \ref{Viterbo functor as a correspondence functor}. While the construction of the Viterbo restriction functor will not be repeated in this paper, we shall describe elements in the moduli spaces used to define it, in a sketchy way. More detailed definition and generalization of the Viterbo restriction functor will be discussed in the upcoming work \cite{Gao2}. \par

\begin{remark}
	The graph correspondence $\Gamma$ can also be viewed as a Lagrangian correspondence from $U$ to $M$. If we ask whether it defines a functor from $\mathcal{W}(U)$ to $\mathcal{W}_{im}(M)$, the answer is not always yes. This is because the projection $\Gamma \to M$ is not always proper, contrary to $\Gamma \to U$. In that case, the module-valued functor is not necessarily representable in our sense.
\end{remark}

	Consider an object $L$ in $\mathcal{W}(M)$. Suppose it is the completion of an exact cylindrical Lagrangian submanifold $L_{0}$ of $M_{0}$ which intersects $U_{0}$ nicely in the following sense. \par

\begin{assumption}\label{strong exactness condition}
	$\partial L' = L_{0} \cap \partial U_{0}$ is a Legendrian submanifold of $\partial U_{0}$; and moreover, the primitive of $L_{0}$ is locally constant near both $\partial M_{0}$ and $\partial U_{0}$. 
\end{assumption}

	If $L_{0}$ satisfies Assumption \ref{strong exactness condition}, then $L'_{0} = L_{0} \cap U_{0}$ is an exact cylindrical Lagrangian submanifold of $U_{0}$ and can therefore be completed to an exact cylindrical Lagrangian submanifold $L' \subset U$. Because of the geometric nature of the $A_{\infty}$-functor 
\begin{equation*}
\Theta_{\Gamma}: \mathcal{W}(M) \to \mathcal{W}_{im}(U),
\end{equation*}
we expect that it behaves like a restriction functor, and takes $L$ to $L'$ on the level of objects. \par
	However, the true story is a bit more complicated, and in fact the above expectation does not always hold true. In order for the geometric composition $L \circ_{H_{M}} \Gamma$ to be related to $L'$, we further need the following condition: \par
	 
\begin{assumption}\label{invariance assumption}
	$L$ is invariant under the Liouville flow in the intermediate region $M_{0} \setminus int(U_{0})$.
\end{assumption}

	This assumption means that, the Lagrangian $L$ is the Liouville completion of the restriction $L'_{0} = L \cap U_{0}$, in the sense that $L$ is the union of $L'_{0}$ and all trajectories of the Liouville flow on $M$ starting from points on $\partial L'$. \par 

\begin{example}
	Let $M_{0}$ and $U_{0}$ be (critical) Weinstein domains. If $M_{0} \setminus int(U_{0})$ is a Weinstein cobordism, and $L$ is the union of some cylinders over Legendrians and the stable or unstable manifolds of those critical points of a Lyapunov function which are contained in the cobordism $M_{0} \setminus int(U_{0})$, then $L$ satisfies Assumption \ref{invariance assumption}.
\end{example}

	Under Assumption \ref{invariance assumption}, it is easy to see that the geometric composition $L \circ \Gamma$ is precisely $L'$, the completion of the restriction $L'_{0}$ of $L$ to $U_{0}$. In fact, the geometric composition in the usual sense (without Hamiltonian perturbation) is the map
\begin{equation*}
L \circ \Gamma = \{(p_{1}, p_{2}, q) \in M \times M \times U: p_{1} \in L, p_{1} = p_{2}, i(q) = p_{2}\} \to U
\end{equation*}
which sends $(p_{1}, p_{2}, q)$ to $q \in U$. In this case this is an exact Lagrangian embedding, whose image is
\begin{equation*}
\{q \in U: \exists p \in M, \text{ such that } p \in L, i(q) = p\}.
\end{equation*}
By the definition of the map $i: U \to M$ as in \eqref{the natural embedding of the completion}, this is the set of all points $q \in U$ such that either $q \in L \cap U_{0}$, or $q = (y, r) \in \partial U \times [1, +\infty)$ such that $\psi_{M}^{r}(y) = p$. By Assumption \ref{invariance assumption}, this is exactly $L'$. \par

\begin{remark}
	In general, every exact cylindrical Lagrangian $L$ which intersects $\partial U$ to a Legendrian submanifold is invariant under the Liouville flow in a neighborhood of $\partial U$ and one of $\partial M$, after suitable Hamiltonian perturbation. 
\end{remark}

	On the other hand, under assumption Assumption \ref{strong exactness condition}, we can show that the bounding cochain in fact vanishes by Proposition \ref{vanishing of the bounding cochain}, so that the image of the functor $\Theta_{\Gamma}$ lands in the completed wrapped Fukaya category $\mathcal{W}(U)$ whose objects are embedded Lagrangian submanifolds. This is the first half of the statement of Theorem \ref{Viterbo functor as a correspondence functor}. \par
	The full sub-category of $\mathcal{W}(M)$ consisting of Lagrangian submanifolds which satisfy Assumption \ref{strong exactness condition} and Assumption \ref{invariance assumption} will be denoted by $\mathcal{B}_{0}(M)$. This is a full sub-category of $\mathcal{B}(M)$. The restriction of the functor $\Theta_{\Gamma}$ to this sub-category is
\begin{equation}
\Theta_{\Gamma}: \mathcal{B}_{0}(M) \to \mathcal{W}(U),
\end{equation}
such that the image of any $L \in Ob \mathcal{B}_{0}(M)$ is $L' \in Ob \mathcal{W}(U)$. This finishes the proof of the first half of Theorem \ref{Viterbo functor as a correspondence functor}. \par

\subsection{The Viterbo restriction functor revisited}\label{the Viterbo restriction functor}
	To prove the second half of Theorem \ref{Viterbo functor as a correspondence functor}, we shall briefly review the definition of the Viterbo restriction functor first, originally due to \cite{Abouzaid-Seidel}. As the setup of the wrapped Fukaya category is slightly different (but equivalent), we shall consider inhomogeneous pseudoholomorphic maps which behave somewhat differently than cascade maps. However, they should yield an equivalent definition of the Viterbo restriction functor, although we will not attempt to include a proof. \par
	The relevant inhomogeneous pseudoholomorphic maps that are used to define the Viterbo restriction functor will be called "climbing disks". Climbing disks are in fact very similar to cascades introduced in \cite{Abouzaid-Seidel}. Roughly speaking, these are inhomogeneous pseudoholomorphic disks in $M$, for which the Floer data consist of a family of Hamiltonians and a family of almost complex structures interpolating the Floer data on $M$ and those on $U$, and the boundary conditions are given by a family of Lagrangian submanifolds interpolating $L$ and $L'$. In order to visualize Floer theory of $U$ in $M$, we shall use a rescaling trick. More details are presented below. \par
	To define climbing disks, we introduce the following geometric data. Conceptually, we want to consider geometric objects which reflect the way how the Liouville manifolds $M$ and $U$ can be interpolated in a suitable sense. The geometric data on $M$ and those on $U$ should be related in a one-parameter family. While in \cite{Abouzaid-Seidel} they think of the size of the collar neighborhood of $\partial U$ inside $U_{0}$ as the parameter (shrinking the sub-domain by a conformal factor $\rho \in (0, 1]$ which is taken to be the parameter), we fix the size of the collar neighborhood (as in \cite{Viterbo1}), but rather change the height of the Hamiltonian and regard that as a parameter. \par
	Note that the Liouville structure on $U_{0}$ induced by that on $M_{0}$ provides a symplectic embedding of a collar neighborhood $\partial U \times [1-\epsilon, 1]$ to $U_{0}$ such that $\partial U \times \{1\}$ is mapped to the boundary $\partial U$. Suppose we have chosen the quadratic Hamiltonian $H_{M}$ on $M$ and $H_{U}$ on $U$, such that they agree on $U_{0} \subset M_{0}$. 
Define a family of piecewise smooth (and lower semi-continuous) Hamiltonians $K_{A}$ depending on a parameter $A \in [0, +\infty)$,
\begin{equation}
K_{A} =
\begin{cases}
0, & \text{ on } U_{0} \setminus (\partial U \times (1-\epsilon, 1),\\
A(r-1+\epsilon)^{2}, &\text{ if } p = (y, r) \in \partial U \times (1-\epsilon, 1),\\
A, & \text{ on } M_{0} \setminus U_{0},\\
A + h(r), & \text{ if } x = (y, r) \in \partial M \times [1, +\infty),
\end{cases}
\end{equation}
where $h(r)$ is the same radial function as that for $H_{M}$, i.e. $h(r) = r^{2}$ if $r$ is slightly bigger than $1$. Note that when $A = 0$, this is simply an admissible Hamiltonian on $M$ which vanishes in the interior part $M_{0}$ and is quadratic in the cylindrical end $\partial M \times [1, +\infty)$. Let $H_{A}$ be a $C^{2}$-small perturbation of $K_{A}$, which is smooth and non-degenerate, such that $H_{A}$ agrees with $K_{A}$ in the collar neighborhood $\partial U \times [1-\epsilon, 1]$ as well as in the cylindrical end $\partial M \times [1, +\infty)$. \par
	We also need a suitable class of almost complex structures in order to write down the inhomogeneous Cauchy-Riemann equations for climbing disks. Suppose we have chosen almost complex structures $J_{M} = J_{M}$ on $M$ and $J_{U} = J_{U}$ on $U$. What is needed is a family of almost complex structures interpolating these two families, defined below. \par

\begin{definition}\label{interpolating family of almost complex structures}
	An interpolating family of almost complex structures is a family $J_{A}$ of compatible almost complex structures on $M$ parametrized by $A \in [0, +\infty)$ such that the following properties are satisfied:
\begin{enumerate}[label=(\roman*)]
	
\item $J_{0} = J_{M}$;

\item For each $A$, $J_{A} = J_{M}$ outside $U_{0}$;

\item For each $A$, the restriction of $J_{A}$ to $M \setminus U_{0}$ is of contact type in the cylindrical end $\partial M \times [1, +\infty)$;

\item There exists $A_{0} > 0$ such that for $A > A_{0}$, the restriction of $J_{A}$ to $U_{0} \setminus (\partial U \times [1-\epsilon, 1])$ agrees with $(\psi_{U}^{A})_{*}J_{U}$, and the restriction of $J_{A}$ to a collar neighborhood of $\partial U$ in $M$ is of contact type on some neighborhood of $\partial U \times \{1-\epsilon\} \subset \partial U \times [1-\epsilon, 1]$.

\end{enumerate}

\end{definition}

	It is not difficult to show that such interpolating families exist, and in fact that they form a contractible space. \par

	Given an exact cylindrical Lagrangian submanifold $L \subset M$ satisfying Assumption \ref{strong exactness condition}, the completion $L'$ of its restriction $L'_{0} = L \cap U_{0}$ is exact cylindrical with a locally constant primitive near $\partial U$. Call that constant $c_{L}$. We define a slightly shrunk Lagrangian submanifold
\begin{equation}\label{shrinking the Lagrangian submanifold}
L^{1-\epsilon} = 
\begin{cases}
\psi_{U}^{1-\epsilon} L'_{0}, &\text{ on } U_{0} \setminus (\partial U \times [1-\epsilon, 1]),\\
\partial L' \times [1-\epsilon, 1], &\text{ on } \partial U \times [1-\epsilon, 1],\\
L
\end{cases}
\end{equation}
This is well-defined, because $L$ is invariant under the Liouville flow in the collar neighborhood $\partial U \times [1-\epsilon, 1]$. \par
	In Floer theory, when defining inhomogeneous pseudoholomorphic maps, we also need to perturb Hamiltonians and almost complex structures. Our strategy is to perturb almost complex structures, in a way such that they depend on domains of the inhomogeneous pseudoholomorphic maps. If the domain is a strip, then an interpolating family should also depend on an additional parameter $t$. That is, an interpolating family for the strip is a family of almost complex structures $J_{A, t}$, such that for every fixed $t$, $J_{A, t}$ is an interpolating family in the sense of Definition \ref{interpolating family of almost complex structures}. \par
	Given $H_{A}$ and $J_{A, t}$ defined as above, we now introduce a Floer datum on the strip. Such a Floer datum, roughly speaking, is a homotopy between $(H_{M}, J_{M, t})$ and $(H_{A}, J_{A, t})$, parametrized by $s \in \mathbb{R}$. \par

\begin{definition}
	A Floer datum on the strip $Z = \mathbb{R} \times [0, 1]$ consists of the following data:
\begin{enumerate}[label=(\roman*)]

\item A family of Hamiltonians $H_{s}$ depending on $s \in \mathbb{R}$, thought of as a family on the strip $Z$ independent of $t$, such that for $s \gg 0$, $H_{s}$ agrees with $H_{M}$, and for $s \ll 0$, $H_{s}$ agrees with $H_{A}$.

\item A family of almost complex structures $J_{(s, t)}$ parametrized by $(s, t) \in Z$, such that for $s \gg 0$, $J_{(s, t)}$ agrees with $J_{M} = J_{M, t}$, and for $s \ll 0$, $J_{(s, t)}$ agrees with $J_{A, t}$.

\end{enumerate}
\end{definition}

	To define a climbing strip, we also need appropriate Lagrangian boundary conditions. \par

\begin{definition}\label{moving Lagrangian labeling}
	A moving Lagrangian labeling for the strip $Z = \mathbb{R} \times [0, 1]$ is a pair $(L_{0, s}, L_{1, s})$ where each $L_{j, s}$ is an exact Lagrangian isotopy (through exact cylindrical Lagrangian submanifolds) parametrized by $s \in \mathbb{R}$, such that for $s \gg 0$, $L_{j, s} = L_{j}$, and for $s \ll 0$, $L_{j, s} = L_{j}^{1-\epsilon}$ as defined in \eqref{shrinking the Lagrangian submanifold}.
\end{definition}

	Note that there is a natural choice of moving Lagrangian labeling as follows. Let $\lambda: \mathbb{R} \to [1-\epsilon, 1]$ be a smooth non-decreasing function which is $1-\epsilon$ for $s \ll 0$ and is $1$ for $s \gg 0$. Then we can define $L_{j, s} = L_{j}^{\lambda(s)}$ in a similar way to that of \eqref{shrinking the Lagrangian submanifold} (replacing $1-\epsilon$ by $\lambda(s)$). Without loss of generality, we may assume that $\lambda(s) = 1$ for $s \ge 0$. \par

\begin{definition}
	A climbing strip is a smooth map
\begin{equation*}
w: \mathbb{R} \times [0, 1] \to M,
\end{equation*}
with the following properties:
\begin{enumerate}[label=(\roman*)]

\item $w$ satisfies the inhomogeneous Cauchy-Riemann equation:
\begin{equation}\label{inhomogeneous Cauchy-Riemann equation for climbing strips}
\partial_{s} w + J_{(s, t)}(\partial_{t} w - X_{H_{s}}(w)) = 0.
\end{equation}

\item The boundary conditions for $w$ are:
\begin{equation}
w(s, 0) \in L_{0, s}, w(s, 1) \in L_{1, s}.
\end{equation}

\item The asymptotic convergence conditions of $w$ are:
\begin{equation*}
\lim\limits_{s \to -\infty} w(s, \cdot) = x_{A}(\cdot),
\end{equation*}
for a time-one $H_{A}$-chord $x_{A}$ from $L_{0, A}$ to $L_{1, A}$, and
\begin{equation*}
\lim\limits_{s \to +\infty} w(s, \cdot) = x(\cdot),
\end{equation*}
for a time-one $H_{M}$-chord $x$ from $L_{0}$ to $L_{1}$. We require that the $H_{A}$-chord $x_{A}$ satisfy a further condition that
\begin{equation}\label{action constraint for the inner chord}
- A\epsilon^{2} \le \mathcal{A}(x_{A}) \le \delta,
\end{equation}
for a universal small constant $\delta$, whose meaning is to be explained later.

\end{enumerate}

\end{definition}

	There are two conditions listed above which need further explanation. First, the asymptotic convergence conditions should be, as usual, Hamiltonian chords for the Hamiltonian $H_{A, s}$, as $s \to \pm \infty$. When $s \to +\infty$, $H_{A, s} = H_{M}$, and the boundary conditions $(L_{0, s}, L_{1, s})$ agree with $(L_{0}, L_{1})$ when $s \gg 0$, so the asymptotic convergence condition as $s \to +\infty$ is given by a time-one $H_{M}$-chord $x$ from $L_{0}$ to $L_{1}$. When $s \to -\infty$, $H_{A, s} = H_{A}$, and the boundary conditions $(L_{0, s}, L_{1, s})$ agree with $(L_{0}^{1-\epsilon}, L_{1}^{1-\epsilon})$ when $s \ll 0$, so the asymptotic convergence condition as $s \to +\infty$ is given by a time-one $H_{A}$-chord $x_{A}$ from $L_{0}^{1-\epsilon}$ to $L_{1}^{1-\epsilon}$. Now we shall see how such a time-one $H_{A}$-chord is related to a time-one $H_{U}$-chord $x'$ from $L'_{0}$ to $L'_{1}$. Note that inside $U_{0} \setminus (\partial U \times [1-\epsilon, 1])$, the Hamiltonian $K_{A}$ vanishes, so that $H_{A}$ is $C^{2}$-small, which can be taken to agree with $\frac{1}{1-\epsilon} H_{U} \circ \psi_{U}^{1-\epsilon}$ (as $H_{U}$ is also small, the rescaled Hamiltonian is close to $H_{U}$); in the collar neighborhood $\partial U \times [1-\epsilon, 1]$, the Hamiltonian $H_{A}$ is the rescaling of a quadratic Hamiltonian on $U$ by a factor $A$, i.e. $\frac{1}{A} H_{U} \circ \psi_{U}^{A}$; and outside $U_{0}$, $H_{A}$ agrees with $H_{M} + A \epsilon^{2}$. Regarding the Lagrangian boundary conditions, inside $U_{0} \setminus (\partial U \times [1-\epsilon, 1])$ the Lagrangian submanifold $L_{j}^{1-\epsilon}$ agrees with $\psi_{U}^{1-\epsilon}L'_{j, 0}$; and outside $U_{0}$, $L_{j}^{1-\epsilon}$ agrees with $L_{j}$. Thus, time-one $H_{A}$-chords from $L_{0}^{1-\epsilon}$ to $L_{1}^{1-\epsilon}$ inside $U_{0} \setminus (\partial U \times [1-\epsilon, 1])$ naturally correspond to time-one $H_{U}$-chords from $L'_{0}$ to $L'_{1}$ inside $U_{0}$, and time-one Hamiltonian chords of $H_{A}$ which are contained in the collar neighborhood $\partial U \times [1-\epsilon, 1]$ are in one-to-one correspondence with time-one Hamiltonian chords for $H_{U}$ from $L'_{0}$ to $L'_{1}$ restricted to a finite part of the cylindrical end, $\partial U \times [1, A\epsilon + 1]$, on which $H_{U}$  is quadratic with leading coefficient $1$. \par
	The second point that needs explanation is the action constraint \eqref{action constraint for the inner chord}. The reason that we impose the constraint on the action of $x'$ is that for the Hamiltonian $K_{A}$, the Hamiltonian chords are either constants in $U_{0} \setminus (\partial U \times [1-\epsilon, 1])$ or $M_{0} \setminus U_{0})$, and non-constant chords in the collar neighborhood $\partial U \times [1-\epsilon, 1]$ or in the cylindrical end $\partial M \times [1, +\infty)$. For the constant chords, the action is uniformly bounded by a small constant, denoted by $\delta'$, which depends only on the primitives of the Lagrangian submanifolds and can be chosen very small. Thus if we make a small perturbation of $K_{A}$ to $H_{A}$, the Hamiltonian chords which are contained inside $U_{0} \setminus (\partial U \times [1-\epsilon, 1])$ also have action bounded by a small constant $\delta$.
For the non-constant chords, we observe that the restriction of the Hamiltonian vector field $X_{K_{A}}$ to some level $\partial U \times \{r\} \subset \partial U \times [1-\epsilon, 1]$ is $2A(r-1+\epsilon)Y_{\partial U}$, where $Y_{\partial U}$ is the Reeb vector field on the contact boundary $\partial U$ of $U_{0}$. Thus a Hamiltonian chord on level $r \in [1-\epsilon, 1]$ has action $-A(r-1+\epsilon)^{2}$ (there is no extra contribution as the primitive is locally constant there), which is at least $-A\epsilon^{2}$. Thus the action constraint \eqref{action constraint for the inner chord} for the Hamiltonian chord $x_{A}$ for $H_{A}$ is reasonable. By imposing this condition, we mean that for each $A$, we only consider climbing strips $w$ whose asymptotic Hamiltonian chord at $-\infty$ satisfies such constraint on its action. \par
	For each given $A$, the moduli space of such climbing strips is denoted by $\mathcal{P}^{A}(x_{A}, x)$, i.e.
\begin{equation}
\begin{split}
\mathcal{P}^{A}(x_{A}, x) = \{w: 
& w \text{ is a climbing strip with respect to the Floer datum } (H_{A, s}, J_{A, (s, t)}),\\
& \text{ with Lagrangian boundary conditions given by } (L_{0, s}, L_{1, s}),\\
& \text{ with asymptotic convergence conditions } x_{A} \text{ at } -\infty \text{ and } x \text{ at } +\infty\},
\end{split}
\end{equation}
Note that the equation \eqref{inhomogeneous Cauchy-Riemann equation for climbing strips} is not translation invariant, and we are not varying the parameter $A$ in this case, so the virtual dimension of the moduli space $\mathcal{P}^{A}(x_{A}, x)$ is
\begin{equation*}
v-\dim \mathcal{P}^{A}(x_{A}, x) = \deg(x_{A}) - \deg(x).
\end{equation*} \par
	If the virtual dimension is zero, the moduli space $\mathcal{P}^{A}(x_{A}, x)$ is a compact smooth manifold of dimension zero, for generic choices of Floer data. Thus, by counting rigid elements in this moduli space, we may define a map
\begin{equation}
\tilde{r}^{1}_{A}: CW^{*}(L_{0}, L_{1}; H_{M}) \to CW^{*}_{-}(L_{0}^{1-\epsilon}, L_{1}^{1-\epsilon}; H_{A}),
\end{equation}
by
\begin{equation}
\tilde{r}^{1}_{A}(x) = \sum_{x_{A}: \deg(x_{A}) = \deg(x)} \sum_{w \in \mathcal{P}^{A}(x_{A}, x)} o_{w},
\end{equation}
where $o_{w}: o_{x} \to o_{x_{A}}$ is the canonical isomorphism of orientation lines induced by $w$. Here $CW^{*}_{-}(L_{0}^{1-\epsilon}, L_{1}^{1-\epsilon}; H_{A})$ is a sub-complex of the wrapped Floer complex $CW^{*}_{-}(L_{0}^{1-\epsilon}, L_{1}^{1-\epsilon}; H_{A})$, generated by time-one $H_{A}$-chords from $L_{0}^{1-\epsilon}$ to $L_{1}^{1-\epsilon}$ which are contained in $U_{0}$. Because of the natural correspondence between time-one $H_{A}$-chords from $L_{0}^{1-\epsilon}$ to $L_{1}^{1-\epsilon}$ which are contained in $U_{0}$ and time-one $H_{U}$-chords from $L'_{0}$ to $L'_{1}$ inside a finite part of $U$, this space $CW^{*}_{-}(L_{0}^{1-\epsilon}, L_{1}^{1-\epsilon}; H_{A})$ is related to the wrapped Floer complex $CW^{*}(L'_{0}, L'_{1}; H_{U})$ by the following map
\begin{equation}
\tau_{A}: CW^{*}_{-}(L_{0}^{1-\epsilon}, L_{1}^{1-\epsilon}; H_{A}) \to CW^{*}_{(-A\epsilon^{2}, \delta)}(L'_{0}, L'_{1}; H_{U}),
\end{equation}
for some fixed small $\delta > 0$ such that all time-one $H_{U}$-chords have action less than $\delta$ (this can be achieved by choosing the Hamiltonian and the primitives appropriately). This map is a chain-level isomorphism, and in fact sends any time-one $H_{A}$-chord $x_{A}$ to its corresponding time-one $H_{U}$-chord $x'$.
By composing the map $\tilde{r}^{1}_{A}$ with $\tau_{A}$, we get a map
\begin{equation}
r^{1}_{A} = \tau_{A} \circ \tilde{r}^{1}_{A}: CW^{*}(L_{0}, L_{1}; H_{M}) \to CW^{*}_{(-A\epsilon^{2}, \delta)}(L'_{0}, L'_{1}; H_{U}).
\end{equation} \par

	To prove that $r^{1}_{A}$ is a cochain map, we need to study the compactification of the moduli space $\mathcal{P}^{A}(x_{A}, x)$. Because of the asymptotic behavior of the elements $w$, it is natural to introduce a compactification
\begin{equation*}
\bar{\mathcal{P}}^{A}(x_{A}, x)
\end{equation*}
of this moduli space, by adding broken climbing strips. In the codimension one boundary strata, broken climbing strips are of the following two types:
\begin{enumerate}[label=(\roman*)]

\item A pair $(w, u)$, where $w$ is a climbing strip, and $u$ is an inhomogeneous pseudoholomorphic strip in $M$. $w$ and $u$ have a common asymptotic convergence condition $x_{0}$ at the positive end of $w$ and the negative end of $u$. This occurs because as $s \to +\infty$, the family $H_{A, s}$ agrees with $H_{M}$ (and similarly for the almost complex structures and Lagrangian boundary conditions), so when the energy of a sequence of climbing strips escapes from $+\infty$, a $(H_{M}, J_{M})$-pseudoholomorphic strip breaks out.

\item A pair $(u', w)$, where $u'$ is pseudoholomorphic with respect to the Floer datum $(H_{A}, J_{A})$, and $w$ is a climbing strip. $u'$ and $w$ have a common asymptotic convergence condition $x_{A, 0}$ at the positive end of $u'$ and the negative end of $w$. This occurs because as $s \to -\infty$, the family $H_{A, s}$ agrees with $H_{A}$ (and similarly for the almost complex structures and Lagrangian boundary conditions), so when the energy of a sequence of climbing strips escapes from $+\infty$, a $(H_{A}, J_{A})$-pseudoholomorphic strip breaks out. Moreover, for such an inhomogeneous pseudoholomorphic strip, if the input $x_{A, 0}$ is in the sub-complex $CW^{*}_{-}(L_{0}^{1-\epsilon}, L_{1}^{1-\epsilon}; H_{A})$, the output must also be in this sub-complex.

\end{enumerate}
This implies that we have an isomorphism of the codimension one boundary strata:
\begin{equation}\label{boundary strata of moduli space of climbing strips}
\partial \bar{\mathcal{P}}^{A}(x_{A}, x) \cong \coprod_{\substack{x_{0}\\ \deg(x_{0}) = \deg(x)}} \mathcal{P}^{A}(x_{A}, x_{0}) \times \mathcal{M}(x_{0}, x) \cup \coprod_{\substack{x_{A, 0}\\ \deg(x_{A, 0}) = \deg(x_{A})}} \mathcal{M}(x_{A}, x_{A, 0}) \times \mathcal{P}^{A}(x_{A, 0}, x).
\end{equation} 
Here $\mathcal{M}(x_{A}, x_{A, 0})$ is the moduli space of $(H_{A}, J_{A})$-pseudoholomorphic strips with asymptotic convergence conditions $x_{A}, x_{A, 0}$, which can be identified with the moduli space $\mathcal{M}(x', x'_{0})$ of $(H_{U}, J_{U})$-pseudoholomorphic strips in $U$, by a similar rescaling argument. \par
	In order for the above isomorphism to hold, we must assume from now on that we have make conformally consistent choice of Floer data for all kinds of strips. As the Lagrangian submanifolds involved are all exact cylindrical, standard transversality methods allow us to prove that, in the case of virtual dimension being one, the compactified moduli space $\bar{\mathcal{P}}^{A}(x_{A}, x)$ of virtual dimension one is a compact topological manifold of dimension one, if the Floer data are chosen generically. Thus, by a standard gluing argument, based on on the structure of the codimension-one boundary strata of $\bar{\mathcal{P}}^{A}(x_{A}, x)$ as shown in \eqref{boundary strata of moduli space of climbing strips}, we can show: \par

\begin{lemma}
	The map
\begin{equation*}
r^{1}_{A}: CW^{*}(L_{0}, L_{1}; H_{M}) \to CW^{*}_{(-A\epsilon^{2}, \delta)}(L'_{0}, L'_{1}; H_{U})
\end{equation*}
is a cochain map.
\end{lemma}
\begin{proof}
	In addition to the standard gluing argument, one must ensure that no inhomogeneous pseudoholomorphic strip in $U_{0}$ connecting two Hamiltonian chords for $H_{A}$ inside $U_{0}$ escapes from $U_{0}$. This is because that the hypersurface $\partial U \times \{1\}$ is pseudo-convex with respect to the chosen almost complex structure, so we can apply the maximum principle. \par
\end{proof}

	However, the problem is that this depends on an extra parameter $A$. To remove this ambiguity, we introduce the following trick. Note that the action of any $H_{A}$-chord in the collar neighborhood $\partial \times [1-\epsilon, 1]$ satisfies:
\begin{equation*}
\mathcal{A}_{H_{A}}(x_{A}) = -A(r-1+\epsilon)^{2},
\end{equation*}
if $x_{A}$ lies on $\partial U \times \{r\}$. There is no contribution from the primitive $f$, because $f$ is locally constant there, so that $f_{A}(x_{A}(1)) = f_{A}(x_{A}(0))$. For the corresponding $H_{U}$-chord $x'$, the same estimate is satisfied, thus we see that $x'$ is a time-one $H_{U}$-chord contained in $\partial U \times [1, A\epsilon + 1]$. Thus, there exists $\delta > 0$ sufficiently small, such that the map $r^{1}_{A}$ in fact descends to a sub-complex
\begin{equation}
r^{1}_{A}: CW^{*}(L_{0}, L_{1}; H_{M}) \to CW^{*}_{(-A\epsilon^{2}, \delta)}(L'_{0}, L'_{1}; H_{U}).
\end{equation} \par
	To get rid of the dependence of $A$, the idea is to take the direct limit as $A \to +\infty$. In order for such direct limit to exist, we must ensure that these maps are compatible with the natural inclusions
\begin{equation*}
i_{A, A'}: CW^{*}_{(-A\epsilon^{2}, \delta)}(L'_{0}, L'_{1}; H_{U}) \to CW^{*}_{(-A'\epsilon^{2}, \delta)}(L'_{0}, L'_{1}; H_{U}), \text{ for } A < A'.
\end{equation*}

\begin{lemma}\label{linear restriction homomorphism is compatible with inclusion maps of sub-complexes}
	There is a way of choosing families of Hamiltonians when defining the various maps involved, such that for every $A < A'$, the following diagram commutes up to chain homotopy
\begin{equation}
\begin{tikzcd}
CW^{*}(L_{0}, L_{1}; H_{M}) \arrow[r, "\tilde{r}^{1}_{A}"] \arrow[rd, "\tilde{r}^{1}_{A'}"] & CW^{*}_{-}(L_{0}^{1-\epsilon}, L_{1}^{1-\epsilon}; H_{A}) \arrow[r, "\tau_{A}"] \arrow[d, "c_{A, A'}"] & CW^{*}_{(-A\epsilon^{2}, \delta)}(L'_{0}, L'_{1}; H_{U}) \arrow[d, "i_{A, A'}"]\\
	& CW^{*}_{-}(L_{0}^{1-\epsilon}, L_{1}^{1-\epsilon}; H_{A'}) \arrow[r, "\tau_{A'}"] & CW^{*}_{(-A'\epsilon^{2}, \delta)}(L'_{0}, L'_{1}; H_{U})
\end{tikzcd}
\end{equation}
Here
\begin{equation*}
c_{A, A'}: CW^{*}_{-}(L_{0}^{1-\epsilon}, L_{1}^{1-\epsilon}; H_{A}) \to CW^{*}_{-}(L_{0}^{1-\epsilon}, L_{1}^{1-\epsilon}; H_{A'})
\end{equation*}
is the continuation map induced by a monotone homotopy between $H_{A}$ and $H_{A'}$, and
\begin{equation*}
i_{A, A'}: CW^{*}_{(-A\epsilon^{2}, \delta)}(L'_{0}, L'_{1}; H_{U}) \to CW^{*}_{(-A'\epsilon^{2}, \delta)}(L'_{0}, L'_{1}; H_{U})
\end{equation*}
is the natural inclusion map.
\end{lemma}
\begin{proof}
	When defining the continuation map $c_{A, A'}$, we can choose the family of Hamiltonians $H_{A, A', s}$, such that the composition of $H_{A, s}$ and $H_{A, A', s}$ agrees with the family $H_{A', s}$, after identifying the glued strip with $\mathbb{R} \times [0, 1]$ by a suitable reparametrization. Thus, by a standard gluing argument, the left triangle commutes up to chain homotopy, which is unique up to higher homotopies. \par
	To prove the homotopy commutativity of the right square, we need to study the continuation map $c_{A, A'}$ in more details. For this, we first write down a specific choice of the homotopy $H_{A, A', s}$, such that in the collar neighborhood $\partial U \times [1-\epsilon, 1]$, the homotopy takes the form
\begin{equation}\label{specific choice of monotone homotopy of Hamiltonians}
H_{A, A', s}(y, r) = ((1-\lambda(s))A + \lambda(s)A')(r-1+\epsilon)^{2}, (y, r) \in \partial U \times [1-\epsilon, 1],
\end{equation}
where $\lambda: \mathbb{R} \to [0, 1]$ is a smooth non-increasing function, such that $\lambda(s) = 1$ for $s \ll 0$, and $\lambda(s) = 0$ for $s \gg 0$.
Let $x_{A}$ be a time-one $H_{A}$-chord, which corresponds to a time-one $H_{U}$-chord $x'$ under $\tau_{A}$. We shall first prove that, if $A'$ is sufficiently close to $A$, then the image of $x_{A}$ under $c_{A, A'}$ is the unique time-one $H_{A'}$-chord which corresponds to the same time-one $H_{U}$-chord $x'$ under $\tau_{A'}$.
There are two cases to consider:
\begin{enumerate}[label=(\roman*)]

\item $x_{A}$ is a small perturbation of a constant Hamiltonian chord for $K_{A}$. Such a Hamiltonian chord is contained in $U_{0} \setminus (\partial U \times [1-\epsilon, 1])$, where $K_{A}$ and $K_{A'}$ are both zero. Thus we can take their small perturbations $H_{A}$ and $H_{A'}$ such that they agree inside $U_{0} \setminus (\partial U \times [1-\epsilon, 1])$. Since for a constant homotopy of Hamiltonians the continuation map is necessarily the identity, we conclude that $c_{A, A'}(x_{A})$ must be the same Hamiltonian chord as $x_{A}$, now regarded as a Hamiltonian chord for $H_{A'}$.

\item $x_{A}$ lies in the collar neighborhood $\partial U \times [1-\epsilon, 1]$, and corresponds to a non-constant time-one $H_{U}$-chord $x'$ in the cylindrical end. It is a general property shared by continuation maps that the count of rigid continuation strips does not jump except for a discrete set of $A$'s. Thus for $A'$ sufficiently close to $A$, we have that $c_{A, A'}(x_{A}) = x_{A'}$, as $c_{A, A} = id$ is the identity map (on the chain level), as it is the continuation map with respect to the constant homotopy of Hamiltonians.
In particular, for our specific choice of homotopy of Hamiltonians \eqref{specific choice of monotone homotopy of Hamiltonians}, there is an explicit formula for such a continuation strip:
\begin{equation}\label{unique continuation strip}
u_{A, A'}(s, t) = x_{(1-\lambda(s))A + \lambda(s)A'}(t).
\end{equation}
In other words, if $A'$ is sufficiently close to $A$, then the unique continuation strip which asymptotically converges to $x_{A}$ at $+\infty$ is the trace of all Hamiltonian chords of $H_{A''}$ for varying $A'' \in [A, A']$, starting from the given time-one $H_{A}$-chord $x_{A}$.

\end{enumerate}
This proves that if $A'$ is sufficiently close to $A$, the right square strictly commutes on the chain level. \par

	For general $A' > A$, commutativity might not hold on strictly on the chain level, as there might be extra off-diagonal terms appearing in the continuation map. To prove homotopy commutativity we use the following trick. For each $B \in [A, A']$, there is a small $\delta_{B} > 0$ such that for $B' \in (B - \delta_{B}, B + \delta_{B})$, we have that there is a choice of almost complex structure $J_{B'}$ in a small neighborhood of $J_{B}$ homotopic to $J_{B}$, for which the relevant moduli spaces of inhomogeneous pseudoholomorphic strips are regular, and the chain-level continuation map
\begin{equation*}
c_{B', B}: CW^{*}_{-}(L_{0}^{1-\epsilon}, L_{1}^{1-\epsilon}; H_{B'}, J_{B'}) \to CW^{*}_{-}(L_{0}^{1-\epsilon}, L_{1}^{1-\epsilon}; H_{B}, J_{B}),
\end{equation*}
if $B' \le B$, or
\begin{equation*}
c_{B, B'}: CW^{*}_{-}(L_{0}^{1-\epsilon}, L_{1}^{1-\epsilon}; H_{B}, J_{B}) \to CW^{*}_{-}(L_{0}^{1-\epsilon}, L_{1}^{1-\epsilon}; H_{B'}, J_{B'})
\end{equation*}
if $B \le B'$, are diagonal matrices with respect to the natural basis. The intervals $(B - \delta_{B}, B + \delta_{B})$ form an open cover of $[A, A']$, from which we choose a finite sub-cover, $(B_{0} - \delta_{B_{0}}, B_{0} + \delta_{B_{0}}), \cdots, (B_{N} - \delta_{B_{N}}, B_{N} + \delta_{B_{N}})$, where $A < B_{0} < \cdots < B_{N} < A'$. For each $B_{j}$ and $B_{j+1}$, $(B_{j} - \delta_{B_{j}}, B_{j} + \delta_{B_{j}}) \cap (B_{j+1} - \delta_{B_{j+1}}, B_{j+1} + \delta_{B_{j+1}})$ is non-empty, so we can pick a number $B'_{j}$ in this intersection. Then by a standard homotopy argument, the continuation map $c_{B_{j}, B_{j+1}}$ is chain homotopic to the composition $c_{B'_{j}, B_{j+1}} \circ c_{B'_{j}, B_{j+1}}$. With respect to the natural basis (ordered according to action) for these Floer cochain complexes, both $c_{B_{j}, B'_{j}}$ and $c_{B'_{j}, B_{j+1}}$ can be written as matrices of the form $(I, 0)$, where $I$ is the identity matrix (the number of columns of $c_{B_{j}, B'_{j}}$ should be equal to the number of rows of $c_{B'_{j}, B_{j+1}}$). Now the continuation map $c_{A, A'}$ is chain homotopic to the composition $c_{B_{N-1}, B_{N}} \circ \cdots \circ c_{B_{0}, B_{1}}$, and is therefore chain homotopic to a map $\kappa_{A, A'}$ which can be written as a matrix of the form $(I, 0)$. Now it is easy to see that the map $\kappa_{A, A'}$, when replacing $c_{A, A'}$, makes the right square strictly commutes, because we have $\kappa_{A, A'}(x_{A}) = x_{A'}$ by definition, for the time-one $H_{A}$-chord $x_{A}$ and the time-one $H_{A'}$-chord $x_{A'}$ corresponding to the same time-one $H_{U}$-chord $x'$. Thus, for the continuation map $c_{A, A'}$ itself, the right square commutes up to chain homotopy. \par
\end{proof}

	As an immediate consequence of this lemma, we have the following: \par

\begin{corollary}
	The homotopy direct limit map
\begin{equation}\label{linear restriction homomorphism}
r^{1} = \lim\limits_{A \to +\infty} r^{1}_{A}: CW^{*}(L_{0}, L_{1}; H_{M}) \to CW^{*}(L'_{0}, L'_{1}; H_{U})
\end{equation}
is well-defined. It is a cochain map.
\end{corollary}
\begin{proof}
	Simply note that the direct limit $\lim\limits_{A \to +\infty} CW^{*}_{(-A\epsilon^{2}, \delta)}(L'_{0}, L'_{1}; H_{U})$ is homotopy equivalent to the whole wrapped Floer cochain complex $CW^{*}(L_{0}, L_{1}; H_{U})$. \par
	The second item follows immediately from the fact that $r^{1}_{A}$ is a cochain map for every $A$. \par
\end{proof}

	The map \eqref{linear restriction homomorphism} is called the linear Viterbo restriction homomorphism. \par

\subsection{Non-linear terms of the Viterbo restriction functor}
	The higher order terms of the Viterbo restriction functor are defined by counting inhomogeneous pseudoholomorphic maps of similar kind, whose domains are disks with several punctures. These are defined in a slightly different manner, as we shall also include $A$ as a parameter. Let $k \ge 2$ be a positive integer and let $(Z^{k+1}, A) \in \mathcal{S}^{k+1} = \mathcal{R}^{k+1} \times (0, +\infty)$ be an element of (the smooth part of) the multiplihedra, where $Z^{k+1}$ is a disk with $k+1$ boundary punctures. The punctures are labeled by $z_{0}, \cdots, z_{k}$ in a counterclockwise order, where $z_{0}$ is negative, while other punctures are positive. These punctures should come with chosen strip-like ends:
\begin{equation*}
\epsilon_{0}: (-\infty, 0] \times [0, 1] \to Z^{k+1},
\end{equation*}
and
\begin{equation*}
\epsilon_{j}: [0, +\infty) \times [0, 1] \to Z^{k+1},
\end{equation*}
for $j = 1, \cdots, k$. The boundary component between $z_{0}$ and $z_{1}$ is denoted by $I_{0}^{-}$, the boundary component between $z_{0}$ and $z_{k}$ is denoted by $I_{k}^{-}$, and for $j = 1, \cdots, k-1$, the boundary component between $z_{j}$ and $z_{j+1}$ is denoted by $I_{j}^{+}$. \par
	The boundary conditions for $(Z^{k+1}, A)$ is specified by the following definition. \par

\begin{definition}
	A moving Lagrangian label for $(Z^{k+1}, A)$ is a collection of families of Lagrangian submanifolds, one for each boundary component of $Z^{k+1}$. The assignment is as follows:
\begin{enumerate}[label=(\roman*)]

\item Assigned to $I^{+}_{j}$, the family of Lagrangian submanifolds is the constant family $L_{j}$, for $j = 1, \cdots, k-1$;

\item Assigned to $I^{-}_{0}$, the family of Lagrangian submanifolds is the family $L_{0, A, s}$;

\item Assigned to $I^{-}_{k}$, the family of Lagrangian submanifolds is the family $L_{k, A, s}$.

\end{enumerate}
Here $L_{0, A, s}, L_{k, A, s}$ are the exact Lagrangian isotopies introduced in Definition \ref{moving Lagrangian label}.
\end{definition}

	To write down inhomogeneous Cauchy-Riemann equations adapted to our setup, we need to introduce an appropriate class of Floer data on these domains. \par

\begin{definition}
	A Floer datum on $(Z^{k+1}, A)$ consists of the following data:
\begin{enumerate}[label=(\roman*)]

\item A collection of weights $\lambda_{0}, \cdots, \lambda_{k}$.

\item A basic one-form $\beta_{Z^{k+1}, A}$ on $Z^{k+1}$, such that over the $j$-th strip-like ends it agrees with $\lambda_{j} dt$. Here by a basic one-form we mean a sub-closed one-form which vanishes along the boundary of $Z^{k+1}$, and whose differential vanishes in a neighborhood of the boundary of $Z^{k+1}$.

\item A family of Hamiltonians $H_{Z^{k+1}, A}$ depending on $(Z^{k+1}, A)$, such that near the $j$-th strip-like end, it agrees with $\lambda_{j}H_{A}$ up to some conformal rescaling factor $\rho_{j}$;

\item A family of almost complex structures $J_{Z^{k+1}, A}$ depending on $(Z^{k+1}, A)$, such that near the $j$-th strip-like end, it agrees with $J_{A}$ up to the same conformal rescaling factor $\rho_{j}$.

\item A shifting function $\rho_{(Z^{k+1}, A)}: \partial Z^{k+1} \to (0, +\infty)$ which takes the value $\rho_{j}$ on the boundary of the $j$-th strip-like end.

\end{enumerate}
\end{definition}

\begin{definition}
	Suppose we have chosen a moving Lagrangian labeling for $(Z^{k+1}, A)$, as well as a Floer datum on $(Z^{k+1}, A)$. A climbing disk (with $k+1$ punctures) is a triple $(Z^{k+1}, A, w^{k+1})$ where $(Z^{k+1}, A) \in \mathcal{S}^{k+1}$, and
\begin{equation*}
w^{k+1}: Z^{k+1} \to M
\end{equation*}
is a smooth map with the following properties:
\begin{enumerate}[label=(\roman*)]

\item $w^{k+1}$ satisfies the inhomogeneous Cauchy-Riemann equation:
\begin{equation}
(dw - \beta_{Z^{k+1}, A} \otimes X_{H_{Z^{k+1}, A}}(w)) + J_{Z^{k+1}, A} \circ (dw - \beta_{Z^{k+1}, A} \otimes X_{H_{Z^{k+1}}, A}(w)) \circ j_{Z^{k+1}, A} = 0.
\end{equation}

\item The boundary conditions for $w$ are given by the chosen moving Lagrangian labeling $(L_{0, A}, L_{1}, \cdots, L_{k-1}, L_{k, A})$.

\item The asymptotic convergence conditions of $w^{k+1}$ are:
\begin{equation*}
\lim\limits_{s \to -\infty} w^{k+1} \circ \epsilon_{0}(s, \cdot) = x_{0, A}(\cdot)
\end{equation*}
for a time-one $H_{A}$-chord $x_{0, A}$ from $L_{0, A}$ to $L_{k, A}$,
and
\begin{equation*}
\lim\limits_{s \to +\infty} w^{k+1} \circ \epsilon_{j}(s, \cdot) = x_{j}(\cdot)
\end{equation*}
for a time-one $H_{M}$-chord $x_{j}$ from $L_{j-1}$ to $L_{j}$, where $j = 1, \cdots, k$. Similar to the case of climbing strips, we require that the $H_{A}$-chord $x_{A, 0}$ satisfy the following condition
\begin{equation*}
-A\epsilon^{2} \le \mathcal{A}(x_{0, A}) \le \delta.
\end{equation*}

\end{enumerate}
\end{definition}

	The definition of a Floer datum easily generalizes to broken disks in the multiplihedra $\bar{\mathcal{S}}^{k+1}$. There is also a notion of universal and conformally consistent choices of Floer data, similar to that in the case for ordinary $A_{\infty}$-disks and also that in the case of the action-restriction functor. Now let us assume from now on that universal and conformally consistent choices of Floer data have been made for all elements in $\bar{\mathcal{S}}^{k+1}$ and for all $k$. \par
	Note that for each climbing disk $(Z^{k+1}, A, w^{k+1})$, the asymptotic convergence condition $x_{0, A}$ at the negative strip-like end depends on the moduli parameter $A$. In order to define a moduli space of such climbing disks, such that the Hamiltonian chord at the negative strip-like end can be fixed as the output, we must find a way of identifying these $x_{0, A}$ for different $A$. This issue can be resolved using the observation that time-one $H_{A}$-chords inside $U_{0}$ can be identified with certain time-one $H_{U}$-chords. Let $x'_{0}$ be the time-one $H_{U}$-chord corresponding to $x_{0, A}$, under the natural map
\begin{equation*}
\tau_{A}: CW^{*}_{-}(L_{0, A}, L_{k, A}; H_{A}) \to CW^{*}_{(-A\epsilon^{2}, \delta)}(L'_{0}, L'_{k}; H_{U}).
\end{equation*}
As discussed in the proof of Lemma \ref{linear restriction homomorphism is compatible with inclusion maps of sub-complexes}, the image of $x_{0, A}$ under the continuation map $c_{A, A'}$ is a time-one $H_{A'}$-chord $x_{0, A'}$ corresponding to the same $H_{U}$-chord $x'_{0}$ under the  isomorphism $\tau_{A'}$. Thus, we may write that asymptotic convergence condition as
\begin{equation}
\lim\limits_{s \to -\infty} w^{k+1}(s, \cdot) = \tau_{A}(x'_{0})(\cdot),
\end{equation}
for a fixed time-one $H_{U}$-chord $x'_{0}$ from $L'_{0}$ to $L'_{k}$. Then, the moduli space of such climbing disks, with the given asymptotic convergence conditions $x'_{0}$ at the negative strip-like end, and $x_{1}, \cdots, x_{k}$ at the positive strip-like ends, is denoted by $\mathcal{P}_{k+1}(x'_{0}; x_{1}, \cdots, x_{k})$. The virtual dimension of this moduli space is
\begin{equation*}
v-\dim \mathcal{P}_{k+1}(x'_{0}; x_{1}, \cdots, x_{k}) = \deg(x'_{0}) - \deg(x_{1}) - \cdots - \deg(x_{k}) + k - 1,
\end{equation*}
where $k-1$ is the dimension of the moduli space $\mathcal{S}^{k+1}$ of the underlying domains $(Z^{k+1}, A)$. \par
	There is a natural compactification of the moduli space of climbing disks with $k+1$ punctures, denoted by
\begin{equation*}
\bar{\mathcal{P}}_{k+1}(x'_{0}; x_{1}, \cdots, x_{k}).
\end{equation*}
The elements in this compactified moduli space are broken climbing disks, generalizing the broken climbing strips. There are four types of such broken configurations:
\begin{enumerate}[label=(\roman*)]

\item A pair $((Z^{k+1}, A, w^{k+1}), u)$, where $(Z^{k+1}, A, w^{k+1})$ is a climbing disk with $k+1$ punctures, and $u$ is an inhomogeneous pseudoholomorphic strip in $M$. $w^{k+1}$ and $u$ have a common asymptotic convergence condition $x_{new}$ at some positive end of $w^{k+1}$ and the negative end of $u$, for some time-one $H_{M}$-chord $x_{new}$. This occurs as the limit of a sequence of climbing disks where the energy escapes through some positive strip-like end.

\item A pair $(u', (Z^{k+1}, A, w^{k+1}))$, where $u'$ is an inhomogeneous pseudoholomorphic strip in $U$ with respect to the Floer datum $(H_{A}, J_{A})$, and $(Z^{k+1}, A, w^{k+1})$ is a climbing disk with $k+1$ punctures. $u'$ and $w^{k+1}$ have a common asymptotic convergence condition $x'_{new}$ at the positive end of $u'$ and the negative end of $w^{k+1}$ (up to a conformal rescaling), for some time-one $H_{U}$-chord $x'_{new}$. This occurs as the limit of a sequence of climbing disks as the energy escapes through the negative strip-like end.

\item A pair $((Z^{k-m+2}, A, w^{k-m+2}), u)$, where $(Z^{k-m+2}, A, w^{k-m+2})$ is a climbing disk with $k-m+2$ punctures (for some $m > 1$) and $u$ is an inhomogeneous pseudoholomorphic $(m+1)$-punctured disk in $M$. $w^{k-m+2}$ and $u$ have a common asymptotic convergence condition $x_{new}$ at some positive end of $w^{k-m+2}$, and the negative end of $u$, for some time-one $H_{M}$-chord $x_{new}$. This occurs as the limit of a sequence of climbing disks for which the parameter $A$ tends to $0$ or remains finite.

\item A tuple $(u', (Z^{s_{1}+1}, A_{1}, w^{s_{1}+1}), \cdots, (Z^{s_{l}+1}, A_{l}, w^{s_{l}+1}))$, where $(Z^{s_{j}+1}, A_{j}, w^{s_{j}+1})$ is a climbing disk with $s_{j}+1$ punctures, and $u'$ is an inhomogeneous pseudoholomorphic $(l+1)$-punctured disk with respect to the Floer datum $(H_{A}, J_{A})$ (for some $l > 1$). $u'$ and $w^{s_{j}+1}$ have a common asymptotic convergence condition $x'_{new, j}$ at the $j$-th positive end of $u'$ (up to conformal rescaling), and the negative end of $w^{s_{j}+1}$, for some time-one $H_{U}$-chord $x'_{new, j}$. This occurs as the limit of a sequence of climbing disks for which the parameter $A$ tends to $+\infty$.

\end{enumerate}
Note that the broken configurations of type (i) and type (iii) can be written in a uniform way, as (i) is the special case $m = 1$. And the broken configurations of type (ii) and type (iv) can be written in a uniform way, as (ii) is the special case $l = 1$.
In case (ii), by an $(H_{A}, J_{A})$-pseudoholomorphic disk $u'$ with $l+1$ punctures, we mean it satisfies the inhomogeneous Cauchy-Riemann equation with respect to a domain-dependent family of Hamiltonians which agrees with $H_{A}$ (up to rescaling) over the strip-like ends, and a domain-dependent family of almost complex structures which agrees with $J_{A}$ (up to rescaling) over the strip-like ends. Such an inhomogeneous pseudoholomorphic disk then corresponds to an $(H_{U}, J_{U})$-pseudoholomorphic disk under the conformal rescaling by Liouville flow.
In case (iv), such an $(H_{A}, J_{A})$-pseudoholomorphic strip $u'$ corresponds to an $(H_{U}, J_{U})$-pseudoholomorphic strip under conformal rescaling. 
Thus, we have an isomorphism of the codimension one boundary strata:
\begin{equation}\label{boundary strata of moduli space of climbing disks}
\begin{split}
&\partial \bar{\mathcal{P}}_{k+1}(x'_{0}; x_{1}, \cdots, x_{k})\\
\cong & \coprod_{i} \coprod_{\substack{x_{new}\\ \deg(x_{new}) = \deg(x_{i+1}) + \cdots + \deg(x_{i+m}) + 2 - k}} \mathcal{P}_{k-m+2}(x'_{0}; x_{1}, \cdots, x_{i}, x_{new}, x_{i+m+1}, \cdots, x_{k})\\
& \times \mathcal{M}_{m+1}(x_{new}, x_{i+1}, \cdots, x_{i+m})\\
\cup & \coprod_{\substack{s_{1}, \cdots, s_{l}\\s_{1} + \cdots + s_{l} = k}}
\coprod_{\substack{x'_{new, 1}, \cdots, x'_{new, l}\\ \deg(x'_{new, j}) = \deg(x_{s_{1}+\cdots+s_{j-1}+1}) + \cdots + \deg(x_{s_{1}+\cdots+s_{j}}) + 1 - s_{j}}}
\mathcal{M}_{l+1}(x'_{0}, x'_{new, 1}, \cdots, x'_{new, l})\\
& \times \mathcal{P}_{s_{1}+1}(x'_{new, 1}; x_{1}, \cdots, x_{s_{1}}) \times \cdots \times \mathcal{P}_{s_{l}+1}(x'_{new, l}; x_{s_{1}+\cdots+s_{l-1}+1}, \cdots, x_{k}).
\end{split}
\end{equation}
Note that the components $\mathcal{M}_{l+1}(x'_{0}, x'_{new, 1}, \cdots, x'_{new, l})$ are moduli spaces of $(H_{U}, J_{U})$-pseudoholomorphic disks in $U$, which enter this picture as the identification of the moduli spaces $\mathcal{M}_{l+1}(x_{0, A}; x_{new, 1, A}, \cdots, x_{new, l, A})$ of $(H_{A}, J_{A})$-pseudoholomorphic disks, provided that a rigid climbing disk exists for that $A$. As before, this isomorphism holds if the Floer data are chosen in a consistent way. \par
	If we choose Floer data generically, the compactified moduli space $\bar{\mathcal{P}}_{k+1}(x'_{0}; x_{1}, \cdots, x_{k})$, when the virtual dimension is zero or one, is a compact smooth/topological manifold of dimension zero/one. In particular, $\mathcal{P}_{k+1}(x'_{0}; x_{1}, \cdots, x_{k})$ is compact if the dimension is zero, and therefore consists of finitely many points with natural orientations. Then by counting rigid elements in the moduli space $\mathcal{P}_{k+1}(x'_{0}; x_{1}, \cdots, x_{k})$, we can define multilinear maps
\begin{equation}\label{higher order terms of the Viterbo restriction homomorphism}
r^{k}: CW^{*}(L_{k-1}, L_{k}; H_{M}) \otimes \cdots \otimes CW^{*}(L_{0}, L_{1}; H_{M}) \to CW^{*}(L'_{0}, L'_{k}; H_{U}).
\end{equation}
That is, for each rigid climbing disk $(Z^{k+1}, A, w^{k+1})$, we get a canonical isomorphism of orientation lines
\begin{equation*}
o_{(Z^{k+1}, A, w^{k+1})}: o_{x_{1}} \otimes \cdots \otimes o_{x_{k}} \to o_{x_{0, A}}.
\end{equation*}
We then compose it with the map $\tau_{A}$ to get an isomorphism
\begin{equation*}
o_{x_{1}} \otimes \cdots \otimes o_{x_{k}} \to o_{x'_{0}},
\end{equation*}
and sum over all such $(Z^{k+1}, A, w^{k+1})$ and all $x'_{0}$ to get the desired map $r^{k}$. \par
	Now consider the situation where the virtual dimension of $\bar{\mathcal{P}}(x', x)$ is one. Then the structure of the codimension one boundary strata of the moduli spaces $\bar{\mathcal{P}}(x', x)$ \eqref{boundary strata of moduli space of climbing strips} and the moduli spaces $\bar{\mathcal{P}}_{k+1}(x'_{0}; x_{1}, \cdots, x_{k})$ \eqref{boundary strata of moduli space of climbing disks} immediately implies: \par

\begin{lemma}
	The multilinear maps
\begin{equation*}
r^{k}: CW^{*}(L_{k-1}, L_{k}; H_{M}) \otimes \cdots \otimes CW^{*}(L_{0}, L_{1}; H_{M}) \to CW^{*}(L'_{0}, L'_{k}; H_{U})
\end{equation*}
satisfy the equations for $A_{\infty}$-functors.
\end{lemma}
\begin{proof}
	The non-trivial part involves taking the direct limit $\lim\limits_{A \to +\infty}$ for various Floer cochain spaces and various maps, but that can be dealt with in a similar way as Lemma
\end{proof}

	Varying the objects $L_{j}$'s in the full sub-category $\mathcal{B}(M)$, we get an $A_{\infty}$-functor
\begin{equation}
r: \mathcal{B}(M) \to \mathcal{W}(U).
\end{equation} \par
	To summarize, the spirit of our construction is as follows. Implicit, we have defined a version of homotopy direct limit
\begin{equation*}
\lim\limits_{A \to +\infty} \mathcal{W}_{-}(M; H_{A})
\end{equation*}
where $\mathcal{W}_{-}(M; H_{A})$ is an $A_{\infty}$-category whose morphism spaces are the truncated wrapped Floer complexes $CW^{*}_{-}(L_{0}^{1-\epsilon}, L_{1}^{1-\epsilon}; H_{A})$, and whose $A_{\infty}$-structure maps are defined by counting inhomogeneous pseudoholomorphic disks with respect to Floer data $(H_{A}, J_{A})$. Equivalently, this category is quasi-isomorphic to a truncated version of the wrapped Fukaya category of $U$. There are natural functors
\begin{equation*}
c_{A}: \mathcal{W}(U) \to \mathcal{W}_{-}(M; H_{A}),
\end{equation*}
which are induced by the continuation functors defined in terms of Floer theory in $U$, given a homotopy of Hamiltonians from $H_{U}$, which is quadratic with leading coefficient $1$, to a Hamiltonian which is quadratic with leading coefficient $A$. These functors $c_{A}$ are compatible with natural continuation functors
\begin{equation*}
c_{A, A'}: \mathcal{W}_{-}(M; H_{A}) \to \mathcal{W}_{-}(M; H_{A'}),
\end{equation*}
so we can take the direct limit
\begin{equation*}
\lim\limits_{A \to +\infty} c_{A}: \mathcal{W}(U) \to \lim\limits_{A \to +\infty} \mathcal{W}_{-}(M; H_{A}).
\end{equation*}
By an action filtration argument, this functor is a quasi-isomorphism. Thus we are able to find a homotopy inverse
\begin{equation*}
k: \lim\limits_{A \to +\infty} \mathcal{W}_{-}(M; H_{A}) \to \mathcal{W}(U).
\end{equation*}
From this point of view, the main part of our construction is to use the moduli spaces $\mathcal{P}_{k+1}(x'_{0}; x_{1}, \cdots, x_{k})$ to define an $A_{\infty}$-functor
\begin{equation*}
\lim\limits_{A \to +\infty} r_{A}: \mathcal{B}(M) \to \lim\limits_{A \to +\infty} \mathcal{W}_{-}(M; H_{A}).
\end{equation*}
By composing this with the functor $k$, we get
\begin{equation*}
r = k \circ \lim\limits_{A \to +\infty} r_{A}: \mathcal{B}(M) \to \mathcal{W}(U).
\end{equation*} \par

\begin{definition}
	The $A_{\infty}$-functor 
\begin{equation*}
r: \mathcal{B}(M) \to \mathcal{W}(U)
\end{equation*}
is called the Viterbo restriction functor.
\end{definition}

\subsection{Comparison between the linear terms}\label{section: comparing the two linear restriction functors}
	In this subsection, we shall prove the second half of Theorem \ref{Viterbo functor as a correspondence functor}. Formulated in a more precise way, what we need to prove is that the linear term of the functor $\Theta_{\Gamma}$
\begin{equation*}
\Theta_{\Gamma}^{1}: CW^{*}(L_{0}, L_{1}; H_{M}) \to CW^{*}(L'_{0}, L'_{1}; H_{U})
\end{equation*}
is chain homotopic to the linear Viterbo restriction map
\begin{equation*}
r^{1}: CW^{*}(L_{0}, L_{1}; H_{M}) \to CW^{*}(L'_{0}, L'_{1}; H_{U}).
\end{equation*} \par
	We shall from now on restrict the Viterbo restriction functor to the full sub-category $\mathcal{B}_{0}(M)$, and obtain an $A_{\infty}$-functor
\begin{equation*}
r: \mathcal{B}_{0}(M) \to \mathcal{W}(U).
\end{equation*}
That is to say, we shall consider only Lagrangian submanifolds $L$ in the full sub-category $\mathcal{B}_{0}(M)$, i.e. they satisfy Assumption \ref{strong exactness condition} and Assumption \ref{invariance assumption}. \par
	When comparing $\Theta_{\Gamma}^{1}$ to the linear term $r^{1}$ of the Viterbo restriction functor, we shall instead work with the cochain map $\Pi_\Gamma$ as defined in \eqref{a different realization of the correspondence functor}. 
The reason why we study this map is that its construction is straightforward and geometric, without explicitly referring to representability. \par
	Recall that the map $\Pi_{\Gamma}$ is defined using the moduli spaces
$\bar{\mathcal{U}}(x; x'; e_{0}, e_{1})$ of inhomogeneous pseudoholomorphic quilted maps, which are variants of the moduli spaces
\begin{equation*}
\bar{\mathcal{U}}_{l_{0}, l_{1}}(\alpha, \beta; x; y; y_{0, 1}, \cdots, y_{0, l_{0}}; y_{1, 1}, \cdots, y_{1, l_{1}}; e_{0}, e_{1}),
\end{equation*}
where now there are no those $l_{0} + l_{1}$ punctures on the boundary components of the second patch, as the Lagrangian submanifolds are embedded and the bounding cochains $b_{0}$ and $b_{1}$ vanish. On the other hand, the linear Viterbo restriction map $r^{1}$ is the homotopy direct limit of maps $r^{1}_{A}$, constructed by appropriate count of elements in the moduli spaces of climbing strips $\bar{\mathcal{P}}^{A}(x_{A}, x)$.
The idea of proving that these two maps coincide up to chain homotopy is to compare these two kinds of relevant moduli spaces, so as to establish a natural bijective correspondence between them. For this purpose, we shall investigate the geometric conditions for inhomogeneous pseudoholomorphic quilted maps in the moduli spaces $\bar{\mathcal{U}}(x; x'; e_{0}, e_{1})$, and prove that any inhomogeneous psuedoholomorphic quilted map can be converted to a climbing strip in a canonical way. \par

	Let us first recall the picture of these inhomogeneous pseudoholomorphic quilted maps in $\bar{\mathcal{U}}(x; x'; e_{0}, e_{1})$, or rather, the smooth locus $\mathcal{U}(x; x'; e_{0}, e_{1})$. Let $(u, v)$ be an inhomogeneous pseudoholomorphic quilted map in $\mathcal{U}(x; x'; e_{0}, e_{1})$. Then it satisfies the following conditions:
\begin{enumerate}[label=(\roman*)]

\item The quilted surface has two patches $\underline{S} = (S_{0}, S_{1})$. $S_{0}$ is a punctured disk with a positive puncture $z_{0}^{1}$ and two special punctures $z_{0}^{-, 1}, z_{0}^{-, 2}$, with chosen strip-like ends near the punctures. $S_{1}$ is a punctured disk with a negative puncture $z_{1}^{0}$ and two special punctures $z_{1}^{-, 1}, z_{1}^{-, 2}$. The quilted surface is obtained by seaming the two patches along the boundary component $I_{0}^{-}$ of $S_{0}$ between $z_{0}^{-, 1}, z_{0}^{-, 2}$ and the boundary component $I_{1}^{-}$ of $S_{1}$ between $z_{1}^{-, 1}, z_{1}^{-, 2}$. The strip-like ends near special punctures form quilted strip-like ends for $(z_{0}^{-, 1}, z_{1}^{-, 1}$ and $(z_{0}^{-, 2}, z_{1}^{-, 2})$.

\item $u: S_{0} \to M$ is pseudoholomorphic with respect to the Floer datum $(H_{S_{0}}, J_{S_{0}})$, which maps the two (non-seam) boundary components to $L_{0}$ and $L_{1}$ repsectively, which asymptotically converges to a time-one $H_{M}$-chord $x$ from $L_{0}$ to $L_{1}$ at the puncture $z_{0}^{1}$.

\item $v: S_{1} \to U$ is pseudoholomorphic with respect to the Floer datum $(H_{S_{1}}, J_{S_{1}})$, which maps the two (non-seam) boundary components to $L'$, which asymptotically converges to an $H_{U}$-chord $x'$ from $L'_{0}$ to $L'_{1}$ at $z_{1}^{0}$.

\item Over the seam $(I_{0}^{-}, I_{1}^{-})$, the matching condition for the pair of maps $(u, v)$ is given by the Lagrangian correspondence $\Gamma$, i.e., $(u(z), v(z)) \in \Gamma$ for $z$ on the seam.

\item At the two quilted strip-like ends, $(u, v)$ asymptotically converges to the unique generalized chord representing the cyclic $e_{j} \in CW^{*}(L_{j}, \Gamma, L'_{j})$. In fact, $e_{j}$ comes from the fundamental chain of $L'_{j}$, under the natural identification between the generalized intersections of $(L_{j}, \Gamma, L'_{j})$ and the self-intersections of $L'_{j}$.

\end{enumerate}
Given these conditions, we find that the isomorphism class of the underlying quilted surface $\underline{S}$ is unique, i.e. the moduli space of such quilted surfaces is a singleton. Thus we often denote such a quilted map simply by $(u, v)$. \par

	The main task of this subsection is to prove the following: \par

\begin{proposition}
	Suppose the virtual dimension of $\mathcal{U}(x; x'; e_{0}, e_{1})$ is zero or one. Let $A > 0$ be any positive number such that
\begin{equation}
\mathcal{A}(x') \ge -A\epsilon^{2}.
\end{equation}	
Among generic choices of Floer data which make $\mathcal{U}(x; x'; e_{0}, e_{1})$ regular, we can make a specific choice of Floer data (without losing genericity), for which the moduli space $\mathcal{U}(x; x'; e_{0}, e_{1})$ is isomorphic to a moduli space $\tilde{\mathcal{P}}^{A}(x_{A}, x)$ of inhomogeneous pseudoholomorphic maps, which is orientedly cobordant to the moduli space of climbing strips $\mathcal{P}^{A}(x_{A}, x)$.
\end{proposition}

\begin{remark}
	For technical reasons, our construction does not immediately give rise to a climbing strip from a given inhomogeneous pseudoholomorphic quilted map, but rather yields a map satisfying the inhomogeneous Cauchy-Riemann equation with respect to the same Floer datum, but different boundary conditions. The moduli space of such inhomogeneous pseudoholomorphic maps are cobordant to the moduli space of climbing strips in a natural way by means of exact Lagrangian isotopies of the boundary conditions, which will be explained in the proof.
\end{remark}

	In order to turn such a quilted map $(\underline{S}, (u, v))$ into an inhomogeneous pseudoholomorphic map in $M$, we shall consider the Morse-Bott setup of both wrapped Floer cochain spaces $CW^{*}(L_{j}; H_{M})$ and $CW^{*}(L'_{j}; H_{U})$, for $j = 0, 1$. This process can be done without repeating the details by regarding $L'$ as a Lagrangian immersion. Recall that the asymptotic conditions over the quilted ends are specified by the cyclic element $e_{j} \in CW^{*}(L_{j}, \Gamma, L'_{j})$, which in turn corresponds to the homotopy unit of the $A_{\infty}$-algebra $CW^{*}(L'_{j}; H_{U})$, under the map \eqref{quasi-isomorphism under geometric composition}. This homotopy unit is represented by the minimum of a Morse function on $L'_{j}$, whose image in $U$ is inside the interior part $U_{0}$ where the Hamiltonian vanishes. Thus we may in fact demand that the families of Hamiltonians $(H_{S_{0}}, H_{S_{1}})$ be zero near the quilted ends: this condition is consistent with the Morse-Bott setup of the wrapped Floer theory for $L_{j}$ and $L'_{j}$. \par
	A naive trial is to simply "glue" the two components $u$ and $v$ together. Since $\Gamma$ is the completion of the graph of the natural inclusion $U_{0} \to M_{0}$, the condition that a point $(p, q) \in M \times U$ lies on $\Gamma$ is equivalent to the condition that $p$ is the image of $q$ under the natural map $i: U \to M$. Thus the matching condition on $\Gamma$ in fact allows us to "fold" the quilted map to obtain a map $\bar{w}$ in $M$, after composing the second component $v$ with the embedding $i: U \to M$. That is, glue the two patches of the quilted surface together and define a map $\tilde{w}$ which is $u$ on $S_{0}$ and which is $i \circ v$ on $S_{1}$. From the conditions for $(\underline{S}, (u, v))$ one can immediately see that $\tilde{w}$ satisfies the following properties: 
\begin{enumerate}[label=(\roman*)]

\item The domain of the map $\bar{w}$ is a $4$-punctured disk. One is a negative puncture $\xi^{0}$, which corresponds to the negative puncture of $S_{1}$. One is a positive puncture $\xi^{1}$, which corresponds to the positive puncture of $S_{0}$. The other two are special punctures $\xi^{1, -}, \xi^{2, -}$, which come from gluing the quilted punctures.

\item $\bar{w}$ is smooth, as the matching condition on $\Gamma$ means that the maps $u$ and $i \circ v$ agree on the seam along which the two patches are glued together.

\item $\bar{w}$ has removable singularities at the special punctures $\xi^{1, -}, \xi^{2, -}$. This is because the families of Hamiltonians $(H_{S_{0}}, H_{S_{1}})$ are chosen such that they vanish over the two quilted ends, and the quilted map $(u, v)$ asymptotically converges to $e'_{0}$ and $e'_{k}$, which in the Morse-Bott chain model correspond to the fundamental chains of $L'_{0}$ and $L'_{k}$.
 
\end{enumerate}\par

	We have thus obtained a map $\bar{w}$ to $M$ from the inhomogeneous pseudoholomorphic quilted map $(u, v)$. However, the domain of the map $\bar{w}$ is of an unfamiliar form, and it does not satisfy the desired inhomogeneous Cauchy-Riemann equation for a climbing strip. It is therefore necessary to perform suitable modification on the map $\bar{w}$. By condition (iii), we may take conformal transformations mapping $S_{0}$ to the positive half-strip $Z_{+} = [0, +\infty) \times [0, 1]$, sending the positive puncture to $+\infty$, and the two quilted punctures to the corner points $(0, 0), (0, 1)$, and mapping $S_{1}$ to the negative half-strip $Z_{-} = (-\infty, 0] \times [0, 1]$, sending the negative puncture to $-\infty$, and the two quilted punctures to the corner points $(0, 0), (0, 1)$. Glue these two half-strips together along the common boundary $\{0\} \times [0, 1]$, which correspond to the seam of the quilted surface. The result of gluing is a strip with two marked points $z_{1, +} = (0, 0), z_{2, +} = (0, 1)$. This strip is the domain for our new map. In addition, condition (iii) also implies that the incidence conditions for these two marked points are given by the fundamental chains of $L'_{0}$ and $L'_{1}$, thus these incidence conditions are in fact free conditions. \par

	Now we shall construct Floer datum on this new domain from the original Floer datum for the quilted surface $\underline{S}$. The guiding principle is to use the Liouville structure to rescale the original Floer datum to obtain the new one. 
Without loss of generality, we may assume that $A > 1$, as when taking the direct limit $\lim\limits_{A \to +\infty}$ we only have to consider sufficiently large $A$.
Fix a choice of a smooth increasing function
\begin{equation*}
\rho_{A}: (-\infty, 0] \to [1, A]
\end{equation*}
such that for $s \ll 0$, $\rho_{A}(s) = A$ and $\rho_{A}(s) = 1$ for $s$ close to $0$. We can also extend this function to the whole of $\mathbb{R}$ by setting $\rho_{A}(s) = 1$ for all $s \ge 0$. \par
	First let us discuss how to obtain a family of Hamiltonians from the given families of Hamiltonians $(H_{S_{0}}, H_{S_{1}})$ for the quilted surface. We define a family of Hamiltonians $H_{A, s}$ on $M$ parametrized by $Z$, which depends only on the coordinate $s \in \mathbb{R}$, by setting
\begin{equation}\label{family of Hamiltonians from those on the quilted surface}
H_{A, s} =
\begin{cases}
\frac{1}{\rho_{A}(s)} H_{Z_{-}, s} \circ \psi_{U}^{\rho_{A}(s)} \circ i^{-1}, &\text{ on } Z_{-}, \\
H_{Z_{+}, s}, &\text{ on } Z_{+},
\end{cases}
\end{equation}
where $H_{Z_{-}}$ is the family $H_{S_{1}}$ composed with the conformal transformation from $Z_{-}$ to $S_{1}$, and $H_{Z_{+}}$ is the family $H_{S_{0}}$ composed with the conformal transformation from $Z_{+}$ to $S_{0}$.
Here $\epsilon$ is the small constant that we have fixed for which the collar neighborhood $\partial U \times [1-\epsilon, 1]$ embeds into $U_{0}$. And $\frac{1}{\rho_{A}(s)} H_{Z_{-},s} \circ \psi_{U}^{\rho_{A}(s)} i^{-1}$ means a Hamiltonian on $M$ which takes value $\frac{1}{\rho_{A}(s)} H_{Z_{-}, s} \circ \psi_{U}^{\rho_{A}(s)}(v(s, t))$ at the point $i \circ v(s, t) \in M$, and is constant (equal to $A\epsilon^{2}$) elsewhere. In fact, we do not have to specify what values the Hamiltonian take for points outside $i(U)$, as the image of the point on the domain of the inhomogeneous pseudoholomorphic map at which the domain-dependent family of Hamiltonians is evaluated is inside $i(U)$.
A priori, the resulting family of Hamiltonians also depends on the $t$-coordinate, but we may choose $(H_{S_{0}}, H_{S_{1}})$ appropriately such that the resulting families $H_{Z_{+}}$ and $H_{Z_{-}}$ is independent of $t$ (though this is indeed irrelevant for our purpose). This is possible, because of the following reasons: over the positive strip-like end $S_{0}$, $H_{S_{0}}$ agrees with $H_{M}$, and over the negative strip-like end of $S_{1}$, $H_{S_{1}}$ agrees with $H_{U}$, while over the quilted ends, $(H_{S_{0}}, H_{S_{1}})$ vanish - all of these are independent of $t$. \par
	For the almost complex structure, given the families of almost complex structures $(J_{S_{0}}, J_{S_{1}})$ for the quilted surface, we define a family of almost complex structures $J_{A, (s, t)}$ on $M$ parametrized by $(s, t) \in Z$, by setting
\begin{equation}\label{family of almost complex structures from those on the quilted surface}
J_{A, (s, t)} = 
\begin{cases}
(\psi_{U}^{\rho_{A}(s)})_{*} J_{Z_{-}, (s, t)} \circ i^{-1}, &\text{ on } Z_{-},\\
J_{Z_{+}, (s, t)}, &\text{ on } Z_{+}.
\end{cases}
\end{equation} \par

	We must explain why \eqref{family of Hamiltonians from those on the quilted surface} and \eqref{family of almost complex structures from those on the quilted surface} are well-defined. There are several cases to discuss, depending on the positions of the Hamiltonian chords $x'$ and $x$ to which the quilted map $(\underline{S}, (u, v))$ converges:
\begin{enumerate}[label=(\roman*)]

\item $x'$ is a small perturbation of a constant Hamiltonian chord in $U_{0}$, and $x$ is a small perturbation of a constant Hamiltonian chord in $M_{0}$. In this case, by a maximum principle argument, it is necessary that, for the quilted map $(u, v)$, the image of $u$ is contained in $M_{0}$ and the image of $v$ is contained in $U_{0}$. Thus, the matching condition on $\Gamma$ simply implies that $u(0, t) = v(0, t)$. If furthermore $x$ is contained in $U_{0}$, then in fact $x' = x$ and the quilted map is constant. This is essentially the only case that needs some discussion. In this case, we may assume our choice of $H_{M}$ agrees with the rescaling of $H_{U}$ inside $U_{0}$, so that the above formula \eqref{family of Hamiltonians from those on the quilted surface} is consistent. Now if $x$ is in $M_{0} \setminus U_{0}$, then such a quilted map $(u, v)$ is non-trivial. In this case we may choose $(H_{S_{0}}, H_{S_{1}})$ such that $H_{Z_{-}, 0} = 0$ and $H_{Z_{+}, 0} = 0$. Thus it is automatic that the two formulas in \eqref{family of Hamiltonians from those on the quilted surface} agree when $s = 0$.

\item $x'$ is a small perturbation of a constant Hamiltonian chord in $U_{0}$, while $x$ is a non-constant $H_{M}$-chord in the cylindrical end $\partial M \times [1, +\infty)$. This is similar to the previous case, since such a quilted map is necessarily non-constant, so that we have the freedom to choose the families $(H_{S_{0}}, H_{S_{1}})$ so that they are zero along the seam.

\item $x'$ is a non-constant $H_{U}$-chord in the cylindrical end $\partial U \times [1, +\infty)$, while $x$ is a small perturbation of a constant Hamiltonian chord in $M_{0}$. This $x'$ corresponds to a time-one chord $x_{A}$ for the rescaled Hamiltonian. Then there is in fact no such quilted map or climbing strip, as the action of $x'$ (or $x_{A}$) is very negative, and the action of $x$ is positive, which is of course greater than the action of $x'$ (or $x_{A}$). Thus it is not necessary to consider this case.

\item $x'$ is a non-constant $H_{U}$-chord in the cylindrical end $\partial U \times [1, +\infty)$, and $x$ is a non-constant $H_{M}$-chord in the cylindrical end $\partial M \times [1, +\infty)$. Now $x'$ corresponds to a unique time-one $H_{A}$-chord $x_{A}$, and the quilted map is non-constant. The discussion is similar to those in cases (i) and (ii).

\end{enumerate}
Thus we have shown that \eqref{family of Hamiltonians from those on the quilted surface} is well-defined. For the family of almost complex structures, a parallel argument implies that \eqref{family of almost complex structures from those on the quilted surface} is well-defined. \par
	To construct a climbing strip from the given quilted map $(u, v)$, and $A > 0$, we first define a map
\begin{equation}
w_{0}(s, t)=
\begin{cases}
i \circ \psi_{U}^{\rho_{A}(s)} \circ v(s, t), &\text{ if } s < 0,\\
u(s, t), &\text{ if } s \ge 0.
\end{cases}
\end{equation}
This is well-defined and smooth, since when $s = 0$, we have $\rho_{A}(0) = 1$ so that the first formula reads $i \circ v(0, t)$, which is equal to $u(0, t)$ by the matching condition on $\Gamma$. By definition, the boundary conditions for the map $w_{0}$ are as follows:
\begin{equation}
w_{0}(s, j) \in
\begin{cases}
i(\psi_{U}^{\rho_{A}(s)} L'_{j}), & \text{ if } s < 0,\\
L_{j}, & \text{ if } s \ge 0,
\end{cases}
\text{where } j = 0, 1.
\end{equation}
Note that when $s \le 0$ is close to zero, we have $\rho_{A}(s) = 1$ so that $i(\psi_{U}^{\rho_{A}(s)} L'_{j}) = i(L'_{j}) \subset L_{j}$, because of Assumption \ref{invariance assumption}. Let us rewrite the boundary conditions as
\begin{equation*}
w_{0}(s, j) \in L_{j, A, s},
\end{equation*}
for
\begin{equation}\label{Lagrangian boundary conditions obtained from pasting a quilted map}
L_{j, A, s} :=
\begin{cases}
i(\psi_{U}^{\rho_{A}(s)} L'_{j}), & \text{ if } s < 0,\\
L_{j}, & \text{ if } s \ge 0.
\end{cases}
\end{equation}
When $s \ge 0$, the boundary condition can also be written as $L_{j, A, s} = \psi_{M}^{\rho_{A}(s)} L_{j}$, as we have extended $\rho_{A}$ such that $\rho_{A}(s) = 1$ for all $s \ge 0$.
Since $i$ is the map induced by the Liouville strucutre, it maps $L'_{j}$ into $L_{j}$, and $i(\psi_{U}^{\rho_{A}(s)} L'_{j})$ is mapped into a unique exact cylindrical Lagrangian submanifold of $M$ whose cylindrical end is contained in that of $L_{j}$. In fact, the completion of $i(\psi_{U}^{\rho_{A}(s)} L'_{j})$ is $\psi_{M}^{A} L_{j}$, because $L_{j}$ is assumed to be invariant under the Liouville flow in the Liouville cobordism $M_{0} \setminus int(U_{0})$.
Thus $L_{j, A, s}$ can be regarded as a family of Lagrangian submanifolds of $M$ parametrized by $s \in \mathbb{R}$, if we understand $i(\psi_{U}^{\rho_{A}(s)} L'_{j})$ as its completion. Moreover, this family is an exact Lagrangian isotopy.
Such boundary conditions are slightly different from those for a climbing strip. The way to fix this is to use Lagrangian isotopy to move one kind of boundary conditions to the other - such an argument will be presented later. \par
	Note that $w_{0}$ satisfies the inhomogeneous Cauchy-Riemann equation on the positive half-strip $Z_{+}$, with respect to the Floer datum $(H_{A, s}, J_{A, (s, t)})$ defined as above, as on that region the Floer datum is simply given by the original one $(H_{S_{0}}, J_{S_{0}})$ for the quilted surface. On the negative half-strip $Z_{-}$, $w_{0}$ might not be pseudoholomorphic, because an extra term appears when differentiating $i \circ \psi_{U}^{\rho_{A}(s)} \circ v(s, t)$ with respect to $s$, which is related to the differential of $\psi_{U}^{\rho_{A}(s)}$. By a straightforward calculation, we have
\begin{equation}
\partial_{s}(i \circ \psi_{U}^{\rho_{A}(s)} \circ v(s, t)) = di_{\psi_{U}^{\rho_{A}(s)} \circ v(s, t)} \circ (d\psi_{U}^{\rho_{A}(s)})_{v(s, t)} \rho'_{A}(s) (\partial_{s} v(s, t)),
\end{equation}
where $\rho'_{A}(s)$ is the derivative of the function $\rho_{A}$ with respect to $s$. On the other hand, in the inhomogeneous term contributed by the Hamiltonian vector field, there is not such a factor, as we have for $s < 0$,
\begin{equation}
X_{H_{A, s}} = \rho_{A}(s) di \circ d\psi_{U}^{\rho_{A}(s)} X_{H_{Z_{-}}},
\end{equation}
as calculating the Hamiltonian vector field from the Hamiltonian does not involve differentiation with respect to the variable $s$. In order to obtain a map which satisfies the inhomogeneous Cauchy-Riemann equation, we must therefore perturb the map $w_{0}$. In fact, we may choose the function $\rho_{A}$ suitably such that its derivative is small, say
\begin{equation*}
|\rho'_{A}(s)| < C \cdot inj(M, g),
\end{equation*}
where $inj(M, g)$ is the injectivity radius of a family of metrics $g = g(s, t)$ on $M$ for which the Lagrangian submanifolds $L_{0, A, s}, L_{1, A, s}$ are totally geodesic, and $C$ is an appropriate constant to be determined in the proof of the following perturbation lemma. \par

\begin{lemma}\label{perturbing the glued map to a climbing strip}
	For any given $A > 0$ sufficiently large, and for each $w_{0}$ defined as above, there is a unique map $\tilde{w}: Z \to M$ closest to $w_{0}$, which satisfies the inhomogeneous Cauchy-Riemann equation:
\begin{equation}
\partial_{s} w + J_{A, (s, t)}(\partial_{t} w - X_{H_{A, s}}(w)) = 0,
\end{equation}
and the same boundary conditions $L_{0, A, s}, L_{1, A, s}$ as those for $w_{0}$.
\end{lemma}
\begin{proof}
	For the point $w_{0} \in \mathcal{B}$, there is the exponential map
\begin{equation*}
Exp_{w_{0}}: O \subset T_{w_{0}}\mathcal{B} \to \mathcal{B},
\end{equation*}
defined with respect to a family of metrics $g = g(s, t)$ on $M$ parametrized by $(s, t) \in \mathbb{R} \times [0, 1]$, for which $L_{0, A, s}$ and $L_{1, A, s}$ are totally geodesic.
Here $\mathcal{B}$ is the Banach manifold (with respect to some Sobolev norm $W^{m, p}$) of maps satisfying the same conditions for a climbing strip, except the inhomogeneous Cauchy-Riemann equation. We require that $O$ is an open neighborhood of zero such that the exponential map $Exp_{w_{0}}$ is an isomorphism onto the image: the size of $O$ is at least the injectivity radius of the family of metrics $g$:
\begin{equation*}
inj(M, g) = \min_{(s, t) \in \mathbb{R} \times [0, 1]} inj(M, g(s, t)).
\end{equation*}
In fact, $g = g(s, t)$ is determined by the symplectic form and the chosen family of almost complex structures $J_{A, (s, t)}$. As a result, when $s \ll 0$ or $s \gg 0$, $g(s, t)$ agrees with a family independent of $s$. Therefore, the minimum over $(s, t) \in \mathbb{R} \times [0, 1]$ is indeed taken over a compact set, hence well-defined.
The tangent space $T_{w_{0}}\mathcal{B}$ is
\begin{equation*}
T_{w_{0}}\mathcal{B} = \{V \in W^{m, p}(Z; w_{0}^{*}TM; w_{0}^{*}TL_{0, A, s}, w_{0}^{*}TL_{1, A, s} \}.
\end{equation*}
Here the boundary conditions mean that $V(s, 0) \in T_{w_{0}(s, 0)}L_{0, A, s}, V(s, 1) \in T_{w_{1}(s, 1)}L_{1, A, s}$. \par
	For any $V \in O$, consider the inhomogeneous Cauchy-Riemann equation
\begin{equation}
\partial_{s} Exp_{w_{0}}(V) + J_{A, (s, t)}(\partial_{t} Exp_{w_{0}}(V) - X_{H_{A, s}}(Exp_{w_{0}}(V))) = 0.
\end{equation}
We wish to find a solution $V \in O$ to this equation. We denote the inhomogeneous Cauchy-Riemann operator with respect to the Floer datum $(H_{A, s}, J_{A, (s, t)})$ by
\begin{equation*}
\bar{\partial}_{A}(\cdot) = \partial_{s}(\cdot) + J_{A, (s, t)}(\partial_{t}(\cdot) - X_{H_{A, s}}(\cdot)).
\end{equation*}
Although $w_{0}$ does not satisfy the equation $\bar{\partial}_{A} w_{0} = 0$, we have chosen a connection on the bundle $w_{0}^{*}TM$ relative to $(w_{0}^{*}TL_{0, A, s}, w_{0}^{*}TL_{1, A, s})$ when defining the exponential map $Exp$, so that the linearized operator at $w_{0}$
\begin{equation}
D_{w_{0}} \bar{\partial}_{A}: T_{w_{0}}\mathcal{B} \to \mathcal{E}_{w_{0}}
\end{equation}
is well-defined, where $\mathcal{E}_{w_{0}}$ is the space of $(0, 1)$-forms with values in $w_{0}^{*}TM$, namely $\mathcal{E}_{w_{0}} = W^{k, p}(Z; w_{0}^{*}TM \otimes \Lambda^{0,1}_{Z})$.
The regularity assumption for the moduli space $\mathcal{P}^{A}(x', x)$ implies that the linearized operator $D_{w}\bar{\partial}_{A}$ is surjective for any $w$ which satisfies the equation $\bar{\partial}_{A} w = 0$. However, this is an open condition - for those maps close to an actual solution, the linearized operator is also surjective. In particular, as $w_{0}$ is close to an actual solution, $D_{w_{0}}\bar{\partial}_{A}$ is surjective. It follows that there is a bounded right inverse $Q_{w_{0}, A}$ of $D_{w_{0}} \bar{\partial}_{A}$. Now if we set
\begin{equation*}
C = \frac{1}{||Q_{w_{0}, A}||},
\end{equation*}
then the vector
\begin{equation}
V = Q_{w_{0}}(-\bar{\partial}_{A} w_{0})
\end{equation}
is contained in the neighborhood $O$. This is the desired solution.
The meaning of this formula is as follows. As $w_{0}$ does not satisfy the inhomogeneous Cauchy-Riemann equation, $\bar{\partial}_{A} w_{0}$ is non-zero vector in $\mathcal{E}_{w_{0}}$. Note that the exponential map $Exp$ provides a linear structure on a neighborhood of $w_{0}$ in $\mathcal{B}$ by identifying that with $O \subset T_{w_{0}}\mathcal{B}$. Thus, we may "add" the inverse image of the negative of this non-zero vector $\bar{\partial}_{A} w_{0}$ to the original map $w_{0}$ to kill the deviation from being zero, so that the resulting map satisfies the inhomogeneous Cauchy-Riemann equation. \par
	On the other hand, it is not hard to see such a solution $V$ is unique. \par
\end{proof}

	Let us now denote by $\tilde{\mathcal{P}}^{A}(x_{A}, x)$ the moduli space of maps $\tilde{w}$ as above satisfying the inhomogeneous Cauchy-Riemann equation with respect to the Floer datum $(H_{A, s}, J_{A, (s, t)}$ as well as the Lagrangian boundary conditions $(L_{0, A, s}, L_{1, A, s})$. This moduli space behaves very similarly to the moduli space of climbing strips $\mathcal{P}^{A}(x_{A}, x)$,  and has a natural compactification $\bar{\tilde{\mathcal{P}}}^{A}(x_{A}, x)$, by adding broken maps whose new components are inhomogeneous pseudoholomorphic strips in $M$ with respect to the Floer datum $(H_{M}, J_{M})$ and Lagrangian boundary conditions $(L_{0}, L_{1})$, or those with respect to the Floer datum $(H_{A}, J_{A, t})$ and Lagrangian boundary conditions given by the completions of $(i(\psi_{U}^{A}L'_{0}), i(\psi_{U}^{A}L'_{1})$, whose asymptotic Hamiltonian chords are contained inside $U_{0}$. The latter kind of inhomogeneous pseudoholomorphic strips are also in one-to-one correspondence with inhomogeneous pseudoholomorphic strips in $U$ with respect to the Floer datum $(H_{U}, J_{U})$ and Lagrangian boundary conditions $(L'_{0}, L'_{1})$. \par

	Thus, the above lemma provides a natural bijection between the moduli spaces
\begin{equation*}
\mathcal{U}(x; x'; e_{0}, e_{1}) \cong \tilde{\mathcal{P}}^{A}(x_{A}, x),
\end{equation*}
when the virtual dimension is zero, for generic choices of Floer data.
Then it remains to extend this bijection to the level of compactified moduli spaces. \par

\begin{lemma}\label{isomorphism of moduli spaces of strips}
	Consider the situation where the virtual dimension of the moduli space $\bar{\mathcal{U}}(x, x'; e_{0}, e_{1})$ is zero or one. Suppose we have chosen Floer data generically such that both the compactified moduli space $\bar{\mathcal{U}}(x, x'; e_{0}, e_{1})$ and the compactified moduli space $\bar{\tilde{\mathcal{P}}}^{A}(x_{A}, x)$ are regular. This means that these are compact smooth/topological manifolds of dimension zero/one. Then, among the generic choices of Floer data, there is a specific kind of choices for which there is a natural bijection from $\bar{\mathcal{U}}(x, x'; e_{0}, e_{1})$ to $\bar{\tilde{\mathcal{P}}}^{A}(x_{A}, x)$. \par
	Moreover, when the virtual dimension is one, this bijection is an isomorphism of moduli spaces, in the sense that it comes with a natural virtual isomorphism of Fredholm complexes, and commutes with the gluing maps.
\end{lemma}
\begin{proof}
	The assignment $(u, v) \mapsto w$ gives the natural bijection between elements of the uncompactified moduli spaces:
\begin{equation*}
\mathcal{U}(x, x'; e_{0}, e_{1}) \cong \tilde{\mathcal{P}}^{A}(x_{A}, x).
\end{equation*}
whenever the virtual dimension is zero or one.
If the virtual dimension is zero, these moduli spaces are compact, because we have chosen Floer data generically so that all such moduli spaces as well as moduli spaces of inhomogeneous pseudoholomorphic strips are regular. Thus there is nothing more to prove. \par
	Now consider the case where the virtual dimension is one. Note that the two compactifications are both obtained by adding the same kinds of inhomogeneous pseudoholomorphic strips in $M$ or inhomogeneous pseudoholomorphic strips in $U$. That is to say, there are isomorphisms:
\begin{equation*}
\begin{split}
& \partial \bar{\mathcal{U}}(x, x'; e_{0}, e_{1})\\
\cong & \coprod \mathcal{M}(x', x'_{1}) \times \mathcal{U}(x, x'_{1}; e_{0}, e_{1})\\
& \cup \coprod \mathcal{U}(x_{1}, x'; e_{0}, e_{1}) \times \mathcal{M}(x_{1}, x).
\end{split}
\end{equation*}
and
\begin{equation*}
\begin{split}
& \partial \bar{\tilde{\mathcal{P}}}^{A}(x_{A}, x)\\
\cong & \coprod \tilde{\mathcal{M}}(x_{A}, x_{1, A}) \times \tilde{\mathcal{P}}^{A}(x_{1, A}, x)\\
& \cup \coprod \tilde{\mathcal{P}}^{A}(x_{A}, x_{1}) \times \mathcal{M}(x_{1}, x).
\end{split}
\end{equation*}
Here $\tilde{\mathcal{M}}(x_{A}, x_{1, A})$ is the moduli space of inhomogeneous pseudoholomorphic strips $f_{A}$ with respect to the Floer datum $(H_{A}, J_{A, t})$ and Lagrangian boundary conditions given by the completions of $(i(\psi_{U}^{A} L'_{0}), i(\psi_{U}^{A} L'_{1})$, whose asymptotic Hamiltonian chords are contained inside $U_{0}$. This is naturally isomorphic to $\mathcal{M}(x', x'_{1})$, the moduli space of inhomogeneous pseudoholomorphic strips $f'$ in $U$ with respect to the Floer datum $(H_{U}, J_{U})$ and Lagrangian boundary conditions $(L'_{0}, L'_{1})$, when the time-one $H_{U}$-chords $x'$ and $x'_{1}$ correspond to the time-one $H_{A}$-chords $x_{A}$ and $x_{1, A}$ respectively.
Thus, for any kind of broken quilted map $(f', (u, v))$ or $((u, v), f)$, there is a unique broken map $(f_{A}, \tilde{w})$ or $(\tilde{w}, f)$ associated to it, where $f_{A}$ and $f'$ correspond to each other under the above-mentioned isomorphism between $\tilde{\mathcal{M}}(x_{A}, x_{1, A})$ and $\mathcal{M}(x', x'_{1})$. In this way, the bijection extends over the compactifications. \par
	This bijective correspondence naturally commutes with the gluing maps, because gluing happens near the usual strip-like ends, not the quilted ends. \par

\end{proof}

	As a corollary, this implies that when choosing Floer data generically in such special class, the counts of elements in these moduli spaces are equal. The algebraic consequence of this can be stated as follows. We have the wrapped Floer cochain space $CW^{*}(\psi_{M}^{A} L_{0}, \psi_{M}^{A} L_{1}; H_{A})$, on which the differential is defined by counting rigid elements in the moduli spaces $\tilde{\mathcal{M}}(x_{A}, x_{1, A})$ of inhomogeneous pseudoholomorphic strips. And we also have a sub-complex $CW^{*}_{-}(\psi_{M}^{A} L_{0}, \psi_{M}^{A} L_{1}; H_{A})$, generated by "interior" Hamiltonian chords. In fact, for this sub-complex, one can equivalently write it as $CW^{*}_{-}(i(\psi_{U}^{A} L'_{0}), i(\psi_{U}^{A} L'_{1}); H_{A})$, because any inhomogeneous pseudoholomorphic strip with asymptotic convergence conditions given by those Hamiltonian chords will be contained in the image of $i$.
In a similar way to the definition of the map
\begin{equation*}
\tilde{r}^{1}_{A}: CW^{*}(L_{0}, L_{1}; H_{M}) \to CW^{*}_{-}(L_{0}^{1-\epsilon}, L_{1}^{1-\epsilon}; H_{A}),
\end{equation*}
we may define a map
\begin{equation}
t^{1}_{A}: CW^{*}(L_{0}, L_{1}; H_{M}) \to CW^{*}_{-}(i(\psi_{U}^{A} L'_{0}), i(\psi_{U}^{A} L'_{1}); H_{A})
\end{equation}
by counting rigid elements in the moduli space $\bar{\tilde{\mathcal{P}}}^{A}(x_{A}, x)$. 
Then Lemma \ref{isomorphism of moduli spaces of strips} implies that $\Pi_{\Gamma} = t^{1}_{A}$, when the former map is restricted to the sub-complex generated by those generators whose images under $\Pi_{\Gamma}$ fall within the action filtration window $(-A\epsilon^{2}, \delta)$. \par
	The remaining task is to compare the map $t^{1}_{A}$ with $r^{1}_{A}$. For that purpose, the underlying geometric idea is to relate the moduli space $\tilde{\mathcal{P}}^{A}(x_{A}, x)$ to the moduli space $\mathcal{P}^{A}(x_{A}, x)$ of climbing strips. As mentioned before, the difference between a map $\tilde{w}$ obtained from a quilted map $(u, v)$ and an actual climbing strip $w$ is that they have different boundary conditions. However, as we shall see, their moduli spaces are naturally cobordant to each other. \par

\begin{lemma}\label{cobordism of moduli spaces}
	The moduli space $\tilde{\mathcal{P}}^{A}(x_{A}, x)$ is orientedly cobordant to the moduli space $\mathcal{P}^{A}(x_{A}, x)$ of climbing strips. Moreover, the same holds for compactified moduli spaces $\bar{\tilde{\mathcal{P}}}^{A}(x_{A}, x)$ and $\bar{\mathcal{P}}^{A}(x_{A}, x)$, namely they are also cobordant.
\end{lemma}
\begin{proof}
	For each $s \in \mathbb{R}$, the exact cylindrical Lagrangian submanifold $L_{j, A, s}$ as defined in \eqref{Lagrangian boundary conditions obtained from pasting a quilted map} is exact Lagrangian isotopic to $L_{j, s} = L_{j}^{\lambda(s)}$.
Let $L_{j, A, s, \sigma}$ be such an exact Lagrangian isotopy, parametrized by $\sigma \in [0, 1]$. It is possible to find such isotopies such that the two-dimensional family $L_{j, A, s, \sigma}$ parametrized by $(s, \sigma) \in \mathbb{R} \times [0, 1]$ is smooth. Furthermore, $L_{j, A, s, \sigma}$ is constant (i.e. independent of both $s$ and $\sigma$) for $s \gg 0$, where $L_{j, s} = L_{j}$. The reason is as follows. Since $L_{j}$ is assumed to be invariant under the Liouville flow in the Liouville cobordism $M_{0} \setminus int(U_{0})$, the completion of $i(\psi_{U}^{A} L'_{j})$ is exact Lagrangian isotopic to either $L_{j}$ or $L_{j}^{1-\epsilon}$. Recall that the completion of $i(\psi_{U}^{A} L'_{j})$ is precisely $\psi_{M}^{A} L_{j}$. This is isotopic to $L_{j}$ via the exact Lagrangian isotopy $\psi_{M}^{\rho_{A}(s)} L_{j}$, where $\rho_{A}: \mathbb{R} \to [1, A]$ is the previously used function which is $A$ for $s \ll 0$ and $1$ for $s \ge 0$. On the other hand, $L_{j}^{1-\epsilon}$ is isotopic to $L_{j}$ via the exact Lagrangian isotopy $L_{j}^{\lambda(s)}$. Thus we may compose these two isotopies to obtain an isotopy from $\psi_{M}^{A} L_{j}$ to $L_{j}^{1-\epsilon}$. For each $s$, we reparametrize these isotopies by $\sigma \in [0, 1]$ to obtain an isotopy from $L_{j, A, s}$ to $L_{j, s} = L_{j}^{\lambda(s)}$, namely,
\begin{equation}\label{two-dimensional family of exact cylindrical Lagrangian submanifolds}
L_{j, A, s, \sigma} =
\begin{cases}
\psi_{M}^{\rho_{A}((1 - 2\sigma)s)} L_{j}, & \text{ if } \sigma \in [0, \frac{1}{2}],\\
L_{j}^{\lambda_{1 - 2\sigma}(s)}, & \text{ if } \sigma \in [\frac{1}{2}, 1],
\end{cases}
\end{equation}
where $\lambda_{\sigma}$ is a non-decreasing homotopy between the function $\lambda: \mathbb{R} \to [1-\epsilon, 1]$ and the constant function $1$.
This two-dimensional family \eqref{two-dimensional family of exact cylindrical Lagrangian submanifolds} then satisfies all desired properties. \par
	We can then define a parametrized moduli space 
\begin{equation*}
\mathcal{P}_{+}^{A}(x_{A}, x)
\end{equation*}	
of pairs $(\sigma, w_{\sigma})$, where $\sigma \in [0, 1]$ and $w_{\sigma}: Z \to M$ is an inhomogeneos pseudoholomorphic map with respect to the Floer datum $(H_{A, s}, J_{A, (s, t)})$ and the Lagrangian boundary conditions $(L_{0, A, s, \sigma}, L_{1, A, s, \sigma})$. This moduli space provides the desired cobordism between the moduli spaces $\tilde{\mathcal{P}}^{A}(x_{A}, x)$ and $\mathcal{P}^{A}(x_{A}, x)$. \par
	To obtain a cobordism between the compactified moduli spaces, we just need to compactify the moduli space $\mathcal{P}_{+}^{A}(x_{A}, x)$ in an appropriate way. Such a compactification can be obtained in a similar way to those for $\tilde{\mathcal{P}}^{A}(x_{A}, x)$ and $\mathcal{P}^{A}(x_{A}, x)$, as described below. For each fixed $\sigma$, we have a moduli space $\mathcal{P}_{\sigma}^{A}(x_{A}, x)$ of maps $w_{\sigma}$, such that when $\sigma = 0$, $\mathcal{P}_{0}^{A}(x_{A}, x) = \tilde{\mathcal{P}}^{A}(x_{A}, x)$, and when $\sigma = 1$, $\mathcal{P}_{1}^{A}(x_{A}, x) = \mathcal{P}^{A}(x_{A}, x)$. Each such moduli space $\mathcal{P}_{\sigma}^{A}(x_{A}, x)$ is compactified in the same way to $\bar{\mathcal{P}}^{A}(x_{A}, x)$ and $\bar{\tilde{\mathcal{P}}}^{A}(x_{A}, x)$. Thus we may define the compactification $\bar{\mathcal{P}}_{+}^{A}(x_{A}, x)$ to the the union of these:
\begin{equation}
\bar{\mathcal{P}}_{+}^{A}(x_{A}, x) = \cup_{\sigma \in [0, 1]} \bar{\mathcal{P}}_{\sigma}^{A}(x_{A}, x).
\end{equation}
This compactified moduli space then provides the desired cobordism between $\bar{\tilde{\mathcal{P}}}^{A}(x_{A}, x)$ and 
$\bar{\mathcal{P}}^{A}(x_{A}, x)$. \par
\end{proof}

\begin{corollary}
	Under the assumption of Lemma \ref{isomorphism of moduli spaces of strips}, the cochain maps $\Pi_{\Gamma}$ and $r^{1}$ are chain homotopic.
\end{corollary}
\begin{proof}
	Note that there is a chain homotopy equivalence
\begin{equation*}
k_{A}: CW^{*}_{-}(i(\psi_{U}^{A} L'_{0}), i(\psi_{U}^{A} L'_{1}); H_{A}) \to CW^{*}_{-}(L_{0}^{1-\epsilon}, L_{1}^{1-\epsilon}; H_{A})
\end{equation*}
defined by counting rigid elements in a moduli space of inhomogeneous pseudoholomorphic maps $u: Z \to M$ satisfying the equation
\begin{equation*}
\partial_{s} u + J_{A, t}(\partial_{t} u - X_{H_{A}}(u)) = 0,
\end{equation*}
and the Lagrangian boundary conditions
\begin{equation*}
u(s, j) \in L_{j, A, -\infty, \sigma = \rho(s)}, j = 0, 1,
\end{equation*}
where $L_{j, A, -\infty, \sigma = \rho(s)}$ is obtained from the family $L_{j, A, s, \sigma}$ by first specializing $s = -\infty$, and then substituting $\sigma$ by $\rho(s)$. $L_{j, A, -\infty, \sigma = \rho(s)}$ can be regarded as an exact Lagrangian isotopy parametrized by $s \in \mathbb{R}$ such that for $s \ll 0$, $L_{j, A, -\infty, \sigma = \rho(s)} = L_{j}^{1-\epsilon}$, and for $s \gg 0$, $L_{j, A, -\infty, \sigma = \rho(s)} = i(\psi_{U}^{A} L'_{j})$. \par
	We claim that $\tilde{r}^{1}_{A}$ and $k_{A} \circ t^{1}_{A}$ are chain homotopic. By counting rigid elements in the moduli space $\bar{\mathcal{P}}_{+}^{A}(x_{A}, x)$, we define a map
\begin{equation*}
T_{A}: CW^{*}(L_{0}, L_{1}; H_{M}) \to CW^{*-1}_{-}(L_{0}^{1-\epsilon}, L_{1}^{1-\epsilon}; H_{A})
\end{equation*}
of degree $-1$. By a standard gluing argument, this is a chain homotopy between $\tilde{r}^{1}_{A}$ and $k_{A} \circ t^{1}_{A}$. \par
	Composing $T_{A}$ with the map
\begin{equation*}
\tau_{A}: CW^{*}_{-}(L_{0}^{1-\epsilon}, L_{1}^{1-\epsilon}; H_{A}) \to CW^{*}_{(-A\epsilon^{2}, \delta)}(L'_{0}, L'_{1}; H_{U}),
\end{equation*}
we get a map
\begin{equation*}
S_{A}: CW^{*}(L_{0}, L_{1}; H_{M}) \to CW^{*-1}_{(-A\epsilon^{2}, \delta)}(L'_{0}, L'_{1}; H_{U}).
\end{equation*}
Following a similar homotopy commutativity argument as in Lemma \ref{linear restriction homomorphism is compatible with inclusion maps of sub-complexes}, we may take the direct limit of the directed system of maps $S_{A}$ to get
\begin{equation*}
S = \lim\limits_{A \to +\infty} S_{A}.
\end{equation*}
All the above maps have unique continuous extensions to the completed wrapped Floer cochain spaces, when the energy of inhomogeneous pseudoholomorphic maps is taken into account in the counting definition. \par
	Combining Lemma \ref{isomorphism of moduli spaces of strips} and Lemma \ref{cobordism of moduli spaces}, we conclude that $\Pi_{\Gamma}$ is chain homotopic to $r^{1}_{A}$ when the former map is restricted to given sub-complex generated by those generators whose images under $\Pi_{\Gamma}$ fall within the action filtration window $(-A\epsilon^{2}, \delta)$. Such a chain homotopic is given by the map $S_{A}$. Taking the direct limit over $A$, we conclude that $\Pi_{\Gamma}$ is chain homotopic to $r^{1}$, where the chain homotopy is given by the map $S$. \par
\end{proof}

	Since $\Pi_{\Gamma}$ is chain homotopic to $\Theta_{\Gamma}^{1}$, Theorem \ref{Viterbo functor as a correspondence functor} is therefore complete. \par

\begin{remark}
	A point worth noting is that the Viterbo restriction functor is better defined as a colimit of continuation functors with respect to linear Hamiltonians (or the cascade definition as in \cite{Abouzaid-Seidel}), if one wants to visualize the picture of the Hamiltonian dynamics more directly. We took the current approach simply because of the quadratic Hamiltonians are more convenient for the purpose of constructing functors from Lagrangian correspondences, so that these functors can be compared in the same setup.
\end{remark}

\subsection{Further questions}
	It is therefore natural to ask whether the functors $\Theta_{\Gamma}$ and $r$ as a whole are homotopic to each other, not just limited to their linear terms. While the expectation is yes, an efficient way of proving this is yet to be discovered. At least, there is a very naive case where such coincidence can be easily verified. For example, consider the case where $L_{0}$ is a closed exact Lagrangian submanifold that is contained in $U_{0}$. Then the action of $\Theta_{\mathcal{L}}$ is the identity. This can be easily proved using the maximum principle, which implies that any pseudoholomorphic disk in $M_{0}$ with boundary on $L$ and its Hamiltonian perturbations must be contained in $U_{0}$. In the same way, one sees that the Viterbo restriction functor is also the identity functor on such Lagrangian submanifolds, which implies that the functor $\Theta_{\Gamma}$ agrees with the Viterbo restriction functor on such objects as well. To solve this problem in general, the main difficulty is to find a workable geometric construction of the functor $\Theta_{\Gamma}$, as the definition of the cochain map $\Pi_{\Gamma}$ does not seem to have a straightforward generalization to an $A_{\infty}$-functor. Finding a suitable model of the moduli spaces of quilted surfaces based on which the functor $\Theta_{\Gamma}$ can be constructed directly is the key step of solving this problem. \par
	
	There are of courses many other Lagrangian submanifolds which do not satisfy the geometric conditions we have just discussed. First, there are non-compact exact cylindrical Lagrangian submanifolds of $M$ that does not have very nice restriction to $U_{0}$. A typical example is the cotangent fiber of an annulus restricted to the disjoint union of three cotangent fibers of a deformed sub-annulus, as illustrated in Example 4.2 of \cite{Abouzaid-Seidel}. Second, there are closed exact Lagrangian submanifolds of $M$ which are not entirely contained in $U_{0}$. In such cases, the usual construction of the Viterbo restriction functor does not yield an $A_{\infty}$-functor in general.
However, by analyzing the failure of it being an $A_{\infty}$-functor, we expect that there is an extension of the Viterbo restriction functor to such Lagrangia submanifolds.
Spectacularly, such an extension is related to deformation theory of the wrapped Fukaya category of $U$, and we phrase it as the following conjecture. \par

\begin{conjecture}\label{conjecture on extension of the Viterbo restriction functor}
	Suppose $U_{0} \subset M_{0}$ is a Liouville sub-domain. Let the wrapped Fukaya category of $M$ and that of $U$ consist of a suitable countable collection of Lagrangian submanifolds. Then there is a canonical deformation of the completed wrapped Fukaya category $\mathcal{W}(U)$ of $U$, denoted by $\mathcal{W}(U; B)$, where $B$ is a collection of bounding cochains for objects in $\mathcal{W}(U)$, such that there is a natural $A_{\infty}$-functor
\begin{equation*}
r_{B}: \mathcal{W}(M) \to \mathcal{W}(U; B),
\end{equation*}
which agrees with the Viterbo restriction functor on the full sub-category $\mathcal{B}(M)$.
\end{conjecture}

	The study of such an extension also brings up the question when the extended Viterbo restriction functor can be identified with the functor $\Theta_{\Gamma}$ and when not. That would require more thorough understanding of the bounding cochains in both pictures. Related topics will be discussed in the upcoming work \cite{Gao2}. \par

\bibliography{Lag_functors_wrapped}
\bibliographystyle{alpha}

\end{document}